\documentclass[oneside,english, reqno]{amsart}
\usepackage[T1]{fontenc}
\usepackage[latin9]{inputenc}
\usepackage{prettyref}
\usepackage{mathrsfs}
\usepackage{amsbsy}
\usepackage{amstext}
\usepackage{enumitem}
\usepackage{amsthm}
\usepackage{mathtools}
\usepackage{amssymb}
\usepackage{esint}
\usepackage{graphicx}
\usepackage{xcolor}
\usepackage[noadjust]{cite}
\usepackage{cases}
\usepackage{comment}
\usepackage{tikz}
\usepackage{url}
\usepackage[colorlinks,allcolors=blue]{hyperref}
\usepackage[normalem]{ulem} 

\makeatletter
\numberwithin{equation}{section}
\numberwithin{figure}{section}
\def\th@plain{%
  \thm@notefont{}
  \itshape 
}

\makeatletter
\def\@secnumfont{\bfseries} 
\makeatother

\theoremstyle{plain}
\newtheorem{thm}{\protect\theoremname}[section]
\newtheoremstyle{boldremark}{10pt}{10pt}{}{}{\bfseries}{.}{10pt}{{\thmname{#1}\thmnumber{ #2}\thmnote{ (#3)}}}
\theoremstyle{boldremark}

\newtheorem{rem}[thm]{\protect\remarkname}
\newtheorem{ex}[thm]{\protect\examplename}

\theoremstyle{plain}
\newtheorem{prop}[thm]{\protect\propositionname}
\theoremstyle{plain}
\newtheorem{lem}[thm]{\protect\lemmaname}
\ifx\proof\undefined\
  \newenvironment{proof}[1][\proofname]{\par
    \normalfont\topsep6\p@\@plus6\p@\relax
    \trivlist
    \itemindent\parindent
    \item[\hskip\labelsep
          \scshape
      #1]\ignorespaces
  }{%
    \endtrivlist\@endpefalse
  }
  \providecommand{\proofname}{Proof}
\fi
\theoremstyle{plain}
\newtheorem{cor}[thm]{\protect\corollaryname}
\numberwithin{thm}{section}

\usepackage{fullpage}
\usepackage{hyperref}

\usepackage{bbm}

\newcommand{\cA}{\mathcal{A}}

\newcommand{\cR}{\mathcal{R}}

\newcommand{\R}{\mathbb{R}}

\newcommand{\C}{\mathbb{C}}

\newcommand{\e}{\epsilon}
\renewcommand{\d}{\delta}

\newcommand{\indicator}[1]{\mathbbm{1}_{#1}}

\newcommand{\QFl}{Q_{\rm{Fl}}}
\newcommand{\KFl}{\mathbf{K}_{\rm{Fl}}}
\newcommand{\uFl}{\mathbf{u}_{\rm{Fl}}}
\newcommand{\UFld}{\mathbf{U}_{\rm{Fl},\d}}
\newcommand{\UFlde}{\mathbf{U}_{\rm{Fl},\d,\e}}
\newcommand{\cof}{\mathrm{cof}\,}

\DeclareMathOperator{\skw}{skw}

\DeclareMathOperator{\esssup}{ess\,sup}

\newcommand{\sym}{{\mathrm{sym}}}

\newrefformat{prob}{Problem \ref{#1}}
\newrefformat{prop}{Proposition \ref{#1}}
\newrefformat{cor}{Corollary \ref{#1}}
\newrefformat{sec}{Section \ref{#1}}
\newrefformat{subsec}{Section \ref{#1}}
\newrefformat{fig}{Figure \ref{#1}}
\newrefformat{def}{Definition \ref{#1}}
\newrefformat{ex}{Example \ref{#1}}
\newrefformat{rem}{Remark \ref{#1}}
\newrefformat{conj}{Conjecture \ref{#1}}

\makeatother

\usepackage{babel}
\providecommand{\corollaryname}{Corollary}
\providecommand{\lemmaname}{Lemma}
\providecommand{\propositionname}{Proposition}
\providecommand{\remarkname}{Remark}
\providecommand{\theoremname}{Theorem}
\providecommand{\examplename}{Example}

\begin{document}

\title{Variational derivation of the Flamant solution for a nonlinear elastic wedge}
\author{Dominik Engl, Paul Plucinsky, and Ian Tobasco}
\thanks{E-mail: dominik.engl@ku.de, plucinsk@usc.edu, itobasco@umich.edu}
\date{\today}
\begin{abstract}
Concentrated forces acting at the tip of a two-dimensional wedge give rise to the classical Flamant solution to linear elasticity, whose displacement and strain are singular at the tip of the wedge. Starting from nonlinear elasticity, we prove that the Flamant solution gives the leading order response of a slightly truncated wedge to small boundary displacements or loads. 
This asymptotic result holds for general hyperelastic energies with super-quadratic growth at infinity; it also holds in the borderline case of quadratic growth at infinity, so long as the tip of the wedge is subjected to small enough displacements or loads. A main point of the proof is to restore compactness to low-energy sequences. We do so by applying a logarithmic change of variables sufficiently far from the tip. To justify this change of variables, we prove a geometric rigidity inequality in $L^p$ for truncated wedge domains with a constant that is uniform in the truncation length. This follows from the bi-Lipschitz invariance of the constant in the $L^p$ Friesecke--James--M\"uller inequality. Using this change of variables, we derive an asymptotic variational principle characterizing the Flamant solution in the singular limit of an ideal wedge. 
\tableofcontents{}
\end{abstract}
\subjclass[2020]{74G65 (primary), 74G70, 74B20, 74B05, 49J45, 70G75}
\keywords{Flamant solution; nonlinear elasticity; linearization; asymptotic analysis; energy minimization}

\maketitle

\section{Introduction}

Linear elasticity admits a number of explicit solutions that exhibit singular behavior when loads are concentrated on small regions  or confront geometrically sharp features like corners or cracks. Here we study one such example, called the  \textit{Flamant solution} \cite{flamant1892repartition} (see \cite{Bar23} for a modern presentation). Derived more than a century ago, the Flamant solution  describes the effect of applying a point force to the tip of an infinite two-dimensional wedge. The classical setup assumes isotropy of the elastic medium, but the analog of the Flamant solution for anisotropic linear elasticity is also known \cite{ting1996anisotropic,TS2004}. 
In these solutions, the displacement field $\uFl$ grows logarithmically with the radial distance $r$ from the tip, and the linear strain $\mathbf{e}(\uFl)=\sym\,\nabla\uFl$ and stress $\boldsymbol{\sigma}_{\rm Fl}$ scale as $1/r$. 
For instance, in the isotropic Flamant solution, the stress satisfies
\begin{align}\label{eq:Flamant_Intro}
\boldsymbol{\sigma}_{\text{Fl}}(\mathbf{x}) = \frac{A\cos\theta + B\sin\theta}{r} \mathbf{e}_r \otimes \mathbf{e}_r
\end{align}
where $(r,\theta)$ are polar coordinates centered at the tip and $\{\mathbf{e}_r,\mathbf{e}_\theta\}$ are their unit vector fields. The coefficients $A,B\in\mathbb{R}$ are determined by the force balance condition
\begin{equation}\label{eq:force-balance-defn}
\int_{\alpha}^{\beta}\big(A\cos\theta + B\sin\theta\big)\mathbf{e}_{r}\,d\theta=-\mathbf{f}.
\end{equation}
Equation \eqref{eq:force-balance-defn} balances the force  $\mathbf{f}$ applied to the tip of the wedge with the traction $\boldsymbol{\sigma}_{\text{Fl}}\mathbf{e}_r$ integrated along a circular arc. Figure \ref{fig:IntroFig1}(a) shows the setup we have in mind. 
Note the angular boundaries $\theta = \alpha, \beta$ of the wedge are traction free, since  $\boldsymbol{\sigma}_{\text{Fl}}\mathbf{e}_\theta=\mathbf{0}$.

As a curious point---one that has been the subject of robust discussion in the mechanics community \cite{blog,lazar2006note,vasiliev2021flamant}---the linear elastic energy of the Flamant solution is infinite. Indeed, its energy density is $\frac{1}{2} \left\langle \boldsymbol{\sigma}_{\rm Fl},\mathbf{e}(\uFl)\right\rangle$, and integrating this across the wedge results in a logarithmic divergence due to the $1/r$ singularity in \eqref{eq:Flamant_Intro}. There seems to be a contradiction: the energy  is infinite even though the equations of linear elasticity are satisfied. One is led to ask if there is an asymptotic sense in which the Flamant solution can be derived, say in which the wedge is truncated along a sequence of radii tending to $0$ and $\infty$. In this paper, we prove that the answer is ``yes'' and in fact that the Flamant solution is correct on the basis of nonlinear elasticity, provided one is willing to assume that the material of the wedge is sufficiently stiff at large strains. 

The basic premise of this paper, which is discussed for instance in \cite[Chapter 11.2.1]{Bar23}, is that the singular nature of the Flamant solution reflects its idealizations---the perfectly sharp tip and corresponding point load. 
In physical applications, these features are just a convenient proxy for loads on a finite domain having to pass through some small but finite tip, as depicted in Figure \ref{fig:IntroFig1}(b). Therefore, one should not view the singular nature of the Flamant solution as a deficiency, but instead as a consequence of passing to the limit along a sequence of solutions to elastic boundary value problems in which the scale of the truncated tip is sent to zero along with the magnitude of the loads. Nonlinearity plays an essential role: indeed, in the linear setting, there is no other option than the Flamant solution in the setup just described. With a nonlinear model, one can entertain the possibility that the Flamant solution is wrong---for instance, the material may be so soft that the tip can deform essentially independently of the rest of the wedge. Our analysis identifies a broad and physically motivated set of conditions under which such curiosities can be ruled out and the Flamant solution can be derived. 

\begin{figure}
\centering
\includegraphics[width = .8\textwidth]{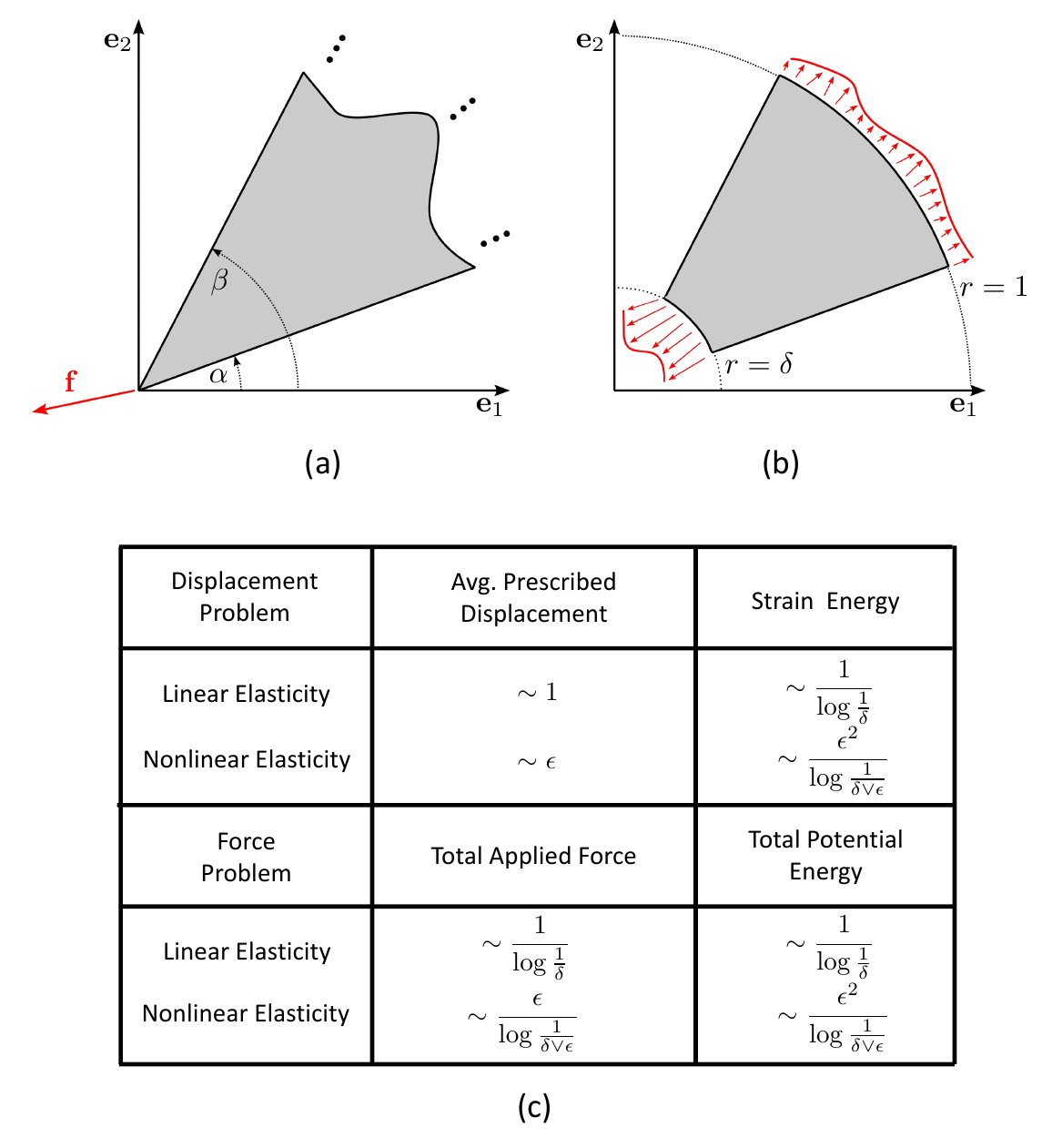}
\caption{Basic setup of the problem. (a) The Flamant solution describes an infinite linear elastic wedge with a point force at its tip. (b) We consider a truncated nonlinear elastic wedge with prescribed forces or displacements at the truncated boundaries $r=\delta,1$. (c) Scaling of the boundary data versus energy for the truncated wedge. The parameters $\delta$ and $\epsilon$ control the truncation length and size of the boundary displacements or forces.}
\label{fig:IntroFig1}
\end{figure}
In a bit more detail, we consider a family of nonlinear elastic and truncated wedge domains subject to small boundary displacements or loads, and prove that deformations with energy approximately equal to the minimum are given by the Flamant solution sufficiently far from the tip of the wedge. This result holds as the truncation length and boundary data go to zero, in terms of a strong $L^2$ asymptotic expansion of the deformation gradient. Similar results have been achieved in the literature on variational linearization in elasticity. This literature goes back to \cite{DNP02} where displacement boundary conditions were used, and has seen much development including a careful study of traction boundary conditions \cite{MPT19a,MPT19b,MaM21}, as well as results for multi-well problems  \cite{AlLaPaWo26,S08,DF25}, prestrained or incompatible elasticity \cite{KuM25,PT09,PT11}, combined homogenization and linearization \cite{MN11,NeR25,GN11}, incompressible materials \cite{MP21,MaP22,MP20,JeS21}, live pressure loads \cite{MR2023}, surface tension \cite{KrM25,CFLoPaM2025}, as well as plasticity and fracture \cite{MS13,F17,ADF23,D14,FSSt25}. The main novelty of this paper stems from our interest in singular solutions of linear elasticity that cannot be  justified on the entire reference domain. Instead, depending on the setup---the size $\delta$ of the truncated tip versus the magnitude $\epsilon$ of the boundary data---there can exist a ``core region'' near the tip where an additional assumption is needed to prevent a nonlinear response. Correspondingly, the natural energy scalings in the problem have logarithmic corrections that account for the core region, see Figure \ref{fig:IntroFig1}(c). The situation reminds of other problems in the calculus of variations where minimizers have point singularities. Examples include vortices in the Ginzburg--Landau theory of superconductivity \cite{bethuel1994ginzburg,sandier2007vortices}, dislocations embedded in an elastic medium \cite{ScZ12,MSZ14,CGO15,GLP10,CGM23}, and elastic cavitation \cite{SST06,Hen09,BrF26}. 
Mathematically, our work is perhaps closest to the work on dislocations, the key difference being that here the singularity is caused by boundary data. 

This paper shows that the Flamant solution gives the asymptotic response of a slightly truncated wedge for all nonlinear hyperelastic energy densities with super-quadratic growth at infinity, regardless of how the truncation length $\delta$ and size of the boundary data $\epsilon$ go to zero. Additionally, we prove the same result for energy densities that grow at least quadratically at infinity, so as long as $\log \frac{\delta \vee \epsilon}{\epsilon}  \ll \log \frac{1}{\epsilon}$. We conjecture that this extra condition is not only sufficient but also necessary, meaning that we do not expect the Flamant solution to apply for models with quadratic growth in other asymptotic regimes (see Remark \ref{rem:open-question-p=2}; see also Remark \ref{rem:sub-quad-growth} for sub-quadratic growth). The mathematical need for these assumptions regarding the growth of the elastic energy comes from the failure of the Sobolev embedding of $W^{1,2}$ into $L^\infty$ in two dimensions. Physically, this means that the tip of the wedge can deform significantly at a finite linear elastic energy cost. 

To achieve these results, we combine the familiar approach to variational linearization from the literature, which is based on geometric rigidity inequalities and Taylor expansion, with a particular ``de-singularizing'' change of variables adapted to the wedge. Using this change of variables, we derive a set of asymptotic variational problems that are solved by the Flamant solution. In other words, we precisely monitor the divergence of the energy in the Flamant solution along with its linear elastic competitors to show that, in the limit, the Flamant solution is optimal. An important technical point is the use of a uniform Friesecke--James--M\"uller rigidity inequality \cite{FJM02}, i.e., one with a constant that is independent of the truncation length of the wedge. While the optimal constants in these inequalities do generally depend on the shape of the domain, they remain bounded under bi-Lipschitz transformations of the domain. This was shown for the $L^2$ version of the Friesecke--James--M\"uller inequality by Lewicka \cite{Lew23}, and for the $L^p$ version by Neukamm and Richter as a case of a more general mixed growth-type rigidity inequality for Jones domains \cite{NeR25}. For the sake of the reader, we include a direct proof of the bi-Lipschitz invariance of the constant in the $L^p$ Friesecke--James--M\"uller rigidity inequality for Lipschitz domains, following  \cite{FJM02,CoS06,Lew23}.

\begin{figure}
\centering
\hspace*{-.5cm} 
\includegraphics[width = 0.8\textwidth]{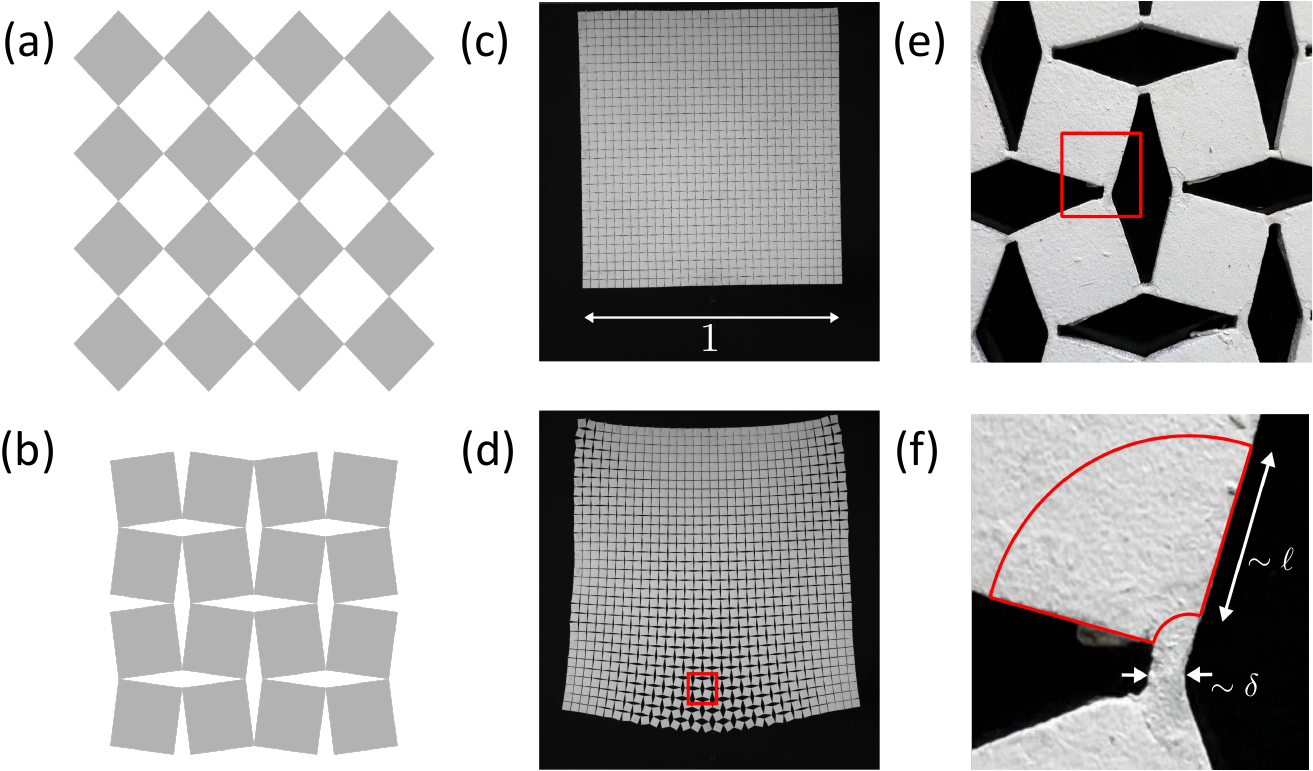}
\caption{Homogenization of kirigami metamaterials. (a) The rotating squares pattern in its checkerboard configuration; (b) a mechanism deformation counter-rotates the panels and causes the pattern to dilate. (c) Fabricated sample of the rotating squares pattern made by laser cutting a rubber sheet, with panels of size $\sim \ell$ and hinges of size $\sim \delta$. (d-e) Soft modes arising under typical loading conditions are given locally by mechanisms. (f) Zooming into one of the panels reveals a truncated wedge domain. Experimental images are courtesy of Paolo Celli.}
\label{fig:Kirigami}
\end{figure}

We end this part of the introduction by discussing our initial motivation for investigating the Flamant solution, namely, for its use in homogenizing kirigami metamaterials.
 Figure \ref{fig:Kirigami} shows a canonical example of kirigami called ``rotating squares''  \cite{GrE00}. Early mathematical work  \cite{berlyand1992asymptotics,berlyand1995effective}  treated the ``checkerboard'' variant of the pattern (Figure \ref{fig:Kirigami}a) as a composite in linear elasticity, and characterized the effective elasticity tensor that emerges when the number of cells goes to infinity. 
 Regarding finite deformations, in \cite{dull2024variational} the first author also modeled the pattern as a composite and studied a limit in which the only viable deformations are ``mechanisms'': deformations where the panels counter-rotate leading to an overall homogeneous shape change (Figure \ref{fig:Kirigami}b). However, if the holes are treated as such, the problem changes dramatically since  mechanisms are not the only available deformations. Instead, one sees a slowly modulated, locally mechanistic response termed a ``soft mode''. Soft modes are ubiquitous in mechanical metamaterials, including origami and kirigami patterns as well as other designs \cite{coulais2018characteristic,deng2020characterization,niu2025geometric,nassar2017curvature}. For the rotating squares pattern  (Figure \ref{fig:Kirigami}c-d), soft modes homogenize to conformal maps \cite{czajkowski2022conformal,zheng2022continuum,li2025nonlinear,li2026effective}. 
The last two authors of this paper have also classified the soft modes of broad families of kirigami and origami designs, and have built continuum models to capture their mechanical response \cite{zheng2022continuum,zheng2023modelling,xu2024derivation,xu2025modeling}.
 
That said, the problem of rigorously passing to the limit in the sequence of many-body elasticity problems describing kirigami remains open. We are presently working on a $\Gamma$-limit approach to this problem for which the results of this paper will play a crucial role. For now, we simply point out that the domain indicated in red in Figure \ref{fig:Kirigami}f is a truncated wedge. 

\subsection{Setup}\label{ssec:Setup}
This paper studies two types of variational problems regarding a truncated nonlinear elastic wedge:
the ``displacement problem'' given in Section \ref{ssec:dispProblem} where the elastic energy is minimized subject to Dirichlet boundary conditions for the deformation of the truncated tip
($r=\delta$) and end of the wedge ($r=1$); and the ``force
problem'' given in Section \ref{ssec:forceSec} where the total potential energy 
is minimized for a family of forces active at $r=\delta$ and $r=1$. After defining these problems, we present our main result in Section \ref{ssec:main}. 

\subsubsection{Elastic energy}
Fix a pair of angles $\alpha,\beta\in\mathbb{R}$ with $\beta - \alpha \in (0,2\pi)$. Given $\delta\in(0,1)$, 
let 
\[
\Omega_{\delta}=\left\{ \mathbf{x}\in\mathbb{R}^{2}:r\in(\delta,1),\theta\in(\alpha,\beta)\right\} 
\]
be the reference domain for the truncated wedge, and define the nonlinear elastic energy 
\begin{equation}\label{eq:nonlinear-energy-defn}
E_{\delta}(\mathbf{y})=\int_{\Omega_{\delta}}W(\nabla\mathbf{y})\,d\mathbf{x}
\end{equation}
for deformations $\mathbf{y}:\Omega_\delta \to \mathbb{R}^2$. The elastic energy density $W: \R^{2\times 2}\to [0,\infty]$ is assumed to be Borel measurable and to have the following properties: 
\begin{itemize}
	\item[(W1)] $W$ is frame-indifferent, i.e., $W(\mathbf{F})=W(\mathbf{R}\mathbf{F})$ for all $\mathbf{F}\in \R^{2\times 2}$ and $\mathbf{R}\in SO(2)$;
	\item[(W2)] $W$ vanishes precisely at $SO(2)$ and obeys $W(\mathbf{F})\sim d^2(\mathbf{F},SO(2))$ nearby $SO(2)$;
    \item[(W3)] $W$ is given by its second-order Taylor expansion nearby $\mathbf{I}$ up to higher order terms; 
    \item[(W4)] $W$ is bounded on compact subsets of $\{\mathbf{F}\in\mathbb{R}^2:\det\mathbf{F}>0\}$;
    \item[(W5)] $W$ grows at least quadratically at infinity.
\end{itemize}
By assumption (W3), we mean that
\[
W(\mathbf{F})=Q(\mathbf{F}-\mathbf{I})+o(|\mathbf{F}-\mathbf{I}|^2)\quad\text{as }\mathbf{F}\to\mathbf{I}
\]
where $Q(\cdot)=\frac{1}{2}\langle\cdot,\mathbb{C}\cdot\rangle$ and $\mathbb{C}=D^2W(\mathbf{I})\in \mathbb{R}^{2\times2\times2\times2}$ is the elasticity tensor of the model. The second part of (W2) is equivalent to the statement that $Q$ vanishes precisely on the set of two-by-two skew-symmetric matrices. 
Assumptions (W2) and (W5) imply the existence of constants $C>0$ and $p\geq 2$ such that
	\begin{align}\label{eq:Wgrowth}
		W(\mathbf{F})\geq C \left(d^2(\mathbf{F}, SO(2)) + d^p(\mathbf{F}, SO(2))\right)\quad\forall\,\mathbf{F}\in\mathbb{R}^{2\times2}.
    \end{align}
As usual, $SO(2)$ is the set of two-by-two rotation matrices and $d(\mathbf{F},SO(2))=\min_{\mathbf{R}\in SO(2)}|\mathbf{F}-\mathbf{R}|$ is the distance to this set in the Frobenius norm.

Assumptions (W1)-(W4) are standard in the variational approach to elasticity. Assumption (W5) is special to our analysis of the Flamant solution, especially in the case of super-quadratic growth ($p>2$). As we shall see, \eqref{eq:Wgrowth} with $p>2$ leads to general asymptotic statements on the validity of the Flamant solution. In contrast, energy densities with quadratic growth ($p=2$) turn out to be a borderline case where our results are limited to a particular asymptotic regime (the Flamant solution may not always apply if $p=2$; see Remark \ref{rem:open-question-p=2} for an open problem in this direction). Finally, we note that our setup allows for the physically realistic condition that $W(\mathbf{F})=\infty$ if $\det \mathbf{F}\leq 0$; consequently, the deformations constructed in this paper satisfy $\det\nabla\mathbf{y}>0$.

\subsubsection{Displacement problem}\label{ssec:dispProblem}

Next, we define the nonlinear displacement and force problems. In the displacement problem, we let
\begin{equation}
\mathcal{E}_{p,\delta,\epsilon}^{\text{disp}}  = \min_{\mathbf{y}\in \mathcal{A}_{p,\delta,\epsilon}
}\,E_{\delta}(\mathbf{y})\label{eq:nonlinear-displacement-problem}
\end{equation}
for the admissible set
\begin{equation}
\begin{aligned}\label{eq:admissibleDispalcements}
\mathcal{A}_{p,\delta,\epsilon}  =\Big\{  \mathbf{y} \in W^{1,p}(\Omega_{\delta};\mathbb{R}^2) & \colon  \mathbf{y}(\mathbf{x})- \mathbf{x} =\epsilon\mathbf{u}_{\delta,\epsilon}^{\pm}(\theta)\ \text{at }r=\delta,1 \Big\}
\end{aligned}
\end{equation}
where $(-, \delta)$ and  $(+, 1)$ are the corresponding pairs. Recall $p\geq 2$ gives the growth of the elastic energy density $W$ in \eqref{eq:Wgrowth}. 
The boundary displacements  $\{\mathbf{u}_{\delta,\epsilon}^{\pm}\}$ are assumed to belong to $W^{1,\infty}((\alpha,\beta);\mathbb{R}^{2})$ and to satisfy the following conditions as $\delta,\epsilon\to 0$: 
\begin{align}
	&\fint_{\alpha}^{\beta} \mathbf{u}_{\delta, \epsilon}^{\pm} \, d\theta \rightarrow \mathbf{u}_{0}^{\pm}\quad \text{for some } \mathbf{u}_{0}^{\pm}\in\mathbb{R}^{2},  \label{eq:meanDispConvergeNL} \\
	 & \sqrt{ \log \frac{1}{\delta \vee \epsilon}}   \| \mathbf{u}_{\delta, \epsilon}^{+}\|_{\dot{W}^{1,\infty}((\alpha,\beta))} \rightarrow 0, \label{eq:mean-freePlus}\\
	 & \sqrt{ \log \frac{1}{\delta \vee \epsilon}}    \|\mathbf{u}_{\delta, \epsilon}^{-} \|_{\dot{W}^{1,\infty}((\alpha, \beta))}\rightarrow 0 \quad \quad \quad  \text{ if }   \delta  \gtrsim  
    \frac{\epsilon}{\sqrt{\log \frac{1}{\delta \vee \epsilon}}},\label{eq:mean-freeMinus} \\ 
    &\limsup \,\frac{\epsilon}{\delta} \|\mathbf{e}_{r} \cdot \frac{d}{d\theta} \mathbf{u}_{\delta, \epsilon}^{-}\|_{L^{\infty}((\alpha,\beta))} < \infty \quad \text{ if } \delta  \ll  
    \frac{\epsilon}{\sqrt{\log \frac{1}{\epsilon}}}, \label{eq:limsup2} \\
    & \limsup \,\frac{\epsilon}{\delta} \|\mathbf{e}_{\theta} \cdot \frac{d}{d\theta} \mathbf{u}_{\delta, \epsilon}^{-}\|_{L^{\infty}((\alpha,\beta))} < 1  \quad  \text{\;\;  if } \delta  \ll  
    \frac{\epsilon}{\sqrt{\log \frac{1}{\epsilon}}}. \label{eq:limsup1}
\end{align}
The notation for trace norms used in this paper is standard, and is recalled in Section \ref{ssec:notation}. 

Let us briefly comment on our assumptions. First, note that nothing in our definition of the displacement problem guarantees that minimizers exist. Instead, this paper concerns the analysis of \emph{almost minimizers}, which by definition are sequences $\mathbf{y}_{\delta,\epsilon}\in \mathcal{A}_{p,\delta,\epsilon}$ such that $E_\delta(\mathbf{y}_{\delta,\epsilon}) = \mathcal{E}_{p,\delta,\epsilon}^{\text{disp}} + o(\mathcal{E}_{p,\delta,\epsilon}^{\text{disp}})$ as $\delta,\epsilon \to 0$. Such sequences always exist, given our setup. A similar remark applies to the force problem introduced below.

Next, we discuss the consequences of \eqref{eq:meanDispConvergeNL}-\eqref{eq:limsup1}. Along the way to justifying the Flamant solution for almost minimizers, we must decide which aspects of the boundary displacements at $r=\delta,1$ remain relevant in the limit. Our assumptions ensure that only the average boundary displacements $\fint_\alpha^\beta\mathbf{u}^\pm_{\delta,\epsilon}\,d\theta$ matter. In particular, we will be able to ``cut off'' the mean-free parts of the data in boundary layers near the tip and end of the wedge. The appearance of the logarithms in these assumptions has to do with the scaling law of the minimum energy $\mathcal{E}_{p,\delta,\epsilon}^{\text{disp}}$, which turns out to be $\sim \epsilon^2/\log(\delta\vee\epsilon)^{-1}$. In order to help preserve the positive determinant constraint in the cutoff procedure, we use $L^\infty$-based norms  rather than the potentially more natural $L^p$-based ones, i.e., we use $W^{1,\infty}$ instead of $W^{1-1/p,p}$. This sort of simplification appears throughout the literature on variational linearization of displacement-controlled problems going back to \cite{DNP02}. 

Finally, we note that while our assumptions rule out rotations of the whole wedge, they do permit rotations of the tip of the wedge along with other localized deformations if the truncation length $\delta$ is sufficiently small. 
To see this, consider boundary displacements of the form
\begin{align*}
\mathbf{u}_{\delta, \epsilon}^{-}(\theta) = \frac{\delta}{\epsilon} (\mathbf{R}^-_{\delta, \epsilon} - \mathbf{I}) \mathbf{e}_r(\theta)+\frac{1}{\epsilon}\mathbf{c}_{\delta,\epsilon}^-,\quad \mathbf{u}_{\delta, \epsilon}^{+}(\theta) = \frac{1}{\epsilon} (\mathbf{R}^+_{\delta, \epsilon} - \mathbf{I}) \mathbf{e}_r(\theta)+\frac{1}{\epsilon}\mathbf{c}_{\delta,\epsilon}^+
\end{align*}  
for $\{\mathbf{R}^\pm_{\delta, \epsilon}\} \subset SO(2)$ and $\{\mathbf{c}_{\delta,\epsilon}^\pm\}\subset\mathbb{R}^2$. 
A straightforward calculation  shows that
\begin{align*}
\frac{\epsilon}{\delta}\big\| \mathbf{u}_{\delta, \epsilon}^{-} \big\|_{\dot{W}^{1,\infty}((\alpha,\beta))} \sim \big| \mathbf{R}^-_{\delta, \epsilon} - \mathbf{I}\big| ,\quad \epsilon\big\| \mathbf{u}_{\delta, \epsilon}^{+} \big\|_{\dot{W}^{1,\infty}((\alpha,\beta))} \sim \big| \mathbf{R}^+_{\delta, \epsilon} - \mathbf{I}\big|.
\end{align*}
From \eqref{eq:mean-freePlus} we see that $\mathbf{R}^+_{\delta,\epsilon}\to\mathbf{I}$ as $\delta,\epsilon \to 0$, so the end of the wedge cannot rotate. Likewise, \eqref{eq:mean-freeMinus} shows that $\mathbf{R}^-_{\delta,\epsilon}\to\mathbf{I}$ if $\delta \gtrsim \epsilon/\sqrt{\log(\delta\vee\epsilon)^{-1}}$.
On the other hand, if $\delta \ll \epsilon/\sqrt{\log\epsilon^{-1}}$ then \eqref{eq:limsup2} and \eqref{eq:limsup1} apply and allow for $\mathbf{R}^-_{\delta,\epsilon}\not\to\mathbf{I}$, i.e., the tip can rotate in this case. The precise form of these last two assumptions has to do with our cutoff procedure. We leave their optimization to future work.

\subsubsection{Force problem}\label{ssec:forceSec} 
In the force problem, we let 
\begin{align}\label{eq:nonlinear-force-problem}
   \mathcal{E}_{p,\delta, \epsilon}^{\text{force}} = \min_{\mathbf{y} \in W^{1,p}(\Omega_{\delta}; \mathbb{R}^2)}\, E_{\delta}(\mathbf{y}) - \frac{\epsilon}{\log \frac{1}{\delta \vee \epsilon}}\Big( V_{\delta,\epsilon}(\mathbf{y}) - V_{\delta,\epsilon}^{\star} \Big)
\end{align} 
where the second term is defined as follows. Introduce 
\begin{align}\label{eq:force_density_nonlinear}
	\mathbf{f}_{\delta,\epsilon}(\mathbf{x})=-\frac{1}{\delta}\mathbf{f}_{\delta,\epsilon}^{-}(\theta)\indicator{\{r=\delta\}}+\mathbf{f}_{\delta,\epsilon}^{+}(\theta)\indicator{\{r=1\}},\quad \mathbf{x}\in\partial\Omega_\delta
\end{align}
for $\{\mathbf{f}_{\delta,\epsilon}^{\pm}\}$ belonging to the dual space $W^{1/2, 2}((\alpha, \beta); \mathbb{R}^2)'$, and define the work integral
\begin{align}\label{eq:defineWorkNL}
    V_{\delta,\epsilon}(\mathbf{y}) = \int_{\partial \Omega_{\delta}} \mathbf{f}_{\delta,\epsilon} \cdot  \mathbf{y} \,ds = \int_\alpha^\beta -\mathbf{f}_{\delta,\epsilon}^{-}\cdot\mathbf{y}|_{r=\delta}+\mathbf{f}_{\delta,\epsilon}^{+}\cdot\mathbf{y}|_{r=1}\,d\theta
\end{align}
where $\mathbf{y}|_{r=a}=\mathbf{y}|_{\partial\Omega_\delta}(a\mathbf{e}_r(\cdot))$. These choices model boundary forces acting on the tip and end of the wedge. Since $\mathbf{y}\in W^{1,p}(\Omega_\delta;\mathbb{R}^2)\subset W^{1,2}(\Omega_\delta;\mathbb{R}^2)$ in \eqref{eq:nonlinear-force-problem}, it suffices to consider $V_{\delta,\epsilon}$ as an element of the dual trace space $W^{1/2,2}(\partial\Omega_\delta;\mathbb{R}^2)'$. We assume that
\begin{align}
&\int_{\alpha}^{\beta}\mathbf{f}_{\delta,\epsilon}^{+}\,d\theta = \int_{\alpha}^{\beta}\mathbf{f}_{\delta,\epsilon}^{-}\,d\theta \to\mathbf{f}_{0}, \label{eq:forceBalanceNL}  \\
& \frac{1}{\sqrt{\log\frac{1}{\d\vee\e}}} \big\| \mathbf{f}^{\pm}_{\delta, \epsilon}  \big\|_{\dot{W}^{\frac{1}{2},2}((\alpha, \beta))'}  \rightarrow 0 \label{eq:meanFreeForcesNL} 
\end{align}
as $\delta,\epsilon \to 0$, for some $\mathbf{f}_0 \in \mathbb{R}^2$ (see Section \ref{ssec:notation} for the notation). The constant $V_{\delta, \epsilon}^{\star}$ is defined in \eqref{eq:VdeltaEpsilonStar}. 

Our assumptions on the forces are similar to the ones for the boundary displacements in the previous setup. In particular, they ensure that only the total force $\int_{\alpha}^{\beta}\mathbf{f}^{\pm}_{\delta,\epsilon}\,d\theta$ applied to the tip and end of the wedge matter as $\delta,\epsilon\to 0$.  The use of the $(\dot{W}^{1/2,2})'$-seminorm in \eqref{eq:meanFreeForcesNL} is consistent with the linearization result we prove. Other choices such as the $(\dot{W}^{1-1/p,p})'$-seminorm are possible, and lead to different assumptions on the parameters.  
The equality in \eqref{eq:forceBalanceNL} implies that the boundary forces are in equilibrium, i.e., $\int_{\partial\Omega_\delta}\mathbf{f_{\delta,\epsilon}}\,ds=\mathbf{0}$. This is necessary for the minimum in \eqref{eq:nonlinear-force-problem} to not be $-\infty$. Unlike linear elasticity, however, the moment balance condition $\int_{\partial \Omega_{\delta}} \mathbf{f}_{\delta, \epsilon} \cdot \mathbf{x}^{\perp} \,ds = 0$ need not be imposed. Instead, moment balance is automatically achieved in the deformed configuration as the wedge can rotate. Finally, as for the displacement problem, we do not assume that minimizers exist and focus instead on \emph{almost minimizers} $\mathbf{y}_{\delta,\epsilon}\in W^{1,p}(\Omega_\delta;\mathbb{R}^2)$ which by definition satisfy $E_{\delta}(\mathbf{y}_{\delta,\epsilon}) - \frac{\epsilon}{\log (\delta\vee\epsilon)^{-1}}\Big( V_{\delta,\epsilon}(\mathbf{y}_{\delta,\epsilon}) - V_{\delta,\epsilon}^{\star} \Big) =\mathcal{E}_{p,\delta, \epsilon}^{\text{force}} + o(\mathcal{E}_{p,\delta, \epsilon}^{\text{force}})$ as $\delta,\epsilon \to 0$.

We come now to the definition of $V_{\delta, \epsilon}^{\star}$. In words, this is the maximum work of the forces attained by rotating the wedge:
\begin{align}\label{eq:VdeltaEpsilonStar}
V_{\delta, \epsilon}^{\star} := \max_{\mathbf{R} \in SO(2)} V_{\delta,\epsilon}(\mathbf{R}\mathbf{x})= \sqrt{ \left( \int_{\partial \Omega_{\delta}} \mathbf{f}_{\delta, \epsilon} \cdot \mathbf{x} \, ds \right)^2 + \left( \int_{\partial \Omega_{\delta}} \mathbf{f}_{\delta, \epsilon} \cdot \mathbf{x}^{\perp} \, ds \right)^2}.
\end{align}
As we will show, including $V_{\delta, \epsilon}^{\star}$ in \eqref{eq:nonlinear-force-problem} ensures that the  minimum scales $\sim \epsilon^2/\log((\delta\vee\epsilon)^{-1})$, just like in the displacement problem. The standard choice for the total potential energy is $E_{\delta}-\frac{\epsilon}{\log((\delta\vee\epsilon)^{-1})}\int_{\partial\Omega_\delta}\mathbf{f}_{\delta,\epsilon}\cdot(\mathbf{y}-\mathbf{x})\,ds$, and of course adding a constant to the energy does not change its minimizers. As a side note, we point out that the maximization in \eqref{eq:VdeltaEpsilonStar} either has a unique maximizing rotation or is maximized by all rotation matrices in $SO(2)$. The first case is generic and happens when $V_{\delta,\epsilon}^{\star}>0$. A similar observation was made more generally in \cite{MaM21}. Here, the claim is easily checked by parameterizing $SO(2)$.

To help present our results on the force problem, we divide it into three cases as follows: we assume that 
\begin{subequations}
\begin{numcases}
{\lim_{\delta,\epsilon\to 0}\,\frac{V_{\d,\e}^{\star}}{\e}=}
	0 \label{super-degenerate}\\
	V_0 >0\label{degenerate}\\
	\infty \label{non-degenerate}
\end{numcases}
\end{subequations}
after possibly passing to a subsequence. 
In cases \eqref{degenerate} and \eqref{non-degenerate},  $V^{\star}_{\d,\e}>0$ for small enough $\delta,\epsilon$, and we define $\varphi_{\delta,\epsilon} \in (-\pi, \pi]$ such that
\begin{equation}\label{eq:phi_deltaeps_defn}
\cos\varphi_{\delta,\epsilon} =  \frac{1}{V_{\delta ,\epsilon}^{\star}}  \int_{\partial \Omega_{\delta}} \mathbf{f}_{\delta,\epsilon} \cdot \mathbf{x} \, ds, \quad     \sin \varphi_{\delta,\epsilon} = \frac{1}{V_{\delta ,\epsilon}^{\star}}  \int_{\partial \Omega_{\delta}} \mathbf{f}_{\delta, \epsilon} \cdot \mathbf{x}^{\perp} \, ds
\end{equation}
using \eqref{eq:VdeltaEpsilonStar}. Finally, we assume in cases \eqref{degenerate} and \eqref{non-degenerate} that
\begin{equation}\label{eq:getVarphiDef}
	 \cos \varphi_{\delta, \epsilon} \rightarrow \cos \varphi, \quad \sin \varphi_{\delta, \epsilon} \rightarrow \sin \varphi.
\end{equation}

\begin{ex}
The following construction realizes each of the cases  in \eqref{super-degenerate}-\eqref{non-degenerate}. Let $\beta = - \alpha = \gamma > 0$ and define $\mathbf{f}_{\delta, \epsilon}^{\pm} \colon (-\gamma, \gamma) \rightarrow \mathbb{R}^2$ by 
\begin{align*}
    \mathbf{f}_{\delta, \epsilon}^{+}(\theta) = \mathbf{f}_{\delta, \epsilon}^{-}(\theta) = \left(f(\theta)  +  \frac{V_0\epsilon^{p'}\cos \varphi}{(1-\delta)2\gamma}  \right) \mathbf{e}_{r}(\theta)  + \left(g(\theta)  +  \frac{V_0\epsilon^{p'}\sin \varphi}{(1-\delta)2\gamma}  \right) \mathbf{e}_{\theta}(\theta)
\end{align*}
for $V_0\in[0,\infty)$, $p'\in(0,\infty)$, $\varphi \in (-\pi, \pi]$, and where $f,g:(-\gamma,\gamma)\to\mathbb{R}$ satisfy 

\[
\int_{-\gamma}^{\gamma} f\,d\theta = \int_{-\gamma}^{\gamma} g\,d\theta = 0.
\]
This construction satisfies  \eqref{eq:forceBalanceNL} and \eqref{eq:meanFreeForcesNL} with 
\[
\mathbf{f}_0 =   \int_{-\gamma}^{\gamma} f\mathbf{e}_r+g\mathbf{e}_\theta\,d\theta.
\]
In fact, any $\mathbf{f}_0\in\mathbb{R}^2$ is achievable by such a construction. 
Furthermore,
\begin{align*}
    \int_{\partial \Omega_{\delta}} \mathbf{f}_{\delta, \epsilon} \cdot \mathbf{x} \, ds &=  \int_{-\gamma}^{\gamma}\big( - \delta \mathbf{f}_{\delta, \epsilon}^{-}+ \mathbf{f}_{\delta, \epsilon}^+ \big)  \cdot \mathbf{e}_r \, d\theta \\
    &= (1 - \delta) \int_{-\gamma}^{\gamma} f  +   \frac{V_0\epsilon^{p'}\cos \varphi}{(1-\delta)2\gamma}   \, d \theta =   V_0  \epsilon^{p'}\cos \varphi
\end{align*}
and similarly
\[
\int_{\partial \Omega_{\delta}} \mathbf{f}_{\delta, \epsilon} \cdot \mathbf{x}^{\perp} \, ds = V_0\epsilon^{p'} \sin \varphi.
\]
Applying \eqref{eq:VdeltaEpsilonStar}, we get that
\begin{align*}
V_{\delta, \epsilon}^{\star} = V_0 \epsilon^{p'}.
\end{align*}
Cases \eqref{super-degenerate}-\eqref{non-degenerate} come from taking $p' > 1$, $p' =1$, and $p' \in (0,1)$ respectively, if $V_0>0$. If $V_0 = 0$ then $V_{\delta,\epsilon}(\mathbf{R}\mathbf{x})=V_{\delta, \epsilon}^{\star}=0$ for any rotation $\mathbf{R}\in SO(2)$, which is an extreme version of \eqref{super-degenerate}.

\end{ex}

\subsection{Main result}\label{ssec:main}
We come to our main result on the asymptotic validity of the Flamant solution for a truncated nonlinear elastic wedge. 
We require some notation. First, define the \emph{Flamant quadratic forms} 
\begin{align}\label{Flamant_energies}
\QFl(\mathbf{u})  = \frac{1}{2}\mathbf{u} \cdot \KFl \mathbf{u} \quad \text{and} \quad \QFl^\ast(\mathbf{f})  = \frac{1}{2}\mathbf{f} \cdot \KFl^{-1}\mathbf{f}
\end{align}
for $\mathbf{u},\mathbf{f}\in \R^2$, where
\begin{align*}
\KFl = \int_{\alpha}^{\beta} \frac{\mathbf{e}_r \otimes \mathbf{e}_r}{(\mathbb{C}^{-1})_{rrrr}}\,d\theta.
\end{align*}
Recall $\mathbb{C} = D^2 W(\mathbf{I})$ is the elasticity tensor and  $(\mathbb{C}^{-1})_{rrrr} = \langle \mathbf{e}_r \otimes \mathbf{e}_r,\mathbb{C}^{-1}\mathbf{e}_r \otimes \mathbf{e}_r\rangle$, with similar expressions for other components.  
Evidently, $\QFl$ and $\QFl^\ast$ are positive definite and symmetric. They are also Legendre transforms, i.e., $\QFl^*(\mathbf{f})=\max_{\mathbf{u}\in\mathbb{R}^2}\mathbf{f}\cdot\mathbf{u}-\QFl(\mathbf{u})$. 
Next, define the \emph{Flamant displacement field}
\begin{align}\label{Flamant_displacement}
\mathbf{u}_{\text{Fl}}(\mathbf{x}) = \mathbf{u}_0(r) + \mathbf{u}_1(\theta) 
\end{align}
for
\begin{equation*}
\begin{aligned}
&\mathbf{u}_0(r) = (\log r ) \mathbf{a}, \\
&\mathbf{u}_1(\theta) = \left[ \int_{\alpha}^{\theta} \Big(\frac{ 2(\mathbb{C}^{-1})_{r\theta rr} }{(\mathbb{C}^{-1})_{rrrr}} \mathbf{e}_{r} \otimes \mathbf{e}_{r}  - \mathbf{e}_r \otimes \mathbf{e}_{\theta} +\frac{(\mathbb{C}^{-1})_{\theta\theta rr}}{(\mathbb{C}^{-1})_{rrrr}}  \mathbf{e}_{\theta}\otimes \mathbf{e}_r\Big) \, d\tilde{\theta} \right] \mathbf{a} 
\end{aligned}
\end{equation*}
where $\mathbf{a} \in \mathbb{R}^2$ is a parameter that will be prescribed. 
Finally, let $\mathbf{R}(\psi)=\cos\psi\mathbf{I}+\sin\psi\mathbf{J}$ for $\psi\in\mathbb{R}$ and $\mathbf{J}=\mathbf{e}_2 \otimes \mathbf{e}_1 - \mathbf{e}_1 \otimes \mathbf{e}_2$, and let $\Omega_{\delta\vee\epsilon}=\Omega_{\delta}\cap\{r>\delta\vee\epsilon\}$ and $\Omega_{\delta,\delta\vee\epsilon}=\Omega_{\delta}\cap\{r\in(\delta,\delta\vee\epsilon)\}$.
\begin{thm}
\label{thm:main-result}
Consider any limit $\delta, \epsilon \rightarrow 0$ if the elastic energy density $W$ has super-quadratic growth ($p >2$ in \eqref{eq:Wgrowth}), or assume furthermore that 
\begin{align}\label{eq:pEqual2Limit}
\log \frac{\epsilon}{\delta} \ll \log \frac{1}{\epsilon} \quad\text{if }\delta \ll \epsilon
\end{align} 
in the case of quadratic growth ($p =2$ in \eqref{eq:Wgrowth}).  
The nonlinear displacement and force
problems \prettyref{eq:nonlinear-displacement-problem} and \prettyref{eq:nonlinear-force-problem}  
obey Flamant-type asymptotics in such a limit. 
Precisely, the following statements hold:
\begin{enumerate}[leftmargin=*]
\item[(I)] The minimum energies in the displacement and force problems obey 
\[
\frac{\log \frac{1}{\delta \vee \epsilon}}{\epsilon^2} \mathcal{E}_{p,\delta, \epsilon}^{\emph{disp}} \to \QFl\left(\mathbf{u}_{0}^{+}-\mathbf{u}_{0}^{-}\right)
\]
and
\begin{align}\label{eq:forceSolve}
\frac{\log \frac{1}{\delta \vee \epsilon}}{\epsilon^2} \mathcal{E}_{p,\delta, \epsilon}^{\emph{force}} \rightarrow 
 \begin{cases}
\displaystyle \min_{\mathbf{R} \in SO(2)} - \QFl^\ast( \mathbf{R}^T \mathbf{f}_0)  &  \text{ if  \eqref{super-degenerate}}\\  
\displaystyle \min_{\mathbf{R} \in SO(2)} - \QFl^\ast( \mathbf{R}^T \mathbf{f}_0) + V_0 (1- \mathbf{R} \mathbf{e}_1 \cdot \mathbf{R}(\varphi) \mathbf{e}_1) & \text{ if \eqref{degenerate}}\\ 
\displaystyle -\QFl^\ast( \mathbf{R}(\varphi)^T \mathbf{f}_0 )  & \text{ if  \eqref{non-degenerate}}  
\end{cases}
\end{align}
where  $\varphi$ is given by \eqref{eq:getVarphiDef}. 

\item[(II)] Let $\{\mathbf{y}_{\delta,\epsilon}\}$ be almost minimizing in the displacement or force problems in that 
\begin{equation*}
\begin{aligned}
 & \mathbf{y}_{\delta, \epsilon}  \in \mathcal{A}_{p, \delta, \epsilon} && \text{ and }  && E_\delta(\mathbf{y}_{\delta, \epsilon}) = \mathcal{E}_{p,\delta,\epsilon}^{\text{disp}} + o\Big( \frac{\epsilon^2}{\log \frac{1}{\delta \vee \epsilon}}\Big),\quad \text{or} \\
  &\mathbf{y}_{\delta, \epsilon} \in W^{1,p}(\Omega_{\delta}; \mathbb{R}^2) && \text{ and }  && E_\delta(\mathbf{y}_{\delta, \epsilon}) - \frac{\epsilon}{\log \frac{1}{\delta \vee \epsilon}} \Big( V_{\delta, \epsilon}(\mathbf{y}_{\delta, \epsilon}) - V_{\delta, \epsilon}^{\star} \Big) = \mathcal{E}_{p,\delta,\epsilon}^{\text{force}} + o\Big( \frac{\epsilon^2}{\log \frac{1}{\delta \vee \epsilon}}\Big).
\end{aligned}
\end{equation*}
After passing to a subsequence (if needed, see Remark \ref{rem:subsequence-thm}), there are rotations  $\{\mathbf{R}_{\delta, \epsilon}\} \subset SO(2)$
such that
\begin{align}\label{eq:keyEstimateNonLinear}
\Big\| \nabla\mathbf{y}_{\delta,\epsilon}-\mathbf{R}_{\delta,\epsilon}\Big(\mathbf{I}+\frac{\epsilon}{\log\frac{1}{\delta\vee\epsilon}} \nabla \mathbf{u}_{\emph{Fl}} \Big)\Big\| _{L^{2}(\Omega_{\delta\vee\epsilon})}  +  \Big\| \nabla \mathbf{y}_{\delta, \epsilon} - \mathbf{R}_{\delta, \epsilon} \Big\|_{L^2(\Omega_{\delta, \delta \vee \epsilon})} \ll\frac{\epsilon}{\sqrt{\log\frac{1}{\delta\vee\epsilon}}}
\end{align}
where $\mathbf{u}_{\emph{Fl}}$ is the Flamant displacement from \eqref{Flamant_displacement} defined using 
\begin{align*}
\mathbf{a}  = \begin{cases}
\mathbf{u}_0^+ - \mathbf{u}_0^{-} & \text{ in the displacement problem } \\
\mathbf{K}_{\emph{Fl}}^{-1} \mathbf{R}_\emph{a}^T \mathbf{f}_0  & \text{ in  case \eqref{super-degenerate} of the force problem } \\
\mathbf{K}_{\emph{Fl}}^{-1} \mathbf{R}_\emph{b}^T \mathbf{f}_0 &  \text{ in case \eqref{degenerate} of the force problem } \\ 
\mathbf{K}_{\emph{Fl}}^{-1} \mathbf{R}^T(\varphi) \mathbf{f}_0  & \text{ in case \eqref{non-degenerate} of the force problem} 
\end{cases}
\end{align*}
for rotations $\mathbf{R}_a,\mathbf{R}_b\in SO(2)$ that solve their respective minimization problems in \eqref{eq:forceSolve} and are further identified below.
 
\item[(III)] In the displacement problem, \eqref{eq:keyEstimateNonLinear} holds for $\mathbf{R}_{\delta,\epsilon}=\mathbf{I}$. In the force problem, \eqref{eq:keyEstimateNonLinear} holds for 
\begin{align*}
\mathbf{R}_{\delta, \epsilon} \in \Big\{ \mathbf{R} \in SO(2) \colon \int_{\Omega_{\delta}} \big| \nabla \mathbf{y}_{\delta, \epsilon}-\mathbf{R} \big|^2 \, d\mathbf{x} = \min_{\mathbf{Q} \in SO(2)} \int_{\Omega_{\delta}}\big| \nabla \mathbf{y}_{\delta, \epsilon} -\mathbf{Q}\big|^2 \, d\mathbf{x} \Big\}
\end{align*}
which consists of a single rotation for small enough $\delta,\epsilon$. This rotation satisfies 
\begin{align}\label{eq:rotConvergeMainTheorem}
\mathbf{R}_{\delta,\epsilon}\to\begin{cases}
 \mathbf{R}_{\emph{a}} & \text{ in  case \eqref{super-degenerate}} \\
 \mathbf{R}_{\emph{b}} & \text{ in case \eqref{degenerate}} \\
\mathbf{R}(\varphi) & \text{ in  case \eqref{non-degenerate}} 
\end{cases}
\end{align}
for the subsequence from (II). In cases \eqref{super-degenerate} and \eqref{degenerate}, this defines $\mathbf{R}_a$ and $\mathbf{R}_b$. In case \eqref{non-degenerate},
\begin{equation}\label{eq:expansion-of-rotation}
\mathbf{R}_{\delta,\epsilon}=\mathbf{R}(\varphi_{\delta,\epsilon}) + o\Big( \sqrt{\frac{\epsilon}{V^{\star}_{\delta,\epsilon}} }\Big) 
\end{equation}
for $\varphi_{\delta,\epsilon}$ given by \eqref{eq:phi_deltaeps_defn}. 
\end{enumerate}
\end{thm}

\begin{rem}\label{rem:subsequence-thm} Parts (II) and (III) of this result hold for the original sequence of almost minimizers in situations where the original sequence of rotations $\{\mathbf{R}_{\delta,\epsilon}\}$ converges. This is always the case 
for the displacement problem and for case \eqref{non-degenerate} of the force problem. However, in cases \eqref{super-degenerate} and \eqref{degenerate} of the force problem, the rotations $\{\mathbf{R}_{\delta,\epsilon}\}$ need not converge. Such non-convergence is related to the possible non-uniqueness of minimizers for the first two minimization problems in \eqref{eq:forceSolve}. See Appendix \ref{sec:degenerate-rotations} for more details.
\end{rem}  
\begin{rem}\label{rem:truncation}
As written, \eqref{eq:keyEstimateNonLinear} guarantees that almost minimizers are given by the Flamant solution above a ``core radius''  $r=\delta\vee\epsilon$, i.e., on $\Omega_{\delta\vee\epsilon}$. Below this core radius (on $\Omega_{\delta,\delta\vee\epsilon}$) the Flamant solution can fail to apply, depending on the parameters. Of course, we must have $\delta<\epsilon$ for this to happen; the most interesting case is where $\delta \lesssim \epsilon/\log\epsilon^{-1}$, since then the maximum strain of the ``bare'' Flamant solution in \eqref{Flamant_displacement} is $\sim \epsilon/(\delta\log\epsilon^{-1})\gtrsim 1$, suggesting that a truncation may indeed be required (as in \eqref{eq:keyEstimateNonLinear}). 
To sort this out, observe that alternate versions of \eqref{eq:keyEstimateNonLinear} hold where $\delta\vee\epsilon$ is replaced by a smaller core radius $r_{\delta,\epsilon}\in [\delta,\epsilon)$ such that $\log \epsilon/r_{\delta,\epsilon} \ll \log 1/\epsilon$. This follows  from \eqref{eq:keyEstimateNonLinear} and the fact that $|\nabla\mathbf{u}_{\text{Fl}}|\sim1/r$.  
The constraint enforces that $\|\frac{\epsilon}{\log\epsilon^{-1}}\nabla\mathbf{u}_{\text{Fl}}\|_{L^{2}(\Omega_{r_{\delta,\epsilon},\epsilon})}\ll\frac{\epsilon}{\sqrt{\log\epsilon^{-1}}}$ where $\Omega_{r_{\delta,\epsilon},\epsilon}=\Omega_\delta\cap\{r\in(r_{\delta,\epsilon},\epsilon)\}$. Evidently, a truncation is required if and only if $\log \epsilon/\delta \gtrsim \log 1/\epsilon$, since then we cannot take $r_{\delta,\epsilon}=\delta$ in the alternate to \eqref{eq:keyEstimateNonLinear}. The only exception is the trivial case where $\mathbf{u}_\text{Fl}=\mathbf{0}$. This remark holds for all models considered in this paper, i.e., for $W$ with super-quadratic or quadratic growth. 
\end{rem}

\begin{rem} \label{rem:open-question-p=2} Turning specifically to $W$ with quadratic growth ($p=2$ in \eqref{eq:Wgrowth}),  we see from the previous remark that the assumption $\log\epsilon/\delta \ll \log 1/\epsilon$ in  \eqref{eq:pEqual2Limit} allows us to conflate the bare, non-truncated Flamant solution with the truncated Flamant solution in \eqref{eq:keyEstimateNonLinear}. We conjecture that \eqref{eq:pEqual2Limit} is actually necessary for \eqref{eq:keyEstimateNonLinear} in the case of quadratic growth.
Suppose for the sake of argument that  $W(\mathbf{F})\sim d^2(\mathbf{F},SO(2))$. By using the bare Flamant solution \eqref{Flamant_displacement} one obtains a competitor for the displacement and force problems, which has significantly less energy than the truncated Flamant solution if (i) $\log\epsilon/\delta \gg \log 1/\epsilon$, and has comparable energy if (ii) $\log\epsilon/\delta \sim \log 1/\epsilon$. In both cases and in both problems, the bare Flamant solution achieves the minimum energy scaling law, which is $\sim \epsilon^2/\log\delta^{-1}$ by a minor modification of the arguments of this paper. Evidently, \eqref{eq:keyEstimateNonLinear} fails to capture almost minimizers in case (i), and there is reason to expect that it also fails in case (ii). We leave this open question to future research. 
\end{rem}

\begin{rem}\label{rem:sub-quad-growth}We do not consider elastic energy densities with sub-quadratic growth, i.e., ones satisfying $W(\mathbf{F})\gtrsim \min\{ d^2(\mathbf{F},SO(2)),d^p(\mathbf{F},SO(2))\}$ for some $p\in (1,2)$ in place of \eqref{eq:Wgrowth}.  It may be possible to justify the Flamant solution for such energies in certain limits $\delta,\epsilon \to 0$, by combining our approach with the mixed growth-type rigidity inequalities from \cite{CDM14} following \cite{ADDe2012}. 
However, if $\delta$ is too small then the Flamant solution does not describe the almost minimizers of sub-quadratic models. This is due to the implied softness of the tip, which can deform nonlinearly while the rest of the wedge is left alone. Such a scenario yields energy $\sim(\epsilon/\delta)^p\delta^2=\epsilon^p\delta^{2-p}$. In comparison, the Flamant solution has energy $\sim \epsilon^2$ corrected by a log. Evidently, the Flamant solution is not optimal for sub-quadratic models and small enough $\delta$. 
\end{rem}

Let us mention a few key ideas from the proof of Theorem \ref{thm:main-result}. 
The standard approach to deriving the Flamant solution uses self-similarity and force balance to deduce that the stress scales as $1/r$. An application of the equations of linear elasticity completes the derivation \cite{Bar23,ting1996anisotropic,TS2004}. (For an alternate approach, see  \cite{unger2002similarity, unger2025nonlinear}.)  Our derivation is based instead on nonlinear elasticity, and must confront the singular nature of the Flamant solution, which has infinite linear elastic energy in the full wedge $\Omega_0=\{\mathbf{x}:r\in(0,1),\theta\in(\alpha,\beta)\}$. A useful observation in this regard is that there is an asymptotic variational principle for the Flamant solution, which can be derived without imposing self-similarity by considering a sequence of truncated linear elastic wedges $\Omega_\delta = \{\mathbf{x}:r\in(\delta,1),\theta\in (\alpha,\beta)\}$ where $\delta \to 0$. To derive this principle, we make a ``de-singularizing'' change of variables in the standard variational formulation of linear elasticity before taking $\delta \to 0$. This change of variables is inspired by the logarithmic dependence of the Flamant displacement \eqref{Flamant_displacement}, and also fixes the reference domain to be independent of $\delta$.
Section \ref{sec:linear-analysis} describes the procedure in detail using displacements. Here, we demonstrate it using stresses, both for flavor and because this gives nicer formulas. At end of this introduction, we connect the present discussion on de-singularizing the Flamant solution to the techniques available from the variational literature for justifying linear elasticity. We also give an outline of the paper.

In linear elasticity, one looks for a stress $\boldsymbol{\sigma}$ and a displacement $\mathbf{u}$ related by $\boldsymbol{\sigma}=\mathbb{C}\mathbf{e}(\mathbf{u})$ where $\mathbb{C}$ is the elasticity tensor and $\mathbf{e}(\mathbf{u})=\sym\,\nabla\mathbf{u}$ is the linear strain. Given boundary conditions prescribing displacements and/or loads, the equations of equilibrium determine $\boldsymbol{\sigma}$ and $\mathbf{u}$. The problem can be equivalently formulated using only stresses or only displacements (see, e.g., \cite[Chapter III.3.5]{DL76}). Adopting the setting of the force problem \eqref{eq:nonlinear-force-problem}-\eqref{eq:meanFreeForcesNL} to fix the boundary data, and using stresses, the linear elastic stress of the truncated wedge $\Omega_\delta$ is found by solving a constrained minimization problem 
\begin{equation}\label{eq:dual-intro}
\min_{\boldsymbol{\sigma}(\mathbf{x})}\,\int_{\Omega_{\delta}}Q^{*}(\boldsymbol{\sigma})\,d\mathbf{x}
\end{equation}
with $Q^*(\cdot)=\frac{1}{2}\langle\cdot,\mathbb{C}^{-1}\cdot\rangle$. The optimization is over all $\boldsymbol{\sigma}:\Omega_\delta\to\text{Sym}_2$ that are divergence-free and whose boundary tractions $\boldsymbol{\sigma}\mathbf{n}$ are specified at $\partial\Omega_\delta$, where $\mathbf{n}$ is the outwards-pointing unit normal. Based on the setup of the force problem, we impose 
\begin{equation}\label{eq:dual-intro-2}
\boldsymbol{\sigma}\mathbf{e}_\theta = \mathbf{0}\quad\text{at }\theta=\alpha,\beta,\qquad \boldsymbol{\sigma}\mathbf{e}_r =  \frac{1}{\delta\log\frac{1}{\delta}}\mathbf{f}^-_{\delta}\quad\text{at }r=\delta,\qquad \boldsymbol{\sigma}\mathbf{e}_r =   \frac{1}{\log\frac{1}{\delta}}\mathbf{f}^+_{\delta}\quad\text{at }r=1
\end{equation}
for $\{\mathbf{f}^{\pm}_{\delta}\}\subset W^{1/2,2}((\alpha,\beta);\mathbb{R}^2)'$ that obey the analog of \eqref{eq:forceBalanceNL} and \eqref{eq:meanFreeForcesNL}. In particular, 
\begin{equation}\label{eq:limit-force-for-intro}
\int_\alpha^\beta\mathbf{f}^{+}_{\delta}\,d\theta=\int_\alpha^\beta\mathbf{f}^{-}_{\delta}\,d\theta \to \mathbf{f}_0\quad\text{and}\quad \frac{1}{\sqrt{\log\frac{1}{\delta}}}\|\mathbf{f}_\delta^\pm\|_{\dot{W}^{\frac{1}{2},2}((\alpha,\beta))'}\to 0
\end{equation}
as $\delta \to 0$, for some $\mathbf{f}_0\in\mathbb{R}^2$.
To ensure existence of a minimizer, we also impose the moment balance condition $\int_{\partial \Omega_{\delta}} \mathbf{f}_{\delta} \cdot \mathbf{x}^{\perp} \,ds = 0$ where $\mathbf{f}_{\delta}=-\frac{1}{\delta}\mathbf{f}_{\delta}^{-}\indicator{\{r=\delta\}}+\mathbf{f}_{\delta}^{+}\indicator{\{r=1\}}$. Given this, we ask for the leading order behavior of the minimizer $\boldsymbol{\sigma}_\delta$ as $\delta \to 0$. For the equivalent setup with displacements, see \eqref{eq:linear-force-problem}-\eqref{eq:linear-balance-laws}. Note $\mathcal{E}_\delta^\text{lin,force}$ is the negative of the minimum in \eqref{eq:dual-intro}.

Evidently, one cannot find the asymptotic behavior of $\boldsymbol{\sigma}_\delta$ simply by taking $\delta = 0$ in the above, since the boundary data in \eqref{eq:dual-intro-2} would not make sense (even after multiplying by $\log\delta^{-1}$).
Instead, we change variables using the following observation: since $Q^*$ is quadratic, $Q^*(\boldsymbol{\sigma})\,rdr=Q^*(r\boldsymbol{\sigma})\,d\log r$. Therefore, if we expect the grouping $r\boldsymbol{\sigma}$ to be important, it makes sense to rewrite the optimization using $\log r$. Introduce the logarithmic change of variables
\begin{equation}\label{eq:log-change-of-vars-intro-stress}
\boldsymbol{\Sigma}(\rho,\theta)=\log\left(\frac{1}{\delta}\right)|\mathbf{x}|\boldsymbol{\sigma}(\mathbf{x})\quad\text{for}\quad\rho=\frac{\log r}{\log\frac{1}{\delta}}
\end{equation}
and note for the integral in \eqref{eq:dual-intro} that 
\[
\int_{\Omega_{\delta}}Q^{*}(\boldsymbol{\sigma})\,d\mathbf{x} =\frac{1}{\log\frac{1}{\delta}}\int_{R}Q^{*}(\boldsymbol{\Sigma})\,d\rho d\theta 
\]
for $R=\{(\rho,\theta):\rho\in(-1,0),\theta\in(\alpha,\beta)\}$. The divergence-free constraint and boundary data in \eqref{eq:dual-intro-2} become 
\begin{equation}\label{eq:new-dual-intro-constraint-1}
\partial_{\rho}(\boldsymbol{\Sigma}\mathbf{e}_{\rho})+\log\left(\frac{1}{\delta}\right)\partial_{\theta}(\boldsymbol{\Sigma}\mathbf{e}_{\theta})=\mathbf{0}\quad \text{in }R
\end{equation}
and
\begin{equation}\label{eq:new-dual-intro-constraint-2}
\boldsymbol{\Sigma}\mathbf{e}_\theta = \mathbf{0}\quad\text{at }\theta=\alpha,\beta,\qquad \boldsymbol{\Sigma}\mathbf{e}_\rho = \mathbf{f}^-_{\delta}\quad\text{at }\rho=-1,\qquad \boldsymbol{\Sigma}\mathbf{e}_\rho =  \mathbf{f}^+_{\delta}\quad\text{at }\rho=0
\end{equation}
where $\mathbf{e}_\rho = \mathbf{e}_r$. (To check \eqref{eq:new-dual-intro-constraint-1}, dot it against a test vector field and integrate by parts.) Therefore, up to a logarithmic prefactor, \eqref{eq:dual-intro}-\eqref{eq:dual-intro-2} is equivalent to the minimization problem
\begin{equation}\label{eq:equiv-dual-problem}
\min_{\boldsymbol{\Sigma}(\mathbf{x})}\,\int_{R}Q^{*}(\boldsymbol{\Sigma})\,d\rho d\theta
\end{equation}
subject to \eqref{eq:new-dual-intro-constraint-1} and \eqref{eq:new-dual-intro-constraint-2}. 

We now send $\delta \to 0$ to deduce the limiting variational problem. 
We permit ourselves to be brief, as this is only the introduction. Given the convex nature of  \eqref{eq:new-dual-intro-constraint-1}-\eqref{eq:equiv-dual-problem}, it suffices to take $\delta \to 0$ in its constraints. Divide \eqref{eq:new-dual-intro-constraint-1} by $\log\delta^{-1}$ and pass to the limit to get that $\partial_\theta(\boldsymbol{\Sigma}\mathbf{e}_\theta)=\mathbf{0}$. By the first part of \eqref{eq:new-dual-intro-constraint-2}, $\boldsymbol{\Sigma}\mathbf{e}_\theta=\mathbf{0}$. Therefore, in the limit, 
\begin{equation}\label{eq:limit-dual-intro-eqn}
\boldsymbol{\Sigma}=\Sigma\mathbf{e}_\rho\otimes\mathbf{e}_\rho
\end{equation}
for some $\Sigma(\rho,\theta)$. Next, we integrate the second and third parts of \eqref{eq:new-dual-intro-constraint-2} and take $\delta \to 0$ to get that
\begin{equation}\label{eq:dual-limit-problem-2}
\int_\alpha^\beta \Sigma\mathbf{e}_\rho\,d\theta = \mathbf{f}_0
\end{equation}
by \eqref{eq:limit-force-for-intro}. The limit problem consists of solving \eqref{eq:equiv-dual-problem} subject to \eqref{eq:limit-dual-intro-eqn} and \eqref{eq:dual-limit-problem-2}. The fact is that the minimum of the original problem in \eqref{eq:new-dual-intro-constraint-1}-\eqref{eq:equiv-dual-problem} converges to the minimum of \eqref{eq:equiv-dual-problem}-\eqref{eq:dual-limit-problem-2}, and that the minimizer of the former converges to the minimizer of the latter. Note we have yet to explicitly use the second part of \eqref{eq:limit-force-for-intro}. This and other  questions about the justification of the procedure just described are addressed in Section \ref{sec:linear-analysis} (see in particular Theorem \ref{thm:main-result-linearized} and Remark \ref{rem:duality}). 

Finally, we solve the limit problem. By Jensen's inequality and averaging, the optimal $\Sigma$ is independent of $\rho$. Since $Q^*(\cdot)=\frac{1}{2}\langle\cdot,\mathbb{C}^{-1}\cdot\rangle$, it solves 
\begin{equation}\label{eq:dual-limit-problem}
\min_{\Sigma(\theta)}\,\int_\alpha^\beta\frac{1}{2}(\mathbb{C}^{-1})_{\rho\rho\rho\rho}\Sigma^2\, d\theta
\end{equation}
subject to \eqref{eq:dual-limit-problem-2}, where $\mathbb{C}_{\rho\rho\rho\rho}=\mathbb{C}_{rrrr}$. The minimizer is $\Sigma = ((\mathbb{C}^{-1})_{\rho\rho\rho\rho})^{-1}\mathbf{a}\cdot\mathbf{e}_\rho$ for $\mathbf{a}\in\mathbb{R}^2$ given by \eqref{eq:dual-limit-problem-2}, as can be shown using Lagrange multipliers. We have just calculated the limit of the optimizers $\{\boldsymbol{\Sigma}_\delta\}$ of \eqref{eq:new-dual-intro-constraint-1}-\eqref{eq:equiv-dual-problem}, which are nothing but the logarithmic versions of the scaled stresses 
$\{\log(\delta^{-1})r\boldsymbol{\sigma}_\delta\}$. Undoing the change of variables in \eqref{eq:log-change-of-vars-intro-stress}, we learn that  
\[
\boldsymbol{\sigma}_\delta=\frac{1}{\log\frac{1}{\delta}}\boldsymbol{\sigma}_{\text{Fl}}+\text{h.o.t.}
\]
where
\begin{align}\label{eq:aniso_Flamant_intro}
\boldsymbol{\sigma}_{\text{Fl}}(\mathbf{x})  = \frac{\Sigma(\theta) }{r} \mathbf{e}_r\otimes\mathbf{e}_r \quad \text{ with } \quad  \Sigma = \frac{\mathbf{a} \cdot \mathbf{e}_r}{(\mathbb{C}^{-1})_{rrrr}}.
\end{align}
For the precise notion of ``higher order terms'', see Theorem \ref{thm:main-result-linearized}. Note $\boldsymbol{\sigma}_{\text{Fl}} = \mathbb{C} \mathbf{e}(\mathbf{u}_{\text{Fl}})$ for the Flamant displacement  $\mathbf{u}_{\text{Fl}}$ in \eqref{Flamant_displacement}, which helps make the connection to Theorem \ref{thm:main-result}. Compared to the formula \eqref{eq:Flamant_Intro} for the stress of an isotropic linear elastic wedge, \eqref{eq:aniso_Flamant_intro} differs only by the normalization factor $(\mathbb{C}^{-1})_{rrrr}$. This factor reduces to a constant in the isotropic case, reproducing \eqref{eq:Flamant_Intro}. 

We just explained how the Flamant solution can be derived by considering a sequence of truncated linear elastic wedges. The proof of Theorem \ref{thm:main-result} combines this remark with the usual variational arguments for passing from nonlinear to linear elasticity. In brief, we apply a variant of the Friesecke--James--M\"uller inequality that we obtain for the truncated wedge in Section \ref{sec:rigidity-korn} to prove for competitors $\{\mathbf{y}_{\delta,\epsilon}\}$ in the nonlinear displacement and force problems that 
\begin{equation}\label{eq:rigidity-intro-inequality}
\int_{\Omega_{\delta\vee\epsilon}}|\nabla\mathbf{y}_{\delta,\epsilon} - \mathbf{R}_{\delta,\epsilon}|^2\,d\mathbf{x}\lesssim \frac{\epsilon^2}{\log\frac{1}{\delta\vee\epsilon}}
\end{equation}
for rotations $\{\mathbf{R}_{\delta,\epsilon}\}\subset SO(2)$. 
The important detail here, which is not entirely obvious, is that the constant in \eqref{eq:rigidity-intro-inequality} can be chosen independently of the truncation length $\delta$. To prove this, we apply a bi-Lipschitz invariant version of the Friesecke--James--M\"uller inequality in Section \ref{sec:rigidity-korn}; we also show the $L^p$ versions of these inequalities, as they are needed in other parts of the proof. 

Then, using \eqref{eq:rigidity-intro-inequality}, we define the displacement $\mathbf{u}_{\delta,\epsilon}(\mathbf{x})=\epsilon^{-1}\mathbf{R}_{\delta,\epsilon}^T(\mathbf{y}_{\delta,\epsilon}(\mathbf{x})-\mathbf{x})$ and show that its logarithmic counterpart
\[
\mathbf{U}_{\delta,\epsilon}(\rho,\theta)=\mathbf{u}_{\delta,\epsilon}(\mathbf{x})\quad \text{for}\quad \rho = \frac{\log r}{\log\frac{1}{\delta\vee\epsilon}}
\]
obeys 
\begin{equation}\label{eq:transformed-bound}
\int_R |D^{\delta\vee\epsilon}\mathbf{U}_{\delta,\epsilon}|^2\,d\rho d\theta \lesssim 1
\end{equation}
for the \emph{logarithmic gradient}
\[
D^{\delta\vee\epsilon}\mathbf{U}_{\delta,\epsilon} = \partial_\rho\mathbf{U}_{\delta,\epsilon}\otimes\mathbf{e}_\rho + \log\left(\frac{1}{\delta\vee\epsilon}\right)\partial_\theta \mathbf{U}_{\delta,\epsilon}\otimes\mathbf{e}_\theta.
\]
This is the displacement analog (and formal adjoint) of the divergence \eqref{eq:new-dual-intro-constraint-1} from the prior stress analysis, with $\delta\vee\epsilon$ in place of $\delta$. 
We analyze the implications of \eqref{eq:transformed-bound} in Section \ref{sec:linear-analysis}, where we pass to the limit in the linear displacement and force problems. Building off of these results, we pass to the limit in the nonlinear problems in Sections \ref{sec:displacement} and \ref{sec:force}. An added complication here is that \eqref{eq:rigidity-intro-inequality}-\eqref{eq:transformed-bound} only controls the displacement on $\Omega_{\delta\vee\epsilon}$, while the boundary data is given at $\partial\Omega_\delta$. When $\delta \ll \epsilon$, this presents issues that we resolve via our assumption \eqref{eq:Wgrowth} on the growth of the elastic energy density. At this stage, we will have shown part (I) of Theorem \ref{thm:main-result}. Parts (II) and (III) regarding almost minimizers are shown in Section \ref{sec:strong-convergence}.

\subsection{Notation}\label{ssec:notation} \textit{Vectors and tensors.} We use standard vector and tensor notation, e.g., we use the Euclidean and Frobenius inner products  $\mathbf{a}\cdot\mathbf{b} = \sum_{i}a_ib_i$ and $\langle\mathbf{A},\mathbf{B}\rangle=\sum_{ij}A_{ij}B_{ij}$. Their norms are both denoted $|\cdot|$.
The outer product of two vectors is given in components by $(\mathbf{a}\otimes\mathbf{b})_{ij}=a_ib_j$, and its symmetric part is $\mathbf{a} \odot \mathbf{b}= ( \mathbf{a} \otimes \mathbf{b} + \mathbf{b} \otimes \mathbf{a})/2$. $SO(n)$ is the set of $n$-by-$n$ rotation matrices. $\mathrm{Sym}_n$ and $\mathrm{Skw}_n$ denote the vector spaces of $n$-by-$n$ symmetric and skew-symmetric matrices. We parameterize $SO(2)$ by  $\mathbf{R}(\psi) = \cos \psi \mathbf{I} + \sin \psi \mathbf{J}$ for $\psi \in \mathbb{R}$ and $\mathbf{J} = \mathbf{e}_2 \otimes \mathbf{e}_1 - \mathbf{e}_1 \otimes \mathbf{e}_2$. We write $\mathbf{a}^\perp=\mathbf{J}\mathbf{a}$. 

\textit{Inequalities.} The notation $a\lesssim b$ means that $a \leq C b$ for some constant $C>0$, and $a\sim b$ means that $a\lesssim b$ and $b\lesssim a$. If $C$ depends on parameters we note this with a subscript, i.e., $a\lesssim_\theta b$ means that $a\leq C(\theta)b$. We track the dependence of such implicit constants on the growth exponent $p$ of the elastic energy density throughout the paper, and also on the wedge angle $\beta-\alpha$ in Section \ref{sec:rigidity-korn} and Appendix \ref{sec:appendix-traces}. On the other hand, we never allow implicit constants to depend on the wedge truncation length $\delta$ or the magnitude of the boundary data $\epsilon$. For positive $a,b$, $a\ll b$ and $a=o(b)$ indicate a limit where $a/b\to 0$ as $b\to 0$, while $a=O(b)$ means that $a\lesssim b$ for sufficiently small $b$. We abbreviate $a\vee b=\max\{a,b\}$. 

\textit{Wedge domains and polar coordinates.} We use the truncated wedge domain $\Omega_{\delta}=\{\mathbf{x}\in\mathbb{R}^2:r\in(\delta,1),\theta\in(\alpha,\beta)\}$ and the ``tip domain'' $\Omega_{s,t}=\Omega_\delta\cap\{r\in(s,t)\}$ where $r,\theta$ are polar coordinates and $\beta-\alpha\in(0,2\pi)$. Provided the meaning is clear, we allow ourselves to mix Cartesian and polar coordinates in a single line. For instance, $\nabla \mathbf{u}(\mathbf{x}) = \partial_r \mathbf{u}(r,\theta) \otimes \mathbf{e}_{r}(\theta) + \frac{1}{r}\partial_{\theta} \mathbf{u}(r,\theta) \otimes \mathbf{e}_{\theta}(\theta)$ where $\mathbf{e}_r,\mathbf{e}_\theta$ are the unit polar basis vectors.  

\textit{Analysis.} We write $\fint_U$ for an average with respect to an appropriate Lebesgue measure, e.g., $\fint_U \cdot \,d\mathbf{x}$ = $\frac{1}{|U|}\int_U \cdot \,d\mathbf{x}$ where $|U|$ is the area of $U\subset\mathbb{R}^2$. The indicator function of $U$ is $\indicator{U}$. 
Given a Sobolev function $\mathbf{u}\in W^{1,p}(\Omega_\delta;\mathbb{R}^2)$, its traces at $r=\delta,1$ are denoted by $\mathbf{u}|_{r=a}=\mathbf{u}|_{\partial\Omega_\delta}(a\mathbf{e}_r(\cdot))$ for $a=\delta,1$. These belong to the fractional Sobolev space $W^{1-1/p,p}((\alpha,\beta);\R^2)$, which consists \cite{Leo23} of all $\mathbf{v}\in L^p((\alpha,\beta);\mathbb{R}^2)$ such that 
\begin{equation}
\label{eq:DotSemiNorm}
\|\mathbf{v}\|_{\dot{W}^{1-\frac1p,p}((\alpha,\beta))} := \left( \int_{\alpha}^\beta \int_{\alpha}^\beta  \frac{|\mathbf{v}(\theta) - \mathbf{v}(\theta')|^p}{|\theta - \theta'|^{p}} \, d\theta d \theta'  \right)^{\frac1p}< \infty
\end{equation}
if $p<\infty$. If $p=\infty$, 
\begin{equation*}
\|\mathbf{v}\|_{\dot{W}^{1,\infty}((\alpha,\beta))} := \esssup_{\substack{\theta,\theta'\in(\alpha,\beta)\\\theta\neq\theta'}} \frac{|\mathbf{v}(\theta) - \mathbf{v}(\theta')|}{|\theta-\theta'|} = \|\frac{d}{d\theta}\mathbf{v}\|_{L^\infty(\alpha,\beta)}< \infty.
\end{equation*}
The dual space $W^{1-1/p,p}((\alpha,\beta);\mathbb{R}^2)'$ consists of all continuous linear functionals $\mathbf{f}:W^{1-1/p,p}((\alpha,\beta);\mathbb{R}^2)\to\mathbb{R}$. We refer to the action of $\mathbf{f}$ as $\int_\alpha^\beta \mathbf{f}\cdot d\theta$, even if $\mathbf{f}$ is not a function.
We denote 
\begin{align}\label{eq:dualDotSemiNorm}
\| \mathbf{f} \|_{\dot{W}^{1-\frac{1}{p},p}((\alpha,\beta))'} := \sup_{\substack{{\mathbf{v} \in W^{1-\frac{1}{p},p}((\alpha,\beta);\mathbb{R}^2)}\\\mathbf{v}\neq\mathbf{0},\ \fint_\alpha^\beta\mathbf{v}\,d\theta=\mathbf{0}}}\,  \frac{\int_\alpha^\beta \mathbf{f}\cdot \mathbf{v}\, d\theta }{\|\mathbf{v}\|_{\dot{W}^{1-\frac{1}{p},p}((\alpha,\beta))}}.
\end{align}
Note $\| \mathbf{f} -\mathbf{c} \|_{\dot{W}^{1-1/p,p}((\alpha,\beta))'} = \| \mathbf{f} \|_{\dot{W}^{1-1/p,p}((\alpha,\beta))'}$ for all $\mathbf{c}\in\mathbb{R}^2$. Though \eqref{eq:dualDotSemiNorm} allows for general $p\in[1,\infty]$, in this paper we only use it for $p=2$.
For further discussion of traces in polar versus Cartesian coordinates, see Appendix \ref{sec:appendix-traces}.

\section{Geometric rigidity for bi-Lipschitz equivalent domains}\label{sec:rigidity-korn}

The geometric rigidity inequality 
\begin{equation}\label{eq:rigidity-inequality-intro}
\min_{\mathbf{R}\in SO(n)}\,\|\nabla\mathbf{y} - \mathbf{R}\|_{L^p(\Omega)} \leq C \|d(\nabla \mathbf{y}, SO(n)\|_{L^p(\Omega)}
\end{equation}
was first proved by Friesecke, James, and M\"uller \cite{FJM02} in the quadratic case $p=2$ with a constant $C$ depending on the choice of bounded, Lipschitz domain $\Omega\subset\mathbb{R}^n$. 
Later, \eqref{eq:rigidity-inequality-intro} was extended to all $p\in(1,\infty)$ by Conti and Schweizer \cite{CoS06} by a modification of the original proof. Since these initial papers, many authors have contributed to the literature on rigidity inequalities in elasticity. As we were finishing this paper, we became aware of the recent work of Neukamm and Richter \cite{NeR25} which includes \eqref{eq:rigidity-inequality-intro} as a subcase of a mixed growth-type rigidity inequality for $\Omega$ that are Jones domains. The key new idea in \cite{NeR25} is to find a way to extend the domain of definition of the deformation $\mathbf{y}$ from $\Omega$ to all of $\mathbb{R}^n$, such that the resulting strain is controlled. This leads to a rigidity inequality with a constant that depends on the parameters defining $\Omega$ as a Jones domain. Since these parameters remain bounded under bi-Lipschitz transformations of $\Omega$, so does the constant in this inequality (see \cite[Section 3.4]{NeR25}). This last fact is crucial for our analysis of the truncated wedge, where we require a rigidity inequality with a constant that is independent of the truncation length.

This section establishes the desired $\delta$-independent family of rigidity inequalities for the truncated wedge domain $\Omega_\delta$. First, for the interested reader, we explain in Section \ref{ssec:rigidity-biLip-proof} how to obtain the bi-Lipschitz invariant version of \eqref{eq:rigidity-inequality-intro} for Lipschitz domains without passing through the theory of Jones domains. Instead, we simply track the dependence of the constants in the original argument of \cite{FJM02} including the modifications of \cite{CoS06} for $p\in(1,\infty)$, as is done for $p=2$ by Lewicka in \cite[Theorem 4.1]{Lew23}. Afterwards, we specialize to truncated wedge domains in Section \ref{ssec:truncated_wedge}. Section \ref{ssec:ChoiceRotSec} optimizes the rotation $\mathbf{R}$. The reader who is eager to get to the main point of the paper---namely, the derivation of the Flamant solution---can safely skip to Section \ref{sec:linear-analysis} and come back later.

\subsection{Rigidity inequalities for bi-Lipschitz equivalent domains}\label{ssec:rigidity-biLip-proof}

Recall $\Omega,\Omega'\subset\mathbb{R}^n$ are \emph{bi-Lipschitz equivalent} if there exists a bijection $\mathbf{\Phi}:\Omega\to\Omega'$ such that $\mathbf{\Phi}$ and $\mathbf{\Phi}^{-1}$ are Lipschitz. The \emph{bi-Lipschitz constant} of $\Phi$ is the smallest $K\geq 0$ such that 
\begin{equation*}
\frac1K|\mathbf{x}-\mathbf{x}'|\leq\left|\mathbf{\Phi}(\mathbf{x})-\mathbf{\Phi}(\mathbf{x}')\right|\leq K|\mathbf{x}-\mathbf{x}'|\quad\forall\,\mathbf{x},\mathbf{x}'\in\Omega.
\end{equation*}
\begin{thm}\label{thm:geometric_rigidity}
	Let $\Omega\subset\mathbb{R}^{n}$, $n\geq 2$, be a bounded Lipschitz domain, and let $p\in(1,\infty)$. For every $\mathbf{y}\in W^{1,p}(\Omega;\R^n)$ there is a rotation $\mathbf{R}\in SO(n)$ such that
	\begin{align}\label{eq:geometric_rigidity}	
		\|\nabla\mathbf{y} - \mathbf{R}\|_{L^p(\Omega)} \leq C \|d(\nabla \mathbf{y}, SO(n)\|_{L^p(\Omega)}.
	\end{align}
	The constant $C$ depends only on $\Omega$ and $p$, and is invariant to translations and dilations of $\Omega$.
	Moreover, the same constant can be used for a family of domains that are bi-Lipschitz equivalent with a uniform bi-Lipschitz constant. 
\end{thm}
\begin{rem} The translation and dilation invariance of $C$ follows directly from \eqref{eq:geometric_rigidity}. The bi-Lipschitz invariance does not. 
\end{rem}
\begin{proof} 

	\textit{Step 0: Rigidity inequality on cubes.} We start by asserting the validity of \eqref{eq:geometric_rigidity} on cubes. As this follows directly from \cite{FJM02,CoS06}, we simply state the result without proof. Let $Q(\mathbf{x}_0,r)=\mathbf{x}_0+(-r/2,r/2)^n$ be the cube centered at $\mathbf{x}_0\in\mathbb{R}^n$ with sides parallel to the axes of length $r>0$. 
    Given $\mathbf{y}\in W^{1,p}(Q(\mathbf{x}_0,r);\R^n)$, there exists $\mathbf{R}\in SO(n)$ such that
	\begin{align}\label{eq:local_rigidity}
		\int_{Q(\mathbf{x}_0,r)} |\nabla \mathbf{y} - \mathbf{R}|^p \, d\mathbf{x} \lesssim_{n,p} \int_{Q(\mathbf{x}_0,r)} d^p(\nabla \mathbf{y},SO(n))\, d\mathbf{x}.
	\end{align}
    The rest of the proof involves going from \eqref{eq:local_rigidity} to a general bounded Lipschitz domain $\Omega$ in a way that is invariant to bi-Lipschitz transformations of the domain.

	\textit{Step 1: Lipschitz truncation.} Next, we recall that a $W^{1,p}$-function is Lipschitz on most of its domain.  
    The $p=2$ version of this is \cite[Theorem 4.3]{Lew23}; a similar result is shown for $p\in[1,\infty)$ in \cite[Proposition A.1]{FJM02}, albeit without the required bi-Lipschitz invariance. Here is the result we use: if $U\subset \R^n$ is a bounded Lipschitz domain and $p\in [1,\infty)$, then for all $\mathbf{y}\in W^{1,p}(U;\mathbb{R}^m)$ and $M>0$ there exists $\tilde{\mathbf{y}}\in W^{1,\infty}(U;\mathbb{R}^m)$ such that
	\begin{align}\label{eq:truncation}
		\begin{split}
		\text{(i)}&\quad \|\nabla \tilde{\mathbf{y}}\|_{L^\infty(U)} \leq M,\\
		\text{(ii)}&\quad \big|\{\mathbf{x}\in U : \mathbf{y}(\mathbf{x}) \neq \tilde{\mathbf{y}}(\mathbf{x})\}\big| \lesssim_U \frac{1}{M}\int_{\{|\nabla \mathbf{y}| > M\}} |\nabla \mathbf{y}|\, d\mathbf{x},\\
		\text{(iii)}&\quad \int_U |\nabla \mathbf{y} - \nabla \tilde{\mathbf{y}}|^p \, d\mathbf{x} \lesssim_{U,p} \int_{\{|\nabla \mathbf{y}| > M\}} |\nabla \mathbf{y}|^p\, d\mathbf{x}.
		\end{split}
	\end{align}
	The constants in (ii) and (iii) depend only on $U$ and $p$ as shown; the dependence on $U$ is uniform for a family of domains that are bi-Lipschitz equivalent with a uniform bi-Lipschitz constant. 
    
    Only (iii) must be checked for $p\neq 2$. It follows directly from (i) and (ii) as in the references. Note that
    \[
    \int_{U}|\nabla \mathbf{y}-\nabla \tilde{\mathbf{y}}|^{p}\,d\mathbf{x}\lesssim_{p}\int_{\{\mathbf{y}\neq\tilde{\mathbf{y}}\}}|\nabla \mathbf{y}|^{p}+|\nabla \tilde{\mathbf{y}}|^{p}\,d\mathbf{x}\lesssim\int_{\{|\nabla \mathbf{y}|>M\}}|\nabla \mathbf{y}|^{p}\,d\mathbf{x}+M^{p}|\left\{ \mathbf{y}\neq\tilde{\mathbf{y}}\right\} |
    \]
    and use that $|\nabla \mathbf{y}|/M > 1$ on the domain of integration on the right-hand side of (ii). 

    This result is used as follows. Let $\mathbf{y}\in W^{1,p}(\Omega;\mathbb{R}^n)$ and produce $\tilde{\mathbf{y}}\in W^{1,\infty}(\Omega;\mathbb{R}^n)$ satisfying \eqref{eq:truncation}. Select $M=M(n)$ such that $|\mathbf{F}|\lesssim d(\mathbf{F},SO(n))$ if $|\mathbf{F}|>M$. If the rigidity inequality \eqref{eq:geometric_rigidity} holds for $\tilde{\mathbf{y}}$, it also holds for $\mathbf{y}$. Indeed, 
    \begin{equation*}
    \int_{\Omega}|\nabla\tilde{\mathbf{y}}-\mathbf{R}|^{p}\,d\mathbf{x}\lesssim_{\Omega,p}\int_{\Omega}d^{p}(\nabla\tilde{\mathbf{y}},SO(n))\,d\mathbf{x}\lesssim_{p}\int_{\Omega}d^{p}(\nabla\mathbf{y},SO(n))+|\nabla\mathbf{y}-\nabla\tilde{\mathbf{y}}|^{p}\,d\mathbf{x}
    \end{equation*}
    and hence
    \begin{align*}
    \int_{\Omega}|\nabla\mathbf{y}-\mathbf{R}|^{p}\,d\mathbf{x}&\lesssim_{p}\int_{\Omega}|\nabla\mathbf{y}-\nabla\tilde{\mathbf{y}}|^{p}+|\nabla\tilde{\mathbf{y}}-\mathbf{R}|^{p}\,d\mathbf{x}\\&\lesssim_{\Omega,p}\int_{\left\{ |\nabla\mathbf{y}|>M\right\} }|\nabla\mathbf{y}|^{p}\,d\mathbf{x}+\int_{\Omega}d^{p}(\nabla\mathbf{y},SO(n))\,d\mathbf{x}\lesssim\int_{\Omega}d^{p}(\nabla\mathbf{y},SO(n))\,d\mathbf{x}
    \end{align*}
    by the definitions of $\tilde{\mathbf{y}}$ and $M$. The constants implicit in the above have the desired bi-Lipschitz invariance, since the same holds for the constants in \eqref{eq:truncation}. Therefore, in the remainder of the proof, we can assume that $\mathbf{y}\in W^{1,\infty}(\Omega;\mathbb{R}^n)$ with $\|\nabla\mathbf{y}\|_{L^\infty(\Omega)}\leq M$. 
    
	\textit{Step 2: Harmonic decomposition.} Next, we write $\mathbf{y} = \mathbf{w}+ \mathbf{z}$ where $\mathbf{w}$ is harmonic. This is motivated by the observation from \cite{FJM02} that if $d(\nabla\mathbf{y},SO(n))$ is small, then $\mathbf{y}$ is approximately harmonic. However, we avoid the choice from \cite{FJM02} as it is specific to $p=2$ and instead follow \cite{CoS06} to handle $p\in(1,\infty)$. Let $\mathbf{z}=\nabla\cdot \boldsymbol{\Psi}$ where $\boldsymbol{\Psi}: \R^n\to \R^{n\times n}$ is the Newtonian potential of
	$(\nabla \mathbf{y}-\cof \nabla \mathbf{y})\mathbbm{1}_{\Omega}$, i.e.,
	\begin{align*}
	    \boldsymbol{\Psi}(\mathbf{x}) = \int_\Omega \Gamma(\mathbf{x}-\mathbf{x}')\left(\nabla \mathbf{y}(\mathbf{x}')-\cof \nabla \mathbf{y}(\mathbf{x}')\right)\,d\mathbf{x}'
	\end{align*}
	where 
    \begin{equation*}
    \Gamma(\mathbf{x})=\begin{cases}
-\frac{1}{n(n-2)\omega_{n}}\frac{1}{|\mathbf{x}|^{n-2}} & n\geq2\\
\frac{1}{2\pi}\log|\mathbf{x}| & n=2
\end{cases}
    \end{equation*}
	and $\omega_n=|B_1(\mathbf{0})|$.
    Since $\Delta \boldsymbol{\Psi} = \nabla \mathbf{y}-\cof \nabla \mathbf{y}$ in $\Omega$ and $\nabla\cdot\cof\nabla\mathbf{y}=\mathbf{0}$, $\mathbf{w}= \mathbf{y}-\mathbf{z}$ is harmonic on $\Omega$.
	A singular integral estimate \cite[Theorem 9.9]{GiT01} shows that 
	\begin{align}\label{eq:non-harmonic}
		\|\nabla \mathbf{z}\|_{L^p(\Omega)} \lesssim_n  \|\nabla^2 \boldsymbol{\Psi}\|_{L^p(\Omega)} \lesssim_{n,p} \|\nabla \mathbf{y}-\cof \nabla \mathbf{y}\|_{L^p(\Omega)} \lesssim_n \|d(\nabla \mathbf{y},SO(n))\|_{L^p(\Omega)}.
	\end{align}
Note we used the pointwise bound $|\nabla \mathbf{y}-\cof \nabla \mathbf{y}|\lesssim_n d(\nabla \mathbf{y},SO(n))$, which holds since $\|\nabla \mathbf{y}\|_{L^\infty(Q)}\leq M$.

	  \textit{Step 3: Whitney decomposition.} Finally, we control the harmonic part $\mathbf{w}$ by a covering argument along with the rigidity inequality for cubes in \eqref{eq:local_rigidity}. Introduce a Whitney decomposition of $\Omega$ made of disjoint cubes $Q(\mathbf{x}_i,r_i)=\mathbf{x}_i+(-r_i/2,r_i/2)^n$ with side lengths $r_i\sim d(\mathbf{x}_i,\partial\Omega)$ such that $\Omega = \cup_{i=1}^\infty \overline{Q(\mathbf{x}_i,r_i)}$ and such that each $\mathbf{x}\in\Omega$ belongs to at most $N=N(n)$ doubled cubes $Q(\mathbf{x}_i,2r_i)$, which we also assume belong to $\Omega$ (see, e.g., \cite[Appendix J]{Gra14}). By \eqref{eq:local_rigidity}, there exist rotations $\{\mathbf{R}_i\}\subset SO(n)$ such that
	\begin{equation*}
    \int_{Q(\mathbf{x}_i,2r_i)}|\nabla \mathbf{w} - \mathbf{R}_i|^p\, d\mathbf{x} \lesssim_{n,p} \int_{Q(\mathbf{x}_i,2r_i)} d^p(\nabla \mathbf{w},SO(n))\, d\mathbf{x}.
    \end{equation*}
    Since $\mathbf{w}$ is harmonic, 
	\begin{equation*}
    r_i^p\int_{Q(\mathbf{x}_i,r_i)}|\nabla^2 \mathbf{w}|^p\, d\mathbf{x} \lesssim_{n,p} \int_{Q(\mathbf{x}_i,2r_i)}|\nabla \mathbf{w} - \mathbf{R}_i|^p\, d\mathbf{x}
    \end{equation*}
by the mean value property applied to $\nabla^2\mathbf{w}$. Summing up over $i$, there follows 
    \begin{equation}\label{eq:global-rigidity-3}
    \int_\Omega |\nabla^2 \mathbf{w}|^p d^p(\mathbf{x},\partial \Omega)\, d\mathbf{x} \lesssim \sum_{i=1}^{\infty} r_i^p\int_{Q(\mathbf{x}_i,r_i)}|\nabla^2 \mathbf{w}|^p\, d\mathbf{x} \lesssim_{n,p} \int_\Omega d^p(\nabla \mathbf{w},SO(n))\, d\mathbf{x}.
    \end{equation}
    Note we absorbed the constant $N$ into the notation.

To finish, we use the weighted $L^p$ Poincar\'e inequality \cite{BoS88}, which holds with a constant that is invariant to bi-Lipschitz transformations of the domain by the same argument as in \cite[Theorem 3.17]{Lew23}.  
We deduce the existence of $\mathbf{A}\in\mathbb{R}^{n\times n}$ such that
    \[
    \int_\Omega |\nabla \mathbf{w}-\mathbf{A}|^p\,d\mathbf{x}\lesssim_{\Omega,p} \int_\Omega |\nabla^2 \mathbf{w}|^p d^p(\mathbf{x},\partial \Omega)\, d\mathbf{x}
    \]
where the constant has the desired invariance. By \eqref{eq:global-rigidity-3}, 
\begin{equation}\label{eq:almost-done}
    \int_\Omega |\nabla \mathbf{w}-\mathbf{A}|^p\,d\mathbf{x}\lesssim_{\Omega,p} \int_\Omega d^p(\nabla \mathbf{w},SO(n))\, d\mathbf{x}
\end{equation}
and the same holds with $\mathbf{A}$ replaced by $\mathbf{R}\in SO(n)$ satisfying $|\mathbf{A} - \mathbf{R}| = d(\mathbf{A},SO(n))$. 
Indeed,
\begin{equation*}
		|\mathbf{A} - \mathbf{R}|^p =  d^p(\mathbf{A},SO(n))\lesssim_p \fint_{\Omega}d^p(\nabla \mathbf{w},SO(n)) + |\nabla \mathbf{w} - \mathbf{A}|^p\, d\mathbf{x}
\end{equation*}
and the last term is controlled by \eqref{eq:almost-done}. To finish, recall that $\mathbf{w}=\mathbf{y}-\mathbf{z}$ where $\mathbf{z}$ obeys \eqref{eq:non-harmonic}. 
\end{proof}

Korn's inequality is well-known in the context of linear elasticity.
With it, one can bound the $L^p$-norm of the displacement gradient by its symmetric part, after subtracting an appropriate constant. (Strictly speaking, we are referring to the $L^p$ version of Korn's second inequality; other versions control the full gradient using boundary conditions or lower order terms.)
For our purposes, we must bound the constant in Korn's inequality for a family of uniformly bi-Lipschitz equivalent domains. This follows by applying  Theorem \ref{thm:geometric_rigidity} to the maps $\mathbf{y}_{\epsilon}(\mathbf{x}) = \mathbf{x} + \epsilon \mathbf{u}(\mathbf{x})$ for $\mathbf{u} \in W^{1,p}(\Omega;\mathbb{R}^n)$, and taking $\epsilon \to 0$. In particular, there are rotations $\mathbf{R}_{\epsilon}\in SO(2)$ and a constant $C =C(\Omega,p)$ with the stated invariances such that $\|\nabla \mathbf{y}_{\epsilon} - \mathbf{R}_{\epsilon}\|_{L^p(\Omega)} \leq C\| d( \nabla \mathbf{y}_{\epsilon}, SO(n)\|_{L^p(\Omega)}$. Expanding this to first order in $\epsilon$ produces the following result.

\begin{cor}\label{cor:Korn_general}
    Let $\Omega\subset \R^n$, $n\geq 2$, be a bounded Lipschitz domain, and let $p\in (1,\infty)$. For every $\mathbf{u}\in W^{1,p}(\Omega;\R^n)$ there exists $\mathbf{W}\in {\rm Skw}_n$ such that
    \[
    \|\nabla \mathbf{u} - \mathbf{W}\|_{L^p(\Omega)} \leq C\|\mathbf{e}(\mathbf{u})\|_{L^p(\Omega)}.
    \]
    The constant $C$ depends only on $\Omega$ and $p$, and is invariant to translations and dilations of $\Omega$.
	Moreover, the same constant can be used for a family of domains that are bi-Lipschitz equivalent with a uniform bi-Lipschitz constant.
\end{cor}

\subsection{Rigidity inequalities for the truncated wedge}\label{ssec:truncated_wedge}
We now obtain $\delta$-independent versions of the geometric rigidity and Korn inequalities for the truncated wedge domain $\Omega_\delta$, for sufficiently small $\delta$. First, we prove a useful bi-Lipschitz equivalence. 
\begin{lem}\label{lemn:truncated_vs_full_wedge}
\label{lem:bi-Lip-param} Fix $\alpha,\beta\in\mathbb{R}$ with $\beta - \alpha \in (0,2\pi)$. The truncated wedge 
\[
\Omega_{\delta}=\left\{ \mathbf{x}\in\mathbb{R}^2:r\in(\delta,1),\theta\in(\alpha,\beta)\right\} 
\]
and the full wedge 
\[
\Omega_{0}=\left\{ \mathbf{x}\in\mathbb{R}^2:r\in(0,1),\theta\in(\alpha,\beta)\right\}
\]
are bi-Lipschitz equivalent with a bi-Lipschitz constant that is uniformly bounded for $\delta\in(0,1/2)$.
\end{lem}
\begin{proof}
To begin, note that if the domains
\[
U_{\delta} =  \left\{ \mathbf{x}:r\in(\delta,1),\theta\in(-\frac{\pi}{4},\frac{\pi}{4})\right\}\quad\text{and}\quad U_0 =   \left\{ \mathbf{x}:r\in(0,1),\theta\in(-\frac{\pi}{4},\frac{\pi}{4})\right\}
\]
are bi-Lipschitz equivalent, then so are $\Omega_{\delta}$ and $\Omega_0$. The proof is sketched in Figure \ref{fig:BiLipschitz}.  The map $\boldsymbol{\Psi} \colon \Omega_{0} \rightarrow U_0$  and its inverse  $\boldsymbol{\Psi}^{-1} \colon U_{0} \rightarrow \Omega_0$ are defined by   
\begin{align*}
\boldsymbol{\Psi}(\mathbf{x}) = r \mathbf{e}_r \big(\tfrac{\pi/2}{\beta- \alpha} [ \theta - \tfrac{1}{2} (\beta + \alpha) ] \big) , \quad \boldsymbol{\Psi}^{-1}(\mathbf{x}) = r \mathbf{e}_r \big(\tfrac{\beta - \alpha}{\pi/2} \theta + \tfrac{1}{2} (\beta + \alpha)  \big) .
\end{align*}
Evidently,
\begin{align*}
|\nabla \boldsymbol{\Psi}| \leq  | \partial_r \boldsymbol{\Psi}|  + |\tfrac{1}{r} \partial_\theta \boldsymbol{\Psi} | = 1 +  \tfrac{\pi/2}{\beta- \alpha} 
\end{align*}
and likewise
\[
|\nabla \boldsymbol{\Psi}^{-1}| \leq 1 + \tfrac{\beta - \alpha}{\pi/2}.
\]
If  $\boldsymbol{\Phi}_{\delta}$ maps  $U_{\delta}$ bijectively onto  $U_0$, then $\boldsymbol{\Psi}^{-1} \circ \boldsymbol{\Phi}_{\delta} \circ \boldsymbol{\Psi}$ maps $\Omega_\delta$ bijectively onto $\Omega_0$ with the estimates 
\begin{align*}
\| \nabla \big( \boldsymbol{\Psi}^{-1} \circ \boldsymbol{\Phi}_{\delta} \circ \boldsymbol{\Psi}\big)\|_{L^{\infty}(\Omega_{\delta})} \leq \| \nabla \boldsymbol{\Psi}^{-1}\|_{L^{\infty}(U_0)} \| \nabla \boldsymbol{\Phi}_{\delta} \|_{L^{\infty}(U_\delta)}  \| \nabla \boldsymbol{\Psi} \|_{L^{\infty}(\Omega_{\delta})} \lesssim \| \nabla \boldsymbol{\Phi}_{\delta} \|_{L^{\infty}(U_\delta)}
\end{align*}
and
\begin{align*}
\| \nabla \big( \boldsymbol{\Psi}^{-1} \circ \boldsymbol{\Phi}_{\delta} \circ \boldsymbol{\Psi}\big)^{-1}\|_{L^{\infty}(\Omega_{\delta})} \leq \| \nabla \boldsymbol{\Psi}^{-1}\|_{L^{\infty}(U_\delta)} \| \nabla \boldsymbol{\Phi}^{-1}_{\delta} \|_{L^{\infty}(U_0)}  \| \nabla \boldsymbol{\Psi} \|_{L^{\infty}(\Omega_0)} \lesssim \| \nabla \boldsymbol{\Phi}^{-1}_{\delta} \|_{L^{\infty}(U_0)}.
\end{align*}
See Figure \ref{fig:BiLipschitz}(a) for a depiction of these claims.

\begin{figure}[t!]
\centering
\includegraphics[width = .7\textwidth]{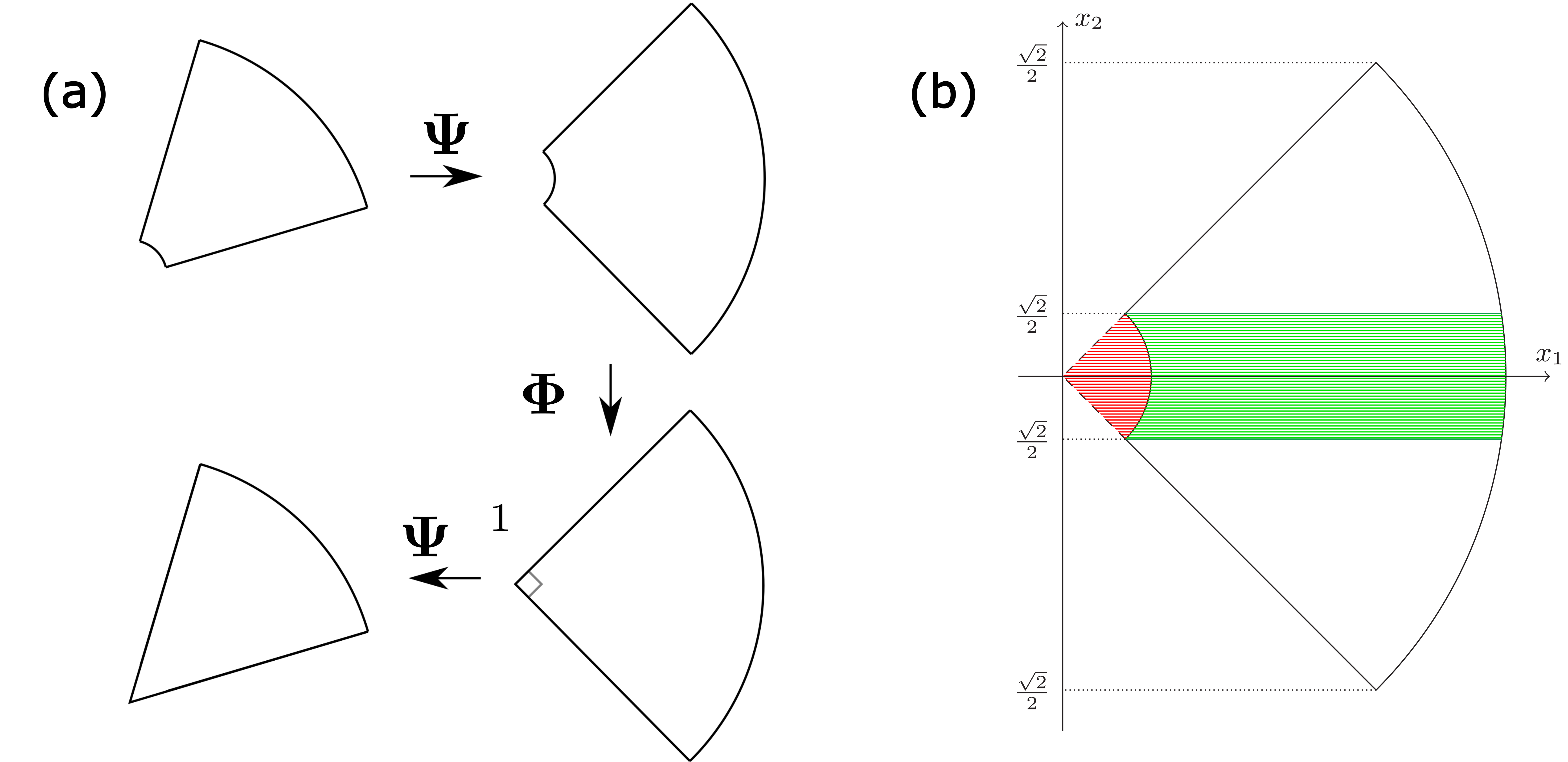}
\caption{Bi-Lipschitz equivalence of $\Omega_{\delta}$ and $\Omega_0$. (a) The map $\boldsymbol{\Psi}$ rescales and shifts the $\theta$-coordinate to transform a wedge with angle $\beta-\alpha$ to a reference wedge with angle $\pi/2$. (b) The map $\boldsymbol{\Phi}_{\delta}$ acts on the $\delta$-truncation of the reference wedge and stretches the $\delta$-strip (in green) so that the $r =\delta$ arc deforms to match the wedge boundary (dashed).}
\label{fig:BiLipschitz}
\end{figure}

We now show that $U_{\delta}$ is bi-Lipschitz equivalent to $U_0$. The map  
\[
\mathbf{\Phi}_{\delta}(\mathbf{x})=\begin{cases}
\mathbf{x} & |x_{2}|> \frac{\delta}{\sqrt{2}}\\
\left(a_{\delta}(x_{2})x_{1}+b_{\delta}(x_{2})\right)\mathbf{e}_{1}+x_{2}\mathbf{e}_{2} & |x_{2}|\leq \frac{\delta}{\sqrt{2}}
\end{cases}
\]
can be used to take $U_\delta$ bijectively onto $U_{0}$ as shown in Figure \ref{fig:BiLipschitz}(b). We simply stretch the   $r = \delta$ arc in the $\mathbf{e}_1$-direction to match the wedge boundary by setting 
\begin{align*}
a_{\delta}=  \frac{\sqrt{1  - x_2^2} - |x_2|}{ \sqrt{1 - x_2^2} - \sqrt{\delta^2 - x_2^2}}, \quad b_{\delta} = \sqrt{1 - x_2^2} (1 - a_{\delta}).
\end{align*}
We must check that this choice satisfies the uniform bi-Lipschitz bound
\begin{equation*}
\sup_{\delta\in(0,\frac12)}\,\|\nabla\mathbf{\Phi}_{\delta}\|_{L^{\infty}(U_{\delta})}\vee\|\nabla\mathbf{\Phi}_{\delta}^{-1}\|_{L^{\infty}(U_{0})}<\infty.
\end{equation*}
Since $\nabla\mathbf{\Phi}_{\delta}^{-1}=\frac{1}{\det\nabla\mathbf{\Phi}_{\delta}}(\text{cof}\,\nabla\mathbf{\Phi}_{\delta})^{T}$
and the cofactor is linear on $\mathbb{R}^{2\times2}$, it suffices
to bound $|\nabla\mathbf{\Phi}_{\delta}|$ from above and $|\det\nabla\mathbf{\Phi}_{\delta}|$
from below. The part where $|x_2|>\delta/\sqrt{2}$ is obvious. For $|x_{2}|\leq\delta/\sqrt{2}$, 
\begin{equation}\label{to-bound-it}
\nabla\mathbf{\Phi}_{\delta}=a_{\delta}\mathbf{e}_{1}\otimes\mathbf{e}_{1}+\left(a_{\delta}'x_{1}+b_{\delta}'\right)\mathbf{e}_{2}\otimes\mathbf{e}_{1}+\mathbf{e}_{2}\otimes\mathbf{e}_{2}
\quad \text{ and } \quad  \det\nabla\mathbf{\Phi}_{\delta}=a_{\delta}.
\end{equation}
In the rest of the proof we simply bound the components of this formula.

First, observe that
\begin{equation}
\begin{aligned}\label{eq:biLip1}
 1\leq a_{\delta}\leq\frac{1}{1-\delta}\leq2   
 \end{aligned}
\end{equation}
for any $\delta\in(0,1/2)$. Indeed, $\mathbf{\Phi}_{\delta}$
stretches each line segment parallel to $\mathbf{e}_{1}$ and with
$|x_{2}|\leq \delta/\sqrt{2}$ by an affine  map that keeps its transverse width fixed. The minimum stretch occurs at $|x_2|= \delta/\sqrt{2}$ and the maximum is at $x_{2}=0$. That is,
$\min\,a_{\delta}=a_{\delta}( \delta/\sqrt{2})=1$ and  $\max\,a_{\delta}=a_{\delta}(0)=1/(1-\delta)$.
This proves  \prettyref{eq:biLip1}.

To finish, we must bound the cross term $a_{\delta}'x_{1}+b_{\delta}'$ in \eqref{to-bound-it}
for $|x_{2}| \leq  \delta/\sqrt{2}$. We make repeated use of the basic estimate  $\frac{|\xi|}{\sqrt{1-\xi^2}} \leq \sqrt{2} |\xi|$ for $|\xi| \leq 1/\sqrt{2}$. In particular, this estimate implies that 
\begin{align}
    &\left|\frac{d}{dx_{2}}\sqrt{1-x_{2}^{2}}\right| =  \frac{|x_{2}|}{\sqrt{1-x_{2}^{2}}}\leq \sqrt{2} |x_2| \leq \delta \leq \frac12, \label{eq:biLip2} \\
     &\left|\frac{d}{d x_{2}}\sqrt{\delta^2-x_{2}^{2}}\right| = \frac{|\frac{x_2}{\delta}|}{\sqrt{1 - |\frac{x_2}{\delta}|^2}} \leq \sqrt{2}\Big|\frac{x_2}{\delta}\Big| \leq 1  \label{eq:biLip3} 
\end{align}
for $|x_{2}|\leq\delta/\sqrt{2}$ and $\delta \in (0,1/2)$. Thus,
\begin{align*}
|b_{\delta}'| &= \left| (1-a_{\delta}) \frac{d}{d\mathbf{x}_{2}}\sqrt{1-x_{2}^{2}} +  \sqrt{1-x_{2}^{2}} a_{\delta}' \right|   \leq \frac{1}{2} + |a_{\delta}'|
\end{align*}
using \eqref{eq:biLip1} and \eqref{eq:biLip2}. Furthermore, 
\begin{align*}
|a_{\delta}'| &=  \left|\frac{\frac{d}{dx_2}\Big(\sqrt{1-x_2^2} - x_2\Big) - a_{\delta} \frac{d}{dx_2}\Big(\sqrt{1 -x_2^2} - \sqrt{\delta^2 - x_2^2}\Big)}{\sqrt{1-x_2^2} - \sqrt{\delta^2 - x_2^2}}  \right|  \leq \frac{7/2}{\left|\sqrt{1-x_2^2} - \sqrt{\delta^2 - x_2^2}\right|}
\end{align*}
using \eqref{eq:biLip1}-\eqref{eq:biLip3}. We conclude that $|a_{\delta}'|\vee|b_{\delta}'| \lesssim 1$ by observing that 
\begin{align*}
  \left|\sqrt{1-x_2^2} - \sqrt{\delta^2 - x_2^2}\right| \geq \sqrt{1 - \frac{1}{2} \delta^2} - \delta   \geq \sqrt{\frac{7}{8}}  - \frac{1}{2} > 0
\end{align*}
for all  $|x_{2}|\leq\delta/\sqrt{2}$ and $\delta \in (0,1/2)$.
\end{proof}

\noindent Lemma \ref{lemn:truncated_vs_full_wedge} and Theorem \ref{thm:geometric_rigidity} 
yield the following geometric rigidity inequalities for $\Omega_\delta$. The second inequality uses boundary data to control the \emph{a priori} unknown rotation from the first inequality; it will come in handy when we discuss the displacement problem. 
\begin{cor}\label{cor:p-rigidity-wedge}
Given $\delta\in(0,1/2)$, $p\in (1,\infty)$, and $\mathbf{y}\in W^{1,p}(\Omega_\d;\R^2)$, there exists $\mathbf{R}\in SO(2)$ such that
\begin{align*}
	\|\nabla\mathbf{y}-\mathbf{R}\|_{L^{p}(\Omega_{\delta})} \lesssim_{\beta-\alpha, p} \|d(\nabla \mathbf{y},SO(2))\|_{L^{p}(\Omega_{\delta})}.
\end{align*}
Furthermore,  
\begin{align*}
    \| \nabla \mathbf{y} - \mathbf{I} \|_{L^p(\Omega_{\delta})} \lesssim_{\beta - \alpha, p} \| d (\nabla \mathbf{y}, SO(2)) \|_{L^p(\Omega_{\delta})}  + \|( \mathbf{y} - \mathbf{x})|_{r = 1} \|_{\dot{W}^{1-\frac1p,p}((\alpha,\beta))}.
\end{align*}
The constants in these estimates are independent of $\d$. 
\end{cor}
\begin{proof}
The first inequality follows from Lemma \ref{lemn:truncated_vs_full_wedge} and Theorem \ref{thm:geometric_rigidity}. The second inequality follows from the first. Fix $\mathbf{y}$ and let $\mathbf{R}$ be the rotation from the first inequality. Observe $|\mathbf{F}| \sim_{\beta-\alpha,p} \|(\mathbf{F} \mathbf{x})|_{r = a}\|_{\dot{W}^{1-1/p,p}((\alpha, \beta))}$ for all $\mathbf{F}\in\mathbb{R}^{2\times2}$ by definition of the seminorm (see \eqref{eq:DotSemiNorm}). By Lemma \ref{traceToGradLemma} from the appendix, 
\begin{align*}
|\mathbf{R} - \mathbf{I}| &\lesssim_{\beta-\alpha,p}
  \big\| \big(\mathbf{y} - \mathbf{R}\mathbf{x} \big)|_{r=1} \big\|_{\dot{W}^{1-\frac1p,p}((\alpha,\beta))}  + \big\| \big(\mathbf{y} - \mathbf{x}\big)|_{r=1} \big\|_{\dot{W}^{1-\frac1p,p}((\alpha,\beta))} \\ 
& \lesssim_{\beta-\alpha,p} \| \nabla \mathbf{y} - \mathbf{R} \|_{L^p(\Omega_{\delta})} + \|\big(\mathbf{y} - \mathbf{x}\big)|_{r=1}\|_{\dot{W}^{1-\frac1p,p}((\alpha,\beta))}\\
&\lesssim_{\beta-\alpha,p} \|d(\nabla\mathbf{y},SO(2))\|_{L^p(\Omega_{\delta})}  + \|( \mathbf{y} - \mathbf{x})|_{r = 1} \|_{\dot{W}^{1-\frac1p,p}((\alpha,\beta))}.
\end{align*}
The second inequality is proved.
\end{proof}
\noindent Korn-type inequalities for $\Omega_\delta$ follow from linearizing Corollary \ref{cor:p-rigidity-wedge} just like in the passage from Theorem \ref{thm:geometric_rigidity} to Corollary \ref{cor:Korn_general}. (Alternatively, they can be obtained from Lemma \ref{lem:bi-Lip-param} and Corollary \ref{cor:Korn_general}.) 
\begin{cor}\label{cor:Korn-inequality}
Given $\delta\in(0,1/2)$, $p \in (1,\infty)$, and $\mathbf{u}\in W^{1,p}(\Omega_\d;\R^2)$, there exists $\mathbf{W}\in {\rm Skw}_2$ such that
\begin{align*}
	\|\nabla\mathbf{u}-\mathbf{W}\|_{L^{p}(\Omega_{\delta})} \lesssim_{\beta-\alpha,p} \|\mathbf{e}(\mathbf{u})\|_{L^{p}(\Omega_{\delta})}.
    \end{align*}
Furthermore,
\begin{align*}
\|\nabla\mathbf{u}\|_{L^{p}(\Omega_{\delta})}\lesssim_{{\beta-\alpha},p}  \begin{cases}
\|\mathbf{e}(\mathbf{u})\|_{L^{p}(\Omega_{\delta})} & \text{ if }\int_{\Omega_{\delta}}\skw \nabla\mathbf{u}\,d\mathbf{x}=\mathbf{0}\\
\|\mathbf{e}(\mathbf{u})\|_{L^{p}(\Omega_{\delta})}+\|\mathbf{u}|_{r=1} \|_{\dot{W}^{1-\frac1p,p}((\alpha,\beta))} & \text{ in general}
\end{cases}.
\end{align*}
The constants in these estimates are independent of $\d$. 
\end{cor}

In the rest of the paper, we drop the subscript regarding the wedge angle $\beta-\alpha$ from the notation, although we continue to highlight the dependence of constants on the growth of the elastic energy density at infinity, i.e., $p$. This is with the exception of the appendix on trace inequalities (Appendix \ref{sec:appendix-traces}).

\subsection{Optimal rotations}\label{ssec:ChoiceRotSec}

In the analysis of the Flamant solution to come, we consider elastic energies of the form $E_\delta(\mathbf{y})=\int_{\Omega_\delta}W(\nabla \mathbf{y})\,d\mathbf{x}$ where $W(\cdot)\gtrsim d^{2}(\cdot,SO(2)) + d^{p}(\cdot,SO(2))$ for some $p\geq 2$ (see  \eqref{eq:Wgrowth}). Given $\mathbf{y}\in W^{1,p}(\Omega_\d;\R^2)$, \prettyref{cor:p-rigidity-wedge} yields two possibly different rotations $\mathbf{R}_2,\mathbf{R}_p\in SO(2)$ satisfying
\begin{equation}
\int_{\Omega_{\delta}}|\nabla\mathbf{y}-\mathbf{R}_2|^{2} \, d \mathbf{x} \lesssim E_{\delta}(\mathbf{y})  \quad \text{ and } \quad  \int_{\Omega_{\delta}} |\nabla\mathbf{y}-\mathbf{R}_p|^{p} \,d\mathbf{x} \lesssim_p E_{\delta}(\mathbf{y}) .\label{eq:L2p-rigidity-recalled}
\end{equation}
We now show how to match the rotation between these estimates. Simultaneously, in preparation for the force problem, we optimize over $\mathbf{R}_2$. (Something similar can be done for the skew-symmetric matrices from Corollary \ref{cor:Korn-inequality}, but in the rest of the paper we only need that result for $p=2$.)

Define the set of \emph{optimal rotations} associated to the deformation $\mathbf{y}$ by
\begin{align*}
	\cR_{\delta}(\mathbf{y}) = \Big\{ \mathbf{R} \in SO(2) \colon \int_{\Omega_\delta} | \nabla \mathbf{y} - \mathbf{R}|^2 \, d\mathbf{x} = \min_{\mathbf{Q} \in SO(2)} \int_{\Omega_\delta} |\nabla \mathbf{y} - \mathbf{Q} |^2 \, d \mathbf{x} \Big\}.
\end{align*}
\begin{cor}\label{cor:mixed-rigidity-energy}
Let $\delta\in(0,1/2)$. For every $\mathbf{y} \in W^{1,p}(\Omega_{\delta}; \mathbb{R}^2)$ and $\mathbf{R} \in \cR_\delta(\mathbf{y})$, 
\[
\int_{\Omega_{\delta}}|\nabla\mathbf{y}-\mathbf{R}|^{2} \, d \mathbf{x} \lesssim E_{\delta}(\mathbf{y}) \quad \text{ and } \quad   \int_{\Omega_{\delta}} |\nabla\mathbf{y}-\mathbf{R}|^{p} \,d\mathbf{x} \lesssim_p  E_{\delta}(\mathbf{y}).
\]
\end{cor}
\begin{proof} 
The first part of the claim follows from the first part of \eqref{eq:L2p-rigidity-recalled} and the definition of $\mathcal{R}_{\delta}(\mathbf{y})$. In other words, we can assume that \eqref{eq:L2p-rigidity-recalled} holds with $\mathbf{R}_2=\mathbf{R}$. For the second part, observe that
\begin{align*}
    |\mathbf{R}_p - \mathbf{R}| \lesssim_p (\mathbf{E}_{\delta}(\mathbf{y}))^{\frac{1}{p}}.
\end{align*}
This is clear if $E_{\delta}(\mathbf{y}) > 1$, otherwise use H\"older's inequality and the assumption $\delta\leq1/2$ to get that
\begin{align*}
|\mathbf{R}_{p}-\mathbf{R}| & \lesssim_p \|\nabla\mathbf{y}-\mathbf{R}_{p}\|_{L^{p}(\Omega_{\delta})}+\|\nabla\mathbf{y}-\mathbf{R}\|_{L^{2}(\Omega_{\delta})} \lesssim_p (E_{\delta}(\mathbf{y}))^{\frac{1}{p}} + (E_{\delta}(\mathbf{y}))^{\frac{1}{2}} \lesssim(E_{\delta}(\mathbf{y}))^{\frac{1}{p}}
\end{align*}
by \prettyref{eq:L2p-rigidity-recalled} with $\mathbf{R}_2=\mathbf{R}$. 
Hence,
\begin{align*}
\int_{\Omega_{\delta}}|\nabla\mathbf{y}-\mathbf{R}|^{p} \,d\mathbf{x} \lesssim_{p} \int_{\Omega_{\delta}}|\nabla\mathbf{y}-\mathbf{R}_{p}|^{p}\,d \mathbf{x}+|\mathbf{R}_{p}-\mathbf{R}|^{p}\lesssim_{p} E_{\delta}(\mathbf{y})
\end{align*}
as claimed.
\end{proof}

\noindent We finish with some elementary properties of $\mathcal{R}_\delta(\mathbf{y})$. Recall $\mathbf{y}^\perp=\mathbf{J}\mathbf{y}$ where $\mathbf{J} = \mathbf{e}_2 \otimes \mathbf{e}_1 - \mathbf{e}_1 \otimes \mathbf{e}_2$.
\begin{lem}\label{skewLemma}
    Let $\delta\in(0,1)$.
	Given $\mathbf{y}\in  W^{1,2}(\Omega_\d;\R^2)$, 
    $\cR_{\delta}(\mathbf{y})=SO(2)$ if and only if $\int_{\Omega_\delta} \partial_1 \mathbf{y} - (\partial_2 \mathbf{y})^{\perp} \, d\mathbf{x} = \mathbf{0}$. Otherwise, $\cR_{\delta}(\mathbf{y})$ consists of a single rotation.   Either way, 
\begin{align}\label{eq:skwRotationIdentity}
\int_{\Omega_{\delta}} \skw \mathbf{R}^T \nabla \mathbf{y} \, d\mathbf{x} = \mathbf{0}
\end{align} 
for all $\mathbf{R} \in \mathcal{R}_\delta(\mathbf{y})$.
\end{lem}
\begin{proof}
Observe that
\begin{align*}
    \min_{\mathbf{Q}\in SO(2)}\,\int_{\Omega_\delta}|\nabla \mathbf{y} - \mathbf{Q}|^2 \, d\mathbf{x} &=  \int_{\Omega_\delta}|\nabla \mathbf{y}|^2 +2\, d\mathbf{x} - 2\max_{Q\in SO(2)}\, \int_{\Omega_\delta}\langle\mathbf{Q},\nabla\mathbf{y}\rangle \, d\mathbf{x}.
\end{align*}
Parameterizing $SO(2)$ through $\mathbf{Q}(\psi)= \cos \psi \mathbf{I} + \sin \psi \mathbf{J}$ with $\psi \in \R$   produces
    \begin{align*}
        \max_{\mathbf{Q}\in SO(2)}\, \int_{\Omega_\delta}\langle\mathbf{Q},\nabla\mathbf{y}\rangle \, d\mathbf{x} = \max_{\psi \in \mathbb{R}} \begin{pmatrix} \cos \psi \\ \sin \psi \end{pmatrix} \cdot \int_{\Omega_\delta} \partial_1 \mathbf{y} - (\partial_2 \mathbf{y})^{\perp} \, d\mathbf{x}.
    \end{align*}
Evidently, if $\int_{\Omega_\delta} \partial_1 \mathbf{y} - (\partial_2 \mathbf{y})^{\perp} \, d\mathbf{x} = \mathbf{0}$, then  all $\psi$ maximize this quantity. Otherwise, $\int_{\Omega_\delta} \partial_1 \mathbf{y} - (\partial_2 \mathbf{y})^{\perp} \, d\mathbf{x} \neq \mathbf{0}$ and  maximizers $\psi^{\star}$ solve 
    \begin{align*}
    \cos \psi^{\star} =  \frac{\mathbf{e}_1 \cdot \int_{\Omega_\delta} \partial_1 \mathbf{y} - (\partial_2 \mathbf{y})^{\perp} \, d\mathbf{x}}{|\int_{\Omega_\delta} \partial_1 \mathbf{y} - (\partial_2 \mathbf{y})^{\perp} \, d\mathbf{x}|}, \quad     \sin \psi^{\star} =  \frac{\mathbf{e}_2 \cdot \int_{\Omega_\delta} \partial_1 \mathbf{y} - (\partial_2 \mathbf{y})^{\perp} \, d\mathbf{x}}{|\int_{\Omega_\delta} \partial_1 \mathbf{y} - (\partial_2 \mathbf{y})^{\perp} \, d\mathbf{x}|}.
    \end{align*}
In this case, the optimal rotation $\mathbf{Q}^{\star} = \mathbf{Q}(\psi^{\star})$ is unique and $\mathcal{R}_\delta(\mathbf{y}) = \{ \mathbf{Q}^{\star} \}$.  
    
Next, we prove \eqref{eq:skwRotationIdentity}. Let $\mathbf{R} \in \mathcal{R}_\delta(\mathbf{y})$ and note that $f(\psi) = \int_{\Omega_\delta} \langle \mathbf{R} \mathbf{Q}(\psi), \nabla \mathbf{y} \rangle \, d\mathbf{x}$ is maximized at $\psi = 0$. Consequently, 
$f'(0) = \int_{\Omega_\delta} \langle \mathbf{J},\mathbf{R}^T \nabla \mathbf{y} \rangle \, d\mathbf{x} = 0$, which is equivalent to \eqref{eq:skwRotationIdentity}.
\end{proof}


\begin{lem}\label{uniquenessCor}
   Let $\mathbf{y}_{\d,\e} \in W^{1,p}(\Omega_\d;\R^2)$ satisfy $\lim_{\delta,\epsilon\to 0} E_{\delta}(\mathbf{y}_{\d,\e})  = 0$. The set $\cR_{\delta}(\mathbf{y}_{\d,\e})$ consists of a single rotation for all sufficiently small $\d,\e$.
\end{lem}
\begin{proof}
  Let $\mathbf{R}_{\d,\e}\in \cR_{\delta}(\mathbf{y}_{\d,\e})$. It follows from Corollary \ref{cor:mixed-rigidity-energy} that $\|\nabla \mathbf{y}_{\d,\e} - \mathbf{R}_{\d,\e}\|_{L^2(\Omega_{\delta})} \lesssim  (E_{\delta}(\mathbf{y}_{\delta, \epsilon}))^{\frac{1}{2}} \rightarrow 0$ as $\delta, \epsilon\to 0$.
From here, we find a constant $C >0$ independent of $\delta, \epsilon$ such that 
    \begin{align*}
        \Big|\int_{\Omega_{\delta}} \partial_1 \mathbf{y}_{\delta, \epsilon}  - (\partial_2 \mathbf{y}_{\delta, \epsilon})^{\perp} \, d\mathbf{x}\Big|  & \geq \Big|\int_{\Omega_{\delta}} \mathbf{R}_{\d,\e}\mathbf{e}_1  - (\mathbf{R}_{\d,\e}\mathbf{e}_2)^\perp\, d\mathbf{x}\Big| \\
        &\qquad  - \Big|\int_{\Omega_{\delta}} (\nabla \mathbf{y}_{\delta, \epsilon} - \mathbf{R}_{\d,\e})\mathbf{e}_1  -  \big((\nabla \mathbf{y}_{\delta, \epsilon}  - \mathbf{R}_{\d,\e})\mathbf{e}_2\big)^\perp \, d\mathbf{x}\Big| \\
        & \geq 2|\Omega_{\delta}| - C\|\nabla \mathbf{y}_{\delta, \epsilon} -\mathbf{R}_{\d,\e}\|_{L^2(\Omega_{\delta})} \geq |\Omega_\delta| > 0
    \end{align*}
for all $\delta, \epsilon$ sufficiently small. The result follows from Lemma \ref{skewLemma}. 
\end{proof}

\section{Roadmap for the analysis: the linear elastic wedge\label{sec:linear-analysis}}
We begin our derivation of the Flamant solution by analyzing the linear response of a truncated wedge. This will lay the foundation for the nonlinear analysis to come. As explained in the introduction, we do not assume any sort of self-similarity or symmetry of the admissible fields, and instead focus on the consequences of almost minimality, i.e., that the  elastic energy in the linear displacement problem, or the total potential energy in the linear force problem, is minimized to leading order in the truncation length $\delta$. We prove that the minimum energy in these problems is $\sim 1/\log\delta^{-1}$ and obtain a pair of limiting variational problems for the prefactors. We also show that the linear strain of almost minimizers is given by the Flamant solution at leading order as $\delta \to 0$. We turn to set up the linear elasticity problems and to state our results.

Recall the reference domain for the truncated wedge is 
\begin{equation*}
\Omega_{\delta}=\left\{ \mathbf{x}\in\mathbb{R}^{2}:r\in(\delta,1),\theta\in(\alpha,\beta)\right\}.
\end{equation*}
The linear elastic energy associated to the nonlinear energy \eqref{eq:nonlinear-energy-defn} from the introduction is
\begin{equation}
E_{\delta}^{\rm{lin}}(\mathbf{u})=\int_{\Omega_{\delta}}Q\left(\mathbf{e}(\mathbf{u})\right)\,d\mathbf{x}\label{eq:linear-elastic-energy}
\end{equation}
where $Q=\frac{1}{2}\left\langle  \cdot,\mathbb{C}\cdot\right\rangle$ and $\mathbb{C} = D^2  W(\mathbf{I})$.
Assumption (W1) from Section \ref{ssec:Setup} and the definition of the elasticity tensor $\mathbb{C}$ imply the major and minor symmetry rules $\C_{ijkl} = \C_{klij}$ and  $\C_{ijkl} = \C_{jikl} = \C_{ijlk}$ for $i,j,k,l \in\{1,2\}$. 
Assumption (W2) proves the existence of constants $C,C'>0$ such that
\begin{align}\label{QGrwoth} 
	C|\sym\,\mathbf{F}|^2 \leq Q(\mathbf{F}) = Q(\sym\, \mathbf{F}) \leq C' |\sym\,\mathbf{F}|^2
\end{align}
for all $\mathbf{F}\in \R^{2\times 2}$. Thus,  $Q(\nabla\mathbf{u})$ depends only on the symmetric part of the displacement gradient $\nabla\mathbf{u}$, i.e., the \emph{linear strain} $\mathbf{e}(\mathbf{u})=\sym\,\nabla\mathbf{u}$. Note $\mathbf{e}(\mathbf{u})=\mathbf{0}$ if and only if $\mathbf{u}$ is an \emph{infinitesimal rigid body motion}, i.e., $\mathbf{u}(\mathbf{x})=\mathbf{W}\mathbf{x}+\mathbf{c}$ for some two-by-two skew-symmetric matrix $\mathbf{W}\in\text{Skw}_2$ and $\mathbf{c}\in\mathbb{R}^2$.

The linear analogs of the nonlinear displacement and force problems from the introduction are as follows:

\vspace{1mm}
\noindent \textbf{Linear displacement problem.} Solve 
\begin{equation}
\mathcal{E}_{\delta}^{\text{lin,disp}} = \min_{\mathbf{u} \in \mathcal{A}^{\text{lin}}_{\delta}
}
\,E_{\delta}^{\rm{lin}}(\mathbf{u})\label{eq:linear-displacement-problem}
\end{equation}
for the admissible set
\begin{align}\label{admissible_linear_displ}
\cA^{\text{lin}}_\d\coloneqq\{\mathbf{u}\in W^{1,2}(\Omega_\d;\R^2) : \mathbf{u}(\mathbf{x})=\mathbf{u}_{\delta}^{\pm}(\theta)\ \text{at }r=\delta,1\}
\end{align}
where $(-,\delta)$ and $(+,1)$ are the corresponding pairs. The boundary displacements $\{\mathbf{u}_{\delta}^{\pm}\}$ belong to the trace space $W^{1/2,2}((\alpha,\beta);\mathbb{R}^{2})$
and satisfy 
\begin{align}\label{eq:linearBoundaryData}
\fint_{\alpha}^{\beta}\mathbf{u}_{\delta}^{\pm}\,d\theta\to\mathbf{u}_{0}^{\pm}\quad\text{and}\quad \sqrt{\log \frac{1}{\delta}} \left\| \mathbf{u}_{\delta}^{\pm}\right\| _{\dot{W}^{\frac12,2}((\alpha,\beta))} \rightarrow 0 
\end{align}
as $\delta\to0$, for some $\mathbf{u}_{0}^{\pm}\in\mathbb{R}^{2}$. Note \eqref{eq:linear-displacement-problem} has a unique minimizer for each $\delta\in (0,1)$, given our setup.

\vspace{1mm}
\noindent \textbf{Linear force problem.} Solve 
\begin{equation}\label{eq:linear-force-problem}
\mathcal{E}_{\delta}^{\text{lin,force}} =\min_{\mathbf{u}\in W^{1,2}(\Omega_{\delta};\mathbb{R}^{2})}
\,E_{\delta}^{\rm{lin}}(\mathbf{u})-\frac{1}{\log\frac{1}{\delta}}V^{\text{lin}}_{\delta}(\mathbf{u})
\end{equation}
for boundary forces
\begin{equation}\label{eq:linear-forces}
\mathbf{f}_{\delta}(\mathbf{x})=-\frac{1}{\delta}\mathbf{f}_{\delta}^{-}(\theta)\indicator{\{r=\delta\}}+\mathbf{f}_{\delta}^{+}(\theta)\indicator{\{r=1\}},\quad\mathbf{x}\in\partial\Omega_\delta
\end{equation}
and the work integral
\begin{equation}
V^{\text{lin}}_{\delta}\label{eq:work-integral}(\mathbf{u})=\int_{\partial\Omega_{\delta}}\mathbf{f}_{\delta}\cdot\mathbf{u}\,ds=\int_\alpha^\beta -\mathbf{f}_\delta^-\cdot\mathbf{u}|_{r=\delta}+\mathbf{f}_\delta^+\cdot\mathbf{u}|_{r=1}\,d\theta.
\end{equation}
Again, we note the slight but useful abuse of notation in the definition of the forces $\{\mathbf{f}_\delta\}$, which belong to the dual trace space $W^{1/2,2}(\partial\Omega_\delta;\mathbb{R}^2)'$. Likewise, 
$\{\mathbf{f}_{\delta}^{\pm}\}\subset W^{1/2,2}((\alpha,\beta);\mathbb{R}^{2})'$ and $\mathbf{u}|_{r=a}=\mathbf{u}(a\mathbf{e}_r(\cdot))$, all as in the introduction. 
We assume that
\begin{align}\label{eq:linearForceData}
\int_{\alpha}^{\beta}\mathbf{f}_{\delta}^{+}\,d\theta = \int_{\alpha}^{\beta}\mathbf{f}_{\delta}^{-}\,d\theta \to\mathbf{f}_{0}\quad\text{and}\quad \frac{1}{\sqrt{\log \frac{1}{\delta}}} \big\| \mathbf{f}_{\delta}^{\pm}\big\| _{\dot{W}^{\frac{1}{2},2}((\alpha,\beta))'} \rightarrow 0
\end{align}
as $\delta\to0$, for some $\mathbf{f}_{0}\in\mathbb{R}^{2}$. Finally, we impose the force and moment balance conditions
\begin{equation}\label{eq:linear-balance-laws}
\int_{\Omega_{\delta}} \mathbf{f}_{\delta} \, ds = \mathbf{0}\quad\text{and}\quad \int_{\partial\Omega_{\delta}}\mathbf{f}_{\delta}\cdot\mathbf{x}^\perp\,ds=0,
\end{equation}
the first of which follows from \eqref{eq:linearForceData}, and the second of which is an important additional assumption for linear elasticity. These balance conditions help ensure the solvability of \eqref{eq:linear-force-problem}. Indeed, given our setup, for each $\delta\in (0,1)$ there exists a minimizer which is unique up to the addition of an infinitesimal rigid body motion.

\vspace{1mm}
The following is the main result of this section. Recall the Flamant quadratic forms 
\begin{equation}\label{eq:Flquadform-recalled}
Q_{\text{Fl}}(\mathbf{u}) = \frac{1}{2} \mathbf{u} \cdot \mathbf{K}_{\text{Fl}} \mathbf{u}\quad\text{and}\quad  Q^{\ast}_{\text{Fl}}(\mathbf{f}) = \frac{1}{2} \mathbf{f} \cdot \mathbf{K}_{\text{Fl}}^{-1} \mathbf{f}
\end{equation}
from \eqref{Flamant_energies}, where $\mathbf{K}_{\text{Fl}} = \int_{\alpha}^{\beta} \frac{ \mathbf{e}_r \otimes \mathbf{e}_r}{( \mathbb{C}^{-1})_{rrrr}} \, d\theta$. Recall also the Flamant displacement field
\begin{equation}
\begin{aligned}\label{Flamant_displacementSec3}
&\mathbf{u}_{\text{Fl}}(\mathbf{x}) = \mathbf{u}_0(r) + \mathbf{u}_1(\theta) \quad \text{ with } \quad \mathbf{u}_0(r) = (\log r ) \mathbf{a}  \quad \text{ and } \\
&\qquad \mathbf{u}_1(\theta) = \left[ \int_{\alpha}^{\theta} \Big(\frac{ 2(\mathbb{C}^{-1})_{r\theta rr} }{(\mathbb{C}^{-1})_{rrrr}} \mathbf{e}_{r} \otimes \mathbf{e}_{r}  - \mathbf{e}_r \otimes \mathbf{e}_{\theta} +\frac{(\mathbb{C}^{-1})_{\theta\theta rr}}{(\mathbb{C}^{-1})_{rrrr}}  \mathbf{e}_{\theta}\otimes \mathbf{e}_r\Big) \, d\tilde{\theta} \right] \mathbf{a}
\end{aligned}
\end{equation}
from \eqref{Flamant_displacement} and the Flamant stress
 \begin{equation}
     \begin{aligned}\label{eq:aniso_Flamant_Sec3}
      \boldsymbol{\sigma}_{\text{Fl}}(\mathbf{x}) = \mathbb{C} \mathbf{e}(\mathbf{u}_{\text{Fl}}(\mathbf{x})) =  \frac{\Sigma(\theta) }{r} \mathbf{e}_r\otimes\mathbf{e}_r \quad \text{ with } \quad  \Sigma = \frac{\mathbf{a} \cdot \mathbf{e}_r}{(\mathbb{C}^{-1})_{rrrr}}   
     \end{aligned}
 \end{equation}
from \eqref{eq:aniso_Flamant_intro}.
\begin{thm}\label{thm:main-result-linearized}
The linear displacement and force problems  \prettyref{eq:linear-displacement-problem} and \prettyref{eq:linear-force-problem}
obey Flamant-type asymptotics as $\delta\to0$. Precisely, the following statements hold:
\begin{enumerate}[leftmargin=*]
\item[(I)] The minimum energies in these problems obey
\begin{equation*}
\log\Big(\frac{1}{\delta}\Big) \mathcal{E}_{\delta}^{\emph{lin,disp}}  \to \QFl(\mathbf{u}_{0}^{+}-\mathbf{u}_{0}^{-})  \quad  \text{ and } \quad \log\Big(\frac{1}{\delta}\Big)\mathcal{E}_{\delta}^{\emph{lin,force}}  \to - \QFl^\ast(\mathbf{f}_0).
\end{equation*}
\item[(II)] Let $\{\mathbf{u}_{\delta}\}$ be almost minimizing in the linear displacement or force problem in that
\begin{equation*}
E_{\delta}^{\rm{lin}}(\mathbf{u}_{\delta})=\mathcal{E}_{\delta}^{\emph{lin,disp}}  +o\Big(\frac{1}{\log\frac{1}{\delta}}\Big)\quad\text{or}\quad E_{\delta}^{\rm{lin}}(\mathbf{u}_{\delta})-\frac{1}{\log\frac{1}{\delta}}V^{\emph{lin}}_{\delta}(\mathbf{u}_{\delta})=\mathcal{E}_{\delta}^{\emph{lin,force}}+o\Big(\frac{1}{\log\frac{1}{\delta}}\Big)
\end{equation*}
respectively. The corresponding linear strains $\mathbf{e}(\mathbf{u}_\d)$ and  stresses $\boldsymbol{\sigma}_{\delta}=\C\mathbf{e}(\mathbf{u}_{\delta})$
obey
\begin{align}\label{eq:stressStrain_L2_statement}
\big\| \mathbf{e}(\mathbf{u}_{\delta}) - \frac{1}{\log\frac{1}{\d}}\mathbf{e}(\uFl) \big\| _{L^{2}(\Omega_{\delta})}\ll\frac{1}{\sqrt{\log\frac{1}{\delta}}}\quad\text{and}\quad\big\| \boldsymbol{\sigma}_{\delta}-\frac{1}{\log\frac{1}{\delta}}\boldsymbol{\sigma}_{\rm{Fl}}\big\| _{L^{2}(\Omega_{\delta})}\ll\frac{1}{\sqrt{\log\frac{1}{\delta}}}
\end{align}
for the Flamant displacement $\uFl$ and stress $\boldsymbol{\sigma}_{\rm{Fl}}$ in \eqref{Flamant_displacementSec3} and \eqref{eq:aniso_Flamant_Sec3} with
\begin{align*}
\mathbf{a} = \begin{cases}
\mathbf{u}_{0}^{+}-\mathbf{u}_{0}^{-} & \text{ in the displacement problem} \\
\mathbf{K}_{\emph{Fl}}^{-1} \mathbf{f}_0 & \text{ in the force problem}
\end{cases}.
\end{align*}
\end{enumerate}
\end{thm}
\begin{rem} As with Theorem \ref{thm:main-result}, it is possible to use Korn's inequality (\prettyref{cor:Korn-inequality}) and the boundary conditions of the linear problems to improve the first part of \eqref{eq:stressStrain_L2_statement} to refer to the full displacement gradient $\nabla\mathbf{u}_\delta$. However, \eqref{eq:stressStrain_L2_statement} will be used in its current form elsewhere in the paper, and these improvements will be achieved for the nonlinear problems in Section \ref{sec:strong-convergence} anyways, so we prefer to leave it as is. 
\end{rem}

The proof of \prettyref{thm:main-result-linearized} serves as a convenient warmup for the nonlinear problems. In particular, it allows us to focus on the treatment of the $1/r$ singularity in the strain. We proceed as follows:
\begin{enumerate}[leftmargin=*]
\item First, we rescale the truncated wedge to a fixed domain via a logarithmic change of coordinates, and use Korn's inequality to deduce weak compactness  from the energy bounds
$E_{\delta}^{\rm{lin}}(\mathbf{u}_{\delta})\lesssim(\log\frac{1}{\delta})^{-1}$ in the displacement problem and  $E_{\delta}^{\rm{lin}}(\mathbf{u}_{\delta})-(\log\frac{1}{\delta})^{-1}V^{\text{lin}}_{\delta}(\mathbf{u}_{\delta})\lesssim (\log\frac{1}{\delta})^{-1}$ in the force problem.
\item Second, we identify the limits of the rescaled minimum energies $\log(\frac{1}{\delta})\mathcal{E}_{\delta}^{\text{lin,disp}}$ and $\log(\frac{1}{\delta})\mathcal{E}_{\delta}^{\text{lin,force}}$ via a pair of limiting variational problems which are solved by the Flamant solution. 
\item Finally, we improve the convergence as $\delta\to 0$ of the strain and stress of almost minimizers from weak to strong $L^2$. This is possible thanks to the quadratic nature of the limit problems in (2).
\end{enumerate} 
We carry out these steps in Sections \ref{sec:logChangeSing}-\ref{sec:AlmostMininimizers}.

\subsection{Compactness in logarithmic variables}\label{sec:logChangeSing}

First we prove a compactness result for displacements whose energy is $\lesssim (\log \frac{1}{\delta})^{-1}$. As discussed above and in the introduction, the key is to pass to the limit with a logarithmic change of variables. Consider the change of variables 
\begin{equation}\label{eq:change-of-vars}
	\mathbf{u}(\mathbf{x})=\mathbf{U}(\rho,\theta)\quad\text{for}\quad\rho=\frac{\log r}{\log\frac{1}{\delta}}
\end{equation}
where  $(r,\theta)$ are the usual polar coordinates of $\mathbf{x}$ centered at the ideal tip of the wedge. The truncated wedge  
\[
\Omega_{\delta}=\{\mathbf{x}:r\in(\delta,1),\theta\in(\alpha,\beta)\}
\]
corresponds to the fixed rectangular domain
\[
R=\left\{ (\rho,\theta):\rho\in(-1,0),\theta\in(\alpha,\beta)\right\} 
\]
with $\rho=-1,0$ corresponding to $r=\delta,1$, respectively. As $\rho$ is an increasing function of $r$, we call $\mathbf{e}_{\rho}=\mathbf{e}_{r}$ and define the \emph{logarithmic gradient}
\begin{equation}
D^{\delta}\mathbf{U}=\partial_{\rho}\mathbf{U}\otimes\mathbf{e}_{\rho}+\log\Big(\frac{1}{\delta}\Big)\partial_{\theta}\mathbf{U}\otimes\mathbf{e}_{\theta}.\label{eq:definition-of-rescaled-gradient}
\end{equation}
This gradient is related to the standard one via
\begin{equation}
D^{\delta}\mathbf{U}=|\mathbf{x}|\log\Big(\frac{1}{\delta}\Big)\nabla\mathbf{u}\label{eq:gradient-rule}
\end{equation}
where the two sides are evaluated at corresponding
$(\rho,\theta)$ and $\mathbf{x}$.

Under the transformation \eqref{eq:change-of-vars}, the linear elastic energy \prettyref{eq:linear-elastic-energy} and work integral \prettyref{eq:work-integral}
obey
\begin{align}\label{eq:EdeltaChange}
E_{\delta}^{\rm{lin}}(\mathbf{u})  =\frac{1}{\log\frac{1}{\delta}}\int_{R}Q(\mathrm{sym}\,D^{\delta}\mathbf{U})\,d\rho d\theta \quad \text{ and } \quad V^{\text{lin}}_{\delta}(\mathbf{u}) =\int_{\alpha}^{\beta} \mathbf{f}_{\delta}^{+}\cdot\mathbf{U}|_{\rho=0}-\mathbf{f}_{\delta}^{-}\cdot\mathbf{U}|_{\rho=-1}\,d\theta.
\end{align}
Other useful relationships include 
\begin{align}\label{eq:strainTologStrain}
	\int_{\Omega_{\delta}}|\mathbf{e}(\mathbf{u})|^{2}\,d\mathbf{x}=\frac{1}{\log\frac{1}{\delta}}\int_{R}|\mathrm{sym}\,D^{\delta}\mathbf{U}|^{2}\,d\rho d\theta\quad\text{and}\quad
	\int_{\Omega_{\delta}}|\nabla\mathbf{u}|^{2}\,d\mathbf{x}=\frac{1}{\log\frac{1}{\delta}}\int_{R}|D^{\delta}\mathbf{U}|^{2}\,d\rho d\theta.
\end{align} 
More generally, if $g_{\delta}(\mathbf{x}) = G_{\delta}(\rho, \theta)$, then 
\begin{align*}
\int_{\Omega_{\delta}} \frac{g_{\delta}}{|\mathbf{x}|^2\log\tfrac{1}{\delta}}  \;d \mathbf{x}  = \int_R G_{\delta} \;d \rho d \theta.
\end{align*}


Before proceeding to compactness, we comment on a few details that are specific to the force problem. First, we control the work integral by the logarithmic gradient of the displacement.
\begin{lem}
\label{lem:force-boundary-control}For every $\mathbf{U}\in W^{1,2}(R;\R^2)$ and $\mathbf{f}^{\pm}\in W^{1/2, 2}((\alpha,\beta);\R^{2})'$ with $\int_{\alpha}^{\beta}\mathbf{f}^{+}\,d\theta=\int_{\alpha}^{\beta}\mathbf{f}^{-}\,d\theta$,
\[
\Big|\int_{\alpha}^{\beta}\mathbf{f}^{+}\cdot\mathbf{U}|_{\rho=0}-\mathbf{f}^{-}\cdot\mathbf{U}|_{\rho=-1}\,d\theta\bigg|\lesssim \bigg(\sum_{(\cdot) = \pm}\frac{1}{\sqrt{\log\frac{1}{\delta}}}\big\| \mathbf{f}^{(\cdot)} \big\|_{\dot{W}^{\frac{1}{2},2}((\alpha, \beta))'} +\Big|\int_{\alpha}^{\beta}\mathbf{f}^{+}\,d\theta\Big|\bigg)\| D^{\delta}\mathbf{U}\|_{L^{2}(R)}
\]
for all $\delta\in(0,1/2)$.
\end{lem}

\begin{proof}
The proof makes use of trace estimates in the appendix on the $r = \delta,1$ boundaries of the wedge.
In particular,  Lemma \ref{traceToGradLemma} for $p=2$  and \eqref{eq:strainTologStrain} furnish 
\begin{align}\label{control_bc_log1}
	\|\mathbf{U}|_{\rho = a}\|_{\dot{W}^{\frac12,2}((\alpha,\beta))} &\sim \| \mathbf{u}|_{r= a'}\|_{\dot{W}^{\frac12,2}((\alpha,\beta))} \lesssim \|\nabla\mathbf{u}\|_{L^{2}(\Omega_{\delta})} = \frac{1}{\sqrt{\log \frac{1}{\delta}}} \big\| D^{\delta} \mathbf{U} \big\|_{L^2(R)}
\end{align}
for $(a,a') = (-1, \delta)$ and $(0,1)$.
It also holds that
\begin{align}\label{control_bc_log2}
\Big|\fint_{\alpha}^{\beta}\mathbf{U}|_{\rho=0}-\mathbf{U}|_{\rho=-1}\,d\theta\Big| & \leq\fint_{\alpha}^{\beta} \int_{-1}^{0} |\partial_{\rho}\mathbf{U}|\,d\rho d\theta  \lesssim
\| D^{\delta}\mathbf{U}\|_{L^{2}(R)}
\end{align}
by the definition of the logarithmic gradient \eqref{eq:definition-of-rescaled-gradient}.

The inequality in the lemma follows from \eqref{control_bc_log1} and \eqref{control_bc_log2} after a manipulation. Subtract averages and use the assumption $\int_{\alpha}^{\beta}\mathbf{f}^{+}\,d\theta=\int_{\alpha}^{\beta}\mathbf{f}^{-}\,d\theta$ to get that
\begin{align}\label{force_vs_boundaryvalue}
\int_{\alpha}^{\beta}\mathbf{f}^{+}\cdot\mathbf{U}|_{\rho=0} - \mathbf{f}^-\cdot \mathbf{U}|_{\rho=-1}\,d\theta
&=\int_{\alpha}^{\beta} \mathbf{f}^{+}\,d\theta\cdot\fint_{\alpha}^{\beta}\mathbf{U}|_{\rho=0} - \mathbf{U}|_{\rho=-1}\,d\theta\nonumber\\
&\quad + \int_{\alpha}^{\beta}\Big(\mathbf{f}^{+}-\fint_{\alpha}^{\beta} \mathbf{f}^{+}\,d\tilde{\theta}\Big)\cdot\Big(\mathbf{U}|_{\rho=0} - \fint_\alpha^\beta \mathbf{U}|_{\rho=0}\,d\tilde{\theta}\Big)\, d\theta\nonumber \\
&\quad+ \int_\alpha^\beta\Big(\mathbf{f}^{-}-\fint_{\alpha}^{\beta} \mathbf{f}^{-}\,d\tilde{\theta}\Big)\cdot\Big(\mathbf{U}|_{\rho=-1} - \fint_\alpha^\beta \mathbf{U}|_{\rho=-1}\,d\tilde{\theta}\Big)\,d\theta.
\end{align}
The first term in \eqref{force_vs_boundaryvalue} is controlled via \eqref{control_bc_log2}, giving
$$\Big|\int_{\alpha}^{\beta} \mathbf{f}^{+}\,d\theta\cdot\fint_{\alpha}^{\beta}\mathbf{U}|_{\rho=0} - \mathbf{U}|_{\rho=-1}\,d\theta\Big|\lesssim \Big|\int_{\alpha}^{\beta} \mathbf{f}^{+}\,d\theta\Big|\|D^\d \mathbf{U}\|_{L^2(R)}.$$
The second term in \eqref{force_vs_boundaryvalue} is handled with the help of \eqref{control_bc_log1}. In particular, we have by duality that  
\begin{align*}
  \Big|\int_{\alpha}^{\beta}\Big(\mathbf{f}^{+}-\fint_{\alpha}^{\beta} \mathbf{f}^{+}\,d\tilde{\theta}\Big)&\cdot\Big(\mathbf{U}|_{\rho=0} - \fint_\alpha^\beta \mathbf{U}|_{\rho=0}\,d\tilde{\theta}\Big)\, d\theta\Big| \\
  &\leq \Big\|\mathbf{f}^+ - \fint_\alpha^\beta \mathbf{f}^+\, d\theta\Big\|_{\dot{W}^{\frac{1}{2},2}((\alpha,\beta))'}\Big\|\mathbf{U}|_{\rho=0} - \fint_\alpha^\beta \mathbf{U}|_{\rho=0}\,d\theta\Big\|_{\dot{W}^{\frac{1}{2},2}((\alpha,\beta))}\\
  &= \big\|\mathbf{f}^+\big\|_{\dot{W}^{\frac{1}{2},2}((\alpha,\beta))'}\|\mathbf{U}|_{\rho=0}\|_{\dot{W}^{\frac12,2}((\alpha,\beta))} \lesssim  \frac{1}{\sqrt{\log\frac{1}{\d}}}\big\|\mathbf{f}^+\big\|_{\dot{W}^{\frac{1}{2},2}((\alpha,\beta))'}\|D^\d \mathbf{U}\|_{L^2(R)}
\end{align*}
where \eqref{control_bc_log1}  is used at the end. The third term in \eqref{force_vs_boundaryvalue} is estimated analogously.
\end{proof}
\color{black}

Next, we observe that compactness in the force problem can only hold up to the addition of an infinitesimal rigid body motion 
\[
\mathbf{u}_{\text{rig}}(\mathbf{x})=\mathbf{W}\mathbf{x} + \mathbf{c}\quad \text{with}\quad \mathbf{W}\in \text{Skw}_2, \mathbf{c}\in \R^2
\]
due to the lack of boundary conditions.  Indeed, subtracting such a motion from $\mathbf{u}$ leaves both the elastic energy and the work of the forces invariant, the latter following from the balance conditions \eqref{eq:linear-balance-laws}. 
As we intend to obtain compactness in logarithmic coordinates, we must be able to recognize infinitesimal rigid body motions in these coordinates. Since $\mathbf{x}=r\mathbf{e}_{r}$ and $r=\delta^{-\rho}$ by \eqref{eq:change-of-vars}, 
	\begin{align}\label{inf_rigid_body_motions_log}
		\mathbf{u}_{\text{rig}}(\mathbf{x})=\mathbf{U}_{\text{rig}}(\rho,\theta) = \omega \d^{-\rho} \mathbf{e}_\theta + \mathbf{c} \quad \text{for}\quad \omega\in\mathbb{R}, \mathbf{c}\in \R^2
	\end{align}	
where $\mathbf{W} \mathbf{v} = \omega \mathbf{v}^{\perp}$.
By \eqref{eq:strainTologStrain}, Korn's inequality in Corollary \ref{cor:Korn-inequality} with $p =2$ reads in these coordinates as
\begin{align}\label{Korn_log}
		\min_{\omega\in\mathbb{R}}\,\int_{R}|D^{\delta}(\mathbf{U}-\omega\delta^{-\rho}\mathbf{e}_{\theta})|^{2}\,d\rho d\theta\lesssim \int_{R}|\mathrm{sym}\,D^{\delta}\mathbf{U}|^{2}\,d\rho d\theta
\end{align}
for all $\mathbf{U}\in W^{1,2}(R;\R^2)$ and $\delta\in(0,1/2)$. 

We come now to compactness. 
\begin{prop}\label{prop:linear-compactness}
Let $\{\mathbf{u}_{\delta}\}$ be admissible in the linear displacement or force problem \prettyref{eq:linear-displacement-problem}
or \prettyref{eq:linear-force-problem} and assume that
\[
\limsup_{\delta\to0}\,\log\Big(\frac{1}{\delta}\Big)E_{\delta}^{\rm{lin}}(\mathbf{u}_{\delta})<\infty\quad\text{or}\quad\limsup_{\delta\to0}\,\log\Big(\frac{1}{\delta}\Big)\Big(E_{\delta}^{\rm{lin}}(\mathbf{u}_{\delta})-\frac{1}{\log\frac{1}{\delta}}V^{\emph{lin}}_{\delta}(\mathbf{u}_{\delta})\Big)<\infty
\]
respectively. Define the corresponding sequence $\{\mathbf{U}_\d\} \subset W^{1,2}(R; \mathbb{R}^2)$ via \eqref{eq:change-of-vars}. There exist $\mathbf{m}\in L^{2}(R;\mathbb{R}^{2})$ and $\mathbf{U}\in W^{1,2}(R;\R^2)$ with $\partial_\theta \mathbf{U}=\mathbf{0}$, and a subsequence  of $\{\mathbf{U}_{\delta}\}$ (not relabeled)  with the following properties:

$i)$ (Displacement problem) It holds that
\begin{align}\label{eq:firstDesiredCompactness}
\mathbf{U}_{\delta}\rightharpoonup\mathbf{U},\quad\partial_{\rho}\mathbf{U}_{\delta}\rightharpoonup\partial_{\rho}\mathbf{U},\quad\log\Big(\frac{1}{\delta}\Big)\partial_{\theta}\mathbf{U}_{\delta}\rightharpoonup\mathbf{m},\quad\partial_{\theta}\mathbf{U}_{\delta}\to\mathbf{0}\quad \text{ in }L^{2}(R;\R^2)
\end{align}
and the boundary traces satisfy 
\begin{align}\label{eq:boundaryTermsLinearTrace}
\mathbf{U}_{\delta}|_{\rho = -1} \rightharpoonup \mathbf{U}|_{\rho = -1} =  \mathbf{u}_0^{-},  \quad  \mathbf{U}_{\delta}|_{\rho = 0} \rightharpoonup \mathbf{U}|_{\rho = 0} =  \mathbf{u}_0^{+}  \quad\text{ in }L^2((\alpha,\beta);\R^2).
\end{align}

$ii)$ (Force problem) There is a sequence of infinitesimal rigid body motions $\{\mathbf{U}_{\emph{rig},\d}\}$ as in \eqref{inf_rigid_body_motions_log} such that
\eqref{eq:firstDesiredCompactness} holds with $\mathbf{U}_\d$ replaced by $\mathbf{U}_\d - \mathbf{U}_{\emph{rig},\d}$. The work integral satisfies
\begin{align}\label{eq:secondDesiredCompactnessForce}
V^{\rm{lin}}_{\delta}(\mathbf{u}_\delta)=\int_{\alpha}^{\beta}\mathbf{f}_{\delta}^{+}\cdot\mathbf{U}_{\delta}|_{\rho=0}-\mathbf{f}_{\delta}^{-}\cdot\mathbf{U}_{\delta}|_{\rho=-1}\,d\theta\to\mathbf{f}_{0}\cdot\left(\mathbf{U}|_{\rho=0}-\mathbf{U}|_{\rho=-1}\right).	
\end{align}
\end{prop}
\begin{proof}
By the change-of-variables formulas  in \eqref{eq:EdeltaChange} and since $Q(\mathbf{F})\gtrsim|\sym\,\mathbf{F}|^2$ per \eqref{QGrwoth}, the assumption is that
\begin{equation}
\begin{aligned}\label{eq:startingCompactness}
 & \limsup_{\delta\to0}\,\int_{R}\left|\mathrm{sym}\,D^{\delta}\mathbf{U}_{\delta}\right|^{2}\,d\rho d\theta<\infty,\\
 & \limsup_{\delta\to0}\,\int_{R}\left|\mathrm{sym}\,D^{\delta}\mathbf{U}_{\delta}\right|^{2}\,d\rho d\theta-\int_{\alpha}^{\beta}\mathbf{f}_{\delta}^{+}\cdot\mathbf{U}_{\delta}|_{\rho=0}-\mathbf{f}_{\delta}^{-}\cdot\mathbf{U}_{\delta}|_{\rho=-1}\,d\theta<\infty
\end{aligned}
\end{equation}
in the displacement and force problems, respectively. We treat each problem in turn.

\textit{i}) In the displacement problem, $\mathbf{U}_{\delta}|_{\rho = 0} = \mathbf{u}_\delta^+$ since $\rho=0$ corresponds to $r=1$ and $\mathbf{u}_{\delta} \in \mathcal{A}^{\text{lin}}_{\delta}$; see \eqref{admissible_linear_displ}. Thus, the Poincar\'e-style inequalities 
\begin{align*}
&\Big\|\mathbf{U}_\d-\fint_{\alpha}^{\beta}\mathbf{U}_\d(\cdot,\theta)\,d\theta\Big\|_{L^{2}(R)}  \lesssim \|\partial_{\theta}\mathbf{U}_\d\|_{L^{2}(R)} \quad \text{and} \quad  \Big\|\fint_{\alpha}^{\beta}\mathbf{U}_\d(\cdot,\theta)\,d\theta\Big\|_{L^{2}(R)}  \lesssim \|\partial_{\rho}\mathbf{U}_\d\|_{L^{2}(R)}+\Big|\fint_{\alpha}^{\beta} \mathbf{u}_\d^+\,d\theta\Big|
\end{align*}
combine to give
\begin{align}\label{poincare_log}
\|\mathbf{U}_\d\|_{L^2(R)}\lesssim \|D^{\delta}\mathbf{U}\|_{L^{2}(R)}+\Big|\fint_{\alpha}^{\beta} \mathbf{u}_\d^+\,d\theta\Big|\lesssim \|D^{\delta}\mathbf{U}\|_{L^{2}(R)}+ 1
\end{align}
using that $|\partial_\rho \mathbf{U}_\d| \vee |\partial_\theta \mathbf{U}_\d| \lesssim |D^\d \mathbf{U}_\d|$ for small enough $\delta$ and the first part of \eqref{eq:linearBoundaryData}. Next, we employ \eqref{eq:strainTologStrain} to get a $(\rho,\theta)$-version of the second inequality in Corollary \ref{cor:Korn-inequality} (with $p=2$) of the form
\begin{align*}
\| D^{\delta}\mathbf{U}_{\delta}\|_{L^{2}(R)} & \lesssim\|\mathrm{sym}\,D^{\delta}\mathbf{U}_{\delta}\|_{L^{2}(R)}+\sqrt{\log\frac{1}{\delta}}\|\mathbf{u}_{\delta}^+\|_{\dot{W}^{\frac12,2}((\alpha,\beta))}\lesssim \|\mathrm{sym}\,D^{\delta}\mathbf{U}_{\delta}\|_{L^{2}(R)}+1.
\end{align*}
Note we used the second part of \eqref{eq:linearBoundaryData}.
Hence,
\begin{equation}
\limsup_{\delta\to0}\,\|\mathbf{U}_{\delta}\|_{L^{2}(R)}+\| D^{\delta}\mathbf{U}_{\delta}\|_{L^{2}(R)}<\infty.\label{eq:rescaled_H1_bd}
\end{equation}
The claimed convergences in \eqref{eq:firstDesiredCompactness} now follow from standard compactness
theorems and the definition of $D^{\delta}\mathbf{U}_{\delta}$ in \prettyref{eq:definition-of-rescaled-gradient}. 

Regarding the boundary terms in \eqref{eq:boundaryTermsLinearTrace}, $\mathbf{U}_{\delta} |_{\rho =0} \rightharpoonup  \mathbf{U}|_{\rho = 0}$ weakly in $L^2((\alpha,\beta);\mathbb{R}^2)$ by the trace theorem. That $\mathbf{U}|_{\rho = 0} = \mathbf{u}_0^+$ is automatic because $\mathbf{U}$ is independent of $\theta$ and the first  part of \eqref{eq:linearBoundaryData} furnishes
\begin{align*}
\mathbf{U}|_{\rho=0} =\fint_\alpha^\beta \mathbf{U}|_{\rho=0} \, d\theta = \lim_{\delta, \epsilon \rightarrow 0} \fint_\alpha^{\beta}  \mathbf{U}_{\delta} |_{\rho = 0} \, d\theta = \lim_{\delta \rightarrow 0} \fint_\alpha^{\beta} \mathbf{u}_{\delta}^+ \, d \theta  = \mathbf{u}_0^+.
\end{align*}
The same argument works for $\mathbf{U}_{\delta} |_{\rho = -1}$.

$ii)$ For the force problem, we apply Korn's inequality in logarithmic coordinates \eqref{Korn_log} to $\mathbf{U}_\delta$ and let $\mathbf{U}_{\text{rig},\d}$ be an associated infinitesimal rigid body of the form \eqref{inf_rigid_body_motions_log}. Let $\omega_{\delta}\in\mathbb{R}$ minimize the left-hand side of this inequality, and select $\mathbf{c}_{\delta}\in\mathbb{R}^2$ such that  $\fint_R (\mathbf{U}_{\delta} - \mathbf{U}_{\text{rig},\delta} )\,d \rho d\theta = \mathbf{0}$. It follows that
\begin{align}\label{eq:preLinCompact1}
	\|\mathbf{U}_{\delta}-\mathbf{U}_{\text{rig},\d}\|_{L^{2}(R)} & \lesssim\| D^{\delta}(\mathbf{U}_{\delta} - \mathbf{U}_{\text{rig},\d})\|_{L^{2}(R)}\lesssim\|\mathrm{sym}\,D^{\delta}\mathbf{U}_{\delta}\|_{L^{2}(R)}.
\end{align}
Note in the first step we used Poincar\'e's inequality with a constant that is independent of $\delta$. This is possible because $|\partial_\rho \mathbf{U}_\d| \vee |\partial_\theta \mathbf{U}_\d| \lesssim |D^\d \mathbf{U}_\d|$ for small enough $\delta$.
Since the work integral $V^{\text{lin}}_{\delta}$ is invariant under subtraction of an infinitesimal rigid body motion, Lemma \ref{lem:force-boundary-control} yields
\begin{equation}
\begin{aligned}\label{eq:preLinCompact2}
\Big|\int_{\alpha}^{\beta}\mathbf{f}_{\delta}^{+}\cdot&\mathbf{U}_{\delta}|_{\rho=0}-\mathbf{f}_{\delta}^{-}\cdot\mathbf{U}_{\delta}|_{\rho=-1}\,d\theta\Big| \\
&\lesssim \bigg(\sum_{(\cdot) = \pm}\frac{1}{\sqrt{\log\frac{1}{\delta}}}\big\|\mathbf{f}_{\delta}^{(\cdot)} \big\|_{\dot{W}^{\frac{1}{2},2}((\alpha, \beta))'}+\Big|\int_{\alpha}^{\beta}\mathbf{f}_{\delta}^{+}\,d\theta\Big|\bigg)\| D^{\delta}(\mathbf{U}_{\delta}-\mathbf{U}_{\text{rig},\d})\|_{L^{2}(R)}.
\end{aligned}
\end{equation}
The second inequality in \eqref{eq:startingCompactness} holds when we replace $\mathbf{U}_{\delta}$ by $\mathbf{U}_{\delta} - \mathbf{U}_{\text{rig},\delta}$. Furthermore, its left-hand side can be bounded from below using \eqref{eq:preLinCompact1} and \eqref{eq:preLinCompact2}, from which we obtain that
\[
\limsup_{\delta\to0}\,\| D^{\delta}(\mathbf{U}_{\delta}-\mathbf{U}_{\text{rig},\delta})\|_{L^{2}(R)}^{2}-c_{1}\| D^{\delta}(\mathbf{U}_{\delta}-\mathbf{U}_{\text{rig},\delta})\|_{L^{2}(R)}<\infty
\]
for some $c_{1}=c_{1}(\{\mathbf{f}_{\delta}^{\pm}\})$ that is bounded independently of $\delta$ by \eqref{eq:linearForceData}.  Thus, \eqref{eq:rescaled_H1_bd} holds for $\mathbf{U}_{\delta}-\mathbf{U}_{\text{rig},\delta}$ and the desired compactness follows.  

Finally, we pass to the limit in the work integral to obtain \eqref{eq:secondDesiredCompactnessForce}. Note for the subsequence that $(\mathbf{U}_{\delta}-\mathbf{U}_{\text{rig},\delta})|_{\rho = -1,0} \rightharpoonup \mathbf{U}|_{\rho= -1,0}$ weakly in $L^2((\alpha,\beta);\mathbb{R}^2)$ by the trace theorem.
The desired result follows from \eqref{force_vs_boundaryvalue} applied with $\mathbf{f}^\pm = \mathbf{f}^\pm_\d$ and $\mathbf{U} = \mathbf{U}_\d-\mathbf{U}_{\text{rig},\d}$. 
The last two terms in \eqref{force_vs_boundaryvalue} vanish by the boundedness of $\|D^\d (\mathbf{U}_\d-\mathbf{U}_{\text{rig},\d})\|_{L^2(R)}$ and the second assumption in \eqref{eq:linearForceData}. The first term converges as claimed, by the first assumption in \eqref{eq:linearForceData}, the weak convergence of $\{(\mathbf{U}_\d-\mathbf{U}_{\text{rig},\d})|_{\rho=-1,0}\}$, and because $\partial_\theta \mathbf{U} = \mathbf{0}$. 
\end{proof}

\subsection{The limit problems}

This section introduces the limiting linear displacement and force problems that appear when $\delta \to 0$.
It is easy to guess the form of the limit problems using the compactness results of Proposition \ref{prop:linear-compactness}. For a complete picture, the reader may wish to compare the present discussion with the derivation using stresses in the introduction following Theorem \ref{thm:main-result} (see Remark \ref{rem:duality}).

According to the compactness result of the prior section, admissible sequences with $E_{\delta}^{\rm{lin}}(\mathbf{u}_{\delta})\lesssim(\log\frac{1}{\delta})^{-1}$ in the displacement problem or  $E_{\delta}^{\rm{lin}}(\mathbf{u}_{\delta})-(\log\frac{1}{\delta})^{-1}V^{\text{lin}}_{\delta}(\mathbf{u}_{\delta})\lesssim (\log\frac{1}{\delta})^{-1}$ in the force problem admit subsequences whose logarithmic counterparts $\mathbf{U}_{\delta}$ defined using \eqref{eq:change-of-vars} obey
\[
\sym \, D^{\delta} \mathbf{U}_{\delta} \rightharpoonup \partial_{\rho} \mathbf{U} \odot \mathbf{e}_{\rho} + \mathbf{m} \odot \mathbf{e}_{\theta} \quad\text{weakly in }L^2(R;\mathbb{R}^{2\times 2})
\]
for some $\mathbf{U}=\mathbf{U}(\rho)$ and $\mathbf{m}=\mathbf{m}(\rho,\theta)$. Proposition \ref{prop:linear-compactness} also explains how to pass to the limit in the boundary traces at $r=\delta,1$, as well as the work integral $V^\text{lin}_\delta$. Combining these results with the change-of-variables formulas in \eqref{eq:EdeltaChange}, we anticipate the following convergences as $\delta \to 0$: for the displacement problem
\begin{equation}\label{eq:anticipated_limit_disp}
    \min_{ \substack{ \mathbf{u}(\mathbf{x})\, \\ \mathbf{u}|_{r=\delta,1}=\mathbf{u}_{\delta,\epsilon}^\pm}}\, \log\Big( \frac{1}{\delta} \Big)E^{\text{lin}}_{\delta} (\mathbf{u}) \to \min_{ \substack{\mathbf{U}(\rho), \mathbf{m}(\rho,\theta)  \\ \mathbf{U}|_{\rho=-1,0}=\mathbf{u}_0^\pm}}\,  \int_R Q \big( \partial \rho \mathbf{U} \odot \mathbf{e}_{\rho} + \mathbf{m} \odot \mathbf{e}_{\theta} \big) \, d \rho d \theta,
\end{equation}
and for the force problem
\begin{equation}\label{eq:anticipated_limit_force}
\begin{split}
  &\min_{  \mathbf{u}(\mathbf{x})}\, \log\Big( \frac{1}{\delta} \Big)\Big( E^{\text{lin}}_{\delta} (\mathbf{u}) - \frac{1}{\log \frac{1}{\delta}} V^{\text{lin}}_{\delta}(\mathbf{u}) \Big) \\
  &\qquad \rightarrow \min_{ \mathbf{U}(\rho), \mathbf{m}(\rho,\theta)}\,  \int_R Q \big( \partial \rho \mathbf{U} \odot \mathbf{e}_{\rho} + \mathbf{m} \odot \mathbf{e}_{\theta} \big) \, d \rho d \theta - \mathbf{f}_{0}\cdot\left(\mathbf{U}|_{\rho=0}-\mathbf{U}|_{\rho=-1}\right).
\end{split}
\end{equation}
These claims are properly stated and proved in Section \ref{sec:AlmostMininimizers}. But first, we show how to solve the limit problems with convex duality.

Consider the Legendre transform of $Q(\mathbf{F})=\frac{1}{2}\langle\mathbf{F},\mathbb{C}\mathbf{F}\rangle$ given by
\[
Q^\ast(\boldsymbol{\Sigma})=\max_{\mathbf{F}\in\mathbb{R}^{2\times2}}\,\left\langle \boldsymbol{\Sigma},\mathbf{F}\right\rangle -Q(\mathbf{F})= \begin{cases}
\frac{1}{2}\left\langle \mathbb{C}^{-1}\boldsymbol{\Sigma},\boldsymbol{\Sigma}\right\rangle &\text{ if } \boldsymbol{\Sigma}\in\text{Sym}_{2}\\ 
\infty & \text{ otherwise}
\end{cases}.
\]
\begin{lem}\label{lem:minmax_lemma}There holds
\[
\min_{\mathbf{m}\in\mathbb{R}^{2}}\,Q(\mathbf{v}\odot\mathbf{e}_{\rho}+\mathbf{m}\odot\mathbf{e}_{\theta})=\max_{\Sigma\in\mathbb{R}}\,\left\langle \mathbf{v}\odot\mathbf{e}_{\rho},\Sigma\mathbf{e}_{\rho}\otimes\mathbf{e}_{\rho}\right\rangle -Q^\ast(\Sigma\mathbf{e}_{\rho}\otimes\mathbf{e}_{\rho})=\frac{1}{2}\frac{|\mathbf{v}\cdot\mathbf{e}_{\rho}|^{2}}{(\mathbb{C}^{-1})_{\rho\rho\rho\rho}}
\]
for all $\mathbf{v}\in\mathbb{R}^{2}$. Given $\mathbf{v}$, the optimal
$\mathbf{m}$ and $\Sigma$ obey 
\begin{align}\label{minM}
\mathbf{m}=\left(2(\mathbb{C}^{-1})_{\rho\theta\rho\rho}\Sigma-\mathbf{v}\cdot\mathbf{e}_{\theta}\right)\mathbf{e}_{\rho}+(\mathbb{C}^{-1})_{\theta\theta\rho\rho}\Sigma\mathbf{e}_{\theta},\quad\Sigma=\frac{\mathbf{v}\cdot\mathbf{e}_{\rho}}{(\mathbb{C}^{-1})_{\rho\rho\rho\rho}}.
\end{align}
\end{lem}
\begin{proof}
Taking $\boldsymbol{\Sigma}=\Sigma\mathbf{e}_{\rho}\otimes\mathbf{e}_{\rho}$
in the definition of $Q^\ast$ yields that
\begin{equation}
Q(\mathbf{v}\odot\mathbf{e}_{\rho}+\mathbf{m}\odot\mathbf{e}_{\theta})\geq\left\langle \mathbf{v}\odot\mathbf{e}_{\rho},\Sigma\mathbf{e}_{\rho}\otimes\mathbf{e}_{\rho}\right\rangle -Q^\ast(\Sigma\mathbf{e}_{\rho}\otimes\mathbf{e}_{\rho})\label{eq:inequality-to-saturate}
\end{equation}
for all $\mathbf{v}, \mathbf{m}\in \R^2$, and $\Sigma\in\R$. To conclude,
we must characterize $\mathbf{m}$ and $\Sigma$ achieving equality
in the above. In fact, 
\[
Q(\mathbf{v}\odot\mathbf{e}_{\rho}+\mathbf{m}\odot\mathbf{e}_{\theta})+Q^\ast(\Sigma\mathbf{e}_{\rho}\otimes\mathbf{e}_{\rho})-\left\langle \mathbf{v}\odot\mathbf{e}_{\rho},\Sigma\mathbf{e}_{\rho}\otimes\mathbf{e}_{\rho}\right\rangle =\frac{1}{2}\left|\mathbb{C}^{\frac12}\left(\mathbf{v}\odot\mathbf{e}_{\rho}+\mathbf{m}\odot\mathbf{e}_{\theta}\right)-\mathbb{C}^{-\frac12}\Sigma\mathbf{e}_{\rho}\otimes\mathbf{e}_{\rho}\right|^{2}
\]
so that equality in \prettyref{eq:inequality-to-saturate} is equivalent
to the statement that
\begin{equation}
\mathbf{v}\odot\mathbf{e}_{\rho}+\mathbf{m}\odot\mathbf{e}_{\theta}=\mathbb{C}^{-1}\Sigma\mathbf{e}_{\rho}\otimes\mathbf{e}_{\rho}.\label{eq:saddle-point-system}
\end{equation}

We solve for $\mathbf{m}$ and $\Sigma$ now. Testing \prettyref{eq:saddle-point-system}
against the basis $\{\mathbf{e}_{\rho}\otimes\mathbf{e}_{\rho},\mathbf{e}_{\theta}\otimes\mathbf{e}_{\theta},\mathbf{e}_{\rho}\odot\mathbf{e}_{\theta}\}$
for symmetric two-by-two matrices yields the constraints
\[
\mathbf{v}\cdot\mathbf{e}_{\rho}=(\mathbb{C}^{-1})_{\rho\rho\rho\rho}\Sigma,\quad\mathbf{m}\cdot\mathbf{e}_{\theta}=(\mathbb{C}^{-1})_{\theta\theta\rho\rho}\Sigma,\quad\mathbf{v}\cdot\mathbf{e}_{\theta}+\mathbf{m}\cdot\mathbf{e}_{\rho}=2 (\mathbb{C}^{-1})_{\rho\theta\rho\rho}\Sigma
\]
as necessary conditions. In fact, these conditions are also sufficient because both sides  of \eqref{eq:saddle-point-system} are symmetric for all choices of $\mathbf{m}$ and $\Sigma$. 
The first constraint determines $\Sigma$, and the second and third
ones give $\mathbf{m}$. 
\end{proof}

\noindent We are ready to solve the proposed limit problems. Recall $\QFl$ and $\QFl^\ast$ from \eqref{eq:Flquadform-recalled}.
\begin{lem}\label{minimzingLemDisp}
Given $\mathbf{u}_0^\pm\in \R^2$, 
\begin{align*}
\min_{\substack{\mathbf{m} \in L^2(R; \mathbb{R}^2) \\ \mathbf{U} \in W^{1,2}(R; \mathbb{R}^2), \,\, \partial_{\theta} \mathbf{U} = \mathbf{0} \\ \mathbf{U}|_{\rho = -1} = \mathbf{u}_0^-, \,\, \mathbf{U}|_{\rho = 0} = \mathbf{u}_0^+} }\,  \int_{R} Q( \partial_{\rho} \mathbf{U} \odot \mathbf{e}_\rho + \mathbf{m} \odot \mathbf{e}_{\theta}) \, d\rho d \theta = \QFl(\mathbf{u}_0^+ - \mathbf{u}_0^{-}).
\end{align*}
Admissible $\mathbf{U}$ and $\mathbf{m}$ are optimal if and only if  $(\mathbf{U}, \mathbf{m}) = (\mathbf{U}_0, \mathbf{m}_0)$ for
\begin{equation}
\begin{aligned}\label{eq:u0M0Disp}
&\mathbf{U}_0(\rho) = \rho ( \mathbf{u}_0^+ - \mathbf{u}_0^{-}) + \mathbf{u}_0^+, \\
&\mathbf{m}_0(\theta)=  \left(\frac{ 2(\mathbb{C}^{-1})_{\rho\theta\rho\rho} }{(\mathbb{C}^{-1})_{\rho\rho\rho\rho}} \mathbf{e}_{\rho} \otimes \mathbf{e}_{\rho}  -\mathbf{e}_{\rho} \otimes \mathbf{e}_{\theta} +\frac{(\mathbb{C}^{-1})_{\theta\theta\rho\rho}}{(\mathbb{C}^{-1})_{\rho\rho\rho\rho}}\mathbf{e}_{\theta} \otimes \mathbf{e}_{\rho} \right)( \mathbf{u}_0^+ - \mathbf{u}_0^{-}).
\end{aligned}
\end{equation}
\end{lem}
\begin{proof}
	Let $\mathbf{U}\in W^{1,2}(R;\R^2)$ satisfy $\partial_\theta \mathbf{U} = \mathbf{0}$, $\mathbf{U}|_{\rho = -1} = \mathbf{u}_0^-$, and $\mathbf{U}|_{\rho = 0} = \mathbf{u}_0^+$, and let $\mathbf{m}\in L^2(R;\R^2)$.  
Note
	\begin{align}\label{min_displ_est1}
		\int_{R}Q(\partial_{\rho}\mathbf{U}\odot\mathbf{e}_{\rho}+\mathbf{m}\odot\mathbf{e}_{\theta})\,d\rho d\theta
  \ge\int_{R}\min_{\mathbf{m}\in\mathbb{R}^{2}}\,Q(\partial_{\rho}\mathbf{U}\odot\mathbf{e}_{\rho}+\mathbf{m}\odot\mathbf{e}_{\theta})\,d\rho d\theta=\int_{R}\frac{1}{2}\frac{|\partial_{\rho}\mathbf{U}\cdot\mathbf{e}_{\rho}|^{2}}{(\mathbb{C}^{-1})_{\rho\rho\rho\rho}}\,d\rho d\theta
	\end{align}
by Lemma \ref{lem:minmax_lemma}. Applying Jensen's inequality and the fundamental theorem of calculus, there follows
	\begin{align}\label{min_displ_est2}
		\begin{split}
			\int_{R}\frac{1}{2}\frac{|\partial_{\rho}\mathbf{U}\cdot\mathbf{e}_{\rho}|^{2}}{(\mathbb{C}^{-1})_{\rho\rho\rho\rho}}\,d\rho d\theta &\geq \int_\alpha^\beta\frac{1}{2}\frac{|\int_{-1}^0\partial_{\rho}\mathbf{U}\cdot\mathbf{e}_{\rho}\, d\rho|^{2}}{(\mathbb{C}^{-1})_{\rho\rho\rho\rho}}\,d\theta \\
            &= \int_\alpha^\beta\frac{1}{2}\frac{|(\mathbf{u}_0^+ - \mathbf{u}_0^-)\cdot\mathbf{e}_{\rho}|^{2}}{(\mathbb{C}^{-1})_{\rho\rho\rho\rho}}\,d\rho d\theta
			= \QFl(\mathbf{u}_0^+-\mathbf{u}_0^-).
		\end{split}
	\end{align}
	In light of \eqref{minM}, the inequalities in \eqref{min_displ_est1} and \eqref{min_displ_est2} are equalities if and only if $\mathbf{U}$ and $\mathbf{m}$ obey \eqref{eq:u0M0Disp}.
\end{proof}

\begin{lem}\label{minimzingLemForce}
Given $\mathbf{f}_0\in\R^2$, 
\begin{equation}\label{eq:displacement-limit-problem-forced}
\min_{\substack{\mathbf{m} \in L^2(R; \mathbb{R}^2) \\ \mathbf{U} \in W^{1,2}(R; \mathbb{R}^2), \,\, \partial_{\theta} \mathbf{U} = \mathbf{0}}}\,  \int_{R} Q( \partial_{\rho} \mathbf{U} \odot \mathbf{e}_\rho + \mathbf{m} \odot \mathbf{e}_{\theta}) \, d\rho d  \theta   - \mathbf{f}_0 \cdot (\mathbf{U}|_{\rho = 0}-\mathbf{U}|_{\rho = -1} ) = -\QFl^\ast( \mathbf{f}_0).
\end{equation}
Admissible $\mathbf{U}$ and $\mathbf{m}$ are optimal if and only if $(\mathbf{U}, \mathbf{m}) = (\mathbf{U}_0+\mathbf{c}_0, \mathbf{m}_0)$ for
\begin{equation}
\begin{aligned}\label{eq:u0M0Force}
&\mathbf{U}_0(\rho) = \rho \KFl^{-1}  \mathbf{f}_0, \\
&\mathbf{m}_0(\theta)=  \left(\frac{ 2(\mathbb{C}^{-1})_{\rho\theta\rho\rho} }{(\mathbb{C}^{-1})_{\rho\rho\rho\rho}} \mathbf{e}_{\rho} \otimes \mathbf{e}_{\rho}  -\mathbf{e}_{\rho} \otimes \mathbf{e}_{\theta} +\frac{(\mathbb{C}^{-1})_{\theta\theta\rho\rho}}{(\mathbb{C}^{-1})_{\rho\rho\rho\rho}}\mathbf{e}_{\theta} \otimes \mathbf{e}_{\rho} \right) \KFl^{-1}  \mathbf{f}_0,
\end{aligned} 
\end{equation}
and $\mathbf{c}_0 \in \mathbb{R}^2$. 
\end{lem}
\begin{rem}\label{rem:duality} The minimization problem \eqref{eq:displacement-limit-problem-forced} is the displacement version of the limiting stress problem obtained in the paragraphs following Theorem \ref{thm:main-result}. This can be understood using Lemma \ref{lem:minmax_lemma}, which turns \eqref{eq:displacement-limit-problem-forced} into the saddle point problem
\[
\min_{\mathbf{U}(\rho)}\max_{\Sigma(\rho,\theta)}\,\int_{R}\partial_{\rho}\mathbf{U}\cdot\Sigma\mathbf{e}_{\rho}-Q^{*}(\Sigma\mathbf{e}_{\rho}\otimes\mathbf{e}_{\rho})\,d\rho d\theta-\mathbf{f}_{0}\cdot(\mathbf{U}|_{\rho=0}-\mathbf{U}|_{\rho=-1}).
\]
Switching the min and max and evaluating the inner minimum over $\mathbf{U}$ forces $\Sigma$ to depend only on $\theta$, since otherwise the minimum would be $-\infty$. Likewise, the boundary term gives  $\int_\alpha^\beta\Sigma\mathbf{e}_\rho\,d\theta = \mathbf{f}_0$. The resulting maximization is the (negative of the) limiting stress problem in \eqref{eq:dual-limit-problem-2}-\eqref{eq:dual-limit-problem}. This observation leads with a little work to the claimed equivalence between  \eqref{eq:displacement-limit-problem-forced} and \eqref{eq:dual-limit-problem-2}-\eqref{eq:dual-limit-problem}. In particular, one must check that switching the order of operations from min-max to max-min is permitted, which can be done using the characterization of optimizers in \eqref{minM}. We leave the details to the reader (see \cite[Chapter III.3.5]{DL76} for a standard example of this sort of duality argument).
\end{rem}
\begin{proof}
	Let $\mathbf{U}\in W^{1,2}(R;\R^2)$ satisfy $\partial_\theta \mathbf{U} = \mathbf{0}$, and let $\mathbf{m}\in L^2(R;\R^2)$.
	Following the proof  of Lemma \ref{minimzingLemDisp}, we obtain the lower bound 
	\begin{align*}
		\int_{R}Q(\partial_{\rho}\mathbf{U}\odot\mathbf{e}_{\rho}+\mathbf{m}\odot\mathbf{e}_{\theta})\,d\rho d\theta \geq \QFl(\mathbf{U}|_{\rho=0}-\mathbf{U}|_{\rho=-1})
	\end{align*}
	from which it follows that 
	\begin{align*}
		\int_{R} Q( \partial_{\rho} \mathbf{U} \odot \mathbf{e}_\rho + \mathbf{m} \odot \mathbf{e}_{\theta}) \, d\rho d  \theta   &
        - \mathbf{f}_0 \cdot  ( \mathbf{U}|_{\rho = 0}-\mathbf{U}|_{\rho = -1} )\geq \min_{\mathbf{v}\in \R^2}\, \QFl(\mathbf{v}) - \mathbf{f}_0\cdot \mathbf{v} = -\QFl^\ast(\mathbf{f}_0).
	\end{align*}
Note $\mathbf{v} = \KFl^{-1} \mathbf{f}_0$ is the unique minimizer in this last step. Appealing to \eqref{minM} once again, we see that equality holds in the inequalities above if and only if $\mathbf{U}$ and $\mathbf{m}$ obey  \eqref{eq:u0M0Force}. In particular, we deduce in the case of equality that $\mathbf{U}|_{\rho=0}-\mathbf{U}|_{\rho=-1} = \KFl^{-1} \mathbf{f}_0$. 
\end{proof}

\subsection{Justification of the limit problems}\label{sec:AlmostMininimizers}

We now prove that the linear displacement and force problems converge to their limits anticipated in \eqref{eq:anticipated_limit_disp} and \eqref{eq:anticipated_limit_force}. First, we prove that the minimum values converge, and then we do the same for almost minimizers. Motivated by the solution formulas for the optimal $\mathbf{U}$ and $\mathbf{m}$ obtained in Lemmas \ref{minimzingLemDisp} and \ref{minimzingLemForce}, we define the Flamant ansatz in logarithmic coordinates by
\begin{equation}
\begin{aligned}\label{eq:FlamantAnsatz}
\mathbf{U}_{\text{Fl},\delta}(\rho,\theta) = \mathbf{U}_0(\rho)  + \frac{1}{\log \frac{1}{\delta}}\mathbf{V}_0(\theta) \quad \text{with}\quad  \mathbf{V}_0(\theta)=\int_\alpha^\theta \mathbf{m}_0 \, d\tilde{\theta}\quad \text{ for  }(\rho,\theta)\in R
\end{aligned}
\end{equation}
where $(\mathbf{U}_0,\mathbf{m}_0)$ are the optimizers in \eqref{eq:u0M0Disp} or \eqref{eq:u0M0Force} depending on whether we consider the displacement or force problem.
The point of these definitions is that
\begin{align}\label{eq:FlamantAnsatzGradient}
D^{\delta} \mathbf{U}_{\text{Fl},\delta} = \partial_{\rho} \mathbf{U}_0\otimes \mathbf{e}_{\rho} + \mathbf{m}_0 \otimes \mathbf{e}_{\theta}
\end{align}
on the $(\rho,\theta)$-domain. On the physical domain, by \eqref{eq:change-of-vars} and \eqref{eq:gradient-rule},
\begin{align}\label{eq:UFltouFlLin}
\mathbf{U}_{\text{Fl}, \delta} (\rho, \theta)  = \frac{1}{\log \frac{1}{\delta}} \mathbf{u}_{\text{Fl}}(\mathbf{x}) + \mathbf{U}_0(0) \quad \text{and}\quad D^\delta \mathbf{U}_{\text{Fl}, \delta} (\rho,\theta) = |\mathbf{x}|\nabla \mathbf{u}_{\text{Fl}}(\mathbf{x}) 
\end{align}
for $\mathbf{u}_{\text{Fl}}$  in \eqref{Flamant_displacementSec3}, with $\mathbf{a} = \mathbf{u}_0^+ - \mathbf{u}_0^{-}$ in the displacement problem and $\mathbf{a}= \mathbf{K}_{\text{Fl}}^{-1} \mathbf{f}_0$ in the force problem. 

\begin{prop}\label{prop:convergence_minima_linear}
It holds that 
\begin{align}
&\log\big(\frac{1}{\d}\big)\mathcal{E}_{\delta}^{\emph{lin,disp}} \rightarrow \QFl(\mathbf{u}_0^+ - \mathbf{u}_0^{-}), \quad 
\log\big(\frac{1}{\d}\big) \mathcal{E}_{\delta}^{\emph{lin,force}} \rightarrow -\QFl^\ast(\mathbf{f}_0)\label{convergence_linear_energy_force}
\end{align}
as $\delta\to 0$, where $\QFl$ and $\QFl^\ast$ are in \eqref{eq:Flquadform-recalled}, and $\mathbf{u}_0^\pm$ and $\mathbf{f}_0$ are given by \eqref{eq:linearBoundaryData} and \eqref{eq:linearForceData}. 
\end{prop}

\begin{proof}
\textit{Step 1: Lower bounds.}  
We start with the displacement problem. Our goal is to show that
\begin{align}\label{eq:LBDispProb}
    \liminf_{\delta \rightarrow 0}\, \log \big(\frac{1}{\delta}\big) \mathcal{E}_{\delta}^{\text{lin,disp}} \geq Q_{\text{Fl}}(\mathbf{u}_0^+ - \mathbf{u}_0^{-}).
\end{align}
If the left-hand side is $\infty$, there is nothing to prove. Otherwise, there is an admissible subsequence  $\mathbf{u}_{\delta} \in \mathcal{A}_{\delta}^{\text{lin}}$ (not relabeled) such that 
\begin{align*}
\lim_{\delta \rightarrow 0}\, \log \big(\frac{1}{\delta}\big) E_{\delta}^{\text{lin}}(\mathbf{u}_{\delta}) = \liminf_{\delta \rightarrow 0}\, \log \big(\frac{1}{\delta}\big) \mathcal{E}_{\delta}^{\text{lin,disp}} < \infty.
\end{align*}
By Proposition \ref{prop:linear-compactness}(i), there are functions $\mathbf{m}\in L^2(R;\R^2)$ and $\mathbf{U}\in W^{1,2}(R;\R^2)$ with $\partial_\theta \mathbf{U}=\mathbf{0}$, $\mathbf{U}|_{\rho=-1}=\mathbf{u}_{0}^{-}$, and $\mathbf{U}|_{\rho=0}=\mathbf{u}_{0}^{+}$, as well as a further subsequence of $\{\mathbf{u}_\d\}$ such that its logarithmic counterpart $\{\mathbf{U}_\d\}$ satisfies
\begin{align}\label{weak_convergence_linear_gradient}
\sym\, D^{\delta}\mathbf{U}_{\delta}=\partial_{\rho}\mathbf{U}_{\delta}\odot \mathbf{e}_{\rho}+\log\Big(\frac{1}{\delta}\Big)\partial_{\theta}\mathbf{U}_{\delta}\odot \mathbf{e}_{\theta}\rightharpoonup\partial_{\rho}\mathbf{U}\odot \mathbf{e}_{\rho}+\mathbf{m}\odot \mathbf{e}_{\theta}\quad\text{weakly in }L^{2}(R;\R^{2\times 2}).
\end{align}
Since $Q$ is convex, we find via the change of variables \eqref{eq:EdeltaChange}  that
\begin{align}\label{sandwich_linear_displ1}
\lim_{\delta\to0}\, \log\big(\frac{1}{\d}\big)E_\d^{\rm{lin}}(\mathbf{u}_\d) &= \liminf_{\delta\to0}\,\int_{R}Q(\mathrm{sym}\,D^{\delta}\mathbf{U}_{\delta})\,d\rho d\theta \nonumber \\
&\geq\int_{R}Q(\partial_{\rho}\mathbf{U}\odot\mathbf{e}_{\rho}+\mathbf{m}\odot\mathbf{e}_{\theta})\,d\rho d\theta \geq \QFl(\mathbf{u}_0^+ - \mathbf{u}_0^-)
\end{align}
where the last step uses Lemma \ref{minimzingLemDisp}. This proves \eqref{eq:LBDispProb}.

We turn  now to the lower bound for force problem, namely,
\begin{align}\label{eq:LBDispForce}
    \liminf_{\delta \rightarrow 0}\, \log \big(\frac{1}{\delta}\big) \mathcal{E}_{\delta}^{\text{lin,force}} \geq -Q_{\text{Fl}}^{\ast}(\mathbf{f}_0).
\end{align}
Again, if the left side is $\infty$, there is nothing to prove.  Otherwise, there is an admissible subsequence  $\mathbf{u}_{\delta} \in W^{1,2}(\Omega_{\delta}; \mathbb{R}^2)$ such that 
\begin{align*}
 \lim_{\delta \rightarrow 0}\, \log \big(\frac{1}{\delta}\big) 
 \Big( E_{\delta}^{\text{lin}}(\mathbf{u}_{\delta})  -  \frac{1}{\log \frac{1}{\delta}} V_{\delta}^{\text{lin}}( \mathbf{u}_{\delta})\Big)  = \liminf_{\delta \rightarrow 0}\, \log \big(\frac{1}{\delta}\big) \mathcal{E}_{\delta}^{\text{lin,force}} < \infty.   
\end{align*}
Proposition \ref{prop:linear-compactness}(ii) yields a subsequence  of $\{ \mathbf{u}_{\delta} \}$ such that its logarithmic counterpart $\{ \mathbf{U}_{\delta} \}$ satisfies \eqref{weak_convergence_linear_gradient}  for some $\mathbf{m}\in L^2(R;\R^2)$ and $\mathbf{U}\in W^{1,2}(R;\R^2)$ with $\partial_\theta \mathbf{U}=\mathbf{0}$ (the infinitesimal rigid body motions in the proposition have zero symmetrized logarithmic gradient by definition). Arguing as above and using the convergence of the work integral in \eqref{eq:secondDesiredCompactnessForce}, we obtain that
\begin{align}\label{sandwich_linear_force1}
\lim_{\delta\to0}\, \log\big(\frac{1}{\d}\big)&\Big(E_\d^{\rm{lin}}(\mathbf{u}_\d) - \frac{1}{\log\big(\frac{1}{\d}\big)}V^{\text{lin}}_{\delta}(\mathbf{u}_\d)\Big)\nonumber \\
&= \liminf_{\delta\to0}\,\int_{R}Q(\mathrm{sym}\,D^{\delta}\mathbf{U}_{\delta})\,d\rho d\theta- \int_{\alpha}^{\beta}\mathbf{f}_{\delta}^{+}\cdot\mathbf{U}_{\delta}|_{\rho=0}-\mathbf{f}_{\delta}^{-}\cdot\mathbf{U}_{\delta}|_{\rho=-1}\,d\theta \nonumber\\
&\geq\int_{R}Q(\partial_{\rho}\mathbf{U}\odot\mathbf{e}_{\rho}+\mathbf{m}\odot\mathbf{e}_{\theta})\,d\rho d\theta - \mathbf{f}_{0}\cdot (\mathbf{U}|_{\rho=0}-\mathbf{U}|_{\rho=-1})\geq -\QFl^\ast(\mathbf{f}_0)
\end{align}
where the last step is by Lemma \ref{minimzingLemForce}. This proves \eqref{eq:LBDispForce}.

\textit{Step 2: Upper bounds.} We begin with the displacement problem. We seek an admissible sequence  $\mathbf{u}_{\delta} \in \mathcal{A}_{\delta}^{\text{lin}}$ such that 
\begin{align}\label{eq:UPDispProb}
    \lim_{\delta \rightarrow 0}\,  \log \Big(\frac{1}{\delta} \Big)E_{\delta}(\mathbf{u}_\delta) = Q_{\text{Fl}}( \mathbf{u}_0^+ - \mathbf{u}_0^{-}). 
\end{align}
Once proved, this result combines with the lower bound in \eqref{eq:LBDispProb} to establish the displacement part of \eqref{convergence_linear_energy_force}.

The sequence we use to prove \eqref{eq:UPDispProb} is based on the Flamant ansatz $\UFld$  in \eqref{eq:FlamantAnsatz}, where $(\mathbf{U}_0,\mathbf{m}_0)$ are from \eqref{eq:u0M0Disp}. 
We must modify $\UFld$ to enforce the Dirichlet boundary conditions in \eqref{admissible_linear_displ}.
Working in logarithmic coordinates, we define $\{\mathbf{u}_{\delta}\}$ by  
$$\mathbf{u}_{\delta}(\mathbf{x}) = \mathbf{U}_{\d}(\rho,\theta)= \UFld(\rho,\theta) - \frac{1}{\log\big(\frac{1}{\d}\big)}\mathbf{V}_\d(\rho,\theta) + \mathbf{A}_\d(\rho) + \bar{\mathbf{U}}_\d(\rho,\theta)\quad $$ 
for $(\rho, \theta) \in R$. Here, $\mathbf{V}_\d$ helps to cut off $\mathbf{V}_0$ near the boundary in that
\begin{align*}
	\mathbf{V}_\delta(\rho,\theta)=\begin{cases}
		\psi\big(\log\big(\frac{1}{\delta}\big)(\rho+1)\big)\mathbf{V}_0(\theta) &\text{ if } \rho \in (-1,-1+\frac{1}{\log(\frac{1}{\delta})})\\
		\mathbf{0} & \text{ if } \rho \in (-1+\frac{1}{\log(\frac{1}{\delta})},-\frac{1}{\log(\frac{1}{\delta})})\\
		\psi\big(-\log\big(\frac{1}{\delta}\big)\rho \big)\mathbf{V}_0(\theta) &\text{ if } \rho \in (-\frac{1}{\log(\frac{1}{\delta})},0)
	\end{cases}
\end{align*}
for a smooth function $\psi:[0,1]\to [0,1]$ with $\psi=1$ on $[0,1/4]$ and $\psi=0$ on $[3/4,1]$. Also,
\begin{align*}
	\mathbf{A}_\delta(\rho) = \fint_\alpha^\beta \mathbf{u}_\delta^+\, d\theta -\mathbf{u}_0^+ + \rho\Big(\fint_\alpha^\beta \mathbf{u}_\delta^+\, d\theta -\mathbf{u}_0^+ - \fint_\alpha^\beta \mathbf{u}_\delta^-\, d\theta -\mathbf{u}_0^-\Big)
\end{align*}
for $\rho \in (-1,0)$, 
and $\bar{\mathbf{U}}_\d\in W^{1,2}(R;\R^2)$ is obtained from Lemma \ref{lem:extensionNonlinear} in the appendix in logarithmic coordinates, with the specified traces 
$\bar{\mathbf{U}}_\d|_{\rho=-1} = \mathbf{u}^-_\delta - \fint_\alpha^\beta \mathbf{u}^-_\delta\, d\theta$ and $\bar{\mathbf{U}}_\d|_{\rho=0} = \mathbf{u}^+_\delta - \fint_\alpha^\beta \mathbf{u}^+_\delta\, d\theta$.
These modifications ensure that  $\mathbf{u}_{\delta} \in \mathcal{A}_{\delta}^{\text{lin}}$. They also satisfy 
\begin{equation}\label{eq:vanishRemainder}
\begin{aligned}
	\frac{1}{\log\frac{1}{\d}}\|D^\d\mathbf{V}_\d\|_{L^2(R)} &\lesssim \frac{1}{\sqrt{\log\frac{1}{\d}}} \rightarrow 0 , \\
	\|D^\d\mathbf{A}_\d\|_{L^2(R)} &\lesssim \Big|\fint_\alpha^\beta \mathbf{u}_\d^-\,d\theta - \mathbf{u}_\d^-\Big| + \Big|\fint_\alpha^\beta \mathbf{u}_\d^-\,d\theta - \mathbf{u}_\d^-\Big| \rightarrow 0 ,\\
	\|D^\d\bar{\mathbf{U}}_\d\|_{L^2(R)} &\lesssim \sqrt{\log\frac{1}{\d}}\big( \|\mathbf{u}^-_\d\|_{\dot{W}^{\frac12,2}((\alpha,\beta))} + \|\mathbf{u}^+_\d\|_{\dot{W}^{\frac12,2}((\alpha,\beta))}\big) \rightarrow 0
\end{aligned}
\end{equation}
as $\delta \rightarrow 0$ by the smoothness of  $\mathbf{V}_0$, the estimates in Lemma \ref{lem:extensionNonlinear} in logarithmic coordinates, and the assumptions \eqref{eq:linearBoundaryData} on the boundary data.
Since $Q$ is a quadratic form and $D^{\delta} \mathbf{U}_{\delta}$ is the sum of $D^{\delta} \mathbf{U}_{\delta,\text{Fl}} = \partial_{\rho} \mathbf{U}_0\otimes \mathbf{e}_{\rho} + \mathbf{m}_0 \otimes \mathbf{e}_{\theta}$ (see \eqref{eq:FlamantAnsatzGradient}) and modifications that vanish in the $L^2$-norm by \eqref{eq:vanishRemainder}, we get that  
\begin{align*}
	\lim_{\d \to 0}\,   \log\big(\frac{1}{\d}\big)E_\d^{\rm{lin}}(\mathbf{u}_{\delta})  &=\lim_{\d\to 0}\, \int_R Q(\sym\, D^\d \mathbf{U}_{\d}) \, d\rho d\theta = \lim_{\d\to 0}\, \int_R Q(\sym\, D^\d \UFld) \, d\rho d\theta \nonumber\\
    &= \int_R Q(\partial_\rho\mathbf{U}_0\odot \mathbf{e}_\rho + \mathbf{m}_0\odot\mathbf{e}_\theta)\, d\rho d\theta = Q_\text{Fl}(\mathbf{u}_0^+-\mathbf{u}_0^-).
\end{align*}
The last equality uses Lemma \ref{minimzingLemDisp}. This proves \eqref{eq:UPDispProb}.

For the force problem, we must instead construct an admissible sequence $\mathbf{u}_{\delta} \in W^{1,2}(\Omega_{\delta}; \mathbb{R}^2)$ such that 
\begin{align}\label{eq:UBDispForce}
  \lim_{\d \to 0}\,  & \log\big(\frac{1}{\d}\big)\Big(E_\d^{\rm{lin}}(\mathbf{u}_{\delta}) - \frac{1}{\log(\frac{1}{\d})}V^{\text{lin}}_{\delta} (\mathbf{u}_{\delta}) \Big) =   - Q^{\ast}_{\text{Fl}}(\mathbf{f}_0).
\end{align}
Given the lower bound in \eqref{eq:LBDispForce}, this will finish the proof of \eqref{convergence_linear_energy_force}. Now there are no boundary conditions, so we can simply take   $\mathbf{u}_{\delta}( \mathbf{x})  =\mathbf{U}_{\text{Fl},\delta}(\rho, \theta)$ where  $(\mathbf{U}_0,\mathbf{m}_0)$ are given by \eqref{eq:u0M0Force}. From the assumptions on the forces in \eqref{eq:linearForceData} and the formula for $\mathbf{U}_{\text{Fl},\delta}$ in \eqref{eq:FlamantAnsatz}, it is straightforward to show using \eqref{eq:EdeltaChange} that 
\begin{align*}
	\lim_{\d \to 0}\,  & \log\big(\frac{1}{\d}\big)\Big(E_\d^{\rm{lin}}(\mathbf{u}_{\delta}) - \frac{1}{\log(\frac{1}{\d})}V^{\text{lin}}_{\delta} (\mathbf{u}_{\delta}) \Big) \nonumber\\
	&=\lim_{\delta\to0}\,\int_{R}Q(\mathrm{sym}\,D^{\delta}\UFld)\,d\rho d\theta- \int_{\alpha}^{\beta}\mathbf{f}_{\delta}^{+}\cdot\UFld|_{\rho=0}-\mathbf{f}_{\delta}^{-}\cdot\UFld|_{\rho=-1}\,d\theta \nonumber\\
	&= \int_R Q(\partial_\rho\mathbf{U}_0\odot \mathbf{e}_\rho + \mathbf{m}_0\odot\mathbf{e}_\theta)\, d\rho d\theta 
    - \mathbf{f}_{0}\cdot (\mathbf{U}_0|_{\rho=0}-\mathbf{U}_0|_{\rho=-1})= -\QFl^\ast(\mathbf{f}_0)
\end{align*}
by Lemma \ref{minimzingLemForce}. This proves \eqref{eq:UBDispForce}.
\end{proof}

Finally, we discuss almost minimizers. Due to the previous result, a sequence of almost minimizers $\{\mathbf{u}_\delta\}$ in the linear displacement or force problems satisfies
 \begin{equation}\begin{aligned}\label{eq:linMinSeqDef}
     &\mathbf{u}_{\delta} \in \mathcal{A}_{\delta}^{\text{lin}}  &&\text{ and } &&   E_{\delta}(\mathbf{u}_{\delta}) = \mathcal{E}_{\delta}^{\text{lin,disp}} + o\Big( \frac{1}{\log \frac{1}{\delta}}\Big), \\
     &\mathbf{u}_{\delta} \in W^{1,2}(\Omega_{\delta}; \mathbb{R}^2)   && \text{ and } &&  E_{\delta}(\mathbf{u}_{\delta})  - \frac{1}{\log \frac{1}{\delta}} V_{\delta}(\mathbf{u}_{\delta}) = \mathcal{E}_{\delta}^{\text{lin,force}}   + o\Big( \frac{1}{\log \frac{1}{\delta}}\Big)
 \end{aligned}
 \end{equation}
respectively, as $\delta \rightarrow 0$.

\begin{prop}\label{almostMinimizersProp}
Let $\{ \mathbf{u}_{\delta} \}$ be almost minimizing in the linear displacement or force problem. The corresponding linear strains $\mathbf{e}(\mathbf{u}_\d)$ and  stresses $\boldsymbol{\sigma}_{\delta}=\C\mathbf{e}(\mathbf{u}_{\delta})$
obey
\begin{align}\label{eq:stressStrain_L2}
\big\| \mathbf{e}(\mathbf{u}_{\delta}) - \frac{1}{\log\frac{1}{\d}}\mathbf{e}(\uFl) \big\| _{L^{2}(\Omega_{\delta})}\ll\frac{1}{\sqrt{\log\frac{1}{\delta}}}\quad\text{and}\quad\big\| \boldsymbol{\sigma}_{\delta}-\frac{1}{\log\frac{1}{\delta}}\boldsymbol{\sigma}_{\rm{Fl}}\big\| _{L^{2}(\Omega_{\delta})}\ll\frac{1}{\sqrt{\log\frac{1}{\delta}}}
\end{align}
for the Flamant displacement $\uFl$  and stress $\boldsymbol{\sigma}_{\rm{Fl}}$ in \eqref{Flamant_displacementSec3} and \eqref{eq:aniso_Flamant_Sec3} defined using $\mathbf{a} = \mathbf{u}_0^+ - \mathbf{u}_0^{-}$ for the displacement problem and $\mathbf{a}= \mathbf{K}_{\emph{Fl}}^{-1} \mathbf{f}_0$ for the force problem.
\end{prop}

\begin{proof} \textit{Step 1: Preliminaries.}
By \eqref{eq:linMinSeqDef} and Proposition \ref{prop:convergence_minima_linear}, the sequence $\{\mathbf{u}_\delta\}$ obeys 
\begin{align}\label{eq:strongConvergeAlmostMin}
    \log \big(\frac{1}{\delta}\big) E_{\delta}(\mathbf{u}_{\delta}) \rightarrow Q_{\text{Fl}}(\mathbf{u}_0^+ - \mathbf{u}_0^{-})\quad\text{or}\quad \log \big(\frac{1}{\delta}\big)\Big( E_{\delta}(\mathbf{u}_{\delta})  - \frac{1}{\log \frac{1}{\delta}} V_{\delta}^{\text{lin}}(\mathbf{u}_{\delta}) \Big) \rightarrow -Q^{\ast}_{\text{Fl}}(\mathbf{f}_0)
\end{align}
depending on whether we treat the displacement or force problem, respectively. This sequence is also compact in the sense of Proposition \ref{prop:linear-compactness}, and thus we can pass to a subsequence (not relabeled) such that its logarithmic counterpart $\{ \mathbf{U}_{\delta}\}$ satisfies
\begin{align}\label{eq:weakConvergeAlmostMin}
\sym\, D^\delta \mathbf{U}_\d \rightharpoonup \partial_\rho \mathbf{U}\odot \mathbf{e}_\rho + \mathbf{m}\odot \mathbf{e}_\theta\quad\text{ weakly in }L^2(R;\R^{2\times 2})
\end{align}
for some $\mathbf{m}\in L^2(R;\R^2)$ and $\mathbf{U}\in W^{1,2}(R;\R^2)$ with $\partial_\theta \mathbf{U}=\mathbf{0}$. 

\textit{Step 2: Weak convergence of the symmetric gradients.}
We first prove  that
\begin{align}\label{eq:forStrongConvergence}
	\begin{split}
		\sym\, D^\delta \mathbf{U}_\d \rightharpoonup \partial_\rho \mathbf{U}_0\odot \mathbf{e}_\rho + \mathbf{m}_0\odot \mathbf{e}_\theta\quad\text{ weakly in }L^2(R;\R^{2\times 2}),\\
		\int_R Q(\sym\, D^\d \mathbf{U}_\d)\, d\rho d\theta \to \int_R Q(\partial_\rho \mathbf{U}_0\odot \mathbf{e}_\rho + \mathbf{m}_0\odot \mathbf{e}_\theta)\, d\rho d\theta
	\end{split}
\end{align}
for the minimizers $(\mathbf{U}_0,\mathbf{m}_0)$ from Lemma \ref{minimzingLemDisp} or \ref{minimzingLemForce} in the displacement or force problem, respectively.

In the displacement problem,  $\{ \mathbf{u}_{\delta}\}$ satisfies the chain of inequalities in \eqref{sandwich_linear_displ1} of Proposition \ref{prop:convergence_minima_linear} since the symmetric part of its logarithmic gradient weakly converges per \eqref{eq:weakConvergeAlmostMin}. At the same time, its rescaled energy in  \eqref{eq:strongConvergeAlmostMin} converges to $Q_{\text{Fl}}( \mathbf{u}_0^{+} - \mathbf{u}_0^{-})$. Since the latter sandwiches the inequalities in \eqref{sandwich_linear_displ1}, we obtain
\begin{align*}
\int_R Q(\sym\, D^\d \mathbf{U}_\d)\, d\rho d\theta \to \int_R Q(\partial_\rho \mathbf{U}\odot \mathbf{e}_\rho + \mathbf{m}\odot \mathbf{e}_\theta)\, d\rho d\theta
\end{align*}
as well as the identities
\begin{align*}
	\int_R Q(\partial_\rho\mathbf{U}\odot \mathbf{e}_\rho + \mathbf{m}\odot\mathbf{e}_\theta)\, d\rho d\theta = \QFl(\mathbf{u}_0^+ - \mathbf{u}_0^-) = \int_R Q(\partial_\rho\mathbf{U}_0\odot \mathbf{e}_\rho + \mathbf{m}_0\odot\mathbf{e}_\theta)\, d\rho d\theta.
\end{align*}
As $(\mathbf{U}_0,\mathbf{m}_0)$ is the unique minimizer in Lemma \ref{minimzingLemDisp},  $\mathbf{U} = \mathbf{U}_0$ and $\mathbf{m} = \mathbf{m}_0$. We conclude \eqref{eq:forStrongConvergence}. 

For the force problem, $\{ \mathbf{u}_{\delta} \}$ satisfies the chain of inequalities in   \eqref{sandwich_linear_force1} in Proposition \ref{prop:convergence_minima_linear}  and its rescaled energy in \eqref{eq:strongConvergeAlmostMin} converges to $-Q_{\text{Fl}}^{\ast}( \mathbf{f}_0)$. The latter again sandwiches the inequalities of the former, thus producing the convergence 
\begin{equation}
\begin{aligned}\label{eq:prelimForceCon}
&\int_{R}Q(\mathrm{sym}\,D^{\delta}\mathbf{U}_{\delta})\,d\rho d\theta- \int_{\alpha}^{\beta}\mathbf{f}_{\delta}^{+}\cdot\mathbf{U}_{\delta}|_{\rho=0}-\mathbf{f}_{\delta}^{-}\cdot\mathbf{U}_{\delta}|_{\rho=-1}\,d\theta \\
&\qquad \rightarrow  \int_R Q(\partial_\rho\mathbf{U}\odot \mathbf{e}_\rho + \mathbf{m}\odot\mathbf{e}_\theta)\, d\rho d\theta - \mathbf{f}_{0}\cdot (\mathbf{U}|_{\rho=0}-\mathbf{U}|_{\rho=-1}) 
\end{aligned}
\end{equation}
along with the  the following statement about the limits $\mathbf{U}$ and $\mathbf{m}$ in \eqref{eq:weakConvergeAlmostMin}:
\begin{align*}
	\int_R Q(\partial_\rho\mathbf{U}\odot \mathbf{e}_\rho &+ \mathbf{m}\odot\mathbf{e}_\theta)\, d\rho d\theta - \mathbf{f}_{0}\cdot(\mathbf{U}|_{\rho=0}-\mathbf{U}|_{\rho=-1})\\
	&= -\QFl^\ast(\mathbf{f}_0) = \int_R Q(\partial_\rho\mathbf{U}_0\odot \mathbf{e}_\rho + \mathbf{m}_0\odot\mathbf{e}_\theta)\, d\rho d\theta - \mathbf{f}_{0}\cdot(\mathbf{U}_0|_{\rho=0}-\mathbf{U}_0|_{\rho=-1}).
\end{align*}
By Lemma \ref{minimzingLemForce}, $\mathbf{m} = \mathbf{m}_0$  and $\mathbf{U} = \mathbf{U}_0+\mathbf{c}_0$ for some $\mathbf{c}_0\in\mathbb{R}^2$.
 The first line in \eqref{eq:forStrongConvergence}  is now a consequence of \eqref{eq:weakConvergeAlmostMin}; the second line follows by combining \eqref{eq:prelimForceCon} with \eqref{eq:secondDesiredCompactnessForce}. 

\textit{Step 3: Strong convergence of the symmetric gradients.} We now prove that 
\begin{align}\label{eq:strongConvergeLogStrain}
\sym\, D^\delta \mathbf{U}_\d \rightarrow \partial_\rho \mathbf{U}_0\odot \mathbf{e}_\rho + \mathbf{m}_0\odot \mathbf{e}_\theta\quad\text{ strongly in }L^2(R;\R^{2\times 2}).
\end{align}
This follows from \eqref{eq:forStrongConvergence} and the algebra identity 
\[
Q(\mathbf{F}_1 - \mathbf{F}_2) = Q(\mathbf{F}_1) - \langle \mathbb{C}\mathbf{F}_1, \mathbf{F}_2 \rangle +  Q(\mathbf{F}_2)
\]
for all $\mathbf{F}_{1},\mathbf{F}_2\in \text{Sym}_{2}$. Indeed, write 
$\mathbf{Z}=\partial_\rho \mathbf{U}_0\odot \mathbf{e}_\rho + \mathbf{m}_0\odot \mathbf{e}_\theta$ to shorten the notation and observe that  
\begin{align*}
&\lim_{\delta \rightarrow 0}\, \int_R Q (\mathrm{sym}\,D^{\delta} \mathbf{U}_{\delta} - \mathrm{sym} \,\mathbf{Z})\,d\rho d \theta \\
&\qquad = \lim_{\delta \rightarrow 0}\, \int_R Q (\mathrm{sym}\,D^{\delta} \mathbf{U}_{\delta} )\,d\rho d \theta + \int_R Q (\mathrm{sym}\,\mathbf{Z})\,d\rho d \theta  - \int_{R} \left\langle \mathbb{C}  \mathrm{sym}\,D^{\delta} \mathbf{U}_{\delta} , \mathbf{Z} \right\rangle \,d\rho d\theta \\
&\qquad = 2 \int_R Q ( \mathrm{sym}\,\mathbf{Z}) - \int_R \langle \mathbb{C}  \mathrm{sym}\, \mathbf{Z} , \mathrm{sym}\,\mathbf{Z} \rangle\, d\rho d\theta = 0.
\end{align*}
As $Q(\mathbf{F}) \sim |\mathbf{F}|^2$ for all $\mathbf{F} \in \text{Sym}_2$, it follows that 
$$\|\mathrm{sym}\,D^{\delta} \mathbf{U}_{\delta} - \partial_\rho \mathbf{U}_0\odot \mathbf{e}_\rho - \mathbf{m}_0\odot \mathbf{e}_\theta\|_{L^2(R)}  = \|\mathrm{sym}\,D^{\delta} \mathbf{U}_{\delta} - \mathrm{sym}\,\mathbf{Z}\|_{L^2(R)} \rightarrow 0.$$

\textit{Step 4: Changing variables back to physical coordinates.}  In comparing $\mathbf{u}_{\delta}$ to its counterpart in logarithmic coordinates $\mathbf{U}_{\delta}$, we conclude from   \eqref{eq:strongConvergeLogStrain} that
\begin{align*}
\log\big(\frac{1}{\d}\big)\int_{\Omega_{\delta}} \big| \mathbf{e}(\mathbf{u}_{\delta}) - \frac{1}{\log \frac{1}{\delta}} \mathbf{e}(\mathbf{u}_{\text{Fl}})\big|^2 \, d\mathbf{x}=\int_{R} | \sym\, D^{\delta} \mathbf{U}_{\delta} - \sym\, D^{\delta} \mathbf{U}_{\text{Fl},\delta} |^2 \, d \rho d \theta  \to 0
\end{align*}
after changing variables via \eqref{eq:strainTologStrain} and \eqref{eq:UFltouFlLin}, and using that $\sym\, D^{\delta} \mathbf{U}_{\text{Fl},\delta} = \partial_\rho \mathbf{U}_0\odot \mathbf{e}_\rho + \mathbf{m}_0\odot \mathbf{e}_\theta$ by \eqref{eq:FlamantAnsatzGradient}. This proves the desired $L^2$-estimate on the strain in \eqref{eq:stressStrain_L2}, at least for the subsequence obtained in Step 1. The analogous result for the stress follows since $\boldsymbol{\sigma}_{\text{Fl}} = \mathbb{C} \mathbf{e}(\mathbf{u}_{\text{Fl}})$.

\textit{Step 5: Uniqueness of the limits.} Finally, we address the original sequence of almost minimizers. Recall we passed to a subsequence to obtain \eqref{eq:stressStrain_L2}. However, $\mathbf{e}(\mathbf{u}_{\text{Fl}})$ and $\boldsymbol{\sigma}_{\text{Fl}}$ are uniquely identified by $\mathbf{u}_0^{+} - \mathbf{u}_0^{-}$ and $\mathbf{f}_0$ in the displacement and force problems. As a result,  starting from any subsequence of almost minimizers, we can find a further subsequence that satisfies \eqref{eq:stressStrain_L2}. Therefore \eqref{eq:stressStrain_L2} holds for the original sequence.
\end{proof}

\section{The displacement problem for the nonlinear elastic wedge}\label{sec:displacement}

The previous section set the stage by establishing the asymptotics of the linear displacement and force problems for the truncated wedge. The rest of this paper concerns the original, nonlinear problems. We have the complication of justifying linear elasticity in a setting where the linear energy is infinite, as described in the introduction. Relatedly, there are two scales, namely $\epsilon$ for the boundary displacements or forces, and $\delta$ for the truncated tip. Depending on their relation, there can exist a portion of the wedge where the response is nonlinear. To rule this out, we impose a condition on the stiffness of the elastic medium, enforced via super-quadratic growth of the energy density at infinity ($p>2$ in \eqref{eq:Wgrowth}), or by an extra assumption on the parameters (see \eqref{eq:pEqual2Limit}). The main point is to find a way to ``propagate'' the boundary data from the very tip of the wedge into its bulk, which can be shown by a separate argument to respond linearly.

In this section, we prove that the nonlinear displacement problem
\begin{align*}
    \mathcal{E}_{p,\delta, \epsilon}^{\text{disp}} = \min_{\mathbf{y} \in \mathcal{A}_{p,\delta,\epsilon}} E_{\delta}(\mathbf{y})
\end{align*}
defined in \eqref{eq:nonlinear-displacement-problem}-\eqref{eq:limsup1} of the introduction satisfies 
\begin{align}\label{sect4-to-prove}
\lim_{\delta, \epsilon \rightarrow 0} \frac{\log \frac{1}{\delta \vee \epsilon}}{\epsilon^2} \mathcal{E}_{p,\delta,\epsilon}^{\text{disp}} = Q_{\text{Fl}}(\mathbf{u}_0^+ - \mathbf{u}_0^{-}) 
\end{align}
where $\mathbf{u}_0^{\pm} \in \mathbb{R}^2$ are the limiting average displacements of the tip and end of the wedge. That is, we prove the statement in part (I) of Theorem \ref{thm:main-result} for the displacement problem (see Proposition \ref{convergence_energy_NL_disp}). 
To help the reader, we recall each assumption from the introduction  the first time it is used in this section and the next. We use the convention that once an assumption is recalled, it remains in effect.

\subsection{Compactness in logarithmic variables} 
We begin with a compactness result for deformations whose energy obeys  $E_\delta(\mathbf{y}_{\delta,\epsilon})\lesssim\epsilon^2(\log \frac{1}{\delta \vee \epsilon})^{-1}$. We expect to find for the displacements $\mathbf{u}_{\delta,\epsilon}=\epsilon^{-1}(\mathbf{y}_{\delta,\epsilon}-\mathbf{x})$ that $\nabla\mathbf{u}_{\delta,\epsilon}\sim1/r$ corrected by the logarithm of the core radius $\delta\vee\epsilon = \max\{\delta,\epsilon\}$. As such, we define the restricted wedge domain 
\[
\Omega_{\delta \vee \epsilon}=\{\mathbf{x}:r\in(\delta\vee\epsilon,1),\theta\in(\alpha,\beta)\}
\]
and the accompanying logarithmic change of variables
\begin{align}\label{eq:changeVarDeltaVeeEpsilon}
\mathbf{U}(\rho,\theta)=\mathbf{u}(\mathbf{x})\quad\text{for}\quad\rho=\frac{\log r}{\log\frac{1}{\delta \vee \epsilon}}
\end{align}
where $\mathbf{x}\in\Omega_{\delta\vee\epsilon}$ corresponds to $(\rho,\theta)\in R= (-1,0) \times (\alpha,\beta)$.
Note $r=\delta\vee\epsilon,1$ become $\rho=-1,0$. The logarithmic gradient $D^{\delta \vee \epsilon}$ is defined as in \eqref{eq:definition-of-rescaled-gradient}, and the identities in \eqref{eq:gradient-rule}-\eqref{eq:strainTologStrain} hold with $\delta \vee \epsilon$ in place of $\delta$.

The following is used to control both the bulk and tip of the wedge. Recall $\Omega_{\delta,\epsilon}=\Omega_{\delta}\cap\{r\in(\delta,\epsilon)\}$.

\begin{lem}\label{logScaleLemma} 
Let $\delta \in (0,1/2)$, $\epsilon \in (0,1)$, $p\geq 2$, and $\mathbf{u}\in W^{1,p}(\Omega_\d;\R^2)$. 
There holds
 \begin{align*}
 \Big|  \fint_{\alpha}^{\beta}   \mathbf{u}|_{r= 1}-\mathbf{u}|_{r = \delta \vee \epsilon}   \,d \theta   \Big| \lesssim \sqrt{ \log \frac{1}{\delta \vee \epsilon}}  \| \nabla \mathbf{u}\|_{L^2(\Omega_{\delta\vee\epsilon})}.
\end{align*}
If $\delta < \epsilon$, then 
\begin{align*}
 \Big|  \fint_{\alpha}^{\beta}  \mathbf{u}|_{r= \epsilon}-\mathbf{u}|_{r = \delta}   \,d \theta   \Big| \lesssim_p \frac{1}{c_p(\delta,\epsilon)}  \| \nabla \mathbf{u}\|_{L^p(\Omega_{\delta,\epsilon})}
 \end{align*}
 for
 \begin{equation}
\begin{aligned}\label{eq:constant-cp1}
c_p(\delta,\epsilon)=\begin{cases}
\left(\log\frac{\epsilon}{\delta}\right)^{-\frac{1}{2}} & \text{ if }p=2\\
\left(\frac{p-1}{p-2}\left(\epsilon^{\frac{p-2}{p-1}}-\delta^{\frac{p-2}{p-1}}\right)\right)^{\frac{1}{p}-1} & \text{ if }p>2
\end{cases}.
\end{aligned}
\end{equation}
 \end{lem}
 \begin{proof} 
The first part is copied from the linear section, with $\delta\vee\epsilon$ in place of $\delta$ (see \eqref{control_bc_log2}).
To make the link with the second part, notice that the key point was to check the identity  
\begin{align*}
 \frac{1}{\log \frac{1}{\delta \vee \epsilon}} \int_{\alpha}^{\beta}\int_{-1}^0 | \partial_\rho \mathbf{U} |^2 \, d\rho d \theta = \int_{\alpha}^{\beta}\int_{\delta \vee \epsilon}^1  | \partial_r \mathbf{u}|^2 r \, dr d \theta 
\end{align*}
for the change of variables \eqref{eq:changeVarDeltaVeeEpsilon}. Then, one can apply standard inequalities in the transformed domain. 

A similar argument works for $L^p$ with $p\geq2$. Focusing on the tip domain $\Omega_{\delta,\epsilon}$, we seek a change of variables $\tilde{\mathbf{U}}(\rho, \theta) = \mathbf{u}(\mathbf{x})$ with $\rho = \rho(r)$ such that
\begin{align}\label{eq:importChange-Vars-2}
 c_p^p \int_{\alpha}^{\beta} \int_{-1}^0 |\partial_\rho \tilde{\mathbf{U}} |^p \, d\rho d \theta = \int_{\alpha}^{\beta} \int_{\delta}^{\epsilon} | \partial_r \mathbf{u}|^p \, r dr d \theta 
\end{align}
for $c_p=c_p(\delta,\epsilon)$. It follows that
\begin{align*}
\Big(\rho'(r)^{p-1} r\Big)'  = 0, \quad \rho(-1) = \delta, \quad \rho(0) = \epsilon.
\end{align*}
Solving this, we  arrive at the general form of the transformation for $\delta < \epsilon$:
\begin{align*}
\tilde{\mathbf{U}}(\rho, \theta) = \mathbf{u}(\mathbf{x})  \quad \text{where}\quad\rho=\begin{cases}
\frac{\log(r/\epsilon)}{\log(\epsilon/\delta)} & \text{ if }p=2\\
\frac{r^{\frac{p-2}{p-1}}-\epsilon^{\frac{p-2}{p-1}}}{\epsilon^{\frac{p-2}{p-1}}-\delta^{\frac{p-2}{p-1}}} & \text{ if }p>2
\end{cases}
\end{align*}
for $\mathbf{x}\in\Omega_{\delta,\epsilon}$ and $(\rho,\theta)\in R$. The constant $c_p$ is in \eqref{eq:constant-cp1}. 
Given \eqref{eq:importChange-Vars-2}, we can now write that
\begin{align*}
\Big|\fint_{\alpha}^{\beta}\tilde{\mathbf{U}}|_{\rho=0}-\tilde{\mathbf{U}}|_{\rho=-1}\,d\theta\Big|^p & \leq\fint_{\alpha}^{\beta} \int_{-1}^{0} |\partial_{\rho}\tilde{\mathbf{U}}|^p\,d\rho d\theta 
\end{align*}
by the fundamental theorem of calculus and Jensen's inequality. Undoing the change of variables proves the result.  
\end{proof}

We come now to compactness. Recall the admissible set in the displacement problem is 
\begin{equation*}
\mathcal{A}_{p,\delta,\epsilon}  =\Big\{  \mathbf{y} \in W^{1,p}(\Omega_{\delta};\mathbb{R}^2)  \colon  \mathbf{y}(\mathbf{x})- \mathbf{x} =\epsilon\mathbf{u}_{\delta,\epsilon}^{\pm}(\theta)\ \text{at }r=\delta,1 \Big\}
\end{equation*}
per \eqref{eq:admissibleDispalcements}.  
We use assumptions \eqref{eq:meanDispConvergeNL} and \eqref{eq:mean-freePlus} from the introduction, which state that
\begin{equation}
\begin{aligned}\label{eq:restateAssumpNLDispCompact}
\fint_{\alpha}^{\beta} \mathbf{u}_{\delta, \epsilon}^{\pm} \, d\theta \rightarrow \mathbf{u}_0^{\pm} \quad \text{ and } \quad \sqrt{\log \frac{1}{\delta \vee \epsilon}}   \| \mathbf{u}_{\delta, \epsilon}^{+} \|_{\dot{W}^{1,\infty}((\alpha, \beta))}  \rightarrow 0
\end{aligned}
\end{equation}
as $\delta, \epsilon \rightarrow 0$. We also use the bound
\begin{equation*}
W(\mathbf{F}) \gtrsim d^2(\mathbf{F}, SO(2)) + d^p(\mathbf{F}, SO(2))
\end{equation*}
which holds for some $p\geq 2$ by \eqref{eq:Wgrowth}. 
In the special case $p =2$, the limit $\delta, \epsilon \rightarrow 0$ is further constrained via 
\begin{align}\label{eq:pEqual2LimitRestate}
\log \frac{\delta \vee \epsilon}{\delta}\ll \log \frac{1}{\epsilon} 
\end{align}
per \eqref{eq:pEqual2Limit}. 
\begin{prop}\label{CompactnessDispThm} Consider any limit  $\delta, \epsilon \rightarrow 0$ if the elastic energy density $W$ has super-quadratic growth ($p >2$) or assume such a limit obeys  \eqref{eq:pEqual2LimitRestate} in the quadratic growth case ($p =2$). Let $\mathbf{y}_{\delta, \epsilon} \in \mathcal{A}_{p,\delta,\epsilon}$ satisfy
\begin{align}\label{eq:boundedEnergyDisp}
\limsup_{\delta,\epsilon \rightarrow 0} \frac{\log \frac{1}{\delta \vee \epsilon}}{\epsilon^2} E_{\delta}(\mathbf{y}_{\delta,\epsilon})  < \infty,
\end{align}
and define $\{ \mathbf{U}_{\delta,\epsilon}\} \subset W^{1,p}(R; \mathbb{R}^2)$ by
\begin{align*}
\mathbf{U}_{\delta, \epsilon}(\rho ,\theta) = \frac{1}{\epsilon}\big( \mathbf{y}_{\delta, \epsilon} (\mathbf{x}) - \mathbf{x} \big)
\end{align*}
as in \eqref{eq:changeVarDeltaVeeEpsilon}. There exists a subsequence (not relabeled) such that 
\begin{align*}
\mathbf{U}_{\delta,\epsilon}\rightharpoonup\mathbf{U},\quad\partial_{\rho}\mathbf{U}_{\delta,\epsilon}\rightharpoonup\partial_{\rho}\mathbf{U},\quad\log\left(\frac{1}{\delta \vee \epsilon}\right)\partial_{\theta}\mathbf{U}_{\delta,\epsilon}\rightharpoonup\mathbf{m}\quad\text{weakly in }L^{2}(R;\R^2)
\end{align*}
for some $\mathbf{m}\in L^{2}(R;\mathbb{R}^{2})$ and  $\mathbf{U} \in W^{1,2}(R; \mathbb{R}^2)$ with $\partial_{\theta} \mathbf{U} = \mathbf{0}$. 
Furthermore, 
\begin{align*}
\mathbf{U}_{\delta, \epsilon} |_{\rho = -1}  \rightharpoonup \mathbf{U}|_{\rho = -1} = \mathbf{u}_0^{-} ,\quad \mathbf{U}_{\delta, \epsilon} |_{\rho = 0}  \rightharpoonup \mathbf{U}|_{\rho = 0} = \mathbf{u}_0^{+} \quad\text{weakly in } L^2((\alpha,\beta);\R^2)
\end{align*}
where $\mathbf{u}_0^\pm$ are the limits given in \eqref{eq:restateAssumpNLDispCompact}.
\end{prop}
\begin{rem}
We assume more with \eqref{eq:restateAssumpNLDispCompact} than is required by compactness, since the results follow even if we weaken the assumptions on the mean-free parts to
\begin{align*}
    \sqrt{\log \frac{1}{\delta \vee \epsilon}} \| \mathbf{u}_{\delta, \epsilon}^+\|_{\dot{W}^{\frac{1}{2},2}((\alpha, \beta))} \lesssim  1, \quad \| \mathbf{u}_{\delta, \epsilon}^+\|_{\dot{W}^{1-\frac{1}{p},p}((\alpha,\beta))} \rightarrow 0.  
\end{align*}
The stronger assumptions in \eqref{eq:restateAssumpNLDispCompact} are used to construct recovery sequences. 
\end{rem}
\begin{proof}
Part of the proof is straightforward, given its similarity to Proposition \ref{prop:linear-compactness}(i). There is, however, nuance to the case $\d < \e$, where we need to propagate the boundary data from $r=\delta$ to $r=\epsilon$. We employ a rigidity inequality along with Lemma \ref{logScaleLemma} to show that this follows from the energy bound \eqref{eq:boundedEnergyDisp}.

To begin, apply the $p=2$ version of the second estimate in    Corollary \ref{cor:p-rigidity-wedge} to obtain 
\begin{align*}
\epsilon^2 \int_{\Omega_{\delta \vee \epsilon}}|\nabla \mathbf{u}_{\d,\e}|^2\,d \mathbf{x} \leq \| \nabla \mathbf{y}_{\d,\e} - \mathbf{I}\|^2_{L^2(\Omega_{\delta})}  \lesssim  E_{\delta}(\mathbf{y}_{\d,\e}) + \epsilon^2 \| \mathbf{u}_{\delta,\epsilon}^{+}\|^2_{\dot{W}^{\frac12,2}((\alpha,\beta))}
\end{align*}
using that $\mathbf{y}_{\delta,\epsilon}=\mathbf{x}+\epsilon\mathbf{u}_{\delta,\epsilon}^+$ at $r = 1$ and since $d^2(\mathbf{F}, SO(2)) \lesssim W(\mathbf{F})$. 
The logarithmic counterpart $\mathbf{U}_{\d,\e}$ of $\mathbf{u}_{\d,\e}$ satisfies
\begin{align*}
\int_{R} | D^{\delta\vee \epsilon} \mathbf{U}_{\delta, \epsilon} |^2 \, d\rho d \theta = \log \big(\frac{1}{\delta \vee \epsilon}\big) \int_{\Omega_{\delta \vee \epsilon}} |\nabla \mathbf{u}_{\delta, \epsilon}|^2 \, d \mathbf{x}  \lesssim 1 + \log\big(\frac{1}{\delta \vee \epsilon}\big) \| \mathbf{u}_{\delta, \epsilon}^+\|^2_{\dot{W}^{1,\infty}((\alpha,\beta))}
\end{align*}
by \eqref{eq:strainTologStrain} with $\d\vee \e$ instead of $\d$, \eqref{eq:boundedEnergyDisp}, and since $\|\mathbf{u}_{\delta,\epsilon}^+\|_{\dot{W}^{\frac{1}{2},2}((\alpha,\beta))} \lesssim \|\mathbf{u}_{\delta,\epsilon}^+\|_{\dot{W}^{1,\infty}((\alpha,\beta))}$. An application of Poincar\'e's inequality akin to the linear section (cf.~\eqref{poincare_log}) gives that 
$$\|\mathbf{U}_{\delta,\epsilon}\|_{L^2(R)} \lesssim \|  D^{\delta\vee \epsilon} \mathbf{U}_{\delta, \epsilon} \|_{L^2(R)} + \Big|\fint_{\alpha}^{\beta}  \mathbf{u}_{\delta, \epsilon}^+ \, d\theta\Big|$$ 
since $\mathbf{U}_{\delta, \epsilon} |_{\rho = 0} =  \mathbf{u}_{\delta, \epsilon}^+$. Combined with the assumptions on $\{ \mathbf{u}_{\delta, \epsilon}^{+}\}$ in  \eqref{eq:restateAssumpNLDispCompact}, these observations prove the desired compactness for $\{ \mathbf{U}_{\delta, \epsilon} \}$.

We discuss the boundary data now. First, note that $\mathbf{U}_{\delta, \epsilon} |_{\rho = -1}  \rightharpoonup \mathbf{U}|_{\rho = -1}$ and   $\mathbf{U}_{\delta, \epsilon} |_{\rho = 0}  \rightharpoonup \mathbf{U}|_{\rho = 0}$ weakly in $L^2((\alpha,\beta);\R^2)$ by the trace theorem. That $\mathbf{U}|_{\rho = 0} = \mathbf{u}_0^+$ is automatic because $\mathbf{U}$ is independent of $\theta$ and 
\begin{align}\label{eq:uChainBC}
\mathbf{U}|_{\rho=0} =\fint_\alpha^\beta \mathbf{U}|_{\rho=0} \, d\theta = \lim_{\delta, \epsilon \rightarrow 0} \fint_\alpha^{\beta}  \mathbf{U}_{\delta, \epsilon} |_{\rho = 0} \, d\theta = \lim_{\delta, \epsilon \rightarrow 0} \fint_\alpha^{\beta} \mathbf{u}_{\delta,\epsilon}^+ \, d \theta  = \mathbf{u}_0^+.\end{align}
On extracting another subsequence if needed, we are free to assume that the limit $\delta,\epsilon \rightarrow 0$ satisfies $\epsilon \leq \delta$ or $\delta < \epsilon$. If $\epsilon \leq \delta$, then $\mathbf{U}|_{\rho = -1} = \mathbf{u}_0^-$ by the same reasoning as in \eqref{eq:uChainBC} since $\fint_\alpha^\beta \mathbf{u}_{\delta, \epsilon}^{-} \, d \theta \rightarrow \mathbf{u}_0^{-}$ in \eqref{eq:restateAssumpNLDispCompact}. 
If instead $\delta < \epsilon$, the second estimate in Lemma \ref{logScaleLemma} and $\mathbf{U}_{\d,\e}|_{\rho=-1} = \mathbf{u}_{\d,\e}^{-}$ furnish the inequality
\begin{align*}
\Big| \fint_{\alpha}^{\beta} \mathbf{U}_{\delta,\epsilon}|_{\rho = -1}\,d\theta -  \fint_{\alpha}^{\beta} \mathbf{u}_{\delta,\epsilon}^-  \,d \theta\Big|^p  
\lesssim_p \frac{1}{c_p^p}\|\nabla\mathbf{u}_{\d,\e}\|_{L^p(\Omega_\d)}^p
\end{align*}
with $c_p=c_p(\d,\e)$ given by \eqref{eq:constant-cp1}. From its definition, this constant satisfies
\begin{align*}
\frac{1}{c_p^p} \sim_p \begin{cases} 
\log(\frac{\epsilon}{\delta}) &  p =2 \\
\epsilon^{p-2} ( 1 - \big(\frac{\delta}{\epsilon})^{\frac{p-2}{p-1}}\big)^{p-1} \sim \epsilon^{p-2}  & p >2
\end{cases}.
\end{align*}
Using the $L^p$ version  of the second estimate  in     Corollary \ref{cor:p-rigidity-wedge},  the growth condition $d^p(\mathbf{F}, SO(2)) \lesssim W(\mathbf{F})$, and the $r =1$ boundary condition on $\mathbf{y}_{\delta, \epsilon}$, we find that 
\begin{align*}
\Big| \fint_{\alpha}^{\beta} \mathbf{U}_{\delta,\epsilon}|_{\rho = -1}\,d\theta -  \fint_{\alpha}^{\beta} \mathbf{u}_{\delta,\epsilon}^-  \,d \theta\Big|^p  \lesssim_p \frac{1}{\e^pc_p^p} E_\d(\mathbf{y}_{\d,\e}) + \| \mathbf{u}_{\delta,\epsilon}^{+}\|_{\dot{W}^{1-\frac1p,p}((\alpha,\beta))} \lesssim_p \frac{\epsilon^{2-p}}{c_p^p \log \frac{1}{\delta \vee \epsilon}} + \| \mathbf{u}_{\delta, \epsilon}^+\|_{\dot{W}^{1,\infty}((\alpha,\beta))}
\end{align*}
by \eqref{eq:boundedEnergyDisp}. This goes to zero for $p>2$ by \eqref{eq:restateAssumpNLDispCompact}; it also goes to zero for $p =2$ by the added assumption \eqref{eq:pEqual2LimitRestate}. In any case, $\mathbf{U}|_{\rho = -1} = \mathbf{u}_0^-$. 
\end{proof}

\subsection{Convergence of the energies} We are ready to prove \eqref{sect4-to-prove}. Like in the linear setting, we use the logarithmic change of variables to obtain a lower bound, and then construct a recovery sequence saturating this bound. The new ingredients involve Taylor expanding the  energy density and incorporating the tip of the wedge. 
Here, we  make full use of the assumptions on the data in  \eqref{eq:restateAssumpNLDispCompact}, as well as the final assumptions  \eqref{eq:mean-freeMinus}-\eqref{eq:limsup1} from the setup of the nonlinear displacement problem:
\begin{equation}
\begin{aligned}\label{eq:finalMeanFree}
&\sqrt{ \log \frac{1}{\delta \vee \epsilon}} \|\mathbf{u}_{\delta, \epsilon}^{-} \|_{\dot{W}^{1,\infty}((\alpha, \beta))}\rightarrow 0 \quad \quad \quad   \text{ if } \delta \gtrsim \frac{\epsilon}{\sqrt{\log \frac{1}{\delta \vee \epsilon}}} , \\
    &\limsup \,\frac{\epsilon}{\delta} \|\mathbf{e}_{r} \cdot \frac{d}{d\theta} \mathbf{u}_{\delta, \epsilon}^{-}\|_{L^{\infty}((\alpha,\beta))} < \infty \quad \text{ if } \delta  \ll  
    \frac{\epsilon}{\sqrt{\log \frac{1}{\epsilon}}}, \\
    & \limsup \,\frac{\epsilon}{\delta} \|\mathbf{e}_{\theta} \cdot \frac{d}{d\theta} \mathbf{u}_{\delta, \epsilon}^{-}\|_{L^{\infty}((\alpha,\beta))} < 1  \quad  \text{\;\;  if } \delta  \ll  
    \frac{\epsilon}{\sqrt{\log \frac{1}{\epsilon}}}.
\end{aligned}
\end{equation}

\begin{prop}\label{convergence_energy_NL_disp}
Consider any limit $\delta, \epsilon \rightarrow 0$ in the super-quadratic growth case ($p>2$) or assume such a limit obeys \eqref{eq:pEqual2LimitRestate} in the quadratic growth case ($p=2$). 
The minimum energies obey
\begin{align*}
\frac{\log \frac{1}{\delta \vee \epsilon}}{\epsilon^2}   \mathcal{E}_{p,\delta,\epsilon}^{\emph{disp}} \to \QFl( \mathbf{u}_0^+ - \mathbf{u}_0^{-})
\end{align*}
where $\mathbf{u}_0^\pm$ are the limits given in \eqref{eq:restateAssumpNLDispCompact} and $\QFl$ is defined in \eqref{Flamant_energies}.
\end{prop}
\begin{proof}
\textit{Step 1: The lower bound.} We first prove that
\begin{align*}
\liminf_{\delta,\epsilon \rightarrow 0}\, \frac{\log \frac{1}{\delta \vee \epsilon}}{\epsilon^2}   \mathcal{E}_{p,\delta,\epsilon}^{\text{disp}} \geq \QFl( \mathbf{u}_0^+ - \mathbf{u}_0^{-}).
\end{align*}
If the left side is $+ \infty$, there is nothing to prove. Otherwise, there exists  a sequence $\{ \mathbf{y}_{\delta, \epsilon}\}$ with $\mathbf{y}_{\delta, \epsilon} \in  \mathcal{A}_{p,\delta,\epsilon}$ such that 
\begin{align}\label{eq:minimizingYdeltaEpsilonDisp}
\lim_{\epsilon, \delta \rightarrow 0}\,  \frac{\log \frac{1}{\delta \vee \epsilon}}{\epsilon^2}  E_{\delta}(\mathbf{y}_{\delta,\epsilon}) = \liminf_{\delta,\epsilon \rightarrow 0}\, \frac{\log \frac{1}{\delta \vee \epsilon}}{\epsilon^2}   \mathcal{E}_{p,\delta,\epsilon}^{\text{disp}} < \infty.
\end{align}
Arguing as in the proof of Proposition \ref{CompactnessDispThm}, the displacement   $\mathbf{u}_{\delta, \epsilon} = \epsilon^{-1} ( \mathbf{y}_{\delta, \epsilon} - \mathbf{x})$ satisfies 
\begin{align}\label{eq:this_rigidityEst}
\int_{\Omega_{\delta\vee\epsilon}}|\nabla\mathbf{u}_{\delta,\epsilon}|^{2}\, d\mathbf{x} \lesssim\frac{1}{\epsilon^{2}}E_{\delta}(\mathbf{y}_{\delta,\epsilon}) + \|\mathbf{u}_{\delta, \epsilon}^+ \|^2_{\dot{W}^{1,\infty}((\alpha,\beta))}  \lesssim \frac{1}{\log\frac{1}{\delta\vee\epsilon}}
\end{align}
by \eqref{eq:restateAssumpNLDispCompact} and 
\eqref{eq:minimizingYdeltaEpsilonDisp}. Define 
\begin{equation}\label{eq:bounded-set}
    B_{\delta,\epsilon}=\left\{ \mathbf{x}\in\Omega_{\delta\vee\epsilon}:|\epsilon\nabla\mathbf{u}_{\delta,\epsilon}(\mathbf{x})|<M_{\delta,\epsilon}\right\} 
\end{equation}
for some $M_{\delta,\epsilon} >0$ to be determined, and note by Chebyshev's inequality that 
\begin{equation}
|\Omega_{\delta\vee\epsilon}\backslash B_{\delta,\epsilon}|\leq\frac{\epsilon^{2}}{M_{\delta,\epsilon}^{2}}\int_{\Omega_{\delta\vee\epsilon}}|\nabla\mathbf{u}_{\delta,\epsilon}|^2 \, d\mathbf{x} \lesssim \frac{\epsilon^{2}}{M_{\delta,\epsilon}^{2}\log\frac{1}{\delta\vee\epsilon}}.\label{eq:measure-estimate}
\end{equation}
So long as $M_{\delta,\epsilon}\to0$,  frame indifference and Taylor
expansion in the form $W(\mathbf{I}+\mathbf{F})=Q(\sym\,\mathbf{F})+o(|\mathbf{F}|^{2})$
yields for $\nabla\mathbf{y}_{\delta,\epsilon}=\mathbf{I}+\epsilon\nabla\mathbf{u}_{\delta,\epsilon}$
that
\begin{align}\label{eq:TaylorExpansionNonlinear}
\int_{B_{\delta,\epsilon}}W(\nabla\mathbf{y}_{\delta,\epsilon}) \, d \mathbf{x} =\int_{B_{\delta,\epsilon}}Q\big(\epsilon\mathbf{e}(\mathbf{u}_{\delta,\epsilon})\big) \, d \mathbf{x}+o\Big(\int_{B_{\delta,\epsilon}}|\epsilon\nabla\mathbf{u}_{\delta,\epsilon}|^{2} \, d \mathbf{x} \Big)  = \epsilon^2\int_{B_{\delta,\epsilon}}Q\big(\mathbf{e}(\mathbf{u}_{\delta,\epsilon})\big) \, d\mathbf{x} +o\Big(\frac{\epsilon^{2}}{\log\frac{1}{\delta\vee\epsilon}}\Big)
\end{align}
by \eqref{eq:this_rigidityEst} and since $Q$ is quadratic. Thus, 
\begin{align}\label{eq:E0dispIneq}
\lim_{\epsilon, \delta \rightarrow 0}\,  \frac{\log \frac{1}{\delta \vee \epsilon}}{\epsilon^2}  E_{\delta}(\mathbf{y}_{\delta,\epsilon}) \geq \liminf_{\delta, \epsilon \rightarrow 0}\, \log  \frac{1}{\delta \vee \epsilon} \int_{B_{\delta, \epsilon}} Q\big(\mathbf{e}( \mathbf{u}_{\delta, \epsilon} )\big) \, d \mathbf{x} 
\end{align}
because $W\geq 0$ and $B_{\d,\e}\subset \Omega_{\d \vee \e}\subset \Omega_{\d}$. We now apply the compactness in Proposition \ref{CompactnessDispThm} to pass to the limit in \eqref{eq:E0dispIneq} and deduce the lower bound. This is possible by the energy bound \eqref{eq:minimizingYdeltaEpsilonDisp}. The only difficulty is that $B_{\delta, \epsilon}$ may be strictly smaller that $\Omega_{\delta \vee \epsilon}$. This is taken care of by an appropriate choice for $M_{\delta, \epsilon} \rightarrow 0$.

Change variables to $\mathbf{U}_{\delta, \epsilon}(\rho, \theta) = \mathbf{u}_{\delta, \epsilon}(\mathbf{x})$ for $\rho = \frac{\log r}{\log \frac{1}{\delta \vee \epsilon}}$ and observe that  
\begin{align}\label{eq:dispGradtoDispGrad}
 \int_{B_{\delta, \epsilon}} Q\big(\mathbf{e}( \mathbf{u}_{\delta, \epsilon} )\big) \, d \mathbf{x}  =  \frac{1}{\log \frac{1}{\delta \vee \epsilon}} \int_R Q( \sym\, \indicator{\hat{B}_{\delta,\epsilon}} D^{\delta \vee \epsilon} \mathbf{U}_{\delta, \epsilon}) \, d\rho d \theta,
\end{align}
where $\hat{B}_{\delta, \epsilon}$ is the set of $(\rho, \theta) \in R$ corresponding to $\mathbf{x} \in B_{\delta,\epsilon}$. There is a subsequence such that 
\begin{align}\label{eq:weakConvergenceDdeltaEpsilon}
D^{\delta \vee \epsilon} \mathbf{U}_{\delta, \epsilon} = \partial_{\rho} \mathbf{U}_{\delta, \epsilon} \otimes \mathbf{e}_\rho + \frac{1}{\log \frac{1}{\delta \vee \epsilon}} \partial_{\theta} \mathbf{U}_{\delta, \epsilon} \otimes \mathbf{e}_{\theta} \rightharpoonup \partial_{\rho} \mathbf{U} \otimes \mathbf{e}_{\rho}  + \mathbf{m} \otimes \mathbf{e}_{\theta}  \quad \text{ weakly in } L^2(R;\R^{2\times 2})
\end{align}
for some $\mathbf{U} \in W^{1,2}(R;\mathbb{R}^2)$ independent of $\theta$ and some $\mathbf{m} \in L^2(R; \mathbb{R}^2)$ by Proposition \ref{CompactnessDispThm}. Since $d\rho d \theta = \frac{1}{|\mathbf{x}|^2 \log \frac{1}{\delta \vee \epsilon}} \, d\mathbf{x}$,  the estimate for $B_{\delta, \epsilon}$ in \eqref{eq:measure-estimate} yields 
\begin{align}\label{eq:RBhatCalc}
|R \setminus \hat{B}_{\delta, \epsilon} | = \int_{R} \indicator{R\setminus \hat{B}_{\delta, \epsilon}} \, d\rho d\theta  =  \int_{\Omega_{\delta \vee \epsilon}}  \frac{\indicator{\Omega_{\delta \vee \epsilon} \setminus B_{\delta, \epsilon}}}{|\mathbf{x}|^2 \log \frac{1}{\delta \vee \epsilon}} \, d\mathbf{x}  \leq  \frac{ |\Omega_{\delta \vee \epsilon} \setminus B_{\delta, \epsilon}| }{(\delta \vee \epsilon)^2  \log \frac{1}{\delta \vee \epsilon}} \lesssim \Big( \frac{\epsilon}{M_{\delta, \epsilon} (\delta \vee \epsilon) \log \frac{1}{\delta \vee \epsilon}} \Big)^2 
\end{align}
By taking $M_{\delta,\epsilon} = (\log \frac{1}{\delta \vee \epsilon})^{-1/2}  \rightarrow 0$, for instance, we  justify the Taylor expansion in \eqref{eq:TaylorExpansionNonlinear} and conclude from \eqref{eq:RBhatCalc} and the identity $\int_R |\indicator{\hat{B}_{\delta, \epsilon}} - 1|^2 \,d \rho d \theta =   \int_{R} \indicator{R\setminus \hat{B}_{\delta, \epsilon}} \, d\rho d\theta$  that 
\begin{align*}
    \indicator{\hat{B}_{\delta, \epsilon}} \rightarrow 1 \text{ in } L^2(R).    
\end{align*}
Combining this convergence with that of \eqref{eq:weakConvergenceDdeltaEpsilon} furnishes 
\begin{align*}
\indicator{\hat{B}_{\delta, \epsilon}}  D^{\delta \vee \epsilon} \mathbf{U}_{\delta, \epsilon} \rightharpoonup \partial_{\rho} \mathbf{U} \otimes \mathbf{e}_{\rho}  + \mathbf{m} \otimes \mathbf{e}_{\theta}  \quad \text{ weakly in } L^2(R;\R^{2\times 2}).
\end{align*}
Hence, by the convexity of $Q$ and the identity in \eqref{eq:dispGradtoDispGrad},
\begin{align}\label{eq:E0dispIneq1}
\liminf_{\delta, \epsilon \rightarrow 0}\, \log \frac{1}{\delta \vee \epsilon} \int_{B_{\delta, \epsilon}} Q\big(\mathbf{e}( \mathbf{u}_{\delta, \epsilon})\big) \, d \mathbf{x} \geq \int_{R} Q \big(  \partial_{\rho} \mathbf{U} \odot \mathbf{e}_{\rho}  + \mathbf{m} \odot \mathbf{e}_{\theta}  \big) \, d\rho d \theta .
\end{align}
The compactness result in Theorem \ref{CompactnessDispThm} also yields $\mathbf{U}_{\delta, \epsilon}|_{\rho = -1}  \rightharpoonup \mathbf{U}|_{\rho = - 1} = \mathbf{u}_0^-$ and  $\mathbf{U}_{\delta, \epsilon}|_{\rho = 0}  \rightharpoonup \mathbf{U}|_{\rho = 0} = \mathbf{u}_0^+$ in $L^2((\alpha, \beta))$. So, by Lemma \ref{minimzingLemDisp},
\begin{align}\label{eq:E0dispIneq2}
 \int_{R} Q \big(  \partial_{\rho} \mathbf{U} \odot \mathbf{e}_{\rho}  + \mathbf{m} \odot \mathbf{e}_{\theta}  \big) \, d\rho d \theta \geq \QFl( \mathbf{u}_0^+ - \mathbf{u}_0^{-}).
\end{align}
Putting together the inequalities in \eqref{eq:minimizingYdeltaEpsilonDisp}, \eqref{eq:E0dispIneq}, \eqref{eq:E0dispIneq1}, and \eqref{eq:E0dispIneq2} completes the proof.

\textit{Step 2: The upper bound.} We now prove that
\begin{align}\label{eq:letsRecover}
\limsup_{\delta,\epsilon \rightarrow 0} \frac{\log \frac{1}{\delta \vee \epsilon}}{\epsilon^2}    \mathcal{E}_{p,\delta,\epsilon}^{\text{disp}} \leq \QFl( \mathbf{u}_0^+ - \mathbf{u}_0^{-})
\end{align}
by constructing a recovery sequence corresponding to the minimizer $(\mathbf{U}_0, \mathbf{m}_0)$ in Lemma \ref{minimzingLemDisp}.
Our approach is analogous to that of the linear model but needs some care near the tip.
Note the change of variables \eqref{eq:changeVarDeltaVeeEpsilon} implies that $\mathbf{x} \in \Omega_{\delta}$ corresponds to 
\begin{align}\label{eq:RdEpsilon}
(\rho, \theta ) \in \Big( \frac{\log \delta}{\log \frac{1}{\delta \vee \epsilon}}, 0 \Big) \times ( \alpha ,\beta ) =: R_{\delta, \epsilon}.
\end{align}

\textit{Step 2a: Definition of the recovery sequence.}
We construct $\mathbf{y}_{\delta, \epsilon}\in \mathcal{A}_{p, \delta, \epsilon}$ as follows.
Let
\begin{align*}
	\mathbf{y}_{\d,\e}(\mathbf{x}) \coloneqq \mathbf{x} +  \epsilon \mathbf{u}_{\d,\e}(\mathbf{x})
\end{align*}
for $\mathbf{x} \in \Omega_{\delta}$, where $\mathbf{u}_{\d,\e}(\mathbf{x}) = \mathbf{U}_{\d,\e}(\rho,\theta)$ is given by
\begin{align}\label{eq:Recovery_NL_Displ}
	\mathbf{U}_{\d,\e}(\rho,\theta) \coloneqq \UFlde(\rho,\theta) - \frac{1}{\log \frac{1}{\d\vee\e}}\mathbf{V}_{\d,\e}(\rho,\theta) + \mathbf{A}_{\d,\e}(\rho) + \bar{\mathbf{U}}_{\d,\e}(\rho,\theta)
\end{align}
for $(\rho,\theta)\in R_{\d,\e}$.
The quantities $\mathbf{u}_{{\rm Fl},\d,\e}$, $\mathbf{v}_{\d,\e}$, $\mathbf{a}_{\d,\e}$, and $\bar{\mathbf{u}}_{\d,\e}$ are the Cartesian counterparts of the functions on the right-hand side.
The Flamant ansatz $\UFlde$ is defined analogously to \eqref{eq:FlamantAnsatz} in that
\begin{align}\label{eq:FlamantAnsatz_NL}
\UFlde(\rho,\theta) = \mathbf{U}_0(\rho)  + \frac{1}{\log \frac{1}{\d\vee\e}}\mathbf{V}_0(\theta) \quad \text{with}\quad  \mathbf{V}_0(\theta)=\int_\alpha^\theta \mathbf{m}_0 \,d\tilde{\theta} 
\end{align}
for $(\rho,\theta)\in R_{\d,\e}$; here $\mathbf{U}_0$ is canonically extended via its constant trace at $\rho = -1$, i.e., $\mathbf{U}(\rho) = \mathbf{u}_0^-$ for all $\rho \in (\tfrac{\log \delta}{\log \frac{1}{\delta \vee \epsilon}}, -1).$
Just as in the linear section, the function $\mathbf{V}_{\d,\e}$ cuts off $\mathbf{V}_0$ in the sense that
\begin{align*}
	\mathbf{V}_{\delta,\epsilon}(\rho,\theta)=\begin{cases} 
		\mathbf{V}_0(\theta) & \text{ if } \rho \in   (\tfrac{\log \delta}{\log \frac{1}{\delta \vee \epsilon}}, -1)  \\ 
		\psi\big(\log(\frac{1}{\delta \vee \epsilon})(\rho+1)\big)\mathbf{V}_0(\theta)&\text{ if } \rho \in (-1,-1+\tfrac{1}{\log\frac{1}{\delta \vee \epsilon}})\\
		\mathbf{0} & \text{ if } \rho \in (-1+\tfrac{1}{\log\frac{1}{\delta \vee \epsilon}},-\tfrac{1}{\log\frac{1}{\delta \vee \epsilon}})\\
		\psi\big(-\log(\frac{1}{\delta\vee \epsilon})\rho\big)\mathbf{V}_0(\theta) &\text{ if } \rho \in (-\frac{1}{\log\frac{1}{\delta\vee \epsilon}},0)
	\end{cases}
\end{align*}
for a smooth function $\psi:[0,1]\to[0,1]$ with $\psi=1$ on $[0,1/4]$ and $\psi=0$ on $[3/4,1]$. Next, $\mathbf{A}_{\d,\e}$ is piecewise affine in $\rho$ with
\begin{align}\label{piecewise_affine_NL}
	\mathbf{A}_{\d,\e}(\rho) = 
	\begin{cases}
		\fint_\alpha^\beta \mathbf{u}_{\d,\e}^-\, d\theta -\mathbf{u}_0^- & \text{ if }\rho \in (\tfrac{\log \delta}{\log \frac{1}{\delta \vee \epsilon}}, -1)\\
		\fint_\alpha^\beta \mathbf{u}_{\d,\e}^+\, d\theta -\mathbf{u}_0^+ + \rho\Big(\fint_\alpha^\beta \mathbf{u}_{\d,\e}^+\, d\theta -\mathbf{u}_0^+ - \fint_\alpha^\beta \mathbf{u}_{\d,\e}^-\, d\theta +\mathbf{u}_0^-\Big) & \text{ if }\rho\in (-1,0)
		\end{cases}.
\end{align}
Finally, $\bar{\mathbf{U}}_{\d,\e}$ is the extension from Lemma \ref{lem:extensionNonlinearInfinity}  in logarithmic coordinates for the two traces
$\mathbf{u}_{\d,\e}^- - \fint_\alpha^\beta \mathbf{u}_{\d,\e}^-\, d\theta$ and 
$\mathbf{u}_{\d,\e}^+ - \fint_\alpha^\beta \mathbf{u}_{\d,\e}^+\, d\theta$ at the radii $r=\d$ and $1$; the free parameter $M>0$ in this extension is case-dependent. We take $M = 1$ if $\delta \gtrsim   \epsilon/\sqrt{\log \frac{1}{\delta \vee \epsilon}}$. Otherwise, $\delta \ll  \epsilon/\sqrt{\log \epsilon^{-1}}$ and $M$ will be chosen in Step 2d in terms of the boundary data. By taking $\delta,\epsilon$ sufficiently small, we assume without loss of generality that 
\begin{align}\label{eq:sufficientlySmall}
    \begin{cases}
    e^{M} \delta < \epsilon  < e \epsilon < \frac{1}{e} & \text{ in the case }  \delta \ll \frac{\epsilon}{\sqrt{\log \epsilon^{-1}}} \\
        e\delta \leq e(\delta \vee \epsilon) < \frac{1}{e} & \text{ in the other case }
    \end{cases}.
\end{align}
In particular, this allows Lemma \ref{lem:extensionNonlinearInfinity} to be applied.

\textit{Step 2b: Taylor expansion on the bulk of the wedge.}
We now analyze the energy of the construction. First, we linearize on the subset $\Omega_{e (\delta \vee \epsilon), \frac1e}$ of $\Omega_\d$ with radii $r\in \big(e (\delta \vee \epsilon), \frac1e\big)$. 
Under the logarithmic change of coordinates, $\mathbf{x} \in \Omega_{e (\delta\vee \epsilon), \frac1e}$ corresponds to 
\begin{align}\label{smallerR}
 (\rho, \theta) \in \Big(-1+\frac{1}{\log\frac{1}{\delta \vee \epsilon}},-\frac{1}{\log\frac{1}{\delta \vee \epsilon}}\Big) \times ( \alpha, \beta) =: \tilde{R}_{\delta, \epsilon}.
\end{align}
As $\bar{\mathbf{u}}_{\d,\e}$ vanishes inside $\Omega_{e (\delta\vee \e) , \frac1e}$ by construction (Lemma \ref{lem:extensionNonlinearInfinity} and \eqref{eq:sufficientlySmall}), we obtain from \eqref{eq:Recovery_NL_Displ}-\eqref{piecewise_affine_NL} that
\begin{equation}
\begin{aligned}\label{eq:DdeltaEpsUtot}
D^{\delta \vee \epsilon} \mathbf{U}_{\delta, \epsilon}   &=  \Big( \partial_\rho \mathbf{U}_0 + \fint_\alpha^\beta \mathbf{u}_{\d,\e}^+\, d\theta -\mathbf{u}_0^+ - \fint_\alpha^\beta \mathbf{u}_{\d,\e}^-\, d\theta -\mathbf{u}_0^- \Big)   \otimes \mathbf{e}_{\rho} + \mathbf{m}_0 \otimes \mathbf{e}_{\theta}  \\
&= \partial_\rho\mathbf{U}_0 \otimes \mathbf{e}_{\rho} + \mathbf{m}_0 \otimes \mathbf{e}_{\theta} + o(1)
\end{aligned}
\end{equation}
on $\tilde{R}_{\delta, \epsilon}$ by the convergence of the averages in \eqref{eq:restateAssumpNLDispCompact}. 
Hence,   
  $\|D^{\delta \vee \epsilon} \mathbf{U}_{\delta, \epsilon} \|_{L^{\infty}(\tilde{R}_{\delta, \epsilon})} \lesssim 1$, which in turn gives   
\begin{align*}
\|\epsilon \nabla  \mathbf{u}_{\delta, \epsilon}  \|_{L^{\infty}(\Omega_{e(\delta \vee \epsilon), \frac1e})} \leq \frac{\epsilon}{(\delta \vee \epsilon) \log \frac{1}{\delta \vee \epsilon}}  \| D^{\delta \vee \epsilon} \mathbf{U}_{\delta, \epsilon} \|_{L^{\infty}(\tilde{R}_{\delta, \epsilon})} \lesssim  \frac{1}{\log \frac{1}{\delta \vee \epsilon}}
\end{align*}
by the $\delta\vee\epsilon$ analog of the gradient identity \eqref{eq:gradient-rule}.
Taylor expanding $W$ on $\Omega_{e(\delta \vee \epsilon),\frac{1}{e}}$ yields
\begin{align*}
\int_{\Omega_{e(\delta \vee \epsilon), \frac1e}} W(\nabla \mathbf{y}_{\delta, \epsilon} ) \, d\mathbf{x} = \int_{\Omega_{e(\delta \vee \epsilon), \frac1e}} Q\big(\epsilon \mathbf{e}( \mathbf{u}_{\delta, \epsilon})\big) \, d\mathbf{x} + o \Big( \epsilon^2 \int_{\Omega_{e(\delta \vee \epsilon), \frac1e}} |\nabla \mathbf{u}_{\delta, \epsilon}|^2 \,d\mathbf{x} \Big) ,
\end{align*}
so in logarithmic coordinates 
\begin{align*}
\int_{\Omega_{e(\delta \vee \epsilon), \frac1e}} Q\big(\epsilon \mathbf{e}(\mathbf{u}_{\delta, \epsilon})\big) \, d\mathbf{x} &= \frac{\epsilon^2}{\log \frac{1}{\delta \vee \epsilon}} \int_{\tilde{R}_{\delta, \epsilon}}  Q (\sym\, D^{\delta \vee \epsilon} \mathbf{U}_{\delta, \epsilon} ) \,d \rho d \theta  \\
&= \frac{\epsilon^2}{\log \frac{1}{\delta \vee \epsilon}} \Big( \int_{R}  Q \big(  \partial_\rho \mathbf{U}_0 \odot \mathbf{e}_{\rho} + \mathbf{m}_0 \odot \mathbf{e}_{\theta}\big) \,d \rho d \theta + o(1) + O(|R \setminus \tilde{R}_{\delta, \epsilon}|) \Big),  \\
 \epsilon^2 \int_{\Omega_{e(\delta \vee \epsilon), \frac1e}} |\nabla \mathbf{u}_{\delta, \epsilon}|^2 \,d\mathbf{x}  &= \frac{\epsilon^2}{\log \frac{1}{\delta \vee \epsilon}} \int_{\tilde{R}_{\delta, \epsilon}} |D^{\delta \vee \epsilon} \mathbf{U}_{\delta, \epsilon} |^2 \, d\rho d \theta \lesssim \frac{\epsilon^2}{\log \frac{1}{\delta \vee \epsilon}}
\end{align*}
using the formula in \eqref{eq:DdeltaEpsUtot} and that $\| D^{\delta \vee \epsilon} \mathbf{U}_{\delta, \epsilon} \|_{L^{\infty}(\tilde{R}_{\delta, \epsilon})} \lesssim 1$. 
Since $|R \setminus \tilde{R}_{\delta, \epsilon}| \rightarrow 0$ (see \eqref{smallerR}), 
\begin{align}\label{eq:TaylorLimit}
\lim_{\delta, \epsilon \rightarrow 0}  \frac{\log \frac{1}{\delta \vee \epsilon}}{\epsilon^2} \int_{\Omega_{e(\delta \vee \epsilon), \frac1e}} W(\nabla \mathbf{y}_{\delta, \epsilon} ) \, d\mathbf{x}  = \int_{R} Q \big( \partial_\rho \mathbf{U}_0 \odot\mathbf{e}_{\rho} + \mathbf{m}_0 \odot \mathbf{e}_{\theta}\big) \,d \rho d \theta  = \QFl( \mathbf{u}_0^+ - \mathbf{u}_0^{-})
\end{align}
where in the last step we used Lemma \ref{minimzingLemDisp}.

\textit{Step 2c: Error estimates when $\delta \gtrsim \epsilon/ \sqrt{\log \frac{1}{\delta \vee \epsilon}}$.}
In this case, the energy in the exceptional sets $\Omega_{\delta, e( \delta \vee \epsilon)}$ and $\Omega_{\frac1e, 1}$ is controlled by Taylor expansion and a set of uniform bounds on the displacement gradient $\epsilon \nabla \mathbf{u}_{\delta, \epsilon}$.

Note from the definitions of $\mathbf{U}_{\text{Fl},\delta, \epsilon}$, $\mathbf{V}_{\delta, \epsilon}$, and $\mathbf{A}_{\delta, \epsilon}$ that
\begin{align*}
&\big\| D^{\delta \vee \epsilon} \big( \UFlde + \frac{1}{\log \frac{1}{\delta \vee \epsilon}} \mathbf{V}_{\delta,\epsilon} + \mathbf{A}_{\d,\e}\big) \big\|_{L^{\infty}(R_{\delta, \epsilon})}  \lesssim 1, \\
&D^{\delta \vee \epsilon} \big( \UFlde + \frac{1}{\log \frac{1}{\delta \vee \epsilon}} \mathbf{V}_{\delta,\epsilon} + \mathbf{A}_{\d,\e}\big) = \mathbf{0} \quad \text{ on } R_{\d,\e}\setminus R.
\end{align*}
Thus, 
\begin{equation}
\begin{aligned}\label{eq:TaylorBound1}
&\big\| \epsilon \nabla \big(\mathbf{u}_{{\rm Fl}, \delta, \epsilon} - \frac{1}{\log \frac{1}{\delta \vee \epsilon}} \mathbf{v}_{\delta, \epsilon} + \mathbf{a}_{\d,\e}\big) \big\|_{L^{\infty}(\Omega_{\delta, e(\delta \vee \epsilon)})} \lesssim \frac{\epsilon}{(\delta \vee \epsilon) \log \frac{1}{\delta \vee \epsilon}} \leq \frac{1}{\log \frac{1}{\delta \vee \epsilon}},\\
&\big\| \epsilon \nabla \big(\mathbf{u}_{{\rm Fl}, \delta, \epsilon} + \frac{1}{\log \frac{1}{\delta \vee \epsilon}} \mathbf{v}_{\delta, \epsilon} + \mathbf{a}_{\d,\e}\big) \big\|_{L^{\infty}(\Omega_{\frac1e, 1})} \lesssim \frac{\epsilon}{\log \frac{1}{\delta \vee \epsilon}}.
\end{aligned}
\end{equation}
By Lemma \ref{lem:extensionNonlinearInfinity}, $\bar{\mathbf{u}}_{\delta,\epsilon}=\mathbf{0}$ on $\Omega_{e\delta,\frac{1}{e}}$ and  
\begin{equation}
\begin{aligned}\label{eq:TaylorBound2}
    &\|\epsilon \nabla \bar{\mathbf{u}}_{\delta, \epsilon}\|_{L^{\infty}(\Omega_{\delta,e(\delta \vee \epsilon)})} \lesssim \frac{\epsilon}{\delta}\| \mathbf{u}^{-}_{\delta, \epsilon} \|_{\dot{W}^{1,\infty}((\alpha,\beta))} \lesssim  \sqrt{ \log \frac{1}{\delta \vee \epsilon}}  \| \mathbf{u}_{\delta, \epsilon}^{-}\|_{\dot{W}^{1,\infty}((\alpha,\beta))}, \\
    &\| \epsilon \nabla \bar{\mathbf{u}}_{\delta, \epsilon}\|_{L^{\infty}(\Omega_{\frac{1}{e},1})}  \lesssim  \epsilon \| \mathbf{u}_{\delta, \epsilon}^+\|_{\dot{W}^{1,\infty}((\alpha, \beta))}
\end{aligned}
\end{equation}
by the definition of the present step.  
As $\mathbf{u}_{\delta,\epsilon} = \mathbf{u}_{\text{Fl},\delta, \epsilon} - (\log \frac{1}{\delta \vee \epsilon})^{-1} \mathbf{v}_{\delta, \epsilon} + \mathbf{a}_{\delta, \epsilon} + \bar{\mathbf{u}}_{\delta, \epsilon}$,  we conclude from \eqref{eq:TaylorBound1}, \eqref{eq:TaylorBound2},  and the assumptions on the mean-free parts of the boundary data in \eqref{eq:restateAssumpNLDispCompact} and \eqref{eq:finalMeanFree} that 
\begin{align*}
\| \epsilon \nabla \mathbf{u}_{\delta, \epsilon} \|_{L^{\infty}( \Omega_{\delta,e(\delta \vee \epsilon)})} + \| \epsilon \nabla \mathbf{u}_{\delta, \epsilon} \|_{L^{\infty}( \Omega_{\frac{1}{e}, 1})}  \ll 1
\end{align*}
as $\delta, \epsilon \rightarrow 0$. Thus,      
\begin{align}\label{eq:getRecoveryNL1}
    \int_{\Omega_{\delta, e (\delta \vee \epsilon)} \cup \Omega_{\frac{1}{e},1}} W ( \nabla \mathbf{y}_{\delta, \epsilon}) \,d \mathbf{x} = O \Big(\int_{\Omega_{\delta, e (\delta \vee \epsilon)} \cup \Omega_{\frac{1}{e},1}}  | \epsilon \nabla \mathbf{u}_{\delta, \epsilon} |^2 \,d\mathbf{x} \Big) 
\end{align}
since $W(\mathbf{I}+ \mathbf{F}) = O(|\mathbf{F}|^2)$. 
Combining \eqref{eq:finalMeanFree}, \eqref{eq:TaylorBound1}, and \eqref{eq:TaylorBound2} yields that
\begin{equation}
\begin{aligned}\label{eq:getRecoveryNL2}
\int_{\Omega_{\delta,e(\delta \vee \epsilon)}}  \big|\epsilon \nabla \mathbf{u}_{\delta, \epsilon}|^2 \,d \mathbf{x} &\lesssim \int_{\Omega_{\delta,e(\delta \vee \epsilon)}} | \epsilon \nabla \big(\mathbf{u}_{{\rm Fl}, \delta, \epsilon} - \frac{1}{\log \frac{1}{\delta \vee \epsilon}} \mathbf{v}_{\delta, \epsilon} + \mathbf{a}_{\d,\e}\big)\big|^2\, d\mathbf{x}  +  \int_{\Omega_{\delta,e \delta}} |\epsilon \nabla \bar{\mathbf{u}}_{\delta, \epsilon}|^2 \, d\mathbf{x}  \\
&\lesssim \frac{(\delta \vee \epsilon)^2}{( \log \frac{1}{\delta \vee \epsilon})^2} + \frac{\epsilon^2}{\log \frac{1}{\delta \vee \epsilon}} \Big( \sqrt{\log \frac{1}{\delta \vee \epsilon}}   \| \mathbf{u}^{-}_{\delta, \epsilon} \|_{\dot{W}^{1,\infty}((\alpha,\beta))} \Big)^2  = o \Big( \frac{\epsilon^2}{\log \frac{1}{\delta \vee \epsilon}} \Big). 
\end{aligned}
\end{equation}
We also use \eqref{eq:TaylorBound1},  \eqref{eq:TaylorBound2}, and \eqref{eq:restateAssumpNLDispCompact} to get that 
\begin{equation}
\begin{aligned}\label{eq:getRecoveryNL3}
\int_{\Omega_{\frac{1}{e},1}}  \big|\epsilon \nabla \mathbf{u}_{\delta, \epsilon}|^2 \,d \mathbf{x}  &\lesssim \int_{\Omega_{\frac{1}{e},1}} | \epsilon \nabla \big(\mathbf{u}_{{\rm Fl}, \delta, \epsilon} - \frac{1}{\log \frac{1}{\delta \vee \epsilon}} \mathbf{v}_{\delta, \epsilon} + \mathbf{a}_{\d,\e}\big)\big|^2\, d\mathbf{x}  +  \int_{\Omega_{\frac{1}{e},1}} |\epsilon \nabla \bar{\mathbf{u}}_{\delta, \epsilon}|^2 \, d\mathbf{x} \\
&\lesssim \frac{\epsilon^2}{(\log \frac{1}{\delta \vee \epsilon})^2} + \frac{\epsilon^2}{\log \frac{1}{\delta \vee \epsilon}}\Big( \sqrt{\log \frac{1}{\delta \vee \epsilon}} \| \mathbf{u}_{\delta, \epsilon}^+ \|_{\dot{W}^{1,\infty}((\alpha,\beta))} \Big)^2 = o \Big( \frac{\epsilon^2}{\log \frac{1}{\delta \vee \epsilon}} \Big) .
\end{aligned}
\end{equation}
 Putting together  \eqref{eq:getRecoveryNL1}, \eqref{eq:getRecoveryNL2}, and \eqref{eq:getRecoveryNL3} furnishes the desired result:
\begin{align}\label{error2}
  \lim_{\delta, \epsilon \rightarrow 0}\, \frac{\log \frac{1}{\delta \vee \epsilon}}{\epsilon^2} \int_{\Omega_{\delta, e (\delta \vee \epsilon)} \cup \Omega_{\frac{1}{e},1}} W ( \nabla \mathbf{y}_{\delta, \epsilon}) \,d  \mathbf{x} = 0
\end{align}

\textit{Step 2d: Error estimates when  $\delta \ll \epsilon/ \sqrt{\log \frac{1}{\epsilon}}$.} 
Finally, there is an asymptotic regime where the tip can deform significantly; we still show  \eqref{error2} in this case. The construction is identical to the previous cases on $\Omega_{\frac{1}{e},1}$, and the same argument as before gives that
\begin{align}\label{eq:nonlinearEst1}
 \lim_{\delta, \epsilon \rightarrow 0}\, \frac{\log \frac{1}{\delta \vee \epsilon}}{\epsilon^2} \int_{\Omega_{\frac{1}{e},1}} W ( \nabla \mathbf{y}_{\delta, \epsilon}) \,d  \mathbf{x} = 0.
\end{align}
Also, $\bar{\mathbf{u}}_{\delta, \epsilon} = \mathbf{0}$ on $\Omega_{\epsilon, e \epsilon}$. Thus, the prior estimates on  $\mathbf{u}_{\text{Fl},\delta, \epsilon} - (\log \frac{1}{\delta \vee \epsilon})^{-1} \mathbf{v}_{\delta, \epsilon} + \mathbf{a}_{\delta, \epsilon}$ show that
\begin{align}\label{eq:nonlinearEst2}
    \lim_{\delta, \epsilon \rightarrow 0}\, \frac{\log \frac{1}{\delta \vee \epsilon}}{\epsilon^2} \int_{\Omega_{\epsilon,e \epsilon}} W ( \nabla \mathbf{y}_{\delta, \epsilon}) \,d  \mathbf{x} = 0
\end{align}
since $W(\mathbf{I}+\mathbf{F})=O(|\mathbf{F}|^2)$. To finish, we estimate the energy on $\Omega_{\delta, \epsilon}$. 

By definition,  $\mathbf{u}_{\text{Fl},\delta, \epsilon} - (\log \frac{1}{\delta \vee \epsilon})^{-1} \mathbf{v}_{\delta, \epsilon} + \mathbf{a}_{\delta, \epsilon}$ is   constant on $\Omega_{\delta, \epsilon}$ and $\bar{\mathbf{u}}_{\delta, \epsilon} = \mathbf{0}$ on $\Omega_{e^{M}\delta, \epsilon}$. Thus,  
\begin{align}\label{eq:nonlinearEst3}
    \int_{\Omega_{\delta, \epsilon}} W(\nabla \mathbf{y}_{\delta, \epsilon}) \, d \mathbf{x} = \int_{\Omega_{\delta, \epsilon}} W(\mathbf{I} + \epsilon \nabla \bar{\mathbf{u}}_{\delta, \epsilon}) \, d\mathbf{x} = \int_{\Omega_{\delta, e^{M}\delta}} W(\mathbf{I} + \epsilon \nabla \bar{\mathbf{u}}_{\delta, \epsilon}) \, d\mathbf{x}. 
\end{align}
The estimates in Lemma \ref{lem:extensionNonlinearInfinity} furnish
\begin{equation}
\begin{aligned}\label{eq:pickM}
&|\epsilon \nabla \bar{\mathbf{u}}_{\delta, \epsilon}|   \leq \frac{\epsilon}{\delta} \Big(  1 + \frac{\beta - \alpha}{2M} \Big) \| \mathbf{u}_{\delta, \epsilon}^{-} \|_{\dot{W}^{1,\infty}((\alpha,\beta))}, \\ 
    &\det ( \mathbf{I} + \epsilon \nabla \bar{\mathbf{u}}_{\delta, \epsilon} ) \geq 1  - \frac{\epsilon}{\delta} \| \mathbf{e}_{\theta} \cdot \frac{d}{d\theta} \mathbf{u}_{\delta, \epsilon}^{-} \|_{L^{\infty}((\alpha, \beta))} - \frac{\beta - \alpha}{2M} \sum_{k = 1,2} \Big(\frac{\epsilon}{\delta} \| \mathbf{u}_{\delta, \epsilon}^{-} \|_{\dot{W}^{1,\infty}((\alpha, \beta))} \Big)^{k} 
\end{aligned}
\end{equation}
on $\Omega_{\delta,e^{M}\delta}$.
By the assumptions in \eqref{eq:finalMeanFree},
\begin{align*}
M^{-} := \limsup_{\delta, \epsilon \rightarrow 0}  \frac{\epsilon}{\delta}  \| \mathbf{u}_{\delta, \epsilon}^{-} \|_{\dot{W}^{1,\infty}((\alpha, \beta))} < \infty
\end{align*}
and there exists $\tau > 0$ such that 
\begin{align*}
    \frac{\epsilon}{\delta} \| \mathbf{e}_{\theta} \cdot \frac{d}{d\theta} \mathbf{u}^{-} \|_{L^{\infty}((\alpha, \beta))} \leq 1 - \tau \quad \text{ and } \quad \sum_{k = 1,2} \Big(\frac{\epsilon}{\delta} \| \mathbf{u}_{\delta, \epsilon}^{-} \|_{\dot{W}^{1,\infty}((\alpha, \beta))} \Big)^{k} \leq M^{-}(1 + M^{-}) + \tau  
\end{align*}
for $\delta, \epsilon$ sufficiently small. Choosing $M = \frac{\beta - \alpha}{\tau} (M^{-}(1 +M^{-}) + \tau)$ in \eqref{eq:pickM} gives that
\begin{align*}
    |\epsilon \nabla \bar{\mathbf{u}}_{\delta, \epsilon} | \lesssim 1  \quad \text{ and } \quad \det ( \mathbf{I} + \epsilon \nabla \bar{\mathbf{u}}_{\delta, \epsilon} ) \geq \frac{\tau}{2} > 0  
\end{align*}
on $\Omega_{\delta,e^{M}\delta}$.  Since $W$ is bounded on compact subsets of $\{\det\mathbf{F}>0\}$ by assumption (W4) from Section \ref{ssec:Setup}, 
\begin{align}\label{eq:nonlinearEst4}
   \int_{\Omega_{\delta, e^{M}\delta}} W(\mathbf{I} + \epsilon \nabla \bar{\mathbf{u}}_{\delta, \epsilon}) \, d\mathbf{x} \lesssim_\tau \delta^2 \ll \frac{\epsilon^2}{\log \frac{1}{\delta \vee \epsilon}}
\end{align}
by the definition of the present case. Combining \eqref{eq:nonlinearEst1}-\eqref{eq:nonlinearEst3} and \eqref{eq:nonlinearEst4} produces \eqref{error2} as desired. 

\textit{Step 2e: Conclusion.} Using   \eqref{eq:TaylorLimit} and  \eqref{error2}, we have  that
\[
\lim_{\delta,\epsilon \rightarrow 0}\, \frac{\log \frac{1}{\delta \vee \epsilon}}{\epsilon^2}  E_{\delta}(\mathbf{y}_{\delta,\epsilon}) = \QFl( \mathbf{u}_0^+ - \mathbf{u}_0^{-}).
\]
This proves \eqref{eq:letsRecover} since $\mathbf{y}_{\delta, \epsilon}\in \mathcal{A}_{p, \delta, \epsilon}$ so that $\mathcal{E}_{p,\delta,\epsilon}^{\text{disp}} \leq E_{\delta}(\mathbf{y}_{\delta, \epsilon})$. 
\end{proof}

\section{The force problem for the nonlinear elastic wedge}\label{sec:force}
In this section, we study the nonlinear force problem 
\begin{align}\label{eq:totPotEnergy}
    \mathcal{E}_{p,\delta,\epsilon}^{\text{force}} = \min_{\mathbf{y} \in W^{1,p}(\Omega_{\delta}; \mathbb{R}^2)} E_{\delta}(\mathbf{y}) - \frac{\epsilon}{\log \frac{1}{\delta \vee \epsilon}}\Big( V_{\delta, \epsilon}(\mathbf{y}) - V_{\delta,\epsilon}^{\star}\Big) 
\end{align}
defined in \eqref{eq:nonlinear-force-problem}-\eqref{eq:meanFreeForcesNL}. 
We prove that
\begin{align}\label{eq:force-main-result-here}
\frac{\log \frac{1}{\delta \vee \epsilon}}{ \epsilon^2}  \mathcal{E}_{p,\delta, \epsilon}^{\text{force}} \rightarrow \begin{cases}
\displaystyle \min_{\mathbf{R} \in SO(2)} - \QFl^\ast( \mathbf{R}^T \mathbf{f}_0)  &  \text{ in case  \eqref{super-degenerate}} \\  
\displaystyle \min_{\mathbf{R} \in SO(2)} - \QFl^\ast( \mathbf{R}^T \mathbf{f}_0) + V_0(1- \mathbf{R} \mathbf{e}_1 \cdot \mathbf{R}(\varphi) \mathbf{e}_1) & \text{ in case \eqref{degenerate}}\\ 
\displaystyle -\QFl^\ast( \mathbf{R}(\varphi)^T \mathbf{f}_0 )  & \text{ in case  \eqref{non-degenerate}}
\end{cases}
\end{align}
as $\delta,\epsilon \to 0$, where $-\mathbf{f}_0$ is the limiting total force applied to the tip of the wedge. The angle $\varphi$ and the three cases above are defined using the asymptotics of the applied forces in \eqref{eq:VdeltaEpsilonStar}-\eqref{eq:getVarphiDef}. These cases arise because the wedge can rotate in the force problem, which is the main difference between it and the displacement problem. This difference requires a few extra preparatory steps, but our overall strategy for proving \eqref{eq:force-main-result-here} is the same as what we used in Section \ref{sec:displacement} for the displacement problem.

We continue to recall the assumptions from the introduction before they are used, as in the previous section. By the end of this section all of our assumptions will have been used, and we will have finished the proof of part (I) of Theorem \ref{thm:main-result}.

\subsection{Compactness in logarithmic variables}
Recall in the previous section we saw how to deduce compactness from a logarithmic bound on the elastic energy $E_\delta$. To use this in the force problem, we must bound $E_\delta$ by the total potential energy in \eqref{eq:totPotEnergy}.
First, we address the work integral. Observe that 
\begin{align}\label{eq:force_work_rewritten}
\frac{1}{\epsilon}\Big( V_{\delta, \epsilon}(\mathbf{y}) - V^{\star}_{\delta, \epsilon} \Big) =\frac{1}{\epsilon}  \int_{\partial \Omega_{\delta}} \mathbf{f}_{\delta, \epsilon} \cdot \big( \mathbf{y} - \mathbf{R} \mathbf{x}) \, d s  - \frac{1}{\epsilon} \max_{\mathbf{Q} \in SO(2)} \int_{\partial \Omega_{\delta}} \mathbf{f}_{\delta, \epsilon} \cdot  \big(   \mathbf{Q} - \mathbf{R} \big) \mathbf{x} \,ds
\end{align}
for any $\mathbf{y}$ and $\mathbf{R}$. 
We highlight this representation because it turns out that the correct displacement in the force problem is given by
$\mathbf{u}(\mathbf{x}) = \epsilon^{-1}(\mathbf{R}^{T}\mathbf{y}(\mathbf{x})-\mathbf{x})$ for  $\mathbf{R}\in\cR_{\d}(\mathbf{y})$, where 
\begin{align}\label{optimal_rotations1}
\mathcal{R}_{\delta } (\mathbf{y}) = \Big\{ \mathbf{R} \in SO(2) \colon \int_{\Omega_{\delta}} | \nabla \mathbf{y} - \mathbf{R}|^2 \, d\mathbf{x} = \min_{\mathbf{Q} \in SO(2)}\,  \int_{\Omega_{\delta}} | \nabla \mathbf{y} - \mathbf{Q}|^2 \, d\mathbf{x} \Big\} .
\end{align}
We estimate the first term on the right-hand side of \eqref{eq:force_work_rewritten} using that 
\begin{align}\label{eq:fDensity1}
    \mathbf{f}_{\delta,\epsilon}(\mathbf{x})=-\frac{1}{\delta}\mathbf{f}_{\delta,\epsilon}^{-}(\theta)\indicator{\{r=\delta\}}+\mathbf{f}_{\delta,\epsilon}^{+}(\theta)\indicator{\{r=1\}},\quad \mathbf{x}\in\partial\Omega_\delta
\end{align}
where $\{\mathbf{f}_{\delta,\epsilon}^{\pm}\}\subset W^{1/2, 2}((\alpha, \beta); \mathbb{R}^2)'$  are in balance,
\begin{align}\label{eq:fDensity2}
\int_{\alpha}^{\beta} \mathbf{f}_{\delta, \epsilon}^+ \, d\theta = \int_{\alpha}^{\beta} \mathbf{f}_{\delta, \epsilon}^- \, d\theta.
\end{align}
As a matter of organization, we define the quantity
\begin{align}\label{eq:MdeltaVeeEpsilon}
    M_{\delta\vee \epsilon}\big( \mathbf{f}^{+}_{\delta, \epsilon}, \mathbf{f}_{\delta,\epsilon}^-\big) = \sum_{(\cdot) = \pm}  \big\| \mathbf{f}^{(\cdot)}_{\delta, \epsilon} \big\|_{\dot{W}^{\frac{1}{2},2}((\alpha, \beta))'}  + \sqrt{ \log \frac{1}{\delta\vee \epsilon}} \Big| \int_{\alpha}^{\beta} \mathbf{f}_{\delta,\epsilon}^+ \,d\theta \Big|
\end{align}
since it is used many times below.

First, we establish some basic estimates on the work of a general displacement $\mathbf{u}$. Later on, we apply this result to $\mathbf{u} = \mathbf{y} - \mathbf{R}\mathbf{x}$ for $\mathbf{R} \in \mathcal{R}_{\delta}(\mathbf{y})$ to bound the work \eqref{eq:force_work_rewritten} by the elastic energy $E_{\delta}$.
\begin{lem}\label{workEstsElasticLem}
Let $\delta\in (0,1/2)$, $\epsilon\in(0,1)$, and $\mathbf{u}\in W^{1,p}(\Omega_\d;\R^2)$, and suppose $\mathbf{f}_{\delta,\epsilon}$ satisfies \eqref{eq:fDensity1} and \eqref{eq:fDensity2}. 
If $\epsilon \leq \delta$, then 
\begin{align*}
 \Big| \int_{\partial \Omega_{\delta}} \mathbf{f}_{\delta, \epsilon} \cdot \mathbf{u} \, d s \Big|   &\lesssim M_{\delta}\big(\mathbf{f}_{\delta, \epsilon}^+, \mathbf{f}_{\delta, \epsilon}^{-} \big) \| \nabla \mathbf{u}\|_{L^2(\Omega_{\delta})}  
\end{align*}
If instead $\delta < \epsilon$, then
\begin{align*}
\Big| \int_{\partial \Omega_{\delta}} \mathbf{f}_{\delta, \epsilon} \cdot \mathbf{u} \, d s \Big|  &\lesssim_p M_{\epsilon}\big(\mathbf{f}_{\delta, \epsilon}^+, \mathbf{f}_{\delta, \epsilon}^{-} \big) \|\nabla \mathbf{u}\|_{L^2(\Omega_{\delta})}   + \frac{1}{c_p}   \Big| \int_{\alpha}^{\beta} \mathbf{f}_{\delta,\epsilon}^+ \,d\theta \Big| \| \nabla \mathbf{u} \|_{L^p(\Omega_{\delta})}
\end{align*}
and 
\begin{equation*}
\begin{aligned}
&\Big|\int_{\alpha}^{\beta} \mathbf{f}_{\delta, \epsilon}^- \cdot \big( \mathbf{u}|_{r = \epsilon}  - \mathbf{u}|_{r = \delta}\big) \,d\theta \Big|   \lesssim_p  \big\| \mathbf{f}^{-}_{\delta, \epsilon}\big\|_{\dot{W}^{\frac{1}{2},2}((\alpha, \beta))'} \| \nabla \mathbf{u} \|_{L^2(\Omega_{\delta})}  \color{black} +\frac{1}{c_p}\Big|\int_{\alpha}^{\beta} \mathbf{f}_{\delta,\epsilon}^- \, d\theta \Big| \| \nabla \mathbf{u} \|_{L^p(\Omega_{\delta})} .
\end{aligned}
\end{equation*}
with $c_p=c_p(\delta,\epsilon)$ given by \eqref{eq:constant-cp1}.
\end{lem}
\begin{proof}
\textit{Step 1: The first two estimates.}
The force balance identity $\int_{\alpha}^{\beta} \mathbf{f}_{\d,\e}^+ \,d\theta = \int_{\alpha}^{\beta}  \mathbf{f}_{\d,\e}^- \,d\theta$ furnishes that 
\begin{align}\label{force_estimate_nonlinear_basic}
&\Big|\int_{\partial \Omega_{\delta}} \mathbf{f}_{\delta,\epsilon} \cdot \mathbf{u}\,ds\Big| = \Big|\int_{\alpha}^{\beta}  \mathbf{f}_{\delta,\epsilon}^+ \cdot \mathbf{u}|_{r =1} -  \mathbf{f}_{\delta,\epsilon}^- \cdot \mathbf{u}|_{r =\delta} \, d \theta \Big|  \nonumber\\
&\qquad \leq   \Big|\int_{\alpha}^{\beta} \big(  \mathbf{f}_{\delta,\epsilon}^+ - \fint_{\alpha}^{\beta} \mathbf{f}_{\delta,\epsilon}^+ \, d\tilde{\theta}\big)  \cdot \mathbf{u}|_{r =1} \,d \theta \Big|    +   \Big|\int_{\alpha}^{\beta} \big(  \mathbf{f}_{\delta,\epsilon}^- - \fint_{\alpha}^{\beta} \mathbf{f}_{\delta,\epsilon}^- \, d\tilde{\theta}\big)  \cdot \mathbf{u}|_{r =\delta} \,d \theta \Big| \nonumber\\
&\qquad \quad  +  \Big|\int_{\alpha}^{\beta}  \mathbf{f}_{\delta,\epsilon}^+ \, d\theta   \cdot \fint_{\alpha}^{\beta}  \cdot \mathbf{u}|_{r =1} -  \mathbf{u}|_{r =\delta \vee \epsilon}   \,d \theta \Big|   +  \Big|\int_{\alpha}^{\beta}  \mathbf{f}_{\delta,\epsilon}^+ \, d\theta   \cdot \fint_{\alpha}^{\beta}  \mathbf{u}|_{r =\delta \vee \epsilon} -  \mathbf{u}|_{r =\delta}   \,d \theta \Big|.
\end{align}
We bound each of the four  terms above to produce the desired result.  Let $\mathbf{c} = \fint_{\alpha}^{\beta}  \mathbf{u}|_{r=1} \,d  \theta$ and observe  that 
the first term on the right-hand side in \eqref{force_estimate_nonlinear_basic} satisfies
\begin{equation}
\begin{aligned}\label{eq:forceDispPair}
 \Big|\int_{\alpha}^{\beta} \big(  \mathbf{f}_{\delta,\epsilon}^{+} - \fint_{\alpha}^{\beta} \mathbf{f}_{\delta,\epsilon}^{+} \, d\tilde{\theta}\big)  \cdot \mathbf{u} |_{r=1} \,d \theta \Big|  &=   \Big|\int_{\alpha}^{\beta} \big(  \mathbf{f}_{\delta,\epsilon}^{+} - \fint_{\alpha}^{\beta} \mathbf{f}_{\delta,\epsilon}^{+} \, d\tilde{\theta}\big)  \cdot \big( \mathbf{u}|_{r=1} - \mathbf{c}\big) \,d \theta \Big|  \\
 & \leq  \big\| \mathbf{f}_{\delta,\epsilon}^{+} - \fint_{\alpha}^{\beta}   \mathbf{f}_{\delta,\epsilon}^{+} \, d\theta \big\|_{\dot{W}^{\frac{1}{2},2}((\alpha, \beta))'}  \| \mathbf{u}|_{r=1}  -\mathbf{c} \|_{\dot{W}^{\frac{1}{2},2}((\alpha, \beta))}\\
 &\lesssim  \big\| \mathbf{f}_{\delta,\epsilon}^{+}  \big\|_{\dot{W}^{\frac{1}{2},2}((\alpha, \beta))'}  \| \nabla \mathbf{u}\|_{L^2(\Omega_{\delta})}
\end{aligned}
\end{equation}
by Lemma \ref{traceToGradLemma} with  $p =2$.
The second term is handled analogously.
Meanwhile, the third term satisfies 
\begin{align*}
	\Big|\int_{\alpha}^{\beta}  \mathbf{f}_{\delta,\epsilon}^+ \, d\theta   \cdot \fint_{\alpha}^{\beta}  \mathbf{u}|_{r =1} -  \mathbf{u}|_{r =\delta \vee \epsilon}   \,d \theta \Big| \leq  \Big|\int_{\alpha}^{\beta}  \mathbf{f}_{\delta,\epsilon}^+ \, d\theta \Big|    \sqrt{ \log \frac{1}{\delta \vee \epsilon}}  \|\nabla \mathbf{u}\|_{L^2(\Omega_{\delta})} 
\end{align*}
by Lemma \ref{logScaleLemma}.
The fourth term vanishes if $\epsilon \leq \delta$. Otherwise, $\delta < \epsilon$ and we use Lemma \ref{logScaleLemma} to show that 
\begin{align*}
	\Big|\int_{\alpha}^{\beta}  \mathbf{f}_{\delta,\epsilon}^+ \, d\theta   \cdot \fint_{\alpha}^{\beta}  \mathbf{u}|_{r =\epsilon} -  \mathbf{u}|_{r =\delta}   \,d \theta \Big| \leq  \Big|\int_{\alpha}^{\beta}  \mathbf{f}_{\delta,\epsilon}^+ \, d\theta \Big|     \frac{1}{c_p} \| \nabla \mathbf{u}\|_{L^p(\Omega_{\delta})} .
\end{align*}

\textit{Step 2: The last estimate.}
Observe that 
\begin{align*}
\Big| \int_{\alpha}^{\beta} \mathbf{f}_{\delta,\epsilon}^- \cdot \mathbf{u}|_{r= \epsilon} - \mathbf{u}|_{r=\delta}) \, d \theta \Big| & \leq \Big|  \int_{\alpha}^{\beta} \Big( \mathbf{f}_{\delta,\epsilon}^- - \fint_{\alpha}^{\beta} \mathbf{f}_{\delta,\epsilon}^- \,d\tilde{\theta} \Big)  \cdot \mathbf{u}|_{r= \epsilon}\,d\theta \Big| +  \Big|  \int_{\alpha}^{\beta} \Big(  \mathbf{f}_{\delta,\epsilon}^- - \fint_{\alpha}^{\beta} \mathbf{f}_{\delta,\epsilon}^- \,d\tilde{\theta} \Big)  \cdot \mathbf{u}|_{r= \delta} \, d\theta \Big|  \\
&\quad + \Big|\int_{\alpha}^{\beta} \mathbf{f}_{\delta,\epsilon}^-  \, d\theta \cdot  \fint_{\alpha}^{\beta} \big( \mathbf{u}|_{r = \epsilon} -  \mathbf{u}|_{r = \delta}\big)\,d \theta \Big|.
\end{align*} 
These terms are then estimated exactly as before. 
\end{proof}

\noindent Next, we bound the elastic energy by the total potential energy up to an error. We also bound the (non-negative) second term from the right-hand side of \eqref{eq:force_work_rewritten}. 
In addition to the previous assumptions on the forces, we also use the ones from \eqref{eq:forceBalanceNL} and \eqref{eq:meanFreeForcesNL}, i.e., 
\begin{align}\label{eq:restateForces}
\int_{\alpha}^{\beta} \mathbf{f}_{\delta, \epsilon}^+ \, d\theta = \int_{\alpha}^{\beta} \mathbf{f}_{\delta, \epsilon}^- \, d\theta \rightarrow \mathbf{f}_0 \quad \text{ and } \quad \frac{1}{\sqrt{\log \frac{1}{\delta \vee \epsilon}}} \big\| \mathbf{f}_{\delta, \epsilon}^{\pm}  \big\|_{\dot{W}^{\frac{1}{2},2}((\alpha,\beta))'} \rightarrow 0
\end{align}
as $\delta, \epsilon \rightarrow 0$.
\begin{lem}\label{lem:EnergyForce}Let $\delta\in (0,1/2)$ and $\epsilon\in(0,1)$. 
Given $\mathbf{y}\in W^{1,p}(\Omega_\d;\R^2)$,  $\mathbf{R} \in \cR_{\delta}(\mathbf{y})$, and $\{\mathbf{f}_{\delta,\epsilon}\}$ as above, 
\begin{equation}
\begin{aligned}\label{eq:getEnergyBound}
&E_{\delta}(\mathbf{y})  + \max_{\mathbf{Q} \in SO(2)} \frac{\epsilon}{\log\frac{1}{\delta \vee \epsilon}}  \int_{\partial \Omega_{\delta}} \mathbf{f}_{\delta, \epsilon} \cdot (\mathbf{Q} - \mathbf{R}) \mathbf{x} \, ds  \\
&\qquad  \lesssim_p \begin{cases}   E_{\delta}(\mathbf{y})  -   \frac{\epsilon}{\log \frac{1}{\delta}}\big( V_{\delta, \epsilon}(\mathbf{y}) - V^{\star}_{\delta, \epsilon} \big) +  \frac{\epsilon^2}{\log \frac{1}{\delta}} & \text{ if } \epsilon \leq \delta, p \geq 2 \\ 
E_{\delta}(\mathbf{y})  -   \frac{\epsilon}{\log \frac{1}{\epsilon}}\big( V_{\delta, \epsilon}(\mathbf{y}) - V^{\star}_{\delta, \epsilon} \big) +   \frac{\epsilon^2}{\log \frac{1}{\epsilon}}   +   \frac{\epsilon^2 \log \frac{\epsilon}{\delta}}{(\log \frac{1}{\epsilon})^2}  & \text{ if } \delta < \epsilon, p =2 \\
 E_{\delta}(\mathbf{y})  -   \frac{\epsilon}{\log \frac{1}{\epsilon}}\big( V_{\delta, \epsilon}(\mathbf{y}) - V^{\star}_{\delta, \epsilon} \big) +  \frac{\epsilon^2}{\log \frac{1}{\epsilon}}   & \text{ if }  \delta < \epsilon, p >2
\end{cases}.
\end{aligned}
\end{equation}

\end{lem}
\begin{proof}
Rewrite the total potential energy using   \eqref{eq:force_work_rewritten}  as
\begin{align*}
&E_{\delta}(\mathbf{y})  -   \frac{\epsilon}{\log \frac{1}{\delta \vee \epsilon}}\Big( V_{\delta, \epsilon}(\mathbf{y}) - V^{\star}_{\delta, \epsilon} \Big) \\
&\qquad  =  E_{\delta}(\mathbf{y}) + \max_{\mathbf{Q} \in SO(2)}  \frac{\epsilon}{\log \frac{1}{\delta \vee \epsilon}}  \int_{\partial \Omega_{\delta}}  \mathbf{f}_{\delta,\epsilon} \cdot (\mathbf{Q}  - \mathbf{R}) \mathbf{x} \, d s   - \frac{\epsilon}{\log \frac{1}{\delta \vee \epsilon}} \int_{\partial \Omega_{\delta}} \mathbf{f}_{\delta, \epsilon} \cdot (\mathbf{y} - \mathbf{R} \mathbf{x}) \, ds,
\end{align*}
and note it suffices to control the last term. We will arrive at the desired bounds by applying the inequalities  from Lemma \ref{workEstsElasticLem} to $\mathbf{u} = \mathbf{y} - \mathbf{R} \mathbf{x}$ for   $\mathbf{R} \in \cR_{\delta}(\mathbf{y})$. This application of the bounds additionally satisfies 
\begin{align}\label{eq:rigidityForce}
   \| \nabla \mathbf{u} \|_{L^2(\Omega_{\delta})} =  \| \nabla \mathbf{y} - \mathbf{R} \|_{L^2(\Omega_{\delta})} \lesssim \big(E_{\delta}(\mathbf{y})\big)^{\frac{1}{2}} \quad \text{ and }  \quad \| \nabla \mathbf{u} \|_{L^p(\Omega_{\delta})} = \| \nabla \mathbf{y} - \mathbf{R} \|_{L^p(\Omega_{\delta})} \lesssim_{p} \big(E_{\delta}(\mathbf{y})\big)^{\frac{1}{p}} \end{align}
by the geometric rigidity result in Corollary \ref{cor:mixed-rigidity-energy}.

First, let $\epsilon \leq \delta$. By the first inequality in  Lemma \ref{workEstsElasticLem}, \eqref{eq:rigidityForce}, and  Young's inequality,  there exists $C > 0$ such that  
\begin{align*}
    - \frac{\epsilon}{\log \frac{1}{\delta}} \int_{\partial \Omega_{\delta}} \mathbf{f}_{\delta, \epsilon} \cdot (\mathbf{y} - \mathbf{R} \mathbf{x}) \, ds &\geq - C \frac{\epsilon}{\log\frac{1}{\delta}}M_{\delta}\big(\mathbf{f}_{\delta, \epsilon}^{+}, \mathbf{f}_{\delta, \epsilon}^{-}\big) \big( E_{\delta}(\mathbf{y}) \big)^{\frac12} \\& \geq - \frac{\lambda}{2} E_{\delta}(\mathbf{y}) - \frac{C^2}{2 \lambda}    \frac{\epsilon^2}{(\log\frac{1}{\delta})^2} \Big(M_{\delta}\big( \mathbf{f}_{\delta, \epsilon}^+, \mathbf{f}_{\delta, \epsilon}^{-} \big)\Big)^2
\end{align*}
for any $\lambda  >0$. The first estimate  in \eqref{eq:getEnergyBound} follows by choosing $\lambda$ sufficiently small, since \eqref{eq:restateForces} implies that $M_{\delta}\big(\mathbf{f}_{\delta, \epsilon}^+, \mathbf{f}_{\delta, \epsilon}^{-} \big) \lesssim (\log \frac{1}{\delta})^{\frac{1}{2}}$ by its definition in \eqref{eq:MdeltaVeeEpsilon}.
Also,  $\max_{\mathbf{Q} \in SO(2)} \frac{\epsilon}{\log\frac{1}{\delta}}  \int_{\partial \Omega_{\delta}} \mathbf{f}_{\delta, \epsilon} \cdot (\mathbf{Q} - \mathbf{R}) \mathbf{x} \, ds\geq 0$.

Next, let $\delta < \epsilon$. Using the second inequality in Lemma \ref{workEstsElasticLem} and both  estimates in \eqref{eq:rigidityForce},  the same strategy as above produces constants  $C, \tilde{C} >0$ depending on  $p$ such that  
\begin{align*}
&E_{\delta}(\mathbf{y})  -   \frac{\epsilon}{\log \frac{1}{ \epsilon}}\Big( V_{\delta, \epsilon}(\mathbf{y}) - V^{\star}_{\delta, \epsilon} \Big)   \geq  \frac{1}{2} E_{\delta}(\mathbf{y}) -  C  \frac{\epsilon^2}{\log \frac{1}{\epsilon}} - \tilde{C} \frac{\epsilon}{\log(\frac{1}{\epsilon}) c_p}   \Big| \int_{\alpha}^{\beta} \mathbf{f}_{\delta}^+ \,d\theta \Big| \big( E_{\delta}(\mathbf{y}) \big)^{\frac1p}.
\end{align*}
Applying Young's inequality $ab \leq \frac{1}{p} a^p + \frac{p-1}{p} b^{\frac{p}{p-1}}$  for $a,b\geq 0$ to the last term furnishes 
\begin{align*}
 \frac{\epsilon}{\log(\frac{1}{\epsilon}) c_p}   \Big| \int_{\alpha}^{\beta} \mathbf{f}_{\delta}^+ \,d\theta \Big| \big( E_{\delta}(\mathbf{y}) \big)^{\frac1p} \leq \frac{\lambda^{p}}{p} E_{\delta}(\mathbf{y})  + \frac{p-1}{p\lambda^{\frac{p}{p-1}}}\Big( \frac{1}{\log \frac{1}{\epsilon}}  \Big| \int_{\alpha}^{\beta} \mathbf{f}_{\delta}^+ \,d\theta \Big|   \Big)^{\frac{p}{p-1}}  \Big(\frac{\epsilon}{c_p} \Big)^{\frac{p}{p-1}}
\end{align*}
for all $\lambda > 0$. From the definition of $c_p=c_p(\delta, \epsilon)$ in \eqref{eq:constant-cp1}, 
\begin{align*}
\Big(\frac{\epsilon}{c_p} \Big)^{\frac{p}{p-1}} \sim_p \begin{cases}
\epsilon^2 \log \frac{\epsilon}{\delta} & \text{ if } p =2 \\
\epsilon^2(1- (\frac{\delta}{\epsilon})^{\frac{p-2}{p-1}})  \leq \epsilon^2 & \text{ if } p > 2.
\end{cases}
\end{align*}
From the first convergence in \eqref{eq:restateForces}, we also have that 
\begin{align*}
\Big( \frac{1}{\log \frac{1}{\epsilon}}  \Big| \int_{\alpha}^{\beta} \mathbf{f}_{\delta,\epsilon}^+ \,d\theta \Big|   \Big)^{\frac{p}{p-1}} \lesssim \begin{cases}
\frac{1}{(\log\frac{1}{\epsilon})^2} & p = 2 \\ 
\frac{1}{(\log\frac{1}{\epsilon})^{\frac{p}{p-1}}} \leq \frac{1}{\log\frac{1}{\epsilon}} & p > 2.
\end{cases}
\end{align*}
The  last two estimates  in \eqref{eq:getEnergyBound} follow by combining all the inequalities and choosing a small enough $\lambda$.
\end{proof}

We come now to compactness. Again, we require a logarithmic change of variables.
Given $\mathbf{y}:\Omega_\delta\to\mathbb{R}^2$ and $\mathbf{R}\in SO(2)$, we define 
$\mathbf{U}:R\to\mathbb{R}^2$ for $R=(-1,0)\times (\alpha,\beta)$ by
\begin{equation}
\mathbf{U}(\rho,\theta)=\frac{1}{\epsilon}\left(\mathbf{R}^{T}\mathbf{y}(\mathbf{x})-\mathbf{x}\right)\quad\text{where}\quad\rho=\frac{\log r}{\log\frac{1}{\delta \vee \epsilon}}. \label{eq:NL-blow-up-definition}
\end{equation}
For completeness, we recall once again that the energy density $W$ is assumed to satisfy 
\begin{equation*}
W(\mathbf{F}) \gtrsim d^2(\mathbf{F}, SO(2)) + d^p(\mathbf{F}, SO(2))
\end{equation*}
for some $p\geq 2$ by \eqref{eq:Wgrowth}, 
and that if $p =2$, the limit $\delta, \epsilon \rightarrow 0$ is taken to satisfy (see  \eqref{eq:pEqual2Limit})
\begin{align}\label{eq:pEqual2Limit2}
\log \frac{\delta \vee \epsilon}{\delta}\ll \log \frac{1}{\epsilon}. 
\end{align}
We distinguish three cases according to the asymptotics of  
$V_{\delta, \epsilon}^{\star} = \max_{\mathbf{Q} \in SO(2)} \int_{\Omega_{\delta}} \mathbf{f}_{\delta, \epsilon} \cdot \mathbf{Q} \mathbf{x} \,ds$, namely
\begin{subequations}
\begin{numcases}{\frac{V_{\d,\e}^{\star}}{\e}\to}
	0 \label{super-degenerate1}\\
	V_0 >0 \label{degenerate1}\\
	\infty \label{non-degenerate1}
\end{numcases}
\end{subequations}
as given in \eqref{super-degenerate}-\eqref{non-degenerate}. In cases \eqref{degenerate} and \eqref{non-degenerate},  $V^{\star}_{\d,\e}>0$ for small enough $\delta,\epsilon$, and we define $\varphi_{\delta,\epsilon} \in (-\pi, \pi]$ such that
\begin{equation}\label{eq:phi_deltaeps_defn1}
\cos\varphi_{\delta,\epsilon} =  \frac{1}{V_{\delta ,\epsilon}^{\star}}  \int_{\partial \Omega_{\delta}} \mathbf{f}_{\delta,\epsilon} \cdot \mathbf{x} \, ds, \quad     \sin \varphi_{\delta,\epsilon} = \frac{1}{V_{\delta ,\epsilon}^{\star}}  \int_{\partial \Omega_{\delta}} \mathbf{f}_{\delta, \epsilon} \cdot \mathbf{x}^{\perp} \, ds
\end{equation}
using \eqref{eq:VdeltaEpsilonStar}. Finally, we assume in cases \eqref{degenerate} and \eqref{non-degenerate} that
\begin{equation}\label{eq:getVarphiDef1}
	 \cos \varphi_{\delta, \epsilon} \rightarrow \cos \varphi, \quad \sin \varphi_{\delta, \epsilon} \rightarrow \sin \varphi.
\end{equation}
for $ \varphi \in (-\pi, \pi]$. 
These last assumptions enable a succinct characterization of the asymptotic behavior of the work term $\frac{1}{\epsilon} (V_{\delta,\epsilon}(\mathbf{y})  - V_{\delta, \epsilon}^{\star})$. Here is the compactness result: 
\begin{prop}\label{propCompactnessForce} Consider any limit $\delta, \epsilon \rightarrow 0$ in the super-quadratic growth case ($p >2$) or assume such a limit obeys \eqref{eq:pEqual2Limit2} in the quadratic growth case ($p=2$).
Let $\mathbf{y}_{\delta, \epsilon} \in W^{1,p}(\Omega_{\delta};\mathbb{R}^2)$ satisfy
\begin{align}\label{eq:boundedEnergyForce}
\limsup_{\delta,\epsilon \rightarrow 0} \frac{\log \frac{1}{\delta \vee \epsilon}}{\epsilon^2} \left(E_{\delta}(\mathbf{y}_{\delta,\epsilon})  -   \frac{\epsilon}{\log \frac{1}{\delta \vee \epsilon}}\Big( V_{\delta, \epsilon}(\mathbf{y}_{\delta,\epsilon}) - V^{\star}_{\delta, \epsilon} \Big) \right)  < \infty
\end{align}
and define $\{ \mathbf{R}_{\delta, \epsilon} \} \subset SO(2)$ and $\{ \mathbf{U}_{\delta,\epsilon}\} \subset W^{1,p}(R; \mathbb{R}^2)$ by
\begin{align*}
\mathbf{R}_{\delta, \epsilon} \in \cR_{\delta}(\mathbf{y}_{\delta, \epsilon} )  \quad \text{ and } \quad  \mathbf{U}_{\delta, \epsilon}(\rho ,\theta) = \frac{1}{\epsilon}\big( \mathbf{R}_{\delta, \epsilon}^T \mathbf{y}_{\delta, \epsilon} (\mathbf{x}) - \mathbf{x} \big)
\end{align*}
as in \eqref{eq:NL-blow-up-definition}.
There exists subsequence (not relabeled) such that
\begin{align}\label{eq:UDeltaEpsilon}
\mathbf{U}_{\delta,\epsilon} - \fint_{R} \mathbf{U}_{\delta, \epsilon} \, d\mathbf{x} \rightharpoonup\mathbf{U},\quad\partial_{\rho}\mathbf{U}_{\delta,\epsilon}\rightharpoonup\partial_{\rho}\mathbf{U},\quad\log\left(\frac{1}{\delta \vee \epsilon}\right)\partial_{\theta}\mathbf{U}_{\delta,\epsilon}\rightharpoonup\mathbf{m}\quad\text{weakly in }L^{2}(R;\R^2)
\end{align}
for some $\mathbf{m}\in L^{2}(R;\mathbb{R}^{2})$ and  $\mathbf{U} \in W^{1,2}(R; \mathbb{R}^2)$ with $\partial_{\theta} \mathbf{U} = \mathbf{0}$. The rotations satisfy
\begin{align}\label{eq:rotCharacterization}
\mathbf{R}_{\delta,\epsilon} \rightarrow \mathbf{R}_0 \quad \text{ for } \quad  \mathbf{R}_0 = \begin{cases}
\text{some rotation in $SO(2)$}  & \text{ if  \eqref{super-degenerate1} or \eqref{degenerate1}} \\
\mathbf{R}(\varphi)    & \text{ if \eqref{non-degenerate1}}.
\end{cases}
\end{align} 
Finally, the work integral satisfies 
\begin{equation}
\begin{aligned}\label{eq:workDueToForcesLimit}
&-\frac{1}{\epsilon}\Big( V_{\delta, \epsilon}(\mathbf{y}_{\delta,\epsilon}) - V^{\star}_{\delta, \epsilon} \Big) \rightarrow   \mathbf{f}_0 \cdot \mathbf{R}_0 \mathbf{U}|_{\rho = -1} -  \mathbf{f}_0 \cdot \mathbf{R}_0 \mathbf{U}|_{\rho = 0}   + e_0, \quad \text{ where }  \\
 &\qquad e_0 = \lim_{\delta,\epsilon\to0}\,\max_{\mathbf{Q} \in SO(2)}\, \frac{1}{\epsilon} \int_{\partial \Omega_{\delta}} \mathbf{f}_{\delta, \epsilon} \cdot ( \mathbf{Q} \mathbf{x} - \mathbf{R}_{\delta, \epsilon} \mathbf{x}) \, d s  = \begin{cases}
0 & \text{ if \eqref{super-degenerate1}} \\
V_0 ( 1- \mathbf{R}_0 \mathbf{e}_1 \cdot  \mathbf{R} (\varphi) \mathbf{e}_1 )  &  \text{ if \eqref{degenerate1}} \\
\text{some number $\geq 0$.} &  \text{ if \eqref{non-degenerate1}.}
\end{cases}
\end{aligned}
\end{equation}
\end{prop}
\begin{rem}\label{rem:unitRotation}
    The rotation $\mathbf{R}_{\delta,\epsilon} \in \mathcal{R}_{\delta}(\mathbf{y}_{\delta, \epsilon})$  is unique if $\delta, \epsilon$ are sufficiently small. Indeed, we show in the proof that $E_{\delta}(\mathbf{y}_{\delta, \epsilon}) \lesssim_p \epsilon^2(\log \frac{1}{\delta \vee \epsilon})^{-1} \rightarrow 0$, so  the uniqueness follows from Lemma \ref{uniquenessCor}.
\end{rem}
\begin{proof}
\textit{Step 1: Basic convergence properties.}
Suppose $p>2$.  Using  Lemma \ref{lem:EnergyForce} and \eqref{eq:boundedEnergyForce}, we have 
\begin{align}\label{eq:firstEstCompactnessForce}
\frac{\log \frac{1}{\delta \vee \epsilon}}{\epsilon^2}  E_{\delta}(\mathbf{y}_{\delta,\epsilon}) \lesssim_p \frac{\log \frac{1}{\delta \vee \epsilon}}{\epsilon^2} \left( E_{\delta}(\mathbf{y}_{\delta,\epsilon})  -   \frac{\epsilon}{\log \frac{1}{\delta \vee \epsilon}}\Big( V_{\delta, \epsilon}(\mathbf{y}_{\delta,\epsilon}) - V^{\star}_{\delta, \epsilon} \Big)  +    \frac{\epsilon^2}{\log \frac{1}{\delta \vee \epsilon}} \right)  \lesssim 1.
\end{align}
Next, by Corollary \ref{cor:mixed-rigidity-energy},  $\mathbf{u}_{\delta,\epsilon} = \epsilon^{-1} ( \mathbf{R}_{\delta,\epsilon}^T \mathbf{y}_{\delta, \epsilon} - \mathbf{x})$ satisfies  $\int_{\Omega_{\delta}} |\nabla \mathbf{u}_{\delta, \epsilon}|^2 \,d\mathbf{x}  \lesssim  \frac{1}{\epsilon^2} E_{\delta}(\mathbf{y}_{\delta,\epsilon})$
since $\mathbf{R}_{\delta, \epsilon} \in \mathcal{R}_{\delta}(\mathbf{y}_{\delta,\epsilon})$. Combining this observation with \eqref{eq:firstEstCompactnessForce} and the definition  $\mathbf{U}_{\delta, \epsilon}(\rho,\theta) = \mathbf{u}_{\delta,\epsilon}(\mathbf{x})$, there follows (see  \eqref{eq:strainTologStrain})
\begin{align}\label{eq:boundedaL2GradU}
\int_{R} |D^{\delta \vee \epsilon} \mathbf{U}_{\delta,\epsilon}|^2 \,d \rho d\theta = \log \frac{1}{\delta \vee \epsilon}  \int_{\Omega_{\delta \vee \epsilon}} |\nabla \mathbf{u}_{\delta, \epsilon}|^2 \, d \mathbf{x}  \leq \log \frac{1}{\delta \vee \epsilon}  \int_{\Omega_{\delta}} |\nabla \mathbf{u}_{\delta, \epsilon}|^2 \, d \mathbf{x}  \lesssim_p 1.
\end{align}
Furthermore, by Poincar\'e's inequality,
$$\Big\|\mathbf{U}_{\delta, \epsilon} - \fint_R \mathbf{U}_{\delta, \epsilon} \, d \rho d \theta  \Big\|_{L^2(R)}  \lesssim  \| (\partial_\rho \mathbf{U}_{\delta, \epsilon}, \partial_{\theta} \mathbf{U}_{\delta, \epsilon}) \|_{L^2(R)}  \lesssim \| D^{\delta \vee \epsilon} \mathbf{U}_{\delta,\epsilon} \|_{L^2(R)}.$$ 
This means that a subsequence of $\{ \mathbf{U}_{\delta, \epsilon}\}$ satisfies \eqref{eq:UDeltaEpsilon} for some $\mathbf{m}\in L^{2}(R;\mathbb{R}^{2})$ and  $\mathbf{U} \in W^{1,2}(R; \mathbb{R}^2)$. Clearly, $\mathbf{U}$ is independent of $\theta$, just as in the linear analysis. 
By the trace theorem, we also have for the subsequence that  $\mathbf{U}_{\delta, \epsilon}|_{\rho = 0} \rightharpoonup \mathbf{U}|_{\rho = 0}$ and $\mathbf{U}_{\delta, \epsilon}|_{\rho = -1} \rightharpoonup \mathbf{U}|_{\rho = -1}$ in $L^2((\alpha, \beta);\R^2)$.  Finally, a subsequence of $\{ \mathbf{R}_{\delta, \epsilon} \}$  satisfies $\mathbf{R}_{\delta,\epsilon} \rightarrow \mathbf{R}_0$ for some $\mathbf{R}_0 \in SO(2)$ since  $SO(2)$ is compact.

The case $p =2$ is much the same, the only difference being that the estimate in  \eqref{eq:firstEstCompactnessForce} is replaced by 
\begin{equation}\label{eq:firstEstCompactnessForce_p=2}
\frac{\log \frac{1}{\delta \vee \epsilon}}{\epsilon^2}  E_{\delta}(\mathbf{y}_{\delta,\epsilon}) \lesssim 1 + \frac{\log \frac{\delta \vee \epsilon}{\delta}}{\log \frac{1}{\epsilon}}\lesssim 1.
\end{equation}
Note we used \eqref{eq:pEqual2Limit2} to handle the extra term. The rest of the proof of \eqref{eq:UDeltaEpsilon} is as above.

\textit{Step 2: Convergence of the force contributions.} Next, we establish  that there is a subsequence such that
\begin{align}
&\frac{1}{\epsilon}  \int_{\partial \Omega_{\delta}} \mathbf{f}_{\delta, \epsilon} \cdot \big( \mathbf{y}_{\delta, \epsilon} - \mathbf{R}_{\delta, \epsilon} \mathbf{x}) \, d s  \rightarrow \mathbf{f}_0 \cdot \mathbf{R}_0 \mathbf{U}|_{\rho = 0} -  \mathbf{f}_0 \cdot \mathbf{R}_0 \mathbf{U}|_{\rho = -1}, \label{eq:firstForceConvergence} \\
&\max_{\mathbf{Q} \in SO(2)} \frac{1}{\epsilon}  \int_{\partial \Omega_{\delta}} \mathbf{f}_{\delta, \epsilon} \cdot \big( \mathbf{Q} \mathbf{x} - \mathbf{R}_{\delta, \epsilon} \mathbf{x} \big) \, d s \rightarrow e_0 \label{eq:secForceConvergence} 
\end{align} 
for some $e_0 \geq 0$. Once proven, the convergence $$-\frac{1}{\epsilon}\Big( V_{\delta, \epsilon}(\mathbf{y}_{\delta,\epsilon}) - V^{\star}_{\delta, \epsilon} \Big) \rightarrow   \mathbf{f}_0 \cdot \mathbf{R}_0 \mathbf{U}|_{\rho = -1} -  \mathbf{f}_0 \cdot \mathbf{R}_0 \mathbf{U}|_{\rho = 0}   + e_0$$  follows due to the decomposition in \eqref{eq:force_work_rewritten}.

We begin with \eqref{eq:secForceConvergence}, which is straightforward. 
Indeed, $\max_{\mathbf{Q} \in SO(2)} \frac{1}{\epsilon}  \int_{\partial \Omega_{\delta}} \mathbf{f}_{\delta, \epsilon} \cdot \big( \mathbf{Q} \mathbf{x} - \mathbf{R}_{\delta, \epsilon} \mathbf{x} \big) \, d s \lesssim_p 1$   in light of Lemma \ref{lem:EnergyForce} and \eqref{eq:boundedEnergyForce} (and \eqref{eq:pEqual2Limit2} if $p=2$). So we can extract a convergent subsequence. As the sequence is non-negative, the limit $e_0$ is $\geq 0$. 

Next we prove \eqref{eq:firstForceConvergence}.  Begin by switching to polar coordinates  and decompose the left-hand side into a component associated to the bulk of the wedge and one near the tip to obtain
\begin{equation}
\begin{aligned}\label{eq:bulkAndWedgeTip}
\frac{1}{\epsilon}  \int_{\partial \Omega_{\delta}} \mathbf{f}_{\delta, \epsilon} \cdot \big( \mathbf{y}_{\delta, \epsilon} - \mathbf{R}_{\delta, \epsilon} \mathbf{x}) \, d s &=  \frac{1}{\epsilon}  \int_{\alpha}^{\beta} \mathbf{f}_{\delta, \epsilon}^+ \cdot \big( \mathbf{y}_{\delta, \epsilon} - \mathbf{R}_{\delta, \epsilon} \mathbf{x}\big)|_{r = 1}  -  \mathbf{f}_{\delta, \epsilon}^- \cdot \big( \mathbf{y}_{\delta, \epsilon} - \mathbf{R}_{\delta, \epsilon} \mathbf{x}\big)|_{r = \delta \vee \epsilon}  \, d\theta  \\
&\qquad +  \frac{1}{\epsilon} \int_{\alpha}^{\beta} \mathbf{f}_{\delta, \epsilon}^- \cdot \big( \mathbf{y}_{\delta, \epsilon} - \mathbf{R}_{\delta, \epsilon} \mathbf{x}\big)|_{r = \delta \vee \epsilon}  -  \mathbf{f}_{\delta, \epsilon}^- \cdot \big( \mathbf{y}_{\delta, \epsilon} - \mathbf{R}_{\delta, \epsilon} \mathbf{x}\big)|_{r = \delta}  \, d\theta.
\end{aligned}
\end{equation}
The second term, associated to the tip, is zero if $\epsilon\leq \d$. Otherwise, we apply Lemma \ref{workEstsElasticLem} to deduce that
\begin{align*}
&\frac{1}{\epsilon} \Big|\int_{\alpha}^{\beta} \mathbf{f}_{\delta, \epsilon}^- \cdot \big( (\mathbf{y}_{\delta, \epsilon}  - \mathbf{R}_{\delta, \epsilon} \mathbf{x})|_{r = \epsilon}  -  (\mathbf{y}_{\delta, \epsilon} -  \mathbf{R}_{\delta, \epsilon} \mathbf{x})|_{r = \delta} \big) \,d\theta \Big|\\
&\qquad \lesssim_p  \frac{1}{\sqrt{\log\frac{1}{\e}}}  \big\| \mathbf{f}^{-}_{\delta, \epsilon} \big\|_{\dot{W}^{\frac{1}{2},2}((\alpha, \beta))'} + \frac{\epsilon^{\frac2p-1}}{c_p (\log \frac{1}{\epsilon} )^{\frac1p}} \Big| \int_{\alpha}^{\beta} \mathbf{f}_{\delta,\epsilon}^- \, d\theta \Big| \rightarrow 0
\end{align*}
by \eqref{eq:restateForces}  and \eqref{eq:firstEstCompactnessForce}, since $c_p^{-1} (\log \frac{1}{\epsilon} )^{-\frac1p} \epsilon^{\frac2p-1}\rightarrow 0$ if $p >2$ and also if $p=2$ by the assumption \eqref{eq:pEqual2Limit2}. To address the first term in \eqref{eq:bulkAndWedgeTip}, associated to the bulk, we use the displacement $\mathbf{u}_{\delta,\epsilon} = \epsilon^{-1} ( \mathbf{R}_{\delta,\epsilon}^T \mathbf{y}_{\delta, \epsilon} - \mathbf{x})$  to break the integral into two parts by writing
\begin{equation}
\begin{aligned}\label{eq:twoPartsWork}
& \frac{1}{\epsilon}  \int_{\alpha}^{\beta} \mathbf{f}_{\delta, \epsilon}^{\pm} \cdot \big( \mathbf{y}_{\delta, \epsilon} - \mathbf{R}_{\delta, \epsilon} \mathbf{x}\big)|_{r = 1, \delta \vee \epsilon} \, d\theta \\
 &\qquad  =   \int_{\alpha}^{\beta} \mathbf{f}_{\delta,\epsilon}^{\pm}  \, d\theta \cdot  \fint_{\alpha}^{\beta}  \mathbf{R}_{\delta, \epsilon} \mathbf{u}_{\delta, \epsilon}|_{r = 1,\delta \vee \epsilon} \, d \theta  \\
 &\qquad \quad  +  \int_{\alpha}^{\beta}  \Big( \mathbf{f}_{\delta,\epsilon}^{\pm}  - \fint_{\alpha}^{\beta} \mathbf{f}_{\delta,\epsilon}^{\pm} \, d\tilde{\theta} \Big) \cdot \Big(\mathbf{R}_{\delta \epsilon}  \mathbf{u}_{\delta, \epsilon} |_{r = 1, \delta \vee \epsilon}  - \fint_{\alpha}^{\beta}  \mathbf{R}_{\delta \epsilon}  \mathbf{u}_{\delta, \epsilon} |_{r = 1, \delta \vee \epsilon} \, d \tilde{\theta} \Big)  \, d\theta.
\end{aligned}
\end{equation}
We treat the resulting integrals separately. As $\mathbf{R}_{\delta, \epsilon} \rightarrow \mathbf{R}_0$  and $\mathbf{U}_{\delta,\epsilon}|_{\rho = 0, -1} \rightharpoonup \mathbf{U}|_{\rho = 0,-1}$ in $L^2((\alpha,\beta);\mathbb{R}^2)$, 
\begin{align*}
& \int_{\alpha}^{\beta} \mathbf{f}_{\delta,\epsilon}^{\pm}  \, d\theta \cdot  \fint_{\alpha}^{\beta}  \mathbf{R}_{\delta, \epsilon} \mathbf{u}_{\delta, \epsilon}|_{r = 1,\delta \vee \epsilon} \, d \theta  \\
 &\qquad   =  \mathbf{R}^T_{\delta, \epsilon} \int_{\alpha}^{\beta} \mathbf{f}_{\delta,\epsilon}^{\pm} \, d\theta \cdot \fint_{\alpha}^{\beta} \mathbf{U}_{\delta, \epsilon}|_{\rho = 0} \, d\theta \rightarrow   \mathbf{R}_0^T \mathbf{f}_0 \cdot \fint_{\alpha}^{\beta}  \mathbf{U}|_{\rho = 0,-1} \, d\theta  =  \mathbf{R}_0^T \mathbf{f}_0 \cdot  \mathbf{U}|_{\rho = 0,-1},
\end{align*}
after making use of the facts that $\mathbf{U}_{\delta, \epsilon} |_{\rho =0,-1} = \mathbf{u}_{\delta, \epsilon}|_{r = 1, \delta \vee \epsilon}$  and $\mathbf{U}|_{\rho = 0,-1}$ is independent of $\theta$. The remaining integral from \eqref{eq:twoPartsWork} satisfies 
\begin{align*}
&\Big| \int_{\alpha}^{\beta}  \Big( \mathbf{f}_{\delta,\epsilon}^{\pm}  - \fint_{\alpha}^{\beta} \mathbf{f}_{\delta,\epsilon}^{\pm} \, d\tilde{\theta} \Big) \cdot \Big(\mathbf{R}_{\delta \epsilon}  \mathbf{u}_{\delta, \epsilon} |_{r = 1, \delta \vee \epsilon}  - \fint_{\alpha}^{\beta}  \mathbf{R}_{\delta \epsilon}  \mathbf{u}_{\delta, \epsilon} |_{r = 1, \delta \vee \epsilon} \, d \tilde{\theta} \Big)  \, d\theta \Big| \\
&\qquad  \qquad \lesssim \big\| \mathbf{f}_{\delta,\epsilon}^{\pm}  \big\|_{\dot{W}^{\frac{1}{2},2}((\alpha,\beta))'} \big\| \nabla \mathbf{u}_{\delta,\epsilon} \|_{L^2(\Omega_{\delta})}  \lesssim \frac{1}{\sqrt{\log \frac{1}{\delta \vee \epsilon}}} \big\| \mathbf{f}_{\delta,\epsilon}^{\pm}\big\|_{\dot{W}^{\frac{1}{2},2}((\alpha,\beta))'} \rightarrow 0.
\end{align*}
For this, we argue similarly to \eqref{eq:forceDispPair} and use the inequalities in  \eqref{eq:boundedaL2GradU} and the assumptions in \eqref{eq:restateForces}. The desired convergence   \eqref{eq:firstForceConvergence} follows.

\textit{Step 3: Characterizing the limits $\mathbf{R}_0$ and $e_0$.} Finally, we characterize $\mathbf{R}_0$ and $e_0$ from \eqref{eq:rotCharacterization} and \eqref{eq:workDueToForcesLimit}. If \eqref{super-degenerate1} holds,  we use 
\begin{align*}
 \frac{1}{\epsilon} \max_{\mathbf{Q} \in SO(2)} \int_{\partial \Omega_{\delta}} \mathbf{f}_{\delta, \epsilon} \cdot \big( \mathbf{Q} \mathbf{x}  - \mathbf{R}_{\delta, \epsilon} \mathbf{x} \big) \, ds \leq  2 \frac{1}{\epsilon} \max_{\mathbf{Q} \in SO(2)} \int_{\partial \Omega_{\delta}} \mathbf{f}_{\delta, \epsilon} \cdot  \mathbf{Q}  \mathbf{x}\,ds= 2 \frac{V_{\delta, \epsilon}^{\star}}{\epsilon}  
\end{align*}
and \eqref{eq:secForceConvergence} to show that  $e_0=0$. In the other two cases,  $V_{\delta, \epsilon}^{\star} >0$ for sufficiently small $\delta, \epsilon$ and $\mathbf{R}(\varphi_{\delta, \epsilon}) = \cos \varphi_{\delta, \epsilon} \mathbf{I} + \sin \varphi_{\delta, \epsilon} \mathbf{J}$  satisfies $\max_{\mathbf{Q} \in SO(2)} \int_{\partial \Omega_{\delta}} \mathbf{f}_{\delta, \epsilon} \cdot \mathbf{Q} \mathbf{x} \,ds = \int_{\partial \Omega_{\delta}} \mathbf{f}_{\delta, \epsilon} \cdot \mathbf{R}(\varphi_{\delta,\epsilon}) \mathbf{x} \,ds$ (see \eqref{eq:phi_deltaeps_defn1} and the discussion around \eqref{eq:VdeltaEpsilonStar}). Thus,
\begin{align}\label{eq:maxQidentity}
\frac{1}{\epsilon} \max_{\mathbf{Q} \in SO(2)} \int_{\partial \Omega_{\delta}} \mathbf{f}_{\delta, \epsilon} \cdot \big( \mathbf{Q} \mathbf{x}  - \mathbf{R}_{\delta, \epsilon} \mathbf{x} \big) \, ds = \frac{V^\star_{\delta, \epsilon}}{\epsilon} \Big( 1 -   \mathbf{R}_{\delta,\epsilon} \mathbf{e}_1  \cdot \mathbf{R}(\varphi_{\delta, \epsilon}) \mathbf{e}_1  \Big)
\end{align}
because $\mathbf{R}_{\delta, \epsilon} = (\mathbf{e}_1 \cdot \mathbf{R}_{\delta, \epsilon} \mathbf{e}_1 )\mathbf{I} + (\mathbf{e}_2 \cdot \mathbf{R}_{\delta, \epsilon} \mathbf{e}_1) \mathbf{J}.$  When  \eqref{degenerate1} holds,  \eqref{eq:maxQidentity} converges to $V_0( 1 - \mathbf{R}_0 \mathbf{e}_1 \cdot \mathbf{R}(\varphi) \mathbf{e}_1)$ by \eqref{eq:getVarphiDef1}  since $\mathbf{R}_{\delta, \epsilon} \rightarrow \mathbf{R}_0$. 
In the last case \eqref{non-degenerate1}, the boundedness of \eqref{eq:maxQidentity} and the estimate $|\mathbf{I} - \mathbf{R}|^2 \lesssim 1 - \mathbf{e}_1 \cdot \mathbf{R} \mathbf{e}_1$ for all $\mathbf{R} \in SO(2)$ give 
\begin{align*}
|\mathbf{R}_{\delta, \epsilon} - \mathbf{R}(\varphi_{\delta, \epsilon})|^2 \lesssim  1 -  \mathbf{R}_{\delta, \epsilon} \mathbf{e}_1  \cdot \mathbf{R}(\varphi_{\delta, \epsilon}) \mathbf{e}_1 \lesssim_p \frac{\epsilon}{V^\star_{\delta, \epsilon}} \rightarrow 0.
\end{align*}
Thus,  $\mathbf{R}_0 = \mathbf{R}(\varphi)$ by \eqref{eq:getVarphiDef1}. Note in this last case $e_0$ is not determined (it turns out to be zero for almost minimizers, as we show later on in Remark \ref{aysmpRemark}).
\end{proof}

\subsection{Convergence of the energies} 
We are ready to prove the main result of this section, namely \eqref{eq:force-main-result-here}.
\begin{prop}\label{convergence_energy_NL_force}
Consider any limit $\delta, \epsilon \rightarrow 0$ in the super-quadratic growth case ($p >2$) or assume such a limit obeys \eqref{eq:pEqual2Limit2} in the quadratic growth case ($p=2$). We have 
\begin{align}\label{eq:limitForceProblemNL}
\frac{\log \frac{1}{\delta \vee \epsilon}}{ \epsilon^2}  \mathcal{E}_{p,\delta, \epsilon}^{\emph{force}} \rightarrow \begin{cases}
\displaystyle \min_{\mathbf{R} \in SO(2)} - \QFl^\ast( \mathbf{R}^T \mathbf{f}_0)  &  \text{ if  \eqref{super-degenerate1}} \\  
\displaystyle \min_{\mathbf{R} \in SO(2)} - \QFl^\ast( \mathbf{R}^T \mathbf{f}_0) + V_0(1- \mathbf{R} \mathbf{e}_1 \cdot \mathbf{R}(\varphi) \mathbf{e}_1) & \text{ if \eqref{degenerate1}}\\ 
\displaystyle -\QFl^\ast( \mathbf{R}(\varphi)^T \mathbf{f}_0 )  & \text{ if  \eqref{non-degenerate1}}  
\end{cases}
\end{align}
 for $\varphi$ in \eqref{eq:getVarphiDef1} and $\mathbf{f}_0$ in \eqref{eq:restateForces}. 
\end{prop}
\begin{proof}
\textit{Step 1: The lower bound.} First, we establish that 
\begin{align}\label{liminf_NL_force}
\liminf_{\delta,\epsilon \rightarrow 0}\, \frac{\log \frac{1}{\delta \vee \epsilon}}{ \epsilon^2}  \mathcal{E}_{p,\delta, \epsilon}^{\text{force}} \geq  - \QFl^{\ast}(\mathbf{R}_0^T \mathbf{f}_0) + e_0
\end{align}
for $\mathbf{R}_0$ and $e_0$ from \eqref{eq:rotCharacterization} and \eqref{eq:workDueToForcesLimit}.
If the left side is $\infty$, the bound holds trivially. Otherwise, there  is a  sequence $\{ \mathbf{y}_{\delta, \epsilon} \}$  with $\mathbf{y}_{\delta, \epsilon} \in   W^{1,p}(\Omega_\d;\R^2)$ such that 
\begin{align*}
\lim_{\epsilon, \delta \rightarrow 0}\,  \frac{\log \frac{1}{\delta \vee \epsilon}}{\epsilon^2}  E_{\delta}(\mathbf{y}_{\d,\e}) - \frac{1}{\epsilon} \big( V_{\delta, \epsilon}(\mathbf{y}_{\d,\e}) - V_{\delta,\epsilon}^{\star} \big) = \liminf_{\delta,\epsilon \rightarrow 0}\, \frac{\log \frac{1}{\delta \vee \epsilon}}{ \epsilon^2}  \mathcal{E}_{p,\delta, \epsilon}^{\text{force}} < \infty.
\end{align*}
Let $\mathbf{R}_{\delta, \epsilon} \in \cR_{\delta}(\mathbf{y}_{\delta,\epsilon})$ (recall the set  \eqref{optimal_rotations1}) and define $\mathbf{U}_{\delta, \epsilon}(\rho, \theta) = \epsilon^{-1}(\mathbf{R}_{\delta,\epsilon}^T \mathbf{y}_{\delta,\epsilon} (\mathbf{x}) - \mathbf{x})$ using \eqref{eq:NL-blow-up-definition}.
By the compactness result in Proposition \ref{propCompactnessForce}, there is a subsequence such that $D^{\delta \vee \epsilon} \mathbf{U}_{\delta, \epsilon} \rightharpoonup \partial_{\rho} \mathbf{U} \otimes \mathbf{e}_{\rho} + \mathbf{m} \otimes \mathbf{e}_{\theta}$ weakly in $L^2(R;\R^{2\times 2})$ for some $\mathbf{U} \in W^{1,2}(R;\mathbb{R}^2)$ independent of $\theta$ and some $\mathbf{m} \in L^2(R; \mathbb{R}^2)$, and where $\mathbf{R}_{\delta, \epsilon} \rightarrow \mathbf{R}_0$ per \eqref{eq:rotCharacterization}. 
Given this, we can repeat the Taylor expansion argument in Step 1 of the proof of Proposition \ref{convergence_energy_NL_disp}, after making some minor modifications to the proof of \eqref{eq:this_rigidityEst}. 
Indeed, this estimate also holds in the force problem due to Lemma \ref{lem:EnergyForce}, because $\mathbf{R}_{\d,\e}\in\cR_{\d}(\mathbf{y}_{\d,\e})$, and by  \eqref{eq:pEqual2Limit2} for $p = 2$.
Continuing as before, we obtain
\begin{align*}
\liminf_{\delta, \epsilon \rightarrow 0}\,  \frac{\log \frac{1}{\delta \vee \epsilon}}{\epsilon^2} E_{\delta}(\mathbf{y}_{\delta,\epsilon}) \geq  \int_{R} Q \big(  \partial_{\rho} \mathbf{U} \odot \mathbf{e}_{\rho}  + \mathbf{m} \odot \mathbf{e}_{\theta}  \big) \, d\rho d \theta.
\end{align*}
Proposition \ref{propCompactnessForce} also gives that 
\begin{align*}
-\frac{1}{\epsilon}\Big( V_{\delta, \epsilon}(\mathbf{y}_{\delta,\epsilon}) - V^{\star}_{\delta, \epsilon} \Big) \rightarrow    \mathbf{R}_0^T\mathbf{f}_0 \cdot  \mathbf{U}|_{\rho = -1} -  \mathbf{R}_0^T \mathbf{f}_0 \cdot  \mathbf{U}|_{\rho = 0}  + e_0 
\end{align*}
with  $e_0 \geq 0$ given by \eqref{eq:workDueToForcesLimit}. Finally, the characterization of limiting force problem in Lemma \ref{minimzingLemForce} applied to $\mathbf{R}_0^T\mathbf{f}_0$ yields 
\begin{align*}
\int_{R} Q \big(  \partial_{\rho} \mathbf{U} \odot \mathbf{e}_{\rho}  + \mathbf{m} \odot \mathbf{e}_{\theta}  \big) \, d\rho d \theta +    \mathbf{R}_0^T\mathbf{f}_0 \cdot  \mathbf{U}|_{\rho = -1}  -  \mathbf{R}_0^T \mathbf{f}_0 \cdot \mathbf{U}|_{\rho = 0}   \geq -\QFl^{\ast}(\mathbf{R}_0^T \mathbf{f}_0).
\end{align*}
Combining these statements gives \eqref{liminf_NL_force}.

To complete the liminf-inequality, we lower bound the right-hand side of \eqref{liminf_NL_force} according to the degeneracy properties of $\{ \mathbf{f}_{\delta, \epsilon}\}$. 
If \eqref{super-degenerate1} holds, then $e_0 = 0$ and  $- \QFl^{\ast}(\mathbf{R}_0^T \mathbf{f}_0)\geq \min_{\mathbf{R} \in SO(2)} - \QFl^{\ast}(\mathbf{R}^T \mathbf{f}_0)$ as desired. In the case \eqref{degenerate1}, it holds that 
\begin{align*}
- \QFl^{\ast}(\mathbf{R}_0^T \mathbf{f}_0) + e_0=  - \QFl^{\ast}(\mathbf{R}_0^T \mathbf{f}_0) + V_0 (1 - \mathbf{R}_0 \mathbf{e}_1 \cdot  \mathbf{R}(\varphi) \mathbf{e}_1)
\end{align*}
and we again minimize out $\mathbf{R}_0$ to get the lower bound. Finally, if \eqref{non-degenerate1} is satisfied, then $\mathbf{R}_0 = \mathbf{R}(\varphi)$ by compactness, and we simply discard $e_0 \geq 0$.

\textit{Step 2: The upper bound.} We turn to the construction of a recovery sequence, i.e., a sequence of deformations that  saturates the lower bound \eqref{liminf_NL_force}.

\textit{Step 2a: Definition of the recovery sequence.} Define $\mathbf{y}_{\delta, \epsilon} \in W^{1,p}(\Omega_{\delta}; \mathbb{R}^2)$ in the form
\begin{align*}
\mathbf{y}_{\delta, \epsilon}(\mathbf{x}) = \mathbf{R}_{\delta, \epsilon} \big( \mathbf{x} + \epsilon \mathbf{u}_{\delta, \epsilon}( \mathbf{x})\big) 
\end{align*}
for $\mathbf{x} \in \Omega_{\delta}$.
Using the logarithmic change of variables \eqref{eq:NL-blow-up-definition}, define $\mathbf{u}_{\delta, \epsilon} \in W^{1,p}(\Omega_{\delta};\R^2)$ by 
\begin{align*}
\mathbf{u}_{\delta, \epsilon}(\mathbf{x}) = \mathbf{U}_{\delta, \epsilon}(\rho, \theta)= \UFlde(\rho, \theta) - \frac{1}{\log \frac{1}{\delta \vee \epsilon}} \mathbf{V}_{\delta, \epsilon}(\rho, \theta)   
\end{align*}
for $(\rho, \theta) \in R_{\delta, \epsilon} = ( (\log \frac{1}{\delta \vee \epsilon})^{-1} \log \delta, 0)\times(\alpha,\beta)$ (recall (\ref{eq:RdEpsilon})).
Similar to the construction in the displacement problem (see \eqref{eq:FlamantAnsatz_NL}),  
\[
\UFlde = \mathbf{U}_0 + \frac{1}{\log \frac{1}{\delta \vee \epsilon}} \mathbf{V}_0
\]
for  
\begin{align*}
\mathbf{V}_0 (\theta) = \int_{\alpha}^\theta \mathbf{m}_0 \, d\tilde{\theta} \quad \text{ and } \quad  
\mathbf{U}_0(\rho) = \begin{cases}
-\KFl^{-1} \mathbf{R}_0^T \mathbf{f}_0 & \text{ if } \rho \in (-\frac{\log \delta}{\log \frac{1}{\delta \vee \epsilon}}, -1) \\ 
\rho \KFl^{-1} \mathbf{R}_0^T \mathbf{f}_0 &\text{ if } \rho \in (-1,0) 
\end{cases} 
\end{align*}
where $\mathbf{R}_0 \in SO(2)$ is to be chosen later, and $\mathbf{m}_0$ is taken from Lemma \ref{minimzingLemForce} with $\mathbf{f}_0 $ replaced by $\mathbf{R}_0^T \mathbf{f}_0$. The map $\mathbf{V}_{\delta, \epsilon} \colon R_{\delta, \epsilon} \rightarrow \mathbb{R}^2$ is defined via
\begin{align*}
	\mathbf{V}_{\delta,\epsilon}(\rho,\theta)=\begin{cases} 
	\mathbf{V}_0(\theta) & \text{ if } \rho \in   (\tfrac{\log \delta}{\log \frac{1}{\delta \vee \epsilon}}, -1)  \\ 
		\psi\Big(\log\frac{1}{\delta \vee \epsilon}(\rho+1)\Big)\mathbf{V}_0(\theta)&\text{ if } \rho \in (-1,-1+\tfrac{1}{\log\frac{1}{\delta \vee \epsilon}})\\
		\mathbf{0}& \text{ if } \rho \in (-1+\tfrac{1}{\log\frac{1}{\delta \vee \epsilon}},0)
	\end{cases}
\end{align*}
for a smooth function $\psi:[0,1]\to[0,1]$ with $\psi=1$ on $[0,1/4]$ and $\psi=0$ on $[3/4,1]$.
Finally, 
\begin{align}\label{Choice_Rde}
\mathbf{R}_{\delta, \epsilon} = \begin{cases}  \mathbf{R}_0 & \text{ in case  \eqref{super-degenerate1} or \eqref{degenerate1}} \\ 
\mathbf{R}(\varphi_{\delta, \epsilon}) & \text{ in case  \eqref{non-degenerate1}}
\end{cases}.
\end{align}

\textit{Step 2b: Taylor expansion in the bulk of the wedge.} Observe that
\begin{align}\label{smallerR2}
D^{\delta \vee \epsilon} \mathbf{U}_{\delta, \epsilon} =   \partial_{\rho} \mathbf{U}_0 \otimes \mathbf{e}_{\rho} + \mathbf{m}_0 \otimes \mathbf{e}_{\theta} \quad \text{ on } \hat{R}_{\delta, \epsilon} := \Big(-1 + \frac{1}{\log \frac{1}{\delta \vee \epsilon}}, 0 \Big) \times(\alpha, \beta)
\end{align} 
and that $\mathbf{x} \in \Omega_{e (\delta \vee \epsilon) , 1}$ corresponds to $(\rho, \theta) \in \hat{R}_{\delta, \epsilon}$. By the $\delta\vee\epsilon$ analog of the gradient identity \eqref{eq:gradient-rule},
\begin{align*}
\| \epsilon \nabla \mathbf{u}_{\delta, \epsilon} \|_{L^{\infty}(\Omega_{e(\delta\vee\epsilon), 1})} \lesssim \frac{\epsilon^2}{(\delta \vee \epsilon) \log \frac{1}{\delta \vee \epsilon}} \| D^{\delta \vee \epsilon} \mathbf{U}_{\delta, \epsilon}\|_{L^{\infty}(\hat{R}_{\delta, \epsilon})} \lesssim  \frac{1}{\log \frac{1}{\delta \vee \epsilon}}.
\end{align*}
We therefore Taylor expand the elastic energy on $\Omega_{e(\delta\vee\epsilon), 1}$ to obtain that
\begin{align*}
\int_{\Omega_{e(\delta\vee\epsilon),1}} W(\nabla \mathbf{y}_{\delta, \epsilon}) \, d \mathbf{x}  &= \int_{\Omega_{e(\delta\vee\epsilon),1}} W\big(\mathbf{R}_{\delta, \epsilon}( \mathbf{I} + \epsilon \nabla \mathbf{u}_{\delta, \epsilon} )\big) \, d \mathbf{x} \\
& =  \epsilon^2 \int_{\Omega_{e(\delta\vee\epsilon),1}} Q\big(\mathbf{e}( \mathbf{u}_{\delta, \epsilon})\big) \, d \mathbf{x}  + o \Big( \epsilon^2 \int_{\Omega_{e(\delta\vee\epsilon),1}} |\nabla \mathbf{u}_{\delta, \epsilon} |^2 \, d\mathbf{x} \Big)  \\
& =  \frac{\epsilon^2}{\log \frac{1}{\delta \vee \epsilon}}  \int_{\hat{R}_{\delta,\epsilon}} Q (\sym\,D^{\delta \vee \epsilon}\mathbf{U}_{\delta, \epsilon}) \, d\rho d \theta  + o \Big( \frac{\epsilon^2}{\log \frac{1}{\delta \vee \epsilon}} \int_{\hat{R}_{\d,\e}} |D^{\delta \vee \epsilon} \mathbf{U}_{\delta, \epsilon}|^2  \, d\rho d\theta \Big)  \\
&= \frac{\epsilon^2}{\log \frac{1}{\delta \vee \epsilon}} \Big( \int_R Q ( \partial_{\rho} \mathbf{U}_0 \odot \mathbf{e}_{\rho} + \mathbf{m}_0 \odot \mathbf{e}_{\theta}) \, d\rho d\theta + O\big( |R \setminus \hat{R}_{\delta, \epsilon}|\big) \Big)  + o \Big( \frac{\epsilon^2}{\log \frac{1}{\delta \vee \epsilon}} \Big). 
\end{align*}
As $|R \setminus \hat{R}_{\delta, \epsilon}|  \rightarrow 0$ by the definition of $\hat{R}_{\delta, \epsilon}$ in \eqref{smallerR2},
\begin{align*}
 \frac{\log \frac{1}{\delta \vee \epsilon}}{ \epsilon^2} \int_{\Omega_{e(\delta\vee\epsilon),1}} W(\nabla \mathbf{y}_{\delta, \epsilon}) \, d \mathbf{x}  \rightarrow \int_R Q ( \partial_{\rho} \mathbf{U}_0 \odot \mathbf{e}_{\rho} + \mathbf{m}_0 \odot \mathbf{e}_{\theta}) \, d\rho d\theta = \QFl(\KFl^{-1}\mathbf{R}_0^T\mathbf{f}_0) = \QFl^{\ast}(\mathbf{R}_0^T \mathbf{f}_0).
\end{align*}
The last two equalities follow from Lemma \ref{minimzingLemForce}, our choices for $\mathbf{m}_0$ and $\mathbf{U}_0$, and the definitions of $\QFl$ and $\QFl^\ast$ in \eqref{Flamant_energies}.

\textit{Step 2c: Error estimates in the tip of the wedge.}
Next we show that the elastic energy is negligible on $\Omega_{\delta, e(\delta \vee \epsilon)}$. The basic point is that $\mathbf{U}_{\delta, \epsilon}$ and $\mathbf{V}_{\delta, \epsilon}$ are constant on the set $\Omega_{\delta, \delta \vee \epsilon}$, so 
\begin{align*}
\| \epsilon  \nabla \mathbf{u}_{\delta, \epsilon} \|_{L^{\infty}(\Omega_{\delta, e(\delta \vee \epsilon)})} = \|\epsilon  \nabla \mathbf{u}_{\delta, \epsilon} \|_{L^{\infty}(\Omega_{\delta\vee\epsilon, e(\delta \vee \epsilon)})}  \lesssim   \frac{\epsilon}{(\delta \vee \epsilon) \log \frac{1}{\delta \vee \epsilon}} \ll 1
\end{align*}
by the gradient identity \eqref{eq:gradient-rule}.
Since also $|\Omega_{\delta, e(\delta \vee \epsilon)}| \lesssim (\delta \vee \epsilon)^2$, we have that
\begin{align*}
\int_{\Omega_{\delta, e(\delta \vee \epsilon)}} W(\nabla \mathbf{y}_{\delta, \epsilon}) \,d \mathbf{x}  = \int_{\Omega_{\delta, e(\delta \vee \epsilon)}} W(\mathbf{I} + \epsilon \nabla \mathbf{u}_{\delta, \epsilon}) \,d \mathbf{x} = O \Big( \int_{\Omega_{\delta, e(\delta \vee \epsilon)}} | \epsilon \nabla \mathbf{u}_{\delta, \epsilon}|^2 \, d\mathbf{x} \Big) = O\Big( \frac{\epsilon^2}{(\log\frac{1}{\delta \vee \epsilon})^2} \Big)  
\end{align*}
by Taylor expansion.  As such, $\frac{\log \frac{1}{\delta \vee \epsilon}}{ \epsilon^2} \int_{\Omega_{\delta, e(\delta \vee \epsilon)}} W(\nabla \mathbf{y}_{\delta, \epsilon}) \, d \mathbf{x}  \rightarrow  0$.

\textit{Step 2d: Asymptotic behavior of the force terms.} Recall from \eqref{eq:force_work_rewritten} that
\begin{align*}
-\frac{1}{\epsilon} \Big( V_{\delta, \epsilon}(\mathbf{y}_{\delta, \epsilon}) - V^{\star}_{\delta, \epsilon}\Big) &= -\int_{\partial \Omega_{\delta}} \mathbf{f}_{\delta, \epsilon} \cdot \mathbf{R}_{\delta, \epsilon} \mathbf{u}_{\delta, \epsilon} \, ds  + \max_{\mathbf{Q} \in SO(2)} \int_{\partial \Omega_{\delta}} \mathbf{f}_{\delta, \epsilon} \cdot \Big( \mathbf{Q} \mathbf{x} - \mathbf{R}_{\delta, \epsilon} \mathbf{x}\Big) \, ds.
\end{align*}
Using the definition of $\mathbf{R}_{\delta, \epsilon}$ in \eqref{Choice_Rde} and considering the cases \eqref{super-degenerate1}-\eqref{non-degenerate1}, it is straightforward to show via the same type of manipulations as in Step 3 of Proposition \ref{propCompactnessForce} that 
\begin{align*}
& \frac{1}{\epsilon}  \max_{\mathbf{Q} \in SO(2)} \int_{\partial \Omega_{\delta}} \mathbf{f}_{\delta, \epsilon} \cdot \Big( \mathbf{Q} \mathbf{x} - \mathbf{R}_{\delta, \epsilon} \mathbf{x}\Big) \, ds \rightarrow \begin{cases}
 0 & \text{ if }  \eqref{super-degenerate1} \\
 V_0 \big(1 - \mathbf{R}_0 \mathbf{e}_1 \cdot \mathbf{R}(\varphi) \mathbf{e}_1\big)  & \text{ if } \eqref{degenerate1} \\
 0 & \text{ if } \eqref{non-degenerate1}
 \end{cases}. 
\end{align*}
For the other work term, we have 
\begin{align*}
-\int_{\partial \Omega_{\delta}} \mathbf{f}_{\delta, \epsilon} \cdot \mathbf{R}_{\delta, \epsilon} \mathbf{u}_{\delta, \epsilon} \, ds  &= -\int_{\alpha}^{\beta} \mathbf{R}_{\delta, \epsilon}^T \mathbf{f}_{\delta, \epsilon}^+ \cdot  \mathbf{u}_{\delta, \epsilon}|_{r = 1}  -  \mathbf{R}_{\delta, \epsilon}^T \mathbf{f}_{\delta, \epsilon}^- \cdot  \mathbf{u}_{\delta, \epsilon}|_{r = \delta}  \, d\theta \\ 
& = \int_{\alpha}^{\beta} \mathbf{R}_{\delta, \epsilon}^T \mathbf{f}_{\delta, \epsilon}^- \cdot \mathbf{U}_{\delta, \epsilon}|_{\rho = \frac{\log \delta}{\log \frac{1}{\delta \vee \epsilon}}}  - \mathbf{R}_{\delta, \epsilon}^T \mathbf{f}_{\delta, \epsilon}^+ \cdot \mathbf{U}_{\delta, \epsilon}|_{\rho =0}  \,d \theta \\
&= - \int_{\alpha}^{\beta}  \mathbf{f}_{\delta, \epsilon}^- \,d\theta\cdot \mathbf{R}_{\delta, \epsilon} \KFl^{-1}  \mathbf{R}_0^T \mathbf{f}_0 - \frac{1}{\log \frac{1}{\delta \vee \epsilon}} \int_{\alpha}^{\beta} \mathbf{R}_{\delta, \epsilon}^T \mathbf{f}^+_{\delta, \epsilon} \cdot  \mathbf{V}_0\, d\theta.
\end{align*}
The second term in the last equality satisfies 
\begin{align*}
&\Big| \frac{1}{\log \frac{1}{\delta \vee \epsilon}} \int_{\alpha}^{\beta} \mathbf{R}_{\delta, \epsilon}^T \mathbf{f}^+_{\delta, \epsilon} \cdot  \mathbf{V}_0\, d\theta \Big|  \\
&\qquad  =\frac{1}{\log \frac{1}{\delta \vee \epsilon}} \Big| \int_{\alpha}^{\beta} \Big( \mathbf{R}_{\delta, \epsilon}^T \mathbf{f}_{\delta, \epsilon}^+ - \fint_{\alpha}^{\beta} \mathbf{R}_{\delta, \epsilon}^T \mathbf{f}_{\delta, \epsilon}^+ \, d\tilde{\theta}\Big) \cdot \Big( \mathbf{V}_0  - \fint_{\alpha}^{\beta} \mathbf{V}_0\, d \tilde{\theta} \Big) \, d\theta + \fint_{\alpha}^{\beta} \mathbf{R}_{\delta, \epsilon}^T \mathbf{f}_{\delta, \epsilon}^+ \, d \tilde{\theta} \cdot \int_{\alpha}^{\beta} \mathbf{V}_0 \,d\theta \Big|  \\
&\qquad \leq \frac{1}{\log \frac{1}{\delta \vee \epsilon}}\Big( \big\| \mathbf{f}_{\delta, \epsilon}^+ \big\|_{\dot{W}^{\frac{1}{2},2}((\alpha, \beta))'} \big\| \mathbf{V}_0 \big\|_{\dot{W}^{\frac{1}{2},2}((\alpha, \beta))} + \Big| \int_{\alpha}^{\beta} \mathbf{f}_{\delta, \epsilon}^+ \, \d\theta\Big|\Big| \fint_{\alpha}^{\beta} \mathbf{V}_0 \, d\theta\Big| \Big)  \rightarrow 0
\end{align*}
by the  hypotheses  in \eqref{eq:restateForces} and since  $\mathbf{V}_0$ is smooth.  The first term satisfies 
\begin{align*}
 - \int_{\alpha}^{\beta}  \mathbf{f}_{\delta, \epsilon}^- \,d\theta\cdot \mathbf{R}_{\delta, \epsilon} \KFl^{-1}  \mathbf{R}_0^T \mathbf{f}_0 \rightarrow \begin{cases}
 - 2\QFl^{\ast}(\mathbf{R}_0^T \mathbf{f}_0)  & \text{ if \eqref{super-degenerate1} or \eqref{degenerate1}} \\
- \mathbf{R}(\varphi)^T \mathbf{f}_0 \cdot \KFl^{-1} \mathbf{R}_0^T \mathbf{f}_0 & \text{ if \eqref{non-degenerate1}}
 \end{cases}.
\end{align*}

\textit{Step 2e: Conclusion.} From the previous steps,
\begin{align*}
\lim _{\delta, \epsilon \rightarrow 0}\, \frac{\log \frac{1}{\delta \vee \epsilon}}{ \epsilon^2} E_{\delta}(\mathbf{y}_{\delta, \epsilon} ) - \frac{1}{\epsilon} \Big( V_{\delta, \epsilon}(\mathbf{y}_{\delta, \epsilon}) - V_{\delta,\epsilon}^{\star} \Big) =\begin{cases}
-\QFl^{\ast}( \mathbf{R}_0^T \mathbf{f}_0)  & \text{ if \eqref{super-degenerate1}}\\
-\QFl^{\ast}( \mathbf{R}_0^T \mathbf{f}_0) + V_0 ( 1 - \mathbf{R}_0 \mathbf{e}_1 \cdot \mathbf{R}(\varphi) \mathbf{e}_1) & \text{ if  \eqref{degenerate1}}\\
\QFl^{\ast}(\mathbf{R}_0^T \mathbf{f}_0) - \mathbf{R}(\varphi)^T \mathbf{f}_0 \cdot \KFl^{-1} \mathbf{R}_0^T \mathbf{f}_0 & \text{ if  \eqref{non-degenerate1}}
\end{cases}.
\end{align*}
We have yet to choose the rotation $\mathbf{R}_0$. To finish, we take it to minimize the limits on the right-hand side above. This gives \eqref{eq:limitForceProblemNL}. In particular,  $\mathbf{R}_0 = \mathbf{R}(\varphi)$ is optimal for \eqref{non-degenerate1}.
\end{proof}

\section{Asymptotics of almost minimizers}\label{sec:strong-convergence}

We come at last to the characterization of almost minimizers in Theorem \ref{thm:main-result}. Let us recall what we wish to prove. By definition, a sequence of almost minimizers $\{\mathbf{y}_{\delta,\epsilon}\}$ in the nonlinear displacement or force problem satisfies 
\begin{equation}\label{eq:almostMinSec6}
\begin{aligned}
 & \mathbf{y}_{\delta, \epsilon}  \in \mathcal{A}_{p, \delta, \epsilon} && \text{ and }  && E_\delta(\mathbf{y}_{\delta, \epsilon}) = \mathcal{E}_{p,\delta,\epsilon}^{\text{disp}} + o\Big( \frac{\epsilon^2}{\log \frac{1}{\delta \vee \epsilon}}\Big), \quad\text{or}\\
  &\mathbf{y}_{\delta, \epsilon} \in W^{1,p}(\Omega_{\delta}; \mathbb{R}^2) && \text{ and }  && E_\delta(\mathbf{y}_{\delta, \epsilon}) - \frac{\epsilon}{\log \frac{1}{\delta \vee \epsilon}} \Big( V_{\delta, \epsilon}(\mathbf{y}_{\delta, \epsilon}) - V_{\delta, \epsilon}^{\star} \Big) = \mathcal{E}_{p,\delta,\epsilon}^{\text{force}} + o\Big( \frac{\epsilon^2}{\log \frac{1}{\delta \vee \epsilon}}\Big)
\end{aligned}
\end{equation}
where $\mathcal{E}_{p,\delta,\epsilon}^{\text{disp}}$ and $\mathcal{E}_{p,\delta,\epsilon}^{\text{force}}$ are the minimum values of these problems. Our claim is that the Flamant solution gives the leading order asymptotics of $\{\mathbf{y}_{\delta,\epsilon}\}$ in a strong $W^{1,2}$-sense (the weak analog of this statement is already contained in the proofs of Propositions \ref{convergence_energy_NL_disp} and \ref{convergence_energy_NL_force}, cf.\ the proof of the fundamental theorem of $\Gamma$-convergence in, e.g., \cite{dal2012introduction}). After passing to a subsequence, we prove in this section that
\begin{align}
\begin{cases}\label{eq:resultFinalSec6}
\Big\| \nabla \mathbf{y}_{\delta, \epsilon} -  \Big( \mathbf{I} + \frac{\epsilon}{\log \frac{1}{\delta \vee \epsilon}} \nabla \mathbf{u}_{\text{Fl}} \Big) \Big\|_{L^2(\Omega_{\delta \vee \epsilon})}   \ll \frac{\epsilon}{\sqrt{\log \frac{1}{\delta \vee \epsilon}}}  &\text{ in the displacement problem}  \\
\Big\| \nabla \mathbf{y}_{\delta, \epsilon} - \mathbf{R}_{\delta, \epsilon} \Big( \mathbf{I} + \frac{\epsilon}{\log \frac{1}{\delta \vee \epsilon}} \nabla \mathbf{u}_{\text{Fl}} \Big) \Big\|_{L^2(\Omega_{\delta \vee \epsilon})} \ll \frac{\epsilon}{\sqrt{\log \frac{1}{\delta \vee \epsilon}}} & \text{ in the force problem}  
\end{cases}
\end{align}
in general for $p > 2$, and under the additional assumption that $\log \frac{\delta \vee \epsilon}{\delta}  \ll \log \frac{1}{\epsilon}$ if $p = 2$. If $\delta < \epsilon$, we show also that
\begin{align}
\begin{cases}\label{eq:resultFinal1Sec6}
\big\| \nabla \mathbf{y}_{\delta, \epsilon} -  \mathbf{I} \big\|_{L^2(\Omega_{\delta, \epsilon})}   \ll \frac{\epsilon}{\sqrt{\log \frac{1}{\epsilon}}}  &\text{ in the displacement problem}  \\
\big\| \nabla \mathbf{y}_{\delta, \epsilon} - \mathbf{R}_{\delta, \epsilon}  \big\|_{L^2(\Omega_{\delta, \epsilon})} \ll \frac{\epsilon}{\sqrt{\log \frac{1}{\epsilon}}} & \text{ in the force problem}  
\end{cases}
\end{align}
where $\Omega_{\delta, \epsilon}=\Omega_{\delta}\cap\{r\in(\delta,\epsilon)\}$ refers to the tip of the wedge.
The rotation 
  $\mathbf{R}_{\delta, \epsilon}$ belongs to the set
  \begin{align}\label{optimal_rotations2}
\mathcal{R}_{\delta } (\mathbf{y}_{\delta, \epsilon}) = \Big\{ \mathbf{R} \in SO(2) \colon \int_{\Omega_{\delta}} | \nabla \mathbf{y}_{\delta, \epsilon} - \mathbf{R}|^2 \, d\mathbf{x} = \min_{\mathbf{Q} \in SO(2)}\,  \int_{\Omega_{\delta}} | \nabla \mathbf{y}_{\delta, \epsilon} - \mathbf{Q}|^2 \, d\mathbf{x} \Big\},
\end{align}
 which consists of a single rotation when $\delta, \epsilon$ are sufficiently small (see Remark \ref{rem:unitRotation}). The displacement $\mathbf{u}_{\text{Fl}}$ is the Flamant ansatz from \eqref{Flamant_displacement}, i.e.,
\begin{equation}
\begin{aligned}\label{eq:uFl1Sec6}
&\mathbf{u}_{\text{Fl}}(\mathbf{x}) = \mathbf{u}_0(r) + \mathbf{u}_1(\theta) \quad \text{ with } \quad \mathbf{u}_0(r) = (\log r ) \mathbf{a}  \quad \text{ and } \\
&\qquad \mathbf{u}_1(\theta) = \left[ \int_{\alpha}^{\theta} \Big(\frac{ 2(\mathbb{C}^{-1})_{r\theta rr} }{(\mathbb{C}^{-1})_{rrrr}} \mathbf{e}_{r} \otimes \mathbf{e}_{r}  - \mathbf{e}_r \otimes \mathbf{e}_{\theta} +\frac{(\mathbb{C}^{-1})_{\theta\theta rr}}{(\mathbb{C}^{-1})_{rrrr}}  \mathbf{e}_{\theta}\otimes \mathbf{e}_r\Big) \, d\tilde{\theta} \right] \mathbf{a}.
\end{aligned}
\end{equation}
 The coefficient $\mathbf{a}$ is chosen according to the limiting value of the average displacement or applied force:  
\begin{align}\label{eq:uFl2Sec6}
\mathbf{a}  = \begin{cases}
\mathbf{u}_0^+ - \mathbf{u}_0^{-} & \text{ in the displacement problem } \\
\mathbf{K}_{\text{Fl}}^{-1} \mathbf{R}_\text{a}^T \mathbf{f}_0  & \text{ in  case \eqref{super-degenerate1} of the force problem } \\
\mathbf{K}_{\text{Fl}}^{-1} \mathbf{R}_\text{b}^T \mathbf{f}_0 &  \text{ in case \eqref{degenerate1} of the force problem } \\ 
\mathbf{K}_{\text{Fl}}^{-1} \mathbf{R}^T(\varphi) \mathbf{f}_0  & \text{ in case \eqref{non-degenerate1} of the force problem } 
\end{cases}
\end{align}
where $\mathbf{R}_\text{a}, \mathbf{R}_{\text{b}}$ are  minimizing rotations in \eqref{eq:limitForceProblemNL} for cases \eqref{super-degenerate1} and \eqref{degenerate1}, respectively. 

These results constitute the majority of parts (II) and (III) of Theorem \ref{thm:main-result} and are proved in Proposition \ref{finalProp}. To complete the theorem, we also show the subsequential convergence of $\{\mathbf{R}_{\delta, \epsilon}\}$   to $\mathbf{R}_{\text{a}}$, $\mathbf{R}_{\text{b}}$, or $\mathbf{R}(\varphi)$  as claimed in \eqref{eq:rotConvergeMainTheorem}, and prove the asymptotic expansion formula \eqref{eq:expansion-of-rotation} in case \eqref{non-degenerate1} of the force problem. See Remark \ref{rotConvergeRemark} for the former and Remark \ref{aysmpRemark} for the latter. 

The proof of these results is involved, but the overall strategy is familiar: for instance, an analogous result was achieved in the variational derivation of linear elasticity in \cite{ADDe2012}, via an argument based on equi-integrability. Here, we use polar decompositions instead. We break the proof up into several steps. First, we extract a subsequence of displacements $\{\mathbf{u}_{\delta, \epsilon}\}$ from the almost minimizers $\{\mathbf{y}_{\delta, \epsilon}\}$ to obtain an analog of the asymptotic expansion of linear strains from Theorem \ref{thm:main-result-linearized}.
Second, we address the polar decomposition of the deformation gradient given by $\nabla \mathbf{y}_{\delta, \epsilon} = \tilde{\mathbf{R}}_{\delta, \epsilon} (\nabla \mathbf{y}_{\delta, \epsilon}^T \nabla \mathbf{y}_{\delta, \epsilon})^{1/2}$ for an orthogonal matrix field  $\tilde{\mathbf{R}}_{\delta, \epsilon}$. 
Finally, we prove  \eqref{eq:resultFinalSec6} and \eqref{eq:resultFinal1Sec6}. The most delicate part involves showing that  $\tilde{\mathbf{R}}_{\delta, \epsilon}$ is approximated  by  $\mathbf{I} + \epsilon(\log \frac{1}{\delta \vee \epsilon})^{-1}  \skw \nabla \mathbf{u}_{\text{Fl}}$ in the displacement problem, and by $\mathbf{R}_{\delta, \epsilon} \big( \mathbf{I} + \epsilon(\log \frac{1}{\delta \vee \epsilon})^{-1}  \skw\nabla \mathbf{u}_{\text{Fl}}\big)$ in the force problem, on the bulk of the wedge $\Omega_{\delta\vee\epsilon}$. 
Per \eqref{eq:resultFinal1Sec6}, this approximation can fail in the tip of the wedge  $\Omega_{\delta,\delta\vee\epsilon}$ depending on how $\delta,\epsilon \to 0$  (see Remark \ref{rem:truncation}). 

\subsection{Asymptotics of the linear strain}
Let $\mathbf{R}_{\delta, \epsilon} \in \cR_{\delta}(\mathbf{y}_{\delta, \epsilon})$ from \eqref{optimal_rotations2} and define the displacements 
\begin{align}\label{eq:theDisplacementStrong}
\mathbf{u}_{\delta, \epsilon}(\mathbf{x}) = \begin{cases}
\frac1\e (\mathbf{y}_{\delta, \epsilon}(\mathbf{x}) - \mathbf{x}) & \text{ in the displacement problem}  \\
\frac1\e (\mathbf{R}_{\delta, \epsilon}^T \mathbf{y}_{\delta, \epsilon}(\mathbf{x}) - \mathbf{x})  & \text{ in the force problem}
\end{cases}.
\end{align}
Recall the set
\begin{equation}
\begin{aligned}\label{good_sets}
B_{\delta,\epsilon}&= \Big\{\mathbf{x}\in \Omega_{\delta\vee\epsilon}: |\epsilon \nabla \mathbf{u}_{\delta,\epsilon}(\mathbf{x})| < \tfrac{1}{\sqrt{\log\frac{1}{\delta\vee \epsilon}}}\Big\}
\end{aligned}
\end{equation}
from the proofs of Propositions \ref{convergence_energy_NL_disp} and \ref{convergence_energy_NL_force} (see \eqref{eq:bounded-set} and the choice $M_{\delta,\epsilon} = (\log \frac{1}{\delta \vee \epsilon})^{-1/2}$ immediately following \eqref{eq:RBhatCalc}). 
\begin{lem}\label{lemmaStartStrongConvergence}
Consider any limit $\delta, \epsilon \rightarrow 0$ in the super-quadratic growth case ($p>2$) or assume such a limit obeys \eqref{eq:pEqual2Limit2} in the quadratic growth case ($p =2$).   For any almost minimizing sequence $\{\mathbf{y}_{\delta,\epsilon}\}$ in \eqref{eq:almostMinSec6} and associated displacements $\{ \mathbf{u}_{\delta, \epsilon}\}$ in \eqref{eq:theDisplacementStrong}, there is a subsequence such that
\begin{equation}
\begin{aligned}\label{eq:toStartStrongConvergence}
&\frac{\log \frac{1}{\delta \vee \epsilon}}{\epsilon^{2}}   \int_{\Omega_{\delta} \setminus B_{\delta, \epsilon}}  W( \nabla \mathbf{y}_{\delta, \epsilon}) \, d\mathbf{x} \rightarrow 0, \\
&\sqrt{\log \frac{1}{\delta \vee \epsilon}} \Big\| \indicator{B_{\delta, \epsilon}} \mathbf{e}(\mathbf{u}_{\delta, \epsilon}) - \frac{1}{\log \frac{1}{\delta \vee \epsilon}} \mathbf{e}(\mathbf{u}_{\emph{Fl}}) \Big\|_{L^2(\Omega_{\delta \vee \epsilon})} \rightarrow 0
\end{aligned}
\end{equation}
for $\mathbf{u}_{\emph{Fl}}$ in \eqref{eq:uFl1Sec6}-\eqref{eq:uFl2Sec6}. The rotations $\mathbf{R}_\emph{a}$ and $\mathbf{R}_\emph{b}$ are given by the subsequence in cases \eqref{super-degenerate1} and \eqref{degenerate1}.
\end{lem}
\begin{rem}\label{rotConvergeRemark}
   In the proof of this result, we refine the convergence from \eqref{eq:rotCharacterization} of Proposition \ref{propCompactnessForce} using almost minimality. We show in particular that $\{\mathbf{R}_{\d,\e}\}$ converges subsequentially to $\mathbf{R}_{\text{a}}$ in case \eqref{super-degenerate1}, $\mathbf{R}_{\text{b}}$ in case \eqref{degenerate1}, and $\mathbf{R}(\varphi)$ in case \eqref{non-degenerate1} of the force problem. In the first two cases, the rotations   $\mathbf{R}_{\text{a}},\mathbf{R}_{\text{b}}$ are shown to solve their respective minimization problems from  \eqref{eq:limitForceProblemNL}. This proves \eqref{eq:rotConvergeMainTheorem} from Theorem \ref{thm:main-result}.
\end{rem}
\begin{rem}\label{aysmpRemark}
A further byproduct of the proof is the fact that 
   \begin{align*}
       \mathbf{R}_{\delta, \epsilon} = \mathbf{R}(\varphi_{\delta, \epsilon}) + o \Big( \sqrt{\frac{\epsilon}{V_{\delta,\epsilon}^{\star}}}\Big)
   \end{align*}
in  case \eqref{non-degenerate1} of the force problem, for  $\varphi_{\delta, \epsilon}$ given by  \eqref{eq:getVarphiDef1}.  In particular, we 
show that the quantity $e_0$ defined in Proposition \ref{propCompactnessForce} as 
   \begin{align*}
    e_0 = \lim_{\delta ,\epsilon\rightarrow 0}\, \frac{1}{\epsilon}  \max_{\mathbf{Q} \in SO(2)}\,\int_{\partial \Omega_{\delta}} \mathbf{f}_{\delta, \epsilon} \cdot ( \mathbf{Q} \mathbf{x} - \mathbf{R}_{\delta, \epsilon} \mathbf{x}) \, ds   
   \end{align*}
is  zero if $\{ \mathbf{y}_{\delta, \epsilon}\}$ is almost minimizing and we are in case \eqref{non-degenerate1}. We can then  write that
   \begin{align*}
       \frac{V_{\delta, \epsilon}^{\star}}{\epsilon}|\mathbf{R}_{\delta, \epsilon} - \mathbf{R}(\varphi_{\delta, \epsilon})|^2  \lesssim \frac{V_{\delta, \epsilon}^{\star}}{\epsilon}  \big(1 - \mathbf{R}_{\delta, \epsilon} \mathbf{e}_1  \cdot \mathbf{R}(\varphi_{\delta, \epsilon}) \mathbf{e}_1\big)  = \frac{1}{\epsilon}  \max_{\mathbf{Q} \in SO(2)}\,\int_{\partial \Omega_{\delta}} \mathbf{f}_{\delta, \epsilon} \cdot ( \mathbf{Q} \mathbf{x} - \mathbf{R}_{\delta, \epsilon} \mathbf{x}) \, ds \to 0
   \end{align*}
   by Step 3 of the proof of Proposition \ref{propCompactnessForce}. 
   This proves \eqref{eq:expansion-of-rotation} from Theorem \ref{thm:main-result}.
   \end{rem}
\begin{rem}
We must allow to pass to a subsequence in this lemma---and only in this lemma---due to the possibility that the original sequence of rotations $\{\mathbf{R}_{\d,\e}\}$ does not converge. In fact, as was discussed in Remark \ref{rem:subsequence-thm} of the introduction, this possibility can only occur in cases \eqref{super-degenerate1} and \eqref{degenerate1} of the force problem. Instead, in case \eqref{non-degenerate1} of the force problem,  $\mathbf{R}_{\d,\e}\to \mathbf{R}(\varphi)$ with $\varphi$ given in the problem statement; likewise, $\mathbf{R}_{\d,\e}\to\mathbf{I}$ in the displacement problem. In these last two setups, $\mathbf{e}(\mathbf{u}_{\text{FL}})$ is uniquely determined by $\mathbf{R}(\varphi)$ and $\mathbf{f}_0$, or by $\mathbf{u}_0^+ - \mathbf{u}_0^{-}$ depending on the problem. By reapplying the lemma, one shows that any subsequence of an almost minimizing sequence has a further subsequence such that \eqref{eq:toStartStrongConvergence} holds with the same $\mathbf{e}(\mathbf{u}_{\text{Fl}})$. This implies \eqref{eq:toStartStrongConvergence} for the original sequence of almost minimizers.
\end{rem}

\begin{proof}
The general idea is to replicate the proof of the analogous linear result in Proposition \ref{almostMinimizersProp}. We use the logarithmic displacements $\mathbf{U}_{\delta, \epsilon}(\rho, \theta) = \mathbf{u}_{\delta, \epsilon}(\mathbf{x})$ for $\rho = (\log \frac{1}{\delta \vee \epsilon})^{-1}\log r$, where $\mathbf{u}_{\delta, \epsilon}$ are from \eqref{eq:theDisplacementStrong}. We also use the set $\hat{B}_{\delta, \epsilon} = \{ (\rho,\theta) \in R \text{ corresponding to } \mathbf{x} \in B_{\delta, \epsilon}\}$ to handle linearization in logarithmic coordinates. Finally, we use the matrix 
 \begin{align*}
\mathbf{M}(\theta) = \Big(\frac{ 2(\mathbb{C}^{-1})_{\rho\theta \rho\rho} }{(\mathbb{C}^{-1})_{\rho\rho\rho\rho}} \mathbf{e}_{\rho} \otimes \mathbf{e}_{\rho}  - \mathbf{e}_\rho \otimes \mathbf{e}_{\theta} +\frac{(\mathbb{C}^{-1})_{\theta\theta \rho \rho}}{(\mathbb{C}^{-1})_{\rho\rho\rho\rho}}  \mathbf{e}_{\theta}\otimes \mathbf{e}_\rho\Big)
 \end{align*}
 where $(\mathbb{C}^{-1})_{\rho \theta \rho \rho} = (\mathbb{C}^{-1})_{r\theta rr}$ and so on (since $\mathbf{e}_r = \mathbf{e}_{\rho}$).

\textit{Step 1: The displacement problem.}
Since $\{\mathbf{y}_{\d,\e}\}$ is almost minimizing, we obtain from Proposition \ref{CompactnessDispThm} that
a subsequence of $\{\mathbf{U}_{\d,\e}\}$ satisfies 
$D^{\delta \vee \epsilon} \mathbf{U}_{\delta, \epsilon} \rightharpoonup   \partial_{\rho} \mathbf{U} \otimes \mathbf{e}_{\rho} + \mathbf{m} \otimes \mathbf{e}_{\theta}$ in $L^2(R;\R^{2\times 2})$ for some $\mathbf{U} \in W^{1,2}(R; \mathbb{R}^2)$ with $\partial_{\theta} \mathbf{U} = \mathbf{0}$,  $\mathbf{U}|_{\rho = 0} = \mathbf{u}_0^+$, and $\mathbf{U}|_{\rho = -1} = \mathbf{u}_0^-$, and some $\mathbf{m} \in L^2(R; \mathbb{R}^2)$.  
Moreover, Proposition \ref{convergence_energy_NL_disp} and the arguments within yield $\indicator{\hat{B}_{\delta, \epsilon}} D^{\delta \vee \epsilon} \mathbf{U}_{\delta, \epsilon} \rightharpoonup \partial_{\rho} \mathbf{U} \otimes \mathbf{e}_{\rho} + \mathbf{m} \otimes \mathbf{e}_{\theta}$ in $L^2(\Omega, \mathbb{R}^{2\times2})$
along with the following chain of equalities and inequalities: 
\begin{equation*}
\begin{aligned}
	\QFl( \mathbf{u}_0^+ - \mathbf{u}_0^{-}) &= \int_{R} Q( \partial_{\rho} \mathbf{U}_0 \otimes \mathbf{e}_{\rho} + \mathbf{m}_0 \otimes \mathbf{e}_{\theta}) \, d \rho d\theta
	= \lim_{\d,\e\to 0}\,\frac{\log \frac{1}{\delta \vee \epsilon}}{\epsilon^2}  \int_{\Omega_{\delta}} W(\nabla \mathbf{y}_{\delta, \epsilon}) \, d \mathbf{x} \\
	&\geq  \limsup_{\d,\e\to 0}\,\frac{\log \frac{1}{\delta \vee \epsilon}}{\epsilon^2}  \int_{B_{\delta, \epsilon}} W(\nabla \mathbf{y}_{\delta, \epsilon})  \, d \mathbf{x}
	=  \limsup_{\d,\e\to 0}\,  \int_{R} Q( \sym \, \indicator{\hat{B}_{\delta, \epsilon}} D^{\delta \vee \epsilon} \mathbf{U}_{\delta, \epsilon})  \, d \rho d \theta  \\
    &\geq \liminf_{\d,\e\to 0}\int_{R}\,  Q( \sym\,\indicator{\hat{B}_{\delta, \epsilon}} D^{\delta \vee \epsilon} \mathbf{U}_{\delta, \epsilon})  \, d \rho d \theta \geq \int_{R} Q( \partial_{\rho} \mathbf{U}\odot \mathbf{e}_{\rho} + \mathbf{m} \odot \mathbf{e}_{\theta}) \, d \rho d\theta \geq  \QFl( \mathbf{u}_0^+ - \mathbf{u}_0^{-}). 
 \end{aligned}
\end{equation*}
Here, $\mathbf{U}_0(\rho) = \rho (\mathbf{u}_0^+ - \mathbf{u}_0^{-}) +  \mathbf{u}_0^+$ and $\mathbf{m}_0(\theta) = \mathbf{M}(\theta)(\mathbf{u}_0^+ - \mathbf{u}_0^{-})$  are the unique minimizers  in Lemma \ref{minimzingLemDisp}; the third equality follows from \eqref{eq:TaylorExpansionNonlinear} and \eqref{eq:dispGradtoDispGrad}. Evidently, then, all of these inequalities are equalities.  
Thus, 
\begin{equation}
\begin{aligned}\label{eq:keyForDispProb}
 &\indicator{\hat{B}_{\delta, \epsilon}} D^{\delta \vee \epsilon} \mathbf{U}_{\delta, \epsilon} \rightharpoonup \partial_{\rho} \mathbf{U}_0 \otimes \mathbf{e}_{\rho} + \mathbf{m}_0 \otimes \mathbf{e}_{\theta} \quad \text{ in $L^2(R; \mathbb{R}^{2\times2})$}, \\ 
 &\int_R Q(\sym\, \indicator{\hat{B}_{\delta, \epsilon}} D^{\delta \vee \epsilon} \mathbf{U}_{\delta, \epsilon}) \, d \rho d \theta \rightarrow \int_R Q( \partial_{\rho} \mathbf{U}_0 \odot \mathbf{e}_{\rho} + \mathbf{m}_0 \odot \mathbf{e}_{\theta}) \, d \rho d \theta.
\end{aligned}
\end{equation}
This computation  also proves the first convergence in \eqref{eq:toStartStrongConvergence}, since it implies in particular that
\begin{align*}
\lim_{\delta, \epsilon \rightarrow 0} \frac{\log \frac{1}{\delta \vee \epsilon}}{\epsilon^2}\,  \int_{\Omega_{\delta}} W(\nabla \mathbf{y}_{\delta, \epsilon}) \, d \mathbf{x}  = \lim_{\delta, \epsilon \rightarrow 0}\,   \frac{\log \frac{1}{\delta \vee \epsilon}}{\epsilon^2}  \int_{B_{\delta, \epsilon}} W(\nabla \mathbf{y}_{\delta, \epsilon}) \, d \mathbf{x} = \QFl( \mathbf{u}_0^+ - \mathbf{u}_0^{-}).
\end{align*}

\textit{Step 2: The force problem.} By the compactness result in Proposition \ref{propCompactnessForce}, we extract a subsequence with the following properties: $\{ \mathbf{U}_{\delta, \epsilon} \}$ satisfies   $D^{\delta \vee \epsilon} \mathbf{U}_{\delta, \epsilon} \rightharpoonup   \partial_{\rho} \mathbf{U} \otimes \mathbf{e}_{\rho} + \mathbf{m} \otimes \mathbf{e}_{\theta}$ in $L^2(R; \mathbb{R}^{2\times2})$  for some $\mathbf{U} \in W^{1,2}(R; \mathbb{R}^2)$ with $\partial_{\theta} \mathbf{U} = \mathbf{0}$ and    some $\mathbf{m} \in L^2(R; \mathbb{R}^2)$;   $\{\mathbf{R}_{\delta, \epsilon}\}$ satisfies $\mathbf{R}_{\delta, \epsilon}  \rightarrow \mathbf{R}_0$ for some $\mathbf{R}_0 \in SO(2)$ in cases \eqref{super-degenerate1} and \eqref{degenerate1} or $\mathbf{R}_{\delta,\epsilon} \rightarrow  \mathbf{R}(\varphi)$ in case \eqref{non-degenerate1} for  $\varphi$ in \eqref{eq:getVarphiDef1}; finally,
\begin{equation}
\begin{aligned}\label{eq:workConverge1}
\lim_{\delta, \epsilon \rightarrow 0} - \frac{1}{\epsilon} \Big( V_{\delta, \epsilon}(\mathbf{y}_{\delta, \epsilon}) - V_{\delta,\epsilon}^{\star}\Big)=    \begin{cases}
 \mathbf{f}_0 \cdot \mathbf{R}_0 \mathbf{U}|_{\rho = 0} - \mathbf{f}_0 \cdot \mathbf{R}_0 \mathbf{U}|_{\rho = -1} & \text{ if  \eqref{super-degenerate1}} \\
 \mathbf{f}_0 \cdot \mathbf{R}_0 \mathbf{U}|_{\rho = 0} - \mathbf{f}_0 \cdot \mathbf{R}_0 \mathbf{U}|_{\rho = -1} + V_0(1 - \mathbf{R}_0 \mathbf{e}_1 \cdot \mathbf{R}(\varphi) \mathbf{e}_1) & \text{ if  \eqref{degenerate1}}\\
 \mathbf{f}_0 \cdot \mathbf{R}(\varphi) \mathbf{U}|_{\rho = 0} - \mathbf{f}_0 \cdot \mathbf{R}(\varphi) \mathbf{U}|_{\rho = -1} + e_0 & \text{ if \eqref{non-degenerate1}}
\end{cases}
\end{aligned} 
\end{equation}
for some $e_0 \geq 0$. Since additionally  $\log \frac{1}{\delta \vee \epsilon} \int_{\Omega_{\delta \vee \epsilon}}| \nabla \mathbf{u}_{\delta, \epsilon}|^2 \, d \mathbf{x}  = \int_{R} |D^{\delta \vee \epsilon} \mathbf{U}_{\delta, \epsilon} |^2  \, d \rho d \theta \lesssim_p 1$, we can repeat the arguments from Step 1 of Proposition \ref{convergence_energy_NL_force}  to show  that $\indicator{\hat{B}_{\delta, \epsilon}} D^{\delta \vee \epsilon} \mathbf{U}_{\delta, \epsilon} \rightharpoonup \partial_{\rho} \mathbf{U} \otimes \mathbf{e}_{\rho} + \mathbf{m} \otimes \mathbf{e}_{\theta}$ in $L^2(R; \mathbb{R}^{2\times2})$ (see Step 1 of  Proposition \ref{convergence_energy_NL_disp} for more details) and thus
\begin{equation}
\begin{aligned}\label{eq:liminfforceQ}
&\liminf_{\epsilon, \delta \rightarrow 0}\, \int_{R}  Q(\sym\, \indicator{\hat{B}_{\delta, \epsilon}} D^{\delta \vee \epsilon} \mathbf{U}_{\delta, \epsilon})  \, d \rho d \theta \geq  \int_{R} Q( \partial_{\rho} \mathbf{U} \odot \mathbf{e}_{\rho} + \mathbf{m} \odot \mathbf{e}_{\theta}) \, d \rho d\theta.
\end{aligned}
\end{equation}
Since $\{ \mathbf{y}_{\delta, \epsilon}\}$ is almost minimizing, the same arguments, when combined with the convergence of the energies in Proposition \ref{convergence_energy_NL_force}, yield that
\begin{equation}
\begin{aligned}\label{eq:limsupforceQ}
&\limsup_{\delta, \epsilon \rightarrow 0}\,  \int_{R}  Q(\sym\, \indicator{\hat{B}_{\delta, \epsilon}} D^{\delta \vee \epsilon} \mathbf{U}_{\delta, \epsilon})  \, d \rho d \theta - \frac{1}{\epsilon} \Big( V_{\delta, \epsilon}(\mathbf{y}_{\delta, \epsilon}) - V_{\delta,\epsilon}^{\star}\Big) \\
&\qquad\leq \lim_{\delta, \epsilon \rightarrow 0}\, \frac{\log \frac{1}{\delta \vee \epsilon}}{ \epsilon^2}  E_{\delta}( \mathbf{y}_{\delta, \epsilon}) - \frac{1}{\epsilon} \Big( V_{\delta, \epsilon}(\mathbf{y}_{\delta, \epsilon}) - \mathbf{V}_{\delta,\epsilon}^{\star} \Big) \\
&\qquad  =  \begin{cases}
\displaystyle\min_{\mathbf{R} \in SO(2)} - \QFl^{\ast}( \mathbf{R}^T \mathbf{f}_0)  &  \text{ if  \eqref{super-degenerate1}}\\  
\displaystyle\min_{\mathbf{R} \in SO(2)} - \QFl^{\ast}( \mathbf{R}^T \mathbf{f}_0) + V_0(1- \mathbf{R} \mathbf{e}_1 \cdot \mathbf{R}(\varphi) \mathbf{e}_1) & \text{ if \eqref{degenerate1}}\\ 
-\QFl^\ast( \mathbf{R}(\varphi)^T \mathbf{f}_0 )  & \text{ if  \eqref{non-degenerate1}}  
\end{cases}.
\end{aligned}
\end{equation}
Comparing \eqref{eq:workConverge1},  \eqref{eq:liminfforceQ}, and \eqref{eq:limsupforceQ}, we get for case \eqref{super-degenerate1} that 
\begin{align*}
- \QFl^\ast( \mathbf{R}_0^T \mathbf{f}_0)  \leq \int_R Q( \partial_{\rho} \mathbf{U} \odot\mathbf{e}_{\rho} + \mathbf{m} \odot \mathbf{e}_{\theta}) \, d \rho d\theta +  \mathbf{R}_0 ^T\mathbf{f}_0 \cdot \mathbf{U}|_{\rho = 0} - \mathbf{R}_0 ^T\mathbf{f}_0 \cdot \mathbf{U}|_{\rho = -1}  \leq  \min_{\mathbf{R} \in SO(2)} - \QFl^{\ast} (\mathbf{R}^T\mathbf{f}_0)
\end{align*}
where Lemma \ref{minimzingLemForce} is used for the first inequality. It follows that $\mathbf{R}_0= \mathbf{R}_{\text{a}}$ for a minimizing rotation on the right-hand side. Hence, 
\begin{align*}
\int_R Q( \partial_{\rho} \mathbf{U} \odot \mathbf{e}_{\rho} + \mathbf{m} \odot \mathbf{e}_{\theta}) \, d \rho d\theta +  \mathbf{R}_\text{a} ^T\mathbf{f}_0 \cdot \mathbf{U}|_{\rho = 0} - \mathbf{R}_\text{a} ^T\mathbf{f}_0 \cdot \mathbf{U}|_{\rho = -1}  =   - \QFl^{\ast} (\mathbf{R}_{\text{a}}^T\mathbf{f}_0)
\end{align*}
and upon applying Lemma \ref{minimzingLemForce} again, we deduce that $\mathbf{U}(\rho) = \rho \mathbf{K}_{\text{Fl}}^{-1} \mathbf{R}_\text{a}^T \mathbf{f}_0 + \mathbf{c}_{\text{a}}$ and $\mathbf{m}(\theta) = \mathbf{M}(\theta)\mathbf{K}_{\text{Fl}}^{-1} \mathbf{R}_{\text{a}}^T \mathbf{f}_0$ for a constant $\mathbf{c}_\text{a} \in \mathbb{R}^2$. 

The same argument works in case \eqref{degenerate1}; we simply have an extra term. In particular, we get in this case that $\mathbf{R}_0 = \mathbf{R}_{\text{b}}$ for a minimizing rotation in \eqref{degenerate1}, 
from which it follows that $\mathbf{U}(\rho) = \rho \mathbf{K}_{\text{Fl}}^{-1} \mathbf{R}_{\text{b}}^T \mathbf{f}_0 + \mathbf{c}_\text{b}$ and $\mathbf{m}(\theta) = \mathbf{M}(\theta) \mathbf{K}_{\text{Fl}}^{-1} \mathbf{R}_{\text{b}}^{T}\mathbf{f}_0$ for some $\mathbf{c}_{\text{b}} \in \mathbb{R}^2$. Finally, in the last case \eqref{non-degenerate1}, we read off from \eqref{eq:workConverge1},  \eqref{eq:liminfforceQ}, and \eqref{eq:limsupforceQ} that
\begin{align*}
\int_R Q( \partial_{\rho} \mathbf{U} \odot \mathbf{e}_{\rho} + \mathbf{m} \odot \mathbf{e}_{\theta}) \, d \rho d\theta +  \mathbf{R}(\varphi)^T\mathbf{f}_0 \cdot  \mathbf{U}|_{\rho = 0} - \mathbf{R}(\varphi) ^T\mathbf{f}_0 \cdot \mathbf{U}|_{\rho = -1}   + e_0 \leq   - \QFl^{\ast} (\mathbf{R}(\varphi)^T\mathbf{f}_0).
\end{align*}
Using  Lemma \ref{minimzingLemForce} and that $e_0 \geq 0$, it follows that $\mathbf{U}(\rho) = \rho \mathbf{K}_{\text{Fl}}^{-1} \mathbf{R}(\varphi)^T \mathbf{f}_0 + \mathbf{c}_{\text{c}}$ and $\mathbf{m}(\theta) = \mathbf{M}(\theta) \mathbf{K}_{\text{Fl}}^{-1} \mathbf{R}(\varphi)^T \mathbf{f}_0$ for some $\mathbf{c}_{\text{c}} \in \mathbb{R}^2$; also, $e_0 = 0$.

In summary, the results just shown sandwich the liminf and limsup in \eqref{eq:liminfforceQ} and \eqref{eq:limsupforceQ} to produce a limit, after plugging in all the identities  into the convergent work term \eqref{eq:workConverge1}. We deduce that  
\begin{equation}
\begin{aligned}\label{eq:keyForForceProb}
&\indicator{\hat{B}_{\delta, \epsilon}} D^{\delta \vee \epsilon} \mathbf{U}_{\delta, \epsilon} \rightharpoonup \partial_{\rho} \mathbf{U}_0 \otimes \mathbf{e}_{\rho} + \mathbf{m}_0 \otimes \mathbf{e}_{\theta} \quad \text{ in $L^2(R; \mathbb{R}^{2\times2})$}, \\ 
&\int_{R}  Q( \sym\, \indicator{\hat{B}_{\delta, \epsilon}} D^{\delta \vee \epsilon} \mathbf{U}_{\delta, \epsilon})  \, d \rho d \theta \rightarrow \int_{R} Q( \partial_{\rho} \mathbf{U}_0 \odot \mathbf{e}_{\rho} + \mathbf{m}_0 \odot \mathbf{e}_{\theta}) \, d \rho d\theta
\end{aligned}
\end{equation}
in all cases, where $\mathbf{U}_0$ and $\mathbf{m}_0$ are defined as
\begin{align*}
\mathbf{U}_0(\rho) = 
\begin{cases}
\rho \mathbf{K}_{\text{Fl}}^{-1} \mathbf{R}_\text{a}^T \mathbf{f}_0  + \mathbf{c}_\text{a} &  \text{ if  \eqref{super-degenerate1}}\\  
\rho \mathbf{K}_{\text{Fl}}^{-1} \mathbf{R}_\text{b}^T \mathbf{f}_0  + \mathbf{c}_\text{b}  & \text{ if \eqref{degenerate1}}\\ 
\rho \mathbf{K}_{\text{Fl}}^{-1} \mathbf{R}(\varphi)^T \mathbf{f}_0  + \mathbf{c}_\text{c}   & \text{ if  \eqref{non-degenerate1}}    
\end{cases} \quad  \text{ and } \quad \mathbf{m}_0(\theta) = \begin{cases}
 \mathbf{M}(\theta) \mathbf{K}_{\text{Fl}}^{-1} \mathbf{R}_{\text{a}}^T \mathbf{f}_0 & \text{ if  \eqref{super-degenerate1}}\\
 \mathbf{M}(\theta) \mathbf{K}_{\text{Fl}}^{-1} \mathbf{R}_{\text{b}}^T \mathbf{f}_0 & \text{ if  \eqref{degenerate1}}\\
 \mathbf{M}(\theta) \mathbf{K}_{\text{Fl}}^{-1} \mathbf{R}(\varphi)^T \mathbf{f}_0 & \text{ if  \eqref{non-degenerate1}}
\end{cases}.
\end{align*}
From here, it is not hard to obtain the identities  
\begin{align*}
\lim_{\delta, \epsilon \rightarrow 0} \frac{\log \frac{1}{\delta \vee \epsilon}}{\epsilon^2}  \int_{\Omega_{\delta}} W(\nabla \mathbf{y}_{\delta, \epsilon}) \, d \mathbf{x}  = \lim_{\delta, \epsilon \rightarrow 0}   \frac{\log \frac{1}{\delta \vee \epsilon}}{\epsilon^2}  \int_{B_{\delta, \epsilon}} W(\nabla \mathbf{y}_{\delta, \epsilon}) \, d \mathbf{x} = \int_{R} Q( \partial_{\rho} \mathbf{U}_0 \odot \mathbf{e}_{\rho} + \mathbf{m}_0 \odot \mathbf{e}_{\theta}) \, d \rho d\theta.
\end{align*}
This proves the first convergence in \eqref{eq:toStartStrongConvergence} just as in the displacement problem.

\textit{Step 3: Convergence of the linear strain.} We now prove the second part of \eqref{eq:toStartStrongConvergence}. For this, we use the results in \eqref{eq:keyForDispProb} and \eqref{eq:keyForForceProb} to repeat the argument from Step 3 of Proposition \ref{almostMinimizersProp} to prove that
\begin{align*}
\indicator{\hat{B}_{\delta, \epsilon}} \sym \,D^{\delta \vee \epsilon} \mathbf{U}_{\delta, \epsilon} \rightarrow \partial_{\rho} \mathbf{U}_0 \odot \mathbf{e}_{\rho} + \mathbf{m}_0 \odot \mathbf{e}_{\theta}  \quad \text{ in } L^2(R; \mathbb{R}^{2\times2})
\end{align*}
in all four pairings of $(\mathbf{U}_0, \mathbf{m}_0)$. 
The desired result follows because 
\begin{align*}
\Big\| \indicator{\hat{B}_{\delta, \epsilon}} \sym D^{\delta \vee \epsilon} \mathbf{U}_{\delta, \epsilon} -\big( \partial_{\rho} \mathbf{U}_0 \odot \mathbf{e}_{\rho} + \mathbf{m}_0 \odot \mathbf{e}_{\theta}\big) \Big\|_{L^2(R)} = \sqrt{\log \frac{1}{\delta \vee \epsilon}} \Big\| \indicator{B_{\delta, \epsilon}} \mathbf{e}(\mathbf{u}_{\delta, \epsilon}) - \frac{1}{\log \frac{1}{\delta \vee \epsilon}} \mathbf{e}(\mathbf{u}_{\text{Fl}}) \Big\|_{L^2(\Omega_{\delta \vee \epsilon})}
\end{align*}
on changing variables back to Cartesian coordinates. 
\end{proof}

In the rest of this paper, we fix $\{ \mathbf{y}_{\delta, \epsilon} \}$ to be the subsequence of almost minimizers from Lemma \ref{lemmaStartStrongConvergence}. We define the limit quantities $\mathbf{R}_a$, $\mathbf{R_b}$, and $\mathbf{u}_{\text{Fl}}$ using this subsequence.

\subsection{Asymptotics of the polar decomposition} Recall the polar decomposition of the deformation gradient is
\begin{align}\label{eq:PolarDecomp}
    \nabla \mathbf{y}_{\delta, \epsilon} = \tilde{\mathbf{R}}_{\delta, \epsilon}  \sqrt{\nabla \mathbf{y}_{\delta, \epsilon}^T \nabla \mathbf{y}_{\delta, \epsilon}}
\end{align}
where $\tilde{\mathbf{R}}_{\delta, \epsilon} \colon \Omega_{\delta} \rightarrow O(2)$ is an orthogonal matrix field that is not necessarily unique. The term $(\nabla \mathbf{y}_{\delta, \epsilon}^T \nabla \mathbf{y}_{\delta, \epsilon})^{1/2}$ is the \emph{stretch tensor}. It is the unique positive semidefinite square root of $\nabla\mathbf{y}_{\delta, \epsilon}$. We control the stretch tensor first, and then the orthogonal matrix field.
\begin{lem}\label{lem:strainStrong}
Let $\{ \mathbf{y}_{\delta, \epsilon} \}$ be the subsequence from Lemma \ref{lemmaStartStrongConvergence}. Then,
\begin{align*}
\frac{\log \frac{1}{\delta \vee \epsilon} }{\epsilon^2} \int_{\Omega_{\delta \vee \epsilon}} \Big| \sqrt{ \nabla \mathbf{y}_{\delta, \epsilon}^T \nabla \mathbf{y}_{\delta, \epsilon}} -  \Big( \mathbf{I} + \frac{\epsilon}{\log \frac{1}{\delta \vee \epsilon}} \mathbf{e}( \mathbf{u}_{\emph{Fl}})\Big) \Big|^2 \, d\mathbf{x} \rightarrow 0.
\end{align*}
\end{lem}
\begin{proof}
Let $\mathbf{E}_{\delta, \epsilon} =  ( \nabla \mathbf{y}_{\delta, \epsilon}^T \nabla \mathbf{y}_{\delta, \epsilon})^{1/2} - \mathbf{I}$ and note that 
\begin{align*}
\frac{1}{\epsilon^2} \int_{\Omega_{\delta \vee \epsilon}} \Big| \mathbf{E}_{\delta, \epsilon} -  \frac{\epsilon}{\log \frac{1}{\delta \vee \epsilon}} \mathbf{e}( \mathbf{u}_{\text{Fl}})\Big) \Big|^2 \, d\mathbf{x} & \lesssim \int_{\Omega_{\delta \vee \epsilon}}  \big| \epsilon^{-1}\mathbf{E}_{\delta, \epsilon}   - \indicator{B_{\delta, \epsilon}}\mathbf{e}(\mathbf{u}_{\delta, \epsilon})\big|^2 +  \big| \indicator{B_{\delta, \epsilon}}\mathbf{e}(\mathbf{u}_{\delta, \epsilon}) - \frac{1}{\log \frac{1}{\delta \vee \epsilon}}  \mathbf{e}(\mathbf{u}_{\text{Fl}})  \big|^2 \, d \mathbf{x} 
\end{align*} 
for $\mathbf{u}_{\delta, \epsilon} $ in \eqref{eq:theDisplacementStrong} and $B_{\delta,\epsilon}$ in \eqref{good_sets}. By the second result in Lemma \ref{lemmaStartStrongConvergence}, 
\begin{align*}
\int_{\Omega_{\delta \vee \epsilon}} \big| \indicator{B_{\delta, \epsilon}}\mathbf{e}(\mathbf{u}_{\delta, \epsilon}) - \frac{1}{\log \frac{1}{\delta \vee \epsilon}}  \mathbf{e}(\mathbf{u}_{\text{Fl}})  \big|^2 \, d \mathbf{x}  = o\Big( \frac{1}{\log \frac{1}{\delta \vee \epsilon}}\Big).
\end{align*}
Next, observe that  
\begin{align*}
 \int_{\Omega_{\delta \vee \epsilon}}  \big| \epsilon^{-1}\mathbf{E}_{\delta, \epsilon}   - \indicator{B_{\delta, \epsilon}}\mathbf{e}(\mathbf{u}_{\delta, \epsilon})\big|^2\,d \mathbf{x} =  \int_{\Omega_{\delta \vee \epsilon}\setminus B_{\delta, \epsilon}}  \big| \epsilon^{-1}\mathbf{E}_{\delta, \epsilon}  \big|^2\,d \mathbf{x}    + \int_{B_{\delta,\epsilon}}  \big| \epsilon^{-1}\mathbf{E}_{\delta, \epsilon}   - \mathbf{e}(\mathbf{u}_{\delta, \epsilon})\big|^2\,d \mathbf{x} .
\end{align*}
As $W(\mathbf{F}) \gtrsim d^2(\mathbf{F},SO(2)) \geq |\sqrt{\mathbf{F}^T\mathbf{F}} - \mathbf{I}|^2$  for all $\mathbf{F}\in \R^{2\times 2}$, the first limit in Lemma \ref{lemmaStartStrongConvergence} furnishes
\begin{align*}
 \int_{\Omega_{\delta \vee \epsilon}\setminus B_{\delta, \epsilon}} \big|   \epsilon^{-1}\mathbf{E}_{\delta, \epsilon} |^2 \,d\mathbf{x}  \lesssim \frac{1}{\epsilon^2}  \int_{\Omega_{\delta \vee \epsilon}\setminus B_{\delta, \epsilon}} W (\nabla \mathbf{y}_{\delta, \epsilon}) \, d\mathbf{x}  = o\Big( \frac{1}{\log \frac{1}{\delta \vee \epsilon}}\Big).
\end{align*}
To complete the proof, note that  $\epsilon |\nabla \mathbf{u}_{\delta, \epsilon}|  \leq (\log \frac{1}{\delta \vee \epsilon})^{-1/2} \ll 1$ on $B_{\delta, \epsilon}$ and $\sqrt{(\mathbf{I}+ \mathbf{F})^T (\mathbf{I} + \mathbf{F})} = \mathbf{I} +  \mathrm{sym}\, \mathbf{F} + o(|\mathbf{F}|)$ by Taylor expansion. Thus,  
\begin{align*}
\mathbf{E}_{\delta, \epsilon} = \sqrt{( \mathbf{I} + \epsilon \nabla \mathbf{u}_{\delta, \epsilon})^T( \mathbf{I} + \epsilon \nabla \mathbf{u}_{\delta, \epsilon}) } - \mathbf{I}   = \epsilon \mathbf{e}( \mathbf{u}_{\delta, \epsilon}) + o( \epsilon | \nabla \mathbf{u}_{\delta, \epsilon}|)
\end{align*}
on $B_{\delta, \epsilon}$. It follows that 
\begin{align*}
\int_{B_{\delta, \epsilon}}  \big| \epsilon^{-1}\mathbf{E}_{\delta, \epsilon}   - \mathbf{e}(\mathbf{u}_{\delta, \epsilon})\big|^2 \, d\mathbf{x}  = o\Big( \int_{B_{\delta, \epsilon}} | \nabla \mathbf{u}_{\delta, \epsilon}|^2 \, d\mathbf{x} \Big).
\end{align*}
Recall we proved in the compactness arguments of Propositions \ref{CompactnessDispThm} and \ref{propCompactnessForce} that the displacements in either problem satisfy $\int_{\Omega_{\delta \vee \epsilon}} |\nabla  \mathbf{u}_{\delta, \epsilon}|^2 \, d\mathbf{x} \lesssim (\log \frac{1}{\delta \vee \epsilon})^{-1}$. Combining these estimates gives the result.
\end{proof}

 Next, we control the orthogonal matrix field $\tilde{\mathbf{R}}_{\delta, \epsilon}$. We show a version of it is close to the identity $\mathbf{I}$ in the displacement problem, or to $\mathbf{R}_{\delta,\epsilon}\in \cR_{\delta}(\mathbf{y}_{\delta, \epsilon})$ from \eqref{optimal_rotations2} in the force problem.   
\begin{lem}\label{lem:rigidity-polar-factor}
Let  $\{ \mathbf{y}_{\delta, \epsilon} \}$ be the subsequence from  Lemma \ref{lemmaStartStrongConvergence}. For any Lebesgue measurable choice of $\tilde{\mathbf{R}}_{\delta, \epsilon} \colon \Omega_{\delta} \rightarrow O(2)$ in  \eqref{eq:PolarDecomp}, it holds that
\begin{align*}
\int_{\Omega_{\delta}}|\tilde{\mathbf{R}}_{\delta,\epsilon}-\mathbf{I}|^{2} \,d \mathbf{x} \lesssim \frac{\epsilon^2}{\log \frac{1}{\delta \vee \epsilon}}   \quad \text{ or } \quad  \int_{\Omega_{\delta}}|\tilde{\mathbf{R}}_{\delta,\epsilon}-\mathbf{R}_{\delta,\epsilon}|^{2} \,d \mathbf{x} \lesssim_p  \frac{\epsilon^2}{\log \frac{1}{\delta \vee \epsilon}} 
\end{align*}
in the displacement or force problems, respectively.
\end{lem}
\begin{rem} Measurability of $\tilde{\mathbf{R}}_{\delta, \epsilon}$ is not guaranteed by  \eqref{eq:PolarDecomp}, but can be arranged. Where $\det\nabla\mathbf{y}_{\delta,\epsilon}\neq 0$, $\tilde{\mathbf{R}}_{\delta, \epsilon}=\nabla \mathbf{y}_{\delta,\epsilon}(\nabla \mathbf{y}_{\delta, \epsilon}^T \nabla \mathbf{y}_{\delta, \epsilon})^{-1/2}$ which is measurable. The complement splits into the part where $\nabla\mathbf{y}_{\delta,\epsilon}\neq\mathbf{0}$ and the rest. On the former, $(\nabla\mathbf{y}^{T}_{\delta,\epsilon}\nabla\mathbf{y}_{\delta,\epsilon})^{1/2}=\lambda_{\delta,\epsilon}\mathbf{v}_{\delta,\epsilon}\otimes\mathbf{v}_{\delta,\epsilon}$ for measurable unit vector and scalar fields $\mathbf{v}_{\delta,\epsilon}$ and  $\lambda_{\delta,\epsilon}$, and we can take $\tilde{\mathbf{R}}_{\delta,\epsilon}$ to be the unique $SO(2)$-valued field satisfying $\tilde{\mathbf{R}}_{\delta,\epsilon}\mathbf{v}_{\delta,\epsilon}=\lambda^{-1}_{\delta,\epsilon}\nabla\mathbf{y}_{\delta,\epsilon}\mathbf{v}_{\delta,\epsilon}$. Where $\nabla\mathbf{y}_{\delta,\epsilon}=\mathbf{0}$, we can take $\tilde{\mathbf{R}}_{\delta,\epsilon}=\mathbf{I}$. 
\end{rem}
\begin{proof}
First, for the displacement problem, write that
\begin{align*}
\int_{\Omega_{\delta}}|\tilde{\mathbf{R}}_{\delta,\epsilon}- \mathbf{I}|^{2} \,d \mathbf{x} \lesssim\int_{\Omega_{\delta}}\bigg|\tilde{\mathbf{R}}_{\delta,\epsilon}\Big(\sqrt{\nabla\mathbf{y}_{\delta,\epsilon}^{T}\nabla\mathbf{y}_{\delta,\epsilon}}-\mathbf{I}\Big)\bigg|^{2}+|\nabla\mathbf{y}_{\delta,\epsilon}- \mathbf{I} |^{2} \,d \mathbf{x}.
\end{align*}
From the inequality $|\sqrt{\mathbf{F}^T \mathbf{F}} - \mathbf{I}|^2 \leq d^{2}(\mathbf{F}, SO(2)) \lesssim W(\mathbf{F})$ and the fact that we are dealing with almost minimizers, there follows 
\begin{align*}
\int_{\Omega_{\delta}}\Big| \sqrt{\nabla\mathbf{y}_{\delta,\epsilon}^{T}\nabla\mathbf{y}_{\delta,\epsilon}}-\mathbf{I}\Big|^{2} \, d\mathbf{x}  \lesssim E_{\delta}(\mathbf{y}_{\delta, \epsilon}) \lesssim \frac{\epsilon^2}{\log \frac{1}{\delta \vee \epsilon}}.
\end{align*} 
By the $L^2$ version of the second estimate in  Corollary \ref{cor:mixed-rigidity-energy} and the $r = 1$ boundary condition of $\mathbf{y}_{\delta, \epsilon} \in \mathcal{A}_{p, \delta, \epsilon}$,
\begin{align*}
\int_{\Omega_{\delta}} | \nabla\mathbf{y}_{\delta,\epsilon}-\mathbf{I} |^{2} \,d \mathbf{x} &\lesssim E_{\delta}(\mathbf{y}_{\delta, \epsilon}) + \epsilon^2 \| \mathbf{u}_{\delta, \epsilon}^{+}\|^2_{\dot{W}^{1,\infty}((\alpha, \beta))} \lesssim \frac{\epsilon^2 }{\log \frac{1}{\delta \vee \epsilon}}.
\end{align*}
The last inequality follows from  assumption \eqref{eq:restateAssumpNLDispCompact} on the boundary conditions. 

Turning to the force problem, we observe that
\begin{align*}
 \int_{\Omega_{\delta}}|\tilde{\mathbf{R}}_{\delta,\epsilon}-\mathbf{R}_{\delta,\epsilon}|^{2} \,d \mathbf{x} \lesssim \int_{\Omega_{\delta}}\Big|\sqrt{\nabla\mathbf{y}_{\delta,\epsilon}^{T}\nabla\mathbf{y}_{\delta,\epsilon}}-\mathbf{I}\Big|^{2}+|\nabla\mathbf{y}_{\delta,\epsilon}-\mathbf{R}_{\delta,\epsilon}|^{2} \,d \mathbf{x} \lesssim E_{\delta}(\mathbf{y}_{\delta, \epsilon}) 
\end{align*}
by a similar argument as above. Since we are dealing with almost minimizers, we can reproduce the estimate in \eqref{eq:firstEstCompactnessForce} to obtain 
\begin{align*}
E_{\delta}(\mathbf{y}_{\delta, \epsilon}) \lesssim_p  E_{\delta}(\mathbf{y}_{\delta, \epsilon}) - \frac{\epsilon}{\log \frac{1}{\delta \vee \epsilon}} \big( V_{\delta,\epsilon}(\mathbf{y}_{\delta, \epsilon}) - V_{\delta, \epsilon}^{\star} \big) + \frac{\epsilon^2}{\log \frac{1}{\delta \vee \epsilon}} \lesssim \frac{\epsilon^2}{\log \frac{1}{\delta \vee \epsilon}}.
\end{align*}
This completes the proof. 
\end{proof}

\subsection{Asymptotics of the deformation gradient} We are finally ready to prove \eqref{eq:resultFinalSec6} and \eqref{eq:resultFinal1Sec6}.

\begin{prop}\label{finalProp}
Let  $\{ \mathbf{y}_{\delta, \epsilon} \}$ be the subsequence from Lemma \ref{lemmaStartStrongConvergence}. Then
\begin{equation}
\begin{aligned}\label{eq:finalfinalLimits}
&\frac{\log \frac{1}{\delta \vee \epsilon} }{\epsilon^2} \Big( \int_{\Omega_{\delta \vee \epsilon}} \Big| \nabla \mathbf{y}_{\delta, \epsilon} -  \Big( \mathbf{I} + \frac{\epsilon}{\log \frac{1}{\delta \vee \epsilon}} \nabla \mathbf{u}_{\emph{Fl}}  \Big)    \Big|^2 \, d\mathbf{x} + \int_{\Omega_{\delta, \delta \vee \epsilon}} \Big| \nabla \mathbf{y}_{\delta,\epsilon} - \mathbf{I} \Big|^2  \, d\mathbf{x}  \Big) \rightarrow 0, \\ &\frac{\log \frac{1}{\delta \vee \epsilon} }{\epsilon^2} \Big( \int_{\Omega_{\delta \vee \epsilon}} \Big| \nabla \mathbf{y}_{\delta, \epsilon} -  \mathbf{R}_{\delta, \epsilon} \Big( \mathbf{I} + \frac{\epsilon}{\log \frac{1}{\delta \vee \epsilon}} \nabla \mathbf{u}_{\emph{Fl}}  \Big)    \Big|^2 \, d\mathbf{x}  + \int_{\Omega_{\delta, \delta \vee \epsilon}} \Big| \nabla \mathbf{y}_{\delta,\epsilon} - \mathbf{R}_{\delta,\epsilon} \Big|^2  \, d\mathbf{x}  \Big) \rightarrow 0 
\end{aligned}
\end{equation}
in the displacement and force problems, respectively.
\end{prop}

\begin{proof}\textit{Step 1: A preliminary objective.}
Let $\tilde{\mathbf{R}}_{\delta, \epsilon} \colon \Omega_{\delta} \rightarrow O(2)$ be a Lebesgue measurable orthogonal matrix field in the polar decomposition \eqref{eq:PolarDecomp} of $\nabla\mathbf{y}_{\delta,\epsilon}$, and let $(\Omega_{\delta} \setminus B_{\delta, \epsilon})_{> 0} = \{ \mathbf{x} \in \Omega_{\delta} \setminus B_{\delta, \epsilon} \colon \det \nabla \mathbf{y}_{\delta, \epsilon} (\mathbf{x}) > 0\}$ and $(\Omega_{\delta} \setminus B_{\delta, \epsilon})_{\leq 0}  = \{ \mathbf{x} \in \Omega_{\delta} \setminus B_{\delta, \epsilon} \colon \det \nabla \mathbf{y}_{\delta, \epsilon}(\mathbf{x})\leq 0\}$ for $B_{\delta, \epsilon}$ in \eqref{good_sets}. 
Define, for $\delta, \epsilon$ sufficiently small, the rotation field $\mathbf{R}_{\delta, \epsilon}^{\text{field}} \colon \Omega_{\delta} \rightarrow SO(2)$ by 
\begin{align}\label{eq:weirdRotDef}
\mathbf{R}_{\delta, \epsilon}^{\text{field}}= \begin{cases} 
\tilde{\mathbf{R}}_{\delta, \epsilon} \exp \Big(- \frac{\epsilon}{\log \frac{1}{\delta \vee \epsilon}} \skw \nabla \mathbf{u}_{\text{Fl}} \Big) & \text{ on } B_{\delta, \epsilon}  \\
\tilde{\mathbf{R}}_{\delta, \epsilon}& \text{ on } (\Omega_{\delta} \setminus B_{\delta, \epsilon})_{>  0}  \\ 
\mathbf{R}^{\text{const}}_{\delta, \epsilon} & \text{ on } (\Omega_{\delta} \setminus B_{\delta, \epsilon})_{\leq 0} 
\end{cases}
\end{align}
where for notational convenience we let
$\mathbf{R}_{\delta, \epsilon}^{\text{const}} = \mathbf{I}$ in the displacement problem and $\mathbf{R}_{\delta, \epsilon}^{\text{const}} = \mathbf{R}_{\delta, \epsilon}$ in the force problem. 
Recall that the displacements $\mathbf{u}_{\delta, \epsilon}$ from \eqref{eq:theDisplacementStrong} satisfy $|\epsilon \nabla \mathbf{u}_{\delta, \epsilon}| \leq (\log \frac{1}{\delta \vee \epsilon})^{-\frac{1}{2}} \ll 1$ on $B_{\delta, \epsilon}$. So, $\det \nabla \mathbf{y}_{\delta, \epsilon} > 0$ on $B_{\delta, \epsilon}$ for small enough $\delta,\epsilon$, which ensures that \eqref{eq:weirdRotDef} is $SO(2)$-valued. Finally, to succinctly include the case  $\delta < \epsilon$, we truncate the Flamant solution via 
\begin{align}\label{eq:Flamant-regularized}
\mathbf{u}_{\text{Fl}, \epsilon}(\mathbf{x}) = \begin{cases}
    \mathbf{u}_0(r) + \mathbf{u}_1(\theta) & \text{ on } \Omega_{\delta \vee \epsilon} \\
    \mathbf{u}_0(\epsilon) + \frac{r}{\epsilon}\mathbf{u}_1(\theta) & \text{ on } \Omega_{\delta,\delta \vee  \epsilon}\end{cases}.
\end{align}
Note $\mathbf{u}_{\text{Fl},\epsilon} \in W^{1,\infty}(\Omega_{\delta}; \mathbb{R}^2)$ and  $\mathbf{u}_{\text{Fl},\epsilon}=\mathbf{u}_{\text{Fl}}$ on $\Omega_{\delta \vee \epsilon}$ (see \eqref{eq:uFl1Sec6}).
With these definitions in hand, our first objective is to show that 
\begin{align}\label{eq:ultimateClaim}
\int_{\Omega_{\delta}} \big| \nabla \mathbf{y}_{\delta, \epsilon} - \mathbf{R}_{\delta, \epsilon}^{\text{field}}  - \frac{\epsilon}{\log \frac{1}{\delta \vee \epsilon}} \mathbf{R}^{\text{const}}_{\delta, \epsilon} \nabla \mathbf{u}_{\text{Fl},\epsilon} \big|^2 \, d \mathbf{x} = o\Big( \frac{\epsilon^2}{\log \frac{1}{\delta \vee \epsilon}}\Big). 
\end{align}

\textit{Step 2: Estimates on $B_{\delta,\epsilon}$.} We begin by showing that 
\begin{align}\label{eq:firstClaimStrong}
\int_{B_{\delta, \epsilon}}  \big| \nabla \mathbf{y}_{\delta, \epsilon} - \mathbf{R}_{\delta, \epsilon}^{\text{field}}  - \frac{\epsilon}{\log \frac{1}{\delta \vee \epsilon}}  \mathbf{R}^{\text{const}}_{\delta, \epsilon} \nabla \mathbf{u}_{\text{Fl},\epsilon} \big|^2 \,d \mathbf{x} = o\Big( \frac{\epsilon^2}{\log \frac{1}{\delta \vee \epsilon}}\Big).
\end{align}
Since $\frac{\epsilon}{\log \frac{1}{\delta \vee \epsilon}} \|  \nabla \mathbf{u}_{\text{Fl}}\|_{L^{\infty}(\Omega_{\delta \vee \epsilon})}  \lesssim \frac{1}{\log \frac{1}{\delta \vee \epsilon}} \ll 1$, we can Taylor expand the matrix exponential to obtain
\begin{align*}
\mathbf{R}_{\delta, \epsilon}^{\text{field}} = \tilde{\mathbf{R}}_{\delta, \epsilon}  \Big( \mathbf{I} - \frac{\epsilon}{\log \frac{1}{\delta \vee \epsilon}} \skw \nabla \mathbf{u}_{\text{Fl}}  + o\Big( \frac{\epsilon}{\log \frac{1}{\delta \vee \epsilon}} |\nabla \mathbf{u}_{\text{Fl}}|\Big)\Big)
\end{align*}
on $B_{\delta,\epsilon}$. Using this and the polar decomposition \eqref{eq:PolarDecomp}, there follows
\begin{align*}
&\nabla \mathbf{y}_{\delta, \epsilon} - \mathbf{R}_{\delta, \epsilon}^{\text{field}}  - \frac{\epsilon}{\log \frac{1}{\delta \vee \epsilon}}  \mathbf{R}_{\delta, \epsilon}^{\text{const}} \nabla \mathbf{u}_{\text{Fl},\epsilon} \\
&\qquad = \tilde{\mathbf{R}}_{\delta, \epsilon} \sqrt{\nabla \mathbf{y}_{\delta, \epsilon}^T \nabla \mathbf{y}_{\delta, \epsilon}} - \tilde{\mathbf{R}}_{\delta,\epsilon} \Big( \mathbf{I} + \frac{\epsilon}{\log \frac{1}{\delta \vee \epsilon}} \mathbf{e}(\mathbf{u}_{\text{Fl}})  + o\Big( \frac{\epsilon}{\log \frac{1}{\delta \vee \epsilon}} | \nabla \mathbf{u}_{\text{Fl}}|  \Big)  \Big)  +  \frac{\epsilon}{\log \frac{1}{\delta \vee \epsilon}} \big(  \tilde{\mathbf{R}}_{\delta, \epsilon}  - \mathbf{R}^{\text{const}}_{\delta, \epsilon}\big) \nabla \mathbf{u}_{\text{Fl}} 
\end{align*}
on $B_{\delta, \epsilon}$, after adding and subtracting $\frac{\epsilon}{\log \frac{1}{\delta \vee \epsilon}} \tilde{\mathbf{R}}_{\delta, \epsilon} \nabla \mathbf{u}_{\text{Fl}}$ and using  that $\mathbf{u}_{\text{Fl},\epsilon} = \mathbf{u}_{\text{Fl}}$ on this set. Hence, 
\begin{align*}
 &\int_{B_{\delta, \epsilon}}  \big| \nabla \mathbf{y}_{\delta, \epsilon} - \mathbf{R}_{\delta, \epsilon}^{\text{field}}  - \frac{\epsilon}{\log \frac{1}{\delta \vee \epsilon}}  \mathbf{R}^{\text{const}}_{\delta, \epsilon} \nabla \mathbf{u}_{\text{Fl},\epsilon} \big|^2 \,d \mathbf{x} \\
 &\qquad \lesssim  \int_{B_{\delta, \epsilon}}   \Big| \sqrt{ \nabla \mathbf{y}_{\delta, \epsilon}^T \nabla \mathbf{y}_{\delta, \epsilon}} -  \Big( \mathbf{I} + \frac{\epsilon}{\log \frac{1}{\delta \vee \epsilon}} \mathbf{e}( \mathbf{u}_{\text{Fl}})\Big) \Big|^2 +  \Big|   \frac{\epsilon}{\log \frac{1}{\delta \vee \epsilon}} \big(  \tilde{\mathbf{R}}_{\delta, \epsilon}  - \mathbf{R}^{\text{const}}_{\delta, \epsilon}\big) \nabla \mathbf{u}_{\text{Fl}} \Big|^2 \, d\mathbf{x}\\
 &\qquad \qquad + o\Big( \int_{B_{\delta , \epsilon}} \frac{\epsilon^2}{(\log \frac{1}{\delta \vee \epsilon})^2} | \nabla \mathbf{u}_{\text{Fl}}|^2  \, d\mathbf{x} \Big). 
\end{align*}
The first term on the right-hand side is $\ll  \epsilon^2(\log \frac{1}{\delta \vee \epsilon})^{-1} $ by Lemma \ref{lem:strainStrong}. To handle the second term, use $\|\nabla \mathbf{u}_{\text{Fl}}\|_{L^{\infty}(\Omega_{\delta \vee \epsilon})} \lesssim \frac{1}{\delta \vee \epsilon}$ to obtain that
\begin{align*}
\int_{B_{\delta, \epsilon}}  \Big|   \frac{\epsilon}{\log \frac{1}{\delta \vee \epsilon}} \big(  \tilde{\mathbf{R}}_{\delta, \epsilon}  - \mathbf{R}^{\text{const}}_{\delta, \epsilon}\big) \nabla \mathbf{u}_{\text{Fl}} \Big|^2 \, d\mathbf{x} \lesssim \frac{1}{(\log \frac{1}{\delta \vee \epsilon})^2}  \int_{\Omega_{\delta}} |\tilde{\mathbf{R}}_{\delta, \epsilon} - \mathbf{R}^{\text{const}}_{\delta, \epsilon}|^2 \, d\mathbf{x} \lesssim_p  \frac{\epsilon^2 }{(\log \frac{1}{\delta \vee \epsilon})^3}
\end{align*}
by Lemma \ref{lem:rigidity-polar-factor}.  For the last term, simply note that $|\nabla\mathbf{u}_{\text{Fl}}|\lesssim 1/r$ by its definition in \eqref{eq:uFl1Sec6}, so that
\begin{equation}\label{eq:uFlCalcEst}
\int_{\Omega_{\delta \vee \epsilon}} \frac{\epsilon^2}{(\log \frac{1}{\delta \vee \epsilon})^2} |  \nabla \mathbf{u}_{\text{Fl}}|^2  \, d\mathbf{x} \lesssim \int_{\delta\vee\epsilon}^1 \frac{\epsilon^2}{(\log \frac{1}{\delta \vee \epsilon})^2} \frac{1}{r}\,dr\lesssim \frac{\epsilon^2}{\log \frac{1}{\delta \vee \epsilon}}. 
\end{equation}
The limit in \eqref{eq:firstClaimStrong} follows since $B_{\delta, \epsilon} \subset \Omega_{\delta \vee \epsilon}$. 

\textit{Step 3: Estimates on $(\Omega_{\delta}\setminus B_{\delta, \epsilon})_{\leq 0}$.}  Continuing the proof of \eqref{eq:ultimateClaim}, we now consider the set $(\Omega_{\delta} \setminus B_{\delta, \epsilon})_{\leq 0}  = (\Omega_{\delta}\setminus B_{\delta, \epsilon})\cap \{ \det \nabla \mathbf{y}_{\delta, \epsilon}\leq 0\}$ and show that 
 \begin{align}\label{eq:secClaimStrong}
 \int_{(\Omega_{\delta}\setminus B_{\delta, \epsilon})_{\leq 0}}  \big| \nabla \mathbf{y}_{\delta, \epsilon} - \mathbf{R}_{\delta, \epsilon}^{\text{field}}  - \frac{\epsilon}{\log \frac{1}{\delta \vee \epsilon}}  \mathbf{R}^{\text{const}}_{\delta, \epsilon} \nabla \mathbf{u}_{\text{Fl},\epsilon} \big|^2 \,d \mathbf{x} = o\Big( \frac{\epsilon^2}{\log \frac{1}{\delta \vee \epsilon}}\Big).
 \end{align}
The key points are that, because $\{ \mathbf{y}_{\delta, \epsilon}\}$ is almost minimizing, the elastic energy in the ``large displacement set'' $\Omega_{\delta} \setminus B_{\delta, \epsilon}$ is negligible by Lemma \ref{lemmaStartStrongConvergence}, and therefore  $(\Omega_{\delta}\setminus B_{\delta, \epsilon})_{\leq 0}$ is exceedingly small.   
To make use of these observations, write 
 \begin{equation}
\begin{aligned}\label{eq:TriangleOnLeq0}
 &\int_{(\Omega_{\delta} \setminus B_{\delta, \epsilon})_{\leq 0}}   \big| \nabla \mathbf{y}_{\delta, \epsilon} - \mathbf{R}^{\text{field}}_{\delta, \epsilon}  - \frac{\epsilon}{\log \frac{1}{\delta \vee \epsilon}} \mathbf{R}^{\text{const}}_{\delta, \epsilon} \nabla \mathbf{u}_{\text{Fl},\epsilon} \big|^2 \, d \mathbf{x} \\
 &\qquad \lesssim \int_{(\Omega_{\delta} \setminus B_{\delta, \epsilon})_{\leq 0}}\big| \sqrt{ \nabla \mathbf{y}_{\delta, \epsilon}^T \nabla \mathbf{y}_{\delta, \epsilon}} - \mathbf{I} |^2  + |\tilde{\mathbf{R}}_{\delta, \epsilon} - \mathbf{R}^{\text{field}}_{\delta, \epsilon}|^2 + \big|  \frac{\epsilon}{\log \frac{1}{\delta \vee \epsilon}} \nabla \mathbf{u}_{\text{Fl},\epsilon} \big|^2 \, d \mathbf{x}.
\end{aligned}
\end{equation}
Since $ W(\mathbf{F})\gtrsim d^2(\mathbf{F},SO(2))\geq |\sqrt{\mathbf{F}^T \mathbf{F}} - \mathbf{I}|^2 $ for all $\mathbf{F} \in \mathbb{R}^{2\times2}$, the first term in this estimate satisfies
\begin{equation}\label{eq:first-bound}
\int_{(\Omega_{\delta} \setminus B_{\delta, \epsilon})_{\leq 0}}\big| \sqrt{ \nabla \mathbf{y}_{\delta, \epsilon}^T \nabla \mathbf{y}_{\delta, \epsilon}} - \mathbf{I} |^2 \,d\mathbf{x}\lesssim \int_{\Omega_{\delta }  \setminus B_{\delta, \epsilon}} W(\nabla \mathbf{y}_{\delta, \epsilon})\, d \mathbf{x}=o \Big( \frac{\epsilon^2}{\log \frac{1}{\delta \vee \epsilon}} \Big) 
\end{equation}
by the first limit in \eqref{eq:toStartStrongConvergence}. Also, 
\begin{equation*}
|(\Omega_{\delta} \setminus B_{\delta, \epsilon})_{\leq 0} |  \lesssim \int_{(\Omega_{\delta \vee \epsilon} \setminus B_{\delta, \epsilon})_{\leq 0}} W(\nabla \mathbf{y}_{\delta, \epsilon}) \, d\mathbf{x} = o \Big( \frac{\epsilon^2}{\log \frac{1}{\delta \vee \epsilon}} \Big) 
\end{equation*}
since $W(\mathbf{F}) \gtrsim d^{2}(\mathbf{F}, SO(2)) \gtrsim 1$ if $\det \mathbf{F} \leq 0$. Consequently, the last two terms in \eqref{eq:TriangleOnLeq0} satisfy 
\begin{align*}
 &\int_{(\Omega_{\delta} \setminus B_{\delta, \epsilon})_{\leq 0}} |\tilde{\mathbf{R}}_{\delta, \epsilon} - \mathbf{R}^{\text{field}}_{\delta,\epsilon}|^2 + \big|  \frac{\epsilon}{\log \frac{1}{\delta \vee \epsilon}} \nabla \mathbf{u}_{\text{Fl},\epsilon} \big|^2  \, d\mathbf{x} \\
 &\qquad \quad \lesssim |(\Omega_{\delta} \setminus B_{\delta, \epsilon})_{\leq 0} |\Big( 1+ \frac{\epsilon^2}{(\delta \vee \epsilon )^2\big(\log \frac{1}{\delta  \vee \epsilon}\big)^2} \Big) = o\Big ( \frac{\epsilon^2}{\log \frac{1}{\delta \vee \epsilon}}\Big)
\end{align*}
since $| \tilde{\mathbf{R}}_{\delta,\epsilon} - \mathbf{R}_{\delta, \epsilon}^{\text{field}}| \lesssim 1$ and $\| \nabla \mathbf{u}_{\text{Fl},\epsilon} \|_{L^{\infty}(\Omega_{\delta})} \lesssim 1/(\delta \vee \epsilon)$ by \eqref{eq:Flamant-regularized}. The limit in \eqref{eq:secClaimStrong} follows.

\textit{Step 4: Estimates on $(\Omega_{\delta}\setminus B_{\delta, \epsilon})_{> 0}$.} Finally, we show that  
\begin{align}\label{eq:thirdClaimStrong}
 \int_{(\Omega_{\delta}\setminus B_{\delta, \epsilon})_{> 0}}  \big| \nabla \mathbf{y}_{\delta, \epsilon} - \mathbf{R}_{\delta, \epsilon}^{\text{field}}  - \frac{\epsilon}{\log \frac{1}{\delta \vee \epsilon}}  \mathbf{R}^{\text{const}}_{\delta, \epsilon} \nabla \mathbf{u}_{\text{Fl},\epsilon} \big|^2 \,d \mathbf{x} = o\Big( \frac{\epsilon^2}{\log \frac{1}{\delta \vee \epsilon}}\Big)
 \end{align}
for $(\Omega_{\delta} \setminus B_{\delta, \epsilon})_{> 0}  = (\Omega_{\delta}\setminus B_{\delta, \epsilon})\cap \{ \det \nabla \mathbf{y}_{\delta, \epsilon}> 0\}$. Since  $\mathbf{R}_{\delta, \epsilon}^{\text{field}} = \tilde{\mathbf{R}}_{\delta, \epsilon}$ on this set,
 \begin{align*}
& \int_{(\Omega_{\delta} \setminus B_{\delta, \epsilon})_{> 0}} \big| \nabla \mathbf{y}_{\delta, \epsilon} - \mathbf{R}^{\text{field}}_{\delta, \epsilon}  - \frac{\epsilon}{\log \frac{1}{\delta \vee \epsilon}} \mathbf{R}^{\text{const}}_{\delta, \epsilon} \nabla \mathbf{u}_{\text{Fl},\epsilon} \big|^2 \, d \mathbf{x}  \\
 &\qquad \lesssim  \int_{(\Omega_{\delta} \setminus B_{\delta, \epsilon})_{> 0}} \big| \sqrt{\nabla \mathbf{y}_{\delta, \epsilon}^T \nabla \mathbf{y}_{\delta,\epsilon}} - \mathbf{I} |^2 +  \big| \frac{\epsilon}{\log \frac{1}{\delta \vee \epsilon}} \nabla \mathbf{u}_{\text{Fl},\epsilon} \big|^2 \, d \mathbf{x}.
 \end{align*}
  The first term on the right is $\ll \epsilon^2(\log \frac{1}{\delta \vee \epsilon})^{-1}$ just as in \eqref{eq:first-bound}. The second term, meanwhile, satisfies 
 \begin{align*}
\int_{(\Omega_{\delta } \setminus B_{\delta, \epsilon})_{>0}} \big|  \frac{\epsilon}{\log \frac{1}{\delta \vee \epsilon}} \nabla \mathbf{u}_{\text{Fl},\epsilon} \big|^2 \,d\mathbf{x} \lesssim  \frac{\epsilon^2(|\Omega_{\delta \vee \epsilon} \setminus B_{\delta, \epsilon}| + |\Omega_{\delta, \delta \vee \epsilon}|)}{(\delta \vee \epsilon)^2 (\log \frac{1}{\delta \vee \epsilon} )^2 }   \lesssim_p \frac{\epsilon^2}{(\log \frac{1}{\delta \vee \epsilon})^2} 
\end{align*}
since $\| \nabla \mathbf{u}_{\text{Fl},\epsilon}\|_{L^{\infty}(\Omega_{\delta})} \lesssim 1/(\delta \vee \epsilon)$ and because $|\Omega_{\delta\vee\epsilon}\setminus B_{\delta,\epsilon}|\lesssim_p\epsilon^2$ by \eqref{eq:measure-estimate} and the analogous result for the force problem.
The estimate  \eqref{eq:thirdClaimStrong} follows.
 
 \textit{Step 5: Geometric rigidity.}  The results in \eqref{eq:firstClaimStrong}, \eqref{eq:secClaimStrong}, and \eqref{eq:thirdClaimStrong} justify  \eqref{eq:ultimateClaim}. 
Since $\mathbf{R}_{\delta, \epsilon}^\text{field}$ is $SO(2)$-valued, it follows from \eqref{eq:ultimateClaim} that 
\begin{align}\label{eq:dSO2Est}
\int_{\Omega_{\delta}}  d^2 \Big( \nabla \mathbf{y}_{\delta, \epsilon} - \frac{\epsilon}{\log \frac{1}{\delta \vee \epsilon}}\mathbf{R}^{\text{const}}_{\delta, \epsilon}  \nabla \mathbf{u}_{\text{Fl},\epsilon}, SO(2)\Big) \,d \mathbf{x}  = o \Big( \frac{\epsilon^2}{\log \frac{1}{\delta \vee \epsilon}}\Big).
\end{align}
Since also $\nabla \mathbf{y}_{\delta, \epsilon} - \frac{\epsilon}{\log \frac{1}{\delta \vee \epsilon}}\mathbf{R}^{\text{const}}_{\delta, \epsilon}  \nabla \mathbf{u}_{\text{Fl},\epsilon} = \nabla \Big( \mathbf{y}_{\delta, \epsilon} - \frac{\epsilon}{\log \frac{1}{\delta \vee \epsilon}}  \mathbf{R}^{\text{const}}_{\delta, \epsilon} \mathbf{u}_{\text{Fl},\epsilon}\Big)$ is a gradient field, there is a constant rotation  $\mathbf{Q}_{\delta, \epsilon} \in SO(2)$  such that
\begin{align}\label{eq:ConstantQEst}
\int_{\Omega_{\delta}} \big| \nabla \mathbf{y}_{\delta, \epsilon} - \frac{\epsilon}{\log \frac{1}{\delta \vee \epsilon}}  \mathbf{R}^{\text{const}}_{\delta, \epsilon} \nabla \mathbf{u}_{\text{Fl},\epsilon}  - \mathbf{Q}_{\delta, \epsilon}\big|^2 \, d\mathbf{x}  = o \Big( \frac{\epsilon^2}{\log \frac{1}{\delta \vee \epsilon}}\Big)
\end{align}
by the $L^2$-geometric rigidity inequality in Corollary \ref{cor:p-rigidity-wedge}. To finish, we must learn how to replace $\mathbf{Q}_{\delta,\epsilon}$ in the above with $\mathbf{R}^{\text{const}}_{\delta,\epsilon}$. We separate the analysis into the two problems.

\textit{Step 6: $\mathbf{Q}_{\delta, \epsilon} \approx \mathbf{I}$ in the displacement problem.} In the displacement problem, $\mathbf{R}_{\delta, \epsilon}^{\text{const}} = \mathbf{I}$ and we claim that $\mathbf{Q}_{\delta, \epsilon}$ can be replaced by $\mathbf{I}$ in \eqref{eq:ConstantQEst}.  Indeed, we apply the $L^2$ version of the second estimate in Corollary \ref{cor:p-rigidity-wedge} to $\mathbf{y}= \mathbf{y}_{\delta,\epsilon} - \frac{\epsilon}{\log \frac{1}{\delta \vee \epsilon}} \mathbf{u}_{\text{Fl},\epsilon}$  to get 
\begin{align*}
\big\| \nabla \mathbf{y}_{\delta, \epsilon} - \frac{\epsilon}{\log \frac{1}{\delta \vee \epsilon}} \nabla \mathbf{u}_{\text{Fl},\epsilon} - \mathbf{I} \big\|_{L^2(\Omega_{\delta})} \lesssim \big\| \epsilon \mathbf{u}_{\delta, \epsilon}^{+} - \frac{\epsilon}{\log \frac{1}{\delta \vee \epsilon}} \mathbf{u}_{\text{Fl}}\big|_{r=1} \big\|_{\dot{W}^{\frac12,2}((\alpha,\beta))} + o\Big( \frac{\epsilon}{\sqrt{\log \frac{1}{\delta \vee \epsilon}}}\Big) 
\end{align*}
by \eqref{eq:dSO2Est}, the boundary conditions on  $\mathbf{y}_{\delta, \epsilon} \in \mathcal{A}_{p, \delta, \epsilon}$ for $r = 1$, and since $\mathbf{u}_{\text{Fl},\epsilon}|_{r=1} = \mathbf{u}_{\text{Fl}}|_{r=1}$. In view of the assumption in \eqref{eq:restateAssumpNLDispCompact} on the mean-free part of $\mathbf{u}_{\delta,\epsilon}^+$, $\| \epsilon \mathbf{u}_{\delta, \epsilon}^{+}\|_{\dot{W}^{\frac{1}{2},2}((\alpha,\beta))} \ll \epsilon(\log \frac{1}{\delta \vee \epsilon})^{-\frac{1}{2}}.$     
Thus, 
\begin{align}\label{eq:almostDone1}
 \int_{\Omega_{\delta}} \Big| \nabla \mathbf{y}_{\delta, \epsilon} - \Big( \mathbf{I} +  \frac{\epsilon}{\log \frac{1}{\delta \vee \epsilon}} \nabla \mathbf{u}_{\text{Fl},\epsilon} \Big) \Big|^2\,d\mathbf{x} = o \Big(\frac{\epsilon^2}{\log \frac{1}{\delta \vee \epsilon}} \Big).
\end{align}

\textit{Step 7: $\mathbf{Q}_{\delta, \epsilon} \approx \mathbf{R}_{\delta,\epsilon}$ in the force problem.} In the force problem, $\mathbf{R}_{\delta,\epsilon}^{\text{const}} = \mathbf{R}_{\delta, \epsilon}\in \cR_{\d}(\mathbf{y}_{\d,\e})$. Consequently, we show that $\mathbf{Q}_{\delta, \epsilon}$ can be replaced by $\mathbf{R}_{\delta, \epsilon}$ in \eqref{eq:ConstantQEst}. First, observe that 
\begin{align*}
|\mathbf{Q}_{\delta, \epsilon} - \mathbf{R}_{\delta, \epsilon} |^2 \lesssim \int_{\Omega_{\delta }}   \big| \nabla \mathbf{y}_{\delta, \epsilon} - \frac{\epsilon}{\log \frac{1}{\delta \vee \epsilon}}\mathbf{R}_{\delta, \epsilon}  \nabla \mathbf{u}_{\text{Fl},\epsilon}  - \mathbf{Q}_{\delta, \epsilon}\big|^2 +  \big| \nabla \mathbf{y}_{\delta, \epsilon} -  \mathbf{R}_{\delta, \epsilon}\Big( \mathbf{I} + \frac{\epsilon}{\log \frac{1}{\delta \vee \epsilon}} \nabla \mathbf{u}_{\text{Fl},\epsilon}\Big) \big|^2 \, d\mathbf{x}.
\end{align*}
 The first term on the right is $\ll  \epsilon^2(\log \frac{1}{\delta\vee \epsilon})^{-1} $ by \eqref{eq:ConstantQEst}, while the second satisfies
\begin{align*}
\int_{\Omega_{\delta}}  \big| \nabla \mathbf{y}_{\delta, \epsilon} -  \mathbf{R}_{\delta, \epsilon}\Big( \mathbf{I} + \frac{\epsilon}{\log \frac{1}{\delta \vee \epsilon}} \nabla \mathbf{u}_{\text{Fl},\epsilon}\Big) \big|^2 \, d\mathbf{x}  \lesssim \int_{\Omega_{\delta}} |\nabla \mathbf{y}_{\delta,\epsilon} - \mathbf{R}_{\delta, \epsilon} |^2 +  \big| \frac{\epsilon}{\log \frac{1}{\delta \vee \epsilon}} \nabla \mathbf{u}_{\text{Fl},\epsilon} \Big|^2 \, d \mathbf{x} \lesssim_p \frac{\epsilon^2}{\log \frac{1}{ \delta \vee \epsilon}}
\end{align*}
by the $L^2$-rigidity inequality in Corollary \ref{cor:mixed-rigidity-energy} and since $E_{\delta}(\mathbf{y}_{\delta,\epsilon})\lesssim_p \epsilon^2(\log \frac{1}{\delta\vee \epsilon})^{-1}$ (see \eqref{eq:firstEstCompactnessForce} or \eqref{eq:firstEstCompactnessForce_p=2} depending on the growth rate $p$ of the energy density). To estimate the term involving $\mathbf{u}_{\text{Fl},\epsilon}$, we used \eqref{eq:uFlCalcEst}, that $\mathbf{u}_{\text{Fl}, \epsilon} = \mathbf{u}_{\text{Fl}}$ on $\Omega_{\delta \vee \epsilon}$, and the fact that
\begin{align}\label{eq:extra-bit}
    \int_{\Omega_{\delta, \delta \vee \epsilon}} \big| \frac{\epsilon}{\log \frac{1}{\delta \vee \epsilon}}  \nabla \mathbf{u}_{\text{Fl},\epsilon}\big|^2 \, d\mathbf{x} \lesssim  \frac{\epsilon^2}{( \log \frac{1}{\delta \vee \epsilon})^2}.
\end{align}
This holds because $\|\nabla \mathbf{u}_{\text{Fl},\epsilon}\|_{L^{\infty}(\Omega_{\delta})} \lesssim 1/(\delta \vee \epsilon)$ and $|\Omega_{\delta, \delta \vee \epsilon}| \lesssim (\delta \vee \epsilon)^2$.
Evidently, then, upon passing to a subsequence we can arrange that 
\begin{equation}\label{eq:extracted-limit}
  \frac{\sqrt{\log \frac{1}{\delta \vee \epsilon}}}{\epsilon}  \Big( \mathbf{R}_{\delta, \epsilon}^T \mathbf{Q}_{\delta, \epsilon} - \mathbf{I} \Big)  \rightarrow   \mathbf{W}
\end{equation}
for some $\mathbf{W} \in \text{Skw}_{2}$. 

We claim now that $\mathbf{W}=\mathbf{0}$ regardless of the subsequence. Note it follows immediately that \eqref{eq:extracted-limit} holds for the original sequence (the limit is zero). 
Substituting \eqref{eq:extracted-limit} into \eqref{eq:ConstantQEst} yields that
\begin{equation}\label{eq:final-result}
\int_{\Omega_{\delta}}  \Big| \nabla \mathbf{y}_{\delta, \epsilon} - \mathbf{R}_{\delta, \epsilon} \Big(\mathbf{I} + \frac{\epsilon}{\log \frac{1}{\delta \vee \epsilon}} \nabla \mathbf{u}_{\text{Fl},\epsilon} + \frac{\epsilon}{\sqrt{\log \frac{1}{\delta \vee \epsilon}}}\mathbf{W} \Big) \Big|^2 \,d\mathbf{x} = o \Big( \frac{\epsilon^2}{\log \frac{1}{\delta \vee \epsilon}}\Big)
\end{equation}
for the subsequence. From here, the last identity in Lemma \ref{skewLemma} along with Jensen's inequality furnishes
\begin{align*}
&\Big|\fint_{\Omega_{\delta}} \frac{\epsilon}{\log \frac{1}{\delta \vee \epsilon}} \skw \nabla \mathbf{u}_{\text{Fl},\epsilon} - \frac{\epsilon}{\sqrt{\log \frac{1}{\delta \vee \epsilon}}} \mathbf{W}  \,d \mathbf{x} \Big|^2 \\
&\qquad =\Big| \fint_{\Omega_{\delta}} \text{skw}\,\Big( \mathbf{R}_{\delta, \epsilon}^T \nabla \mathbf{y}_{\delta, \epsilon} - \mathbf{I}  - \frac{\epsilon}{\log \frac{1}{\delta \vee \epsilon}} \nabla \mathbf{u}_{\text{Fl},\epsilon} - \frac{\epsilon}{\sqrt{\log \frac{1}{\delta \vee \epsilon}}} \mathbf{W} \Big)  \,d \mathbf{x}   \Big|^2   \\ 
&\qquad  \qquad \leq \fint_{\Omega_{\delta}}  \Big| \nabla \mathbf{y}_{\delta, \epsilon} - \mathbf{R}_{\delta, \epsilon} \Big(\mathbf{I} + \frac{\epsilon}{\log \frac{1}{\delta \vee \epsilon}} \nabla \mathbf{u}_{\text{Fl},\epsilon} + \frac{\epsilon}{\sqrt{\log \frac{1}{\delta \vee \epsilon}}}\mathbf{W} \Big) \Big|^2 \,d\mathbf{x} = o \Big( \frac{\epsilon^2}{\log \frac{1}{\delta \vee \epsilon}}\Big)
\end{align*}
since $\mathbf{R}_{\delta, \epsilon} \in \mathcal{R}_{\delta} (\mathbf{y}_{\delta, \epsilon})$. Therefore,
\begin{align*}
\frac{\epsilon}{\sqrt{\log \frac{1}{\delta \vee \epsilon}}}  |\mathbf{W}| \leq\Big| \fint_{\Omega_{\delta}} \frac{\epsilon}{\log \frac{1}{\delta \vee \epsilon}} \text{skw}\,\nabla \mathbf{u}_{\text{Fl},\epsilon}  \, d\mathbf{x} \Big| + o \Big( \frac{\epsilon}{\sqrt{\log \frac{1}{\delta \vee \epsilon}}}  \Big). 
\end{align*}
As $|\nabla \mathbf{u}_{\text{Fl},\epsilon}| \lesssim   1/(r\vee \epsilon)$ by \eqref{eq:Flamant-regularized}, the first term on the right-hand side is $\lesssim \epsilon (\log \frac{1}{\delta \vee \epsilon})^{-1}$. Thus, $\mathbf{W} = \mathbf{0}$ and this proves \eqref{eq:extracted-limit} and \eqref{eq:final-result} for the original sequence. In conclusion,  
\begin{align}\label{eq:almostDone2}
\int_{\Omega_{\delta}}  \Big| \nabla \mathbf{y}_{\delta, \epsilon} - \mathbf{R}_{\delta, \epsilon} \Big(\mathbf{I} + \frac{\epsilon}{\log \frac{1}{\delta \vee \epsilon}} \nabla \mathbf{u}_{\text{Fl},\epsilon} \Big) \Big|^2 \,d\mathbf{x} = o \Big( \frac{\epsilon^2}{\log \frac{1}{\delta \vee \epsilon}}\Big).
\end{align}

\textit{Step 8: Eliminating the $\epsilon$-truncation.} Since $\mathbf{u}_{\text{Fl}, \epsilon} = \mathbf{u}_{\text{Fl}}$ on $\Omega_{\delta \vee \epsilon}$, \eqref{eq:almostDone1} and \eqref{eq:almostDone2} imply the statements involving $\Omega_{\delta \vee \epsilon}$ on the left-hand side of \eqref{eq:finalfinalLimits}. For the remaining claims, recall that 
$\int_{\Omega_{\delta, \delta \vee \epsilon}} | \frac{\epsilon}{\log \frac{1}{\delta \vee \epsilon}}  \nabla \mathbf{u}_{\text{Fl},\epsilon}|^2 \, d\mathbf{x} \lesssim  \epsilon^2(\log\frac{1}{\delta\vee\epsilon})^{-2} $
by \eqref{eq:extra-bit}. Therefore, we can simply eliminate the corresponding terms from \eqref{eq:almostDone1} and \eqref{eq:almostDone2} to conclude the rest of \eqref{eq:finalfinalLimits}.  
\end{proof}

\vspace*{0.3cm}
\noindent \textbf{Acknowledgments.} PP and IT  acknowledge support
from ARO (ARO-W911NF2310137).  PP acknowledges support from  NSF (CMMI-CAREER-2237243), and IT acknowledges support from NSF (DMS-CAREER-2350161/DMS-CAREER-2618358). DE acknowledges support from the BayIntAn cooperation funding programme by the Bavarian Research Alliance (BayIntAn\textunderscore KUEI\textunderscore2024\textunderscore27).

\vspace*{0.3cm}
\noindent \textbf{Data availability.} There is no supporting data for this manuscript.

\vspace*{0.3cm} 
\noindent \textbf{Conflict of interest.}
The authors declare that they have no conflict of interest.

\appendix
\section{Degeneracies of the limiting rotations}\label{sec:degenerate-rotations}
Theorem \ref{thm:main-result} proves for the nonlinear force problem that
\begin{align}\label{eq:for-appendix-discussion}
\frac{\log \frac{1}{\delta \vee \epsilon}}{\epsilon^2} \mathcal{E}_{p,\delta, \epsilon}^{\text{force}} \rightarrow 
 \begin{cases}
\displaystyle \min_{\mathbf{R} \in SO(2)} - \QFl^\ast( \mathbf{R}^T \mathbf{f}_0)  &  \text{ if  \eqref{super-degenerate}}\\  
\displaystyle \min_{\mathbf{R} \in SO(2)} - \QFl^\ast( \mathbf{R}^T \mathbf{f}_0) + V_0 (1- \mathbf{R} \mathbf{e}_1 \cdot \mathbf{R}(\varphi) \mathbf{e}_1) & \text{ if \eqref{degenerate}}\\ 
\displaystyle -\QFl^\ast( \mathbf{R}(\varphi)^T \mathbf{f}_0 )  & \text{ if  \eqref{non-degenerate}}  
\end{cases}.
\end{align}
This appendix follows up on Remark \ref{rem:subsequence-thm} from the introduction and identifies the conditions under which the first two minimization problems have multiple minimizers. 
\subsection{Case \eqref{super-degenerate}.} In this case, there are always multiple rotations solving \eqref{eq:for-appendix-discussion}. The minimizers rotate the force $ \mathbf{f}_0$ to align with an  eigenvector of the maximum eigenvalue of $\mathbf{K}_{\text{Fl}}^{-1}$. Generically, $\mathbf{K}_{\text{Fl}}$ is not a multiple of the identity and $\mathbf{f}_0 \neq 0$, and there are two such rotations.  Otherwise, all of $SO(2)$ is optimal.  

Interestingly, $\mathbf{K}_{\text{Fl}}=c\mathbf{I}$ can occur, e.g., in isotropic elasticity for special choices of the wedge angles $\alpha, \beta$. Isotropy yields $(\mathbb{C}^{-1})_{rrrr} = 1/E'$ for $E' = E$ in plane strain and  $E' =E/(1-\nu^2)$ in plane stress, where $E$ is  Young's modulus and $\nu$ is Poisson's ratio.  The isotropic Flamant stiffness tensor is   
\begin{align*}
\mathbf{K}_{\text{Fl}}^{\text{iso}} = E' \int_{\alpha}^{ \beta} \mathbf{e}_r \otimes \mathbf{e}_r \, d\theta = E'\mathbf{R} \begin{pmatrix} 
 \gamma + \frac{1}{2}\sin 2 \gamma  & 0 \\ 0 &  \gamma  - \frac{1}{2}\sin 2 \gamma  \end{pmatrix} \mathbf{R}^T  
\end{align*}
where $\mathbf{R}$ is a counterclockwise rotation by the angle $\alpha + \gamma$ with $\gamma = ( \beta - \alpha)/2$. Thus,  $\mathbf{K}_{\text{Fl}}^{\text{iso}}$ is a multiple of the identity if and only if $\beta - \alpha \in\{ \pi/2,\pi\}$.

\subsection{Case \eqref{degenerate}.} In this case, the minimizing rotation in \eqref{eq:for-appendix-discussion} is generically unique, though there are special choices of the parameters that yield exactly  two minimizing rotations. 
To explain, observe the problem in question is equivalent to  
\begin{align}\label{eq:rewriteBMin1}
    \min_{\psi \in (-\pi, \pi]}\, -  \begin{pmatrix}
    \cos \psi\\
    \sin \psi \end{pmatrix} \cdot   (\mathbf{f}_0, \mathbf{f}_0^{\perp})^T \mathbf{K}_{\text{Fl}}^{-1} (\mathbf{f}_0, \mathbf{f}_0^{\perp})  \begin{pmatrix}
    \cos \psi\\
    \sin \psi \end{pmatrix} - V_0 \begin{pmatrix}  \cos \varphi \\ \sin \varphi \end{pmatrix}\cdot \begin{pmatrix} \cos \psi \\ \sin \psi\end{pmatrix}
\end{align}
through the parameterization   $\mathbf{R}(\psi) = \cos \psi \mathbf{I} + \sin \psi\mathbf{J}$. If $\mathbf{f}_0 = \mathbf{0}$, the unique minimizer is $\psi = \varphi$ since $V_0 >0$ by assumption. Otherwise, $(\mathbf{f}_0, \mathbf{f}_0^{\perp})^T \mathbf{K}_{\text{Fl}}^{-1} (\mathbf{f}_0, \mathbf{f}_0^{\perp})$ is  positive definite and \eqref{eq:rewriteBMin1} becomes 
\begin{align*}
\min_{\tilde{\psi} \in (-\pi, \pi]}\, -\lambda_1 \cos^2 \tilde{\psi} - \lambda_2 \sin^{2} \tilde{\psi} - V_0 (\cos \tilde{\varphi} \cos \tilde{\psi} + \sin \tilde{\varphi} \sin \tilde{\psi})  
\end{align*}
after a change of variables $(\psi,\varphi) \leftrightarrow (\tilde{\psi},\tilde{\varphi})$. The terms $\lambda_1, \lambda_2 > 0$ are the eigenvalues of $(\mathbf{f}_0, \mathbf{f}_0^{\perp})^T \mathbf{K}_{\text{Fl}}^{-1} (\mathbf{f}_0, \mathbf{f}_0^{\perp})$. It is useful to express this minimization geometrically. Setting $x = \cos \tilde{\psi}$ and $y = \sin \tilde{\psi}$, completing the square, and dropping  terms independent of $x,y$ leads to the equivalent maximization problem
\begin{align}\label{eq:ellipseMin}
    m^{\star} = \max  \Big\{ \lambda_1 \Big( x + \frac{V_0 \cos \tilde{\varphi}}{\lambda_1} \Big)^2 +  \lambda_2\Big( y+ \frac{V_0 \sin \tilde{\varphi}}{\lambda_2} \Big)^2 \, \colon \, x^2 + y^2 = 1\Big\}. 
\end{align}
The function  $\mathfrak{e}=\mathfrak{e}(x,y)$ being maximized defines an ellipse. Maximizers are points of tangency between the ellipse $\mathfrak{e} = m^{\star}$ and the unit circle $x^2 + y^2 = 1$. 

We now show that \eqref{eq:ellipseMin} has a unique maximizer, generically. The point is that there exists only one point of tangency between the unit circle and the ellipse $\mathfrak{e}=m^{\star}$ for a generic choice of parameters. 
Indeed, suppose there are two points of tangency. Then, by definition,  the origin  is   on the symmetry set of the ellipse, i.e., the closure of the set of centers of circles tangent to (at least) two points on the ellipse. The symmetry set of an ellipse is known \cite{bruce1992curves}: it consists of two line segments, one contained in the major axis and one containing the minor axis. If the origin is in the symmetry set, then $\cos \tilde{\varphi} = 0$ or $\sin \tilde{\varphi} = 0$. This is not generic given the definition of $\tilde{\varphi}$. Therefore, \eqref{eq:ellipseMin} is uniquely solved, generically.

Finally, we explain how \eqref{eq:ellipseMin} can have either one or two solutions if $\sin \tilde{\varphi} = 0$ or $\cos \tilde{\varphi} = 0$. Suppose $\sin \tilde{\varphi} = 0$ and $\cos \tilde{\varphi} \in \{\pm 1\}$. Then \eqref{eq:ellipseMin} becomes 
\begin{align}\label{eq:mStarMax}
    m^{\star} = \max_{x \in [-1,1]}\, \lambda_1 \Big( x \pm \frac{V_0}{\lambda_1} \Big)^2  + \lambda_2 \Big( 1- x^2 \Big).
\end{align}
A maximizer  to this optimization exists; it is unique since $V_0> 0$.   Indeed, 
if $\lambda_1 = \lambda_2$, then this maximization is of a non-constant linear function, giving a unique maximizer at $x^{\star}= -1$ or $1$. Otherwise,  
 \eqref{eq:mStarMax} involves maximizing a quadratic function that is not even, hence the maximizer $x^{\star}$ is still unique. If $x^{\star} \in (-1,1)$,  the ellipse $\mathfrak{e}=m^{\star}$ and the unit circle have two points of tangency at $(x,y) = (x^{\star}, \pm \sqrt{1 - (x^{\star})^2})$. Otherwise, they have one point of tangency at $(x,y) = (1,0)$ or $(-1,0)$. The same conclusions hold in the complementary case where $\cos \tilde{\varphi} = 0$ and $\sin \tilde{\varphi} \in\{\pm 1\}$. 

\section{Estimates on traces}\label{sec:appendix-traces}
Here we record some elementary facts regarding trace spaces in polar versus Cartesian coordinates. We define these spaces using finite differences, as in \cite{Leo23}. 
Recall $\Omega_\delta = \{\mathbf{x}\in\mathbb{R}^2:r\in(\delta,1),\theta\in(\alpha,\beta)\}$ is the truncated wedge domain where $\delta\in(0,1)$ and $\beta-\alpha\in(0,2\pi)$.  
For $p\in[1,\infty)$, the fractional Sobolev space $W^{1-1/p,p}(\partial\Omega_\delta;\mathbb{R}^2)$ consists of all $\mathbf{u}\in L^p(\partial\Omega_\delta;\mathbb{R}^2)$ such that
\begin{equation*}
\|\mathbf{u}\|_{\dot{W}^{1-\frac1p,p}(\partial\Omega_\delta)} := \left( \int_{\partial\Omega_\delta} \int_{\partial\Omega_\delta}  \frac{|\mathbf{u}(\mathbf{x}) - \mathbf{u}(\mathbf{x}')|^p}{|\mathbf{x} - \mathbf{x}'|^{p}} \, ds ds'  \right)^{\frac1p} < \infty.
\end{equation*}
For $p=\infty$, $W^{1,\infty}(\partial\Omega_\delta;\mathbb{R}^2)$ consists of all $\mathbf{u}\in L^\infty(\partial\Omega_\delta;\mathbb{R}^2)$ such that 
\begin{equation*}
\|\mathbf{u}\|_{\dot{W}^{1,\infty}(\partial\Omega_\delta)} := \esssup_{\substack{\mathbf{x},\mathbf{x}'\in\partial\Omega_\delta\\\mathbf{x}\neq\mathbf{x}'}} \frac{|\mathbf{u}(\mathbf{x}) - \mathbf{u}(\mathbf{x}')|}{|\mathbf{x} - \mathbf{x}'|}< \infty.
\end{equation*}
Both formulas use the arclength measure $ds$. These spaces describe the image of the trace map applied to $W^{1,p}(\Omega_\delta;\mathbb{R}^2)$. As such, they are the canonical spaces for Dirichlet boundary data. However, one may wish to impose Dirichlet data along the circular arcs 
\[
\Gamma_a= \{ \mathbf{x}\in\mathbb{R}^2 : r=a,\theta\in(\alpha,\beta)\},\quad a=\delta,1
\]
instead. This leads to $W^{1-1/p,p}(\Gamma_a;\mathbb{R}^2)$ which consists of all $\mathbf{u}\in L^p(\Gamma_a;\mathbb{R}^2)$ such that 
\begin{equation*}
\|\mathbf{u}\|_{\dot{W}^{1-\frac1p,p}(\Gamma_a)} := \left( \int_{\Gamma_a} \int_{\Gamma_a}  \frac{|\mathbf{u}(\mathbf{x}) - \mathbf{u}(\mathbf{x}')|^p}{|\mathbf{x} - \mathbf{x}'|^{p}} \, ds d s'  \right)^{\frac1p} < \infty
\end{equation*}
for $p\in[1,\infty)$, with a similar formula for $p=\infty$. Evidently, the trace of a function in $W^{1,p}(\Omega_\delta;\mathbb{R}^2)$ restricts to an element of $W^{1-1/p,p}(\Gamma_a;\mathbb{R}^2)$ for $a=\delta,1$. Conversely, standard results imply that every function in $W^{1-1/p,p}(\Gamma_a;\mathbb{R}^2)$ for $a=\delta$ or $1$ can be extended to a function in $W^{1,p}(\Omega_\delta;\mathbb{R}^2)$ (see Lemmas \ref{lem:extensionNonlinear} and \ref{lem:extensionNonlinearInfinity}). Finally, $W^{1-1/p,p}((\alpha,\beta);\mathbb{R}^2)$  consists of all $\mathbf{v}\in L^p((\alpha,\beta);\mathbb{R}^2)$ such that 
\begin{equation*}
\|\mathbf{v}\|_{\dot{W}^{1-\frac1p,p}((\alpha,\beta))} := \left( \int_\alpha^\beta \int_\alpha^\beta  \frac{|\mathbf{v}(\theta) - \mathbf{v}(\theta')|^p}{|\theta - \theta'|^{p}} \, d\theta d\theta'  \right)^{\frac1p} < \infty.
\end{equation*}
Again, there is a similar formula for $p=\infty$. 

The first result of this appendix shows that $W^{1-1/p,p}(\Gamma_a;\mathbb{R}^2)$ and $W^{1-1/p,p}((\alpha,\beta);\mathbb{R}^2)$ are isomorphic.
\begin{lem}\label{lem:seminorms_polar} For $a\in(0,\infty)$ and $p\in[1,\infty]$, 
\begin{align*}
\|\mathbf{u} \|_{\dot{W}^{1-\frac1p,p}(\Gamma_a)} \sim_{\beta-\alpha} a^{\frac{2}{p}-1}\|\mathbf{v}\|_{\dot{W}^{1-\frac1p,p}((\alpha, \beta))}
\end{align*} 
under the identification $\mathbf{u}(a\mathbf{e}_r(\cdot))=\mathbf{v}$.
\end{lem}
\begin{proof} By scaling $\mathbf{x}\to a\mathbf{x}$ it suffices to treat $a=1$. Note $ds=d\theta$ in this case, so $(\alpha,\beta) \to \Gamma_1$, $\theta\mapsto \mathbf{e}_r(\theta)$ is an arclength parameterization. Thus $|\mathbf{e}_r(\theta)-\mathbf{e}_r(\theta')|\sim_{\beta-\alpha}|\theta-\theta'|$. The result follows.
\end{proof}
\noindent In the main text, we use boundary displacements that are functions of $\theta$ and belong to $W^{1-1/p,p}((\alpha,\beta);\mathbb{R}^2)$ for appropriate $p$. We proceed now with this in mind. In particular, we abuse notation and write 
\[
\mathbf{u}|_{r=a} = \mathbf{u}|_{\Gamma_a}(a\mathbf{e}_r(\cdot))
\]
in the rest of this appendix. 

Next, we estimate the constants in a trace-type inequality involving $r=\delta,1$. 
\begin{lem}\label{traceToGradLemma} 
Let $\delta\in(0,1/2)$, $p\in[1,\infty]$, and $\mathbf{u} \in W^{1,p}(\Omega_{\delta}; \mathbb{R}^2)$. There holds
\begin{align*}
\delta^{\frac2p-1}\|\mathbf{u}|_{r = \delta} \|_{\dot{W}^{1-\frac1p,p}((\alpha,\beta))} +  \|\mathbf{u}|_{r = 1} \|_{\dot{W}^{1-\frac1p,p}((\alpha,\beta))}  \lesssim_{\beta-\alpha,p} \| \nabla \mathbf{u}\|_{L^{p}(\Omega_{\delta})} . 
\end{align*}
\end{lem}
\begin{proof} The case $p=\infty$ is clear. For $p<\infty$, we first handle $r=1$. Let $\mathbf{c} = \fint_{\Omega_{1/2}} \mathbf{u}\,d\mathbf{x}$ and observe from Lemma \ref{lem:seminorms_polar} that  $\|\mathbf{u}|_{r = 1} \|_{\dot{W}^{1-1/p,p}((\alpha,\beta))} \sim_{\beta-\alpha}  \| \mathbf{u} - \mathbf{c} \|_{\dot{W}^{1-1/p,p} (\Gamma_1)}$. Since $\Gamma_1\subset\partial\Omega_{1/2}$, 
\begin{align*} 
\| \mathbf{u} - \mathbf{c} \|_{\dot{W}^{1-\frac1p,p} (\Gamma_1)} &\leq  \| \mathbf{u} - \mathbf{c} \|_{W^{1-\frac1p,p} (\partial\Omega_{\frac{1}{2}})} \lesssim_{\beta-\alpha,p} \| \mathbf{u} - \mathbf{c} \|_{W^{1,p} (\Omega_{\frac12})} \\
 &=  \|\mathbf{u} - \mathbf{c} \|_{L^p(\Omega_{\frac12})} + \| \nabla \mathbf{u} \|_{L^p(\Omega_{\frac12})} \lesssim_{\beta-\alpha,p} \| \nabla \mathbf{u} \|_{L^p(\Omega_{\frac12})}
\end{align*} 
by the usual trace and Poincar\'{e} inequalities on $\Omega_{1/2}$. Therefore,
\[
\|\mathbf{u}|_{r = 1} \|_{\dot{W}^{1-\frac1p,p}((\alpha,\beta))} \lesssim_{\beta-\alpha,p}  \| \nabla\mathbf{u} \|_{L^p(\Omega_{\frac12})}.
\]

To handle $r=\delta$ we use a change of variables. Rescale $\mathbf{u}$ on  $\Omega_{\delta,2\delta}=\Omega_\delta\cap\{r\in(\delta,2\delta)\}$ to obtain $\tilde{\mathbf{u}} = (2\delta)^{-1} \mathbf{u}(2\delta \cdot)$ on $\Omega_{1/2}$. The same argument as above gives that
\[
\|\tilde{\mathbf{u}}|_{r = \frac12} \|_{\dot{W}^{1-\frac1p,p}((\alpha,\beta))} \lesssim_{\beta-\alpha,p}  \| \nabla \tilde{\mathbf{u}} \|_{L^p(\Omega_{\frac12})}.
\]
Changing variables back, there follows
\[
\delta^{-1} \|\mathbf{u}|_{r = \delta} \|_{\dot{W}^{1-\frac1p,p}((\alpha,\beta))} \lesssim_{\beta-\alpha,p} \delta^{-\frac2p} \| \nabla \mathbf{u} \|_{L^p(\Omega_{\delta})}.
\]
Multiplying by $\delta^{2/p}$ completes the proof. 
\end{proof}

Finally, we provide two extension results for $p<\infty$ and $p=\infty$.
\begin{lem}\label{lem:extensionNonlinear}
	Let $\delta \in (0,e^{-2})$ and $p\in[1,\infty)$, and let $\mathbf{u}^\pm\in W^{1-1/p,p}((\alpha,\beta);\R^2)$ have $\fint_\alpha^\beta \mathbf{u}^\pm\, d\theta =\mathbf{0}$. There exists $\mathbf{u}\in W^{1,p}(\Omega_{\delta}; \mathbb{R}^2)$ such that 
	\begin{equation*}	
	\begin{aligned}
	& \mathbf{u}|_{r = \delta}= \mathbf{u}^-,  \qquad \mathbf{u}|_{r = 1}=\mathbf{u}^+,  \qquad  \mathbf{u} = \mathbf{0} \text{ on } \Omega_{e \delta,\frac1e},  \\ 
	&||\nabla\mathbf{u}||_{L^p(\Omega_{\delta,e\delta})}\lesssim_{\beta-\alpha,p} \delta^{\frac2p -1} \| \mathbf{u}^{-} \|_{\dot{W}^{1-\frac1p,p}((\alpha, \beta))}, \qquad ||\nabla\mathbf{u}||_{L^p(\Omega_{\frac1e,1})}\lesssim_{\beta-\alpha,p}  \| \mathbf{u}^{+} \|_{\dot{W}^{1-\frac1p,p}((\alpha, \beta))}
	\end{aligned}.
	\end{equation*}
\end{lem}
\begin{proof}
Standard extension theorems \cite{Leo23} yield $\mathbf{V}^\pm\in W^{1,p}(\tilde{R};\R^2)$ with $\tilde{R}=(0,1)\times (\alpha,\beta)$ such that $\mathbf{V}^\pm|_{\rho=0}=\mathbf{u}^\pm$ and $\|\mathbf{V}^\pm\|_{W^{1,p}(\tilde{R})}\lesssim_{\beta -\alpha,p} \|\mathbf{u}^\pm\|_{\dot{W}^{1-1/p,p}((\alpha,\beta))}$. In particular, we can first extend $\mathbf{u}^\pm$ to belong to $W^{1-1/p,p}(\mathbb{R};\mathbb{R}^2)$ and then apply a half-space extension and a restriction to get $\mathbf{V}^\pm$.  Let $\psi:[0,1]\to[0,1]$ be smooth with $\psi=0$ on $[0,\frac{1}{4}]$ and $\psi=1$ on $[\frac{3}{4},1]$, and define $\mathbf{u} \colon  \Omega_{\delta}  \rightarrow \mathbb{R}^2$ in polar coordinates by 
	\begin{align*}
		\mathbf{u} (\mathbf{x}) =
		\begin{cases}
			\Big(1- \psi\big(\log r  - \log \delta\big) \Big) \mathbf{V}^-\big(\log r  - \log \delta,\theta\big) &\text{ if } r \in (\delta,e \delta )\\
			\mathbf{0} &\text{ if } r  \in (e \delta ,\frac1e )\\
			\Big(1 - \psi\big(-\log  r \big) \Big) \mathbf{V}^+\big(-\log r ,\theta\big) &\text{ if } r \in (\frac1e,1)\\
		\end{cases}.
	\end{align*}
Note $\mathbf{u}$ satisfies the boundary conditions. For the inequalities, observe using polar coordinates that
\[
	\int_{\Omega_{\delta}}  |\nabla \mathbf{u}|^p \, d\mathbf{x} \lesssim_p \delta^{2-p}\int_{\alpha}^{\beta} \int_{\delta}^{e \delta} \Big( |r \partial_r \mathbf{u}|^p + |\partial_\theta \mathbf{u}|^p \Big) \,\frac{dr}{r} d\theta +  \int_{\alpha}^{\beta} \int_{\frac1e}^{1} \Big( |r \partial_r \mathbf{u}|^p + |\partial_\theta \mathbf{u}|^p \Big) \,\frac{dr}{r} d\theta. 
\]
Set $\tilde{r} = \log r - \log \delta$ and observe for $r \in (\delta, e \delta)$ that
	\begin{align*}
	|r \partial_r \mathbf{u}(r,\theta)| + |\partial_\theta \mathbf{u}(r,\theta)| \lesssim  \big| \mathbf{V}^{-}(\tilde{r}, \theta)  \big| + \big| \partial_{\tilde{r}} \mathbf{V}^{-}(\tilde{r}, \theta) \big| +   \big| \partial_{\theta} \mathbf{V}^{-}(\tilde{r}, \theta) \big|.
	\end{align*}
Thus, as $d \tilde{r}= \frac{dr}{r}$, 
	\begin{align*}
	\int_{\alpha}^{\beta} \int_{\delta}^{e \delta} \Big( |r \partial_r \mathbf{u}|^p + |\partial_\theta \mathbf{u}|^p \Big) \,\frac{dr}{r} d\theta \lesssim_p  \int_{\alpha}^{\beta} \int_0^{1}  \Big( \big| \mathbf{V}^{-} \big|^p + \big| \partial_{\tilde{r}} \mathbf{V}^{-}\big|^p  + \big| \partial_{\theta} \mathbf{V}^{-} \big|^p \Big) \, d\tilde{r} d\theta.
	\end{align*}
Similarly,
\begin{align*}
	\int_{\alpha}^{\beta} \int_{\frac1e}^{1} \Big( |r \partial_r \mathbf{u}|^p + |\partial_\theta \mathbf{u}|^p \Big) \,\frac{dr}{r} d\theta \lesssim_p  \int_{\alpha}^{\beta} \int_0^{1}  \Big( \big| \mathbf{V}^{+} \big|^p + \big| \partial_{\tilde{r}} \mathbf{V}^{+}\big|^p  + \big| \partial_{\theta} \mathbf{V}^{+} \big|^p \Big) \, d\tilde{r} d\theta.
	\end{align*}
 The desired estimate follows.
\end{proof}
\begin{lem}\label{lem:extensionNonlinearInfinity}
Let  $M > 0$, $\delta \in (0,e^{-1-M})$, $\epsilon \in (0,1)$, and $\mathbf{u}^\pm\in W^{1,\infty}((\alpha,\beta);\R^2)$ with $\fint_\alpha^\beta \mathbf{u}^\pm\, d\theta =\mathbf{0}$. There exists $\mathbf{u}\in W^{1,\infty}(\Omega_{\delta}; \mathbb{R}^2)$ such that 
	\begin{equation*}	
	\begin{aligned}
	&\mathbf{u}|_{r = \delta}= \mathbf{u}^-,  \qquad \mathbf{u}|_{r = 1}=\mathbf{u}^+,  \qquad  \mathbf{u} = \mathbf{0} \text{ on } \Omega_{e^{M} \delta,\frac1e},  \\ 
    &\det \big( \mathbf{I} + \epsilon \nabla \mathbf{u} \big) \geq 1  - \frac{\epsilon}{\delta} \| \mathbf{e}_{\theta} \cdot \frac{d}{d\theta} \mathbf{u}^{-} \|_{L^{\infty}((\alpha, \beta))} - \frac{\beta - \alpha}{2M} \sum_{k = 1,2} \Big(\frac{\epsilon}{\delta} \| \mathbf{u}^{-} \|_{\dot{W}^{1,\infty}((\alpha, \beta))} \Big)^{k} \text{ on } \Omega_{\delta,e^{M}\delta},\\
	&\| \epsilon \nabla \mathbf{u}\|_{L^{\infty}(\Omega_{\delta,e^{M}\delta})}  \leq \frac{\epsilon}{\delta} \Big( 1+ \frac{\beta - \alpha}{2M} \Big)   \| \mathbf{u}^{-} \|_{\dot{W}^{1,\infty}((\alpha, \beta))}, \\
    &\| \epsilon \nabla \mathbf{u}\|_{L^{\infty}(\Omega_{\frac{1}{e},1})}   \leq \frac{\epsilon}{e} \Big( 1+\frac{\beta - \alpha}{2}\Big) \| \mathbf{u}^{+} \|_{\dot{W}^{1,\infty}((\alpha, \beta))}.
	\end{aligned}
	\end{equation*}
\end{lem}
\begin{proof}
Define  
$\mathbf{u} \colon  \Omega_{\delta}  \rightarrow \mathbb{R}^2$ in polar coordinates by  
	\begin{align*}
		\mathbf{u} (\mathbf{x}) =
		\begin{cases}
			\Big(1- \frac{1}{M}\big(\log r  - \log \delta\big) \Big) \mathbf{u}^-(\theta) &\text{ if } r \in (\delta,e^{M} \delta )\\
			\mathbf{0} &\text{ if } r  \in (e^M \delta ,\frac1e )\\
			\Big(1 + \log  r  \Big) \mathbf{u}^+(\theta) &\text{ if } r \in (\frac1e,1)\\
		\end{cases}.
	\end{align*}
The boundary conditions hold. For the inequalities, observe that  
    \begin{align*}
    \epsilon \nabla \mathbf{u} =
    \begin{cases}
    \frac{\epsilon}{r}  \Big( - \frac{1}{M}\mathbf{u}^{-} \otimes \mathbf{e}_r + (1- \rho_{\delta})\frac{d}{d\theta} \mathbf{u}^{-} \otimes \mathbf{e}_\theta \Big)  &\text{ on } \Omega_{\delta, e^M\delta},  \\
     \frac{\epsilon}{r}  \Big(\mathbf{u}^{+} \otimes \mathbf{e}_r + (1- \rho_{1})\frac{d}{d\theta} \mathbf{u}^{+} \otimes \mathbf{e}_\theta \Big) &\text{ on } \Omega_{\frac{1}{e},1}
        \end{cases}
    \end{align*}
    where $\rho_{\delta} = M^{-1}( \log r - \log \delta)$ and $\rho_1 =- \log r$. Note $\rho_\delta, \rho_1$ map to $(0,1)$  in the above. 
On $\Omega_{\delta, e^{M} \delta}$, 
    \begin{align*}
        \det \big( \mathbf{I} + \epsilon \nabla \mathbf{u} \big)&= - \big( \mathbf{e}_r - \frac{\epsilon}{r M} \mathbf{u}^{-} \big) \cdot \mathbf{J}\big( \mathbf{e}_{\theta} + \frac{\epsilon}{r}(1- \rho_{\delta}) \frac{d}{d\theta} \mathbf{u}^{-} \big) \\
        &\geq 1 - \frac{\epsilon}{\delta}\big| \mathbf{e}_\theta \cdot \frac{d}{d\theta} \mathbf{u}^{-}  \big|  - \frac{\epsilon}{\delta M}\big| \mathbf{u}^{-} \cdot \mathbf{e}_{r} \big| - \frac{\epsilon^2}{\delta^2 M} \big| \mathbf{u}^{-} \cdot \mathbf{J} \frac{d}{d\theta} \mathbf{u}^{-} \big| \\
        &\geq 1 - \frac{\epsilon}{\delta}\big\| \mathbf{e}_\theta \cdot \frac{d}{d\theta} \mathbf{u}^{-}  \big\|_{L^{\infty}((\alpha, \beta))} - \frac{\epsilon}{\delta M } \| \mathbf{u}^{-} \|_{L^{\infty}((\alpha,\beta))} - \frac{\epsilon^2}{\delta^2 M} \| \mathbf{u}^{-} \|_{L^{\infty}((\alpha,\beta))} \| \frac{d}{d\theta} \mathbf{u}^{-} \|_{L^{\infty}((\alpha,\beta))}
    \end{align*}
for $\mathbf{J}=\mathbf{e}_2\otimes\mathbf{e}_1-\mathbf{e}_1\otimes\mathbf{e}_2$, since $\det \mathbf{F} = -\mathbf{F} \mathbf{e}_r \cdot \mathbf{J} \mathbf{F} \mathbf{e}_\theta$ for  $\mathbf{F} \in \mathbb{R}^{2\times2}$. Furthermore, 
    \begin{align*}
        &\|\nabla \mathbf{u}\|_{L^{\infty}(\Omega_{\delta,e^{M} \delta})} \leq  \frac{1}{\delta}\Big( \frac{1}{M} \|\mathbf{u}^{-}\|_{L^{\infty}((\alpha, \beta))} +  \big\| \frac{d}{d\theta} \mathbf{u}^{-}\big\|_{L^{\infty}((\alpha,\beta))} \Big), \\
         &\|\nabla \mathbf{u}\|_{L^{\infty}(\Omega_{\frac{1}{e},1})} \leq  \frac{1}{e}\Big(  \|\mathbf{u}^{+}\|_{L^{\infty}((\alpha, \beta))} +  \big\| \frac{d}{d\theta} \mathbf{u}^{+}\big\|_{L^{\infty}((\alpha,\beta))} \Big).
    \end{align*}
Since $\mathbf{u}^\pm$ are assumed to average to zero, $\|\mathbf{u}^{\pm}\|_{L^{\infty}((\alpha,\beta))} \leq \frac{\beta - \alpha}{2} \| \frac{d}{d\theta} \mathbf{u}^{\pm}\|_{L^{\infty}((\alpha,\beta))}$. The proof is done. 
\end{proof}

\bibliographystyle{abbrv} 
\bibliography{bibliography} 

@article{GrE00,
	AUTHOR	= {Grima, Joeseph N. and Evans, Kenneth},
	TITLE	= {Auxetic behavior from rotating squares},
	JOURNAL	= {Journal of Materials Science Letters},
	VOLUME	= {19},
	PAGES	= {1563-1565},
	YEAR		= {2000},
	DOI		= {https://doi.org/10.1023/A:1006781224002},	
}

@article{nassar2017curvature,
  title={Curvature, metric and parametrization of origami tessellations: theory and application to the eggbox pattern},
  author={Nassar, Hussein and Leb{\'e}e, Arthur and Monasse, Laurent},
  journal={Proceedings of the Royal Society A: Mathematical, Physical and Engineering Sciences},
  volume={473},
  number={2197},
  year={2017},
  publisher={The Royal Society}
}

@article{deng2020characterization,
  title={Characterization, stability, and application of domain walls in flexible mechanical metamaterials},
  author={Deng, Bolei and Yu, Siqin and Forte, Antonio E and Tournat, Vincent and Bertoldi, Katia},
  journal={Proceedings of the National Academy of Sciences},
  volume={117},
  number={49},
  pages={31002--31009},
  year={2020},
  publisher={National Academy of Sciences}
}

@article{coulais2018characteristic,
  title={A characteristic length scale causes anomalous size effects and boundary programmability in mechanical metamaterials},
  author={Coulais, Corentin and Kettenis, Chris and van Hecke, Martin},
  journal={Nature Physics},
  volume={14},
  number={1},
  pages={40--44},
  year={2018},
  publisher={Nature Publishing Group UK London}
}

@article{niu2025geometric,
  title={Geometric modeling of knitted fabrics},
  author={Niu, Lauren and Dion, Genevi{\`e}ve and Kamien, Randall D},
  journal={Proceedings of the National Academy of Sciences},
  volume={122},
  number={7},
  pages={e2416536122},
  year={2025},
  publisher={National Academy of Sciences}
}

@article {ScZ12,
    AUTHOR = {Scardia, Lucia and Zeppieri, Caterina Ida},
     TITLE = {Line-tension model for plasticity as the {$\Gamma$}-limit of a
              nonlinear dislocation energy},
   JOURNAL = {SIAM J. Math. Anal.},
  FJOURNAL = {SIAM Journal on Mathematical Analysis},
    VOLUME = {44},
      YEAR = {2012},
    NUMBER = {4},
     PAGES = {2372--2400},
      ISSN = {0036-1410,1095-7154},
   MRCLASS = {74C05 (49J45)},
  MRNUMBER = {3023380},
MRREVIEWER = {Jos\'e\ Carlos Pedro Cardoso Matias},
       DOI = {10.1137/110824851},
       URL = {https://doi.org/10.1137/110824851},
}

@article {CGM23,
    AUTHOR = {Conti, Sergio and Garroni, Adriana and Marziani, Roberta},
     TITLE = {Line-tension limits for line singularities and application to
              the mixed-growth case},
   JOURNAL = {Calc. Var. Partial Differential Equations},
  FJOURNAL = {Calculus of Variations and Partial Differential Equations},
    VOLUME = {62},
      YEAR = {2023},
    NUMBER = {8},
     PAGES = {Paper No. 228, 55},
      ISSN = {0944-2669,1432-0835},
   MRCLASS = {74Q05 (35Q74 49J45 49Q20)},
  MRNUMBER = {4642639},
MRREVIEWER = {Hao\ Dong},
       DOI = {10.1007/s00526-023-02552-0},
       URL = {https://doi.org/10.1007/s00526-023-02552-0},
}

@article {GLP10,
    AUTHOR = {Garroni, Adriana and Leoni, Giovanni and Ponsiglione,
              Marcello},
     TITLE = {Gradient theory for plasticity via homogenization of discrete
              dislocations},
   JOURNAL = {J. Eur. Math. Soc. (JEMS)},
  FJOURNAL = {Journal of the European Mathematical Society (JEMS)},
    VOLUME = {12},
      YEAR = {2010},
    NUMBER = {5},
     PAGES = {1231--1266},
      ISSN = {1435-9855,1435-9863},
   MRCLASS = {35Q74 (35B27 74C05 74E15 74G70 74Q05)},
  MRNUMBER = {2677615},
       DOI = {10.4171/JEMS/228},
       URL = {https://doi.org/10.4171/JEMS/228},
}

@article {MSZ14,
    AUTHOR = {M\"uller, Stefan and Scardia, Lucia and Zeppieri, Caterina
              Ida},
     TITLE = {Geometric rigidity for incompatible fields, and an application
              to strain-gradient plasticity},
   JOURNAL = {Indiana Univ. Math. J.},
  FJOURNAL = {Indiana University Mathematics Journal},
    VOLUME = {63},
      YEAR = {2014},
    NUMBER = {5},
     PAGES = {1365--1396},
      ISSN = {0022-2518,1943-5258},
   MRCLASS = {74C05 (49J45 74B20 74Q05)},
  MRNUMBER = {3283554},
MRREVIEWER = {Heng\ Xiao},
       DOI = {10.1512/iumj.2014.63.5330},
       URL = {https://doi.org/10.1512/iumj.2014.63.5330},
}

@article {CGO15,
    AUTHOR = {Conti, Sergio and Garroni, Adriana and Ortiz, Michael},
     TITLE = {The line-tension approximation as the dilute limit of
              linear-elastic dislocations},
   JOURNAL = {Arch. Ration. Mech. Anal.},
  FJOURNAL = {Archive for Rational Mechanics and Analysis},
    VOLUME = {218},
      YEAR = {2015},
    NUMBER = {2},
     PAGES = {699--755},
      ISSN = {0003-9527,1432-0673},
   MRCLASS = {74B05 (49J45 74N05)},
  MRNUMBER = {3375538},
MRREVIEWER = {Sebasti\'an\ Miguel\ Giusti},
       DOI = {10.1007/s00205-015-0869-7},
       URL = {https://doi.org/10.1007/s00205-015-0869-7},
}

@article {BrF26,
    AUTHOR = {Bresciani, Marco and Friedrich, Manuel},
     TITLE = {Core-{R}adius {A}pproximation of {S}ingular {M}inimizers in
              {N}onlinear {E}lasticity},
   JOURNAL = {Appl. Math. Optim.},
  FJOURNAL = {Applied Mathematics and Optimization},
    VOLUME = {93},
      YEAR = {2026},
    NUMBER = {1},
     PAGES = {Paper No. 21},
      ISSN = {0095-4616,1432-0606},
   MRCLASS = {49J45 (74A45 74B20)},
  MRNUMBER = {5013336},
       DOI = {10.1007/s00245-025-10376-x},
       URL = {https://doi.org/10.1007/s00245-025-10376-x},
}

@article {Hen09,
    AUTHOR = {Henao, Duvan},
     TITLE = {Cavitation, invertibility, and convergence of regularized
              minimizers in nonlinear elasticity},
   JOURNAL = {J. Elasticity},
  FJOURNAL = {Journal of Elasticity. The Physical and Mathematical Science
              of Solids},
    VOLUME = {94},
      YEAR = {2009},
    NUMBER = {1},
     PAGES = {55--68},
      ISSN = {0374-3535,1573-2681},
   MRCLASS = {74B20 (49J45 74G10 74G20 74G65)},
  MRNUMBER = {2471341},
MRREVIEWER = {Elvira\ Zappale},
       DOI = {10.1007/s10659-008-9184-y},
       URL = {https://doi.org/10.1007/s10659-008-9184-y},
}

@article {SST06,
    AUTHOR = {Sivaloganathan, Jeyabal and Spector, Scott J. and Tilakraj,
              Viveka},
     TITLE = {The convergence of regularized minimizers for cavitation
              problems in nonlinear elasticity},
   JOURNAL = {SIAM J. Appl. Math.},
  FJOURNAL = {SIAM Journal on Applied Mathematics},
    VOLUME = {66},
      YEAR = {2006},
    NUMBER = {3},
     PAGES = {736--757},
      ISSN = {0036-1399,1095-712X},
   MRCLASS = {74G65 (49K20 49M30 74B20)},
  MRNUMBER = {2216158},
MRREVIEWER = {Elvira\ Zappale},
       DOI = {10.1137/040618965},
       URL = {https://doi.org/10.1137/040618965},
}

@article {KuM25,
    AUTHOR = {Kupferman, Raz and Maor, Cy},
     TITLE = {Linearization in incompatible elasticity for general ambient
              spaces},
   JOURNAL = {SIAM J. Math. Anal.},
  FJOURNAL = {SIAM Journal on Mathematical Analysis},
    VOLUME = {57},
      YEAR = {2025},
    NUMBER = {5},
     PAGES = {5598--5627},
      ISSN = {0036-1410,1095-7154},
   MRCLASS = {53Z30 (49J45 74B20)},
  MRNUMBER = {4963460},
       DOI = {10.1137/24M1701198},
       URL = {https://doi.org/10.1137/24M1701198},
}

@article {NeR25,
    AUTHOR = {Neukamm, Stefan and Richter, Kai},
     TITLE = {Linearization and homogenization of nonlinear elasticity close
              to stress-free joints},
   JOURNAL = {Calc. Var. Partial Differential Equations},
  FJOURNAL = {Calculus of Variations and Partial Differential Equations},
    VOLUME = {64},
      YEAR = {2025},
    NUMBER = {6},
     PAGES = {Paper No. 183, 69},
      ISSN = {0944-2669,1432-0835},
   MRCLASS = {74B20 (35B27 49J45 74-10 74E30)},
  MRNUMBER = {4914531},
MRREVIEWER = {Shuji\ Yoshikawa},
       DOI = {10.1007/s00526-025-03018-1},
       URL = {https://doi.org/10.1007/s00526-025-03018-1},
}

@article {MPT19b,
    AUTHOR = {Maddalena, Francesco and Percivale, Danilo and Tomarelli,
              Franco},
     TITLE = {A new variational approach to linearization of traction
              problems in elasticity},
   JOURNAL = {J. Optim. Theory Appl.},
  FJOURNAL = {Journal of Optimization Theory and Applications},
    VOLUME = {182},
      YEAR = {2019},
    NUMBER = {1},
     PAGES = {383--403},
      ISSN = {0022-3239,1573-2878},
   MRCLASS = {49J45 (74B20)},
  MRNUMBER = {3961364},
MRREVIEWER = {Elvira\ Zappale},
       DOI = {10.1007/s10957-019-01533-8},
       URL = {https://doi.org/10.1007/s10957-019-01533-8},
}

@article {MPT19a,
    AUTHOR = {Maddalena, Francesco and Percivale, Danilo and Tomarelli,
              Franco},
     TITLE = {The gap between linear elasticity and the variational limit of
              finite elasticity in pure traction problems},
   JOURNAL = {Arch. Ration. Mech. Anal.},
  FJOURNAL = {Archive for Rational Mechanics and Analysis},
    VOLUME = {234},
      YEAR = {2019},
    NUMBER = {3},
     PAGES = {1091--1120},
      ISSN = {0003-9527,1432-0673},
   MRCLASS = {74B20 (49J45 74B05)},
  MRNUMBER = {4011693},
MRREVIEWER = {Giuliano\ Lazzaroni},
       DOI = {10.1007/s00205-019-01408-2},
       URL = {https://doi.org/10.1007/s00205-019-01408-2},
}

@article{flamant1892repartition,
  title={Sur la r{\'e}partition des pressions dans un solide rectangulaire charg{\'e} transversalement},
  author={Flamant, A},
  journal={CR Acad. Sci. Paris},
  volume={114},
  number={1892},
  pages={1465--1468},
  year={1892}
}

@book{ting1996anisotropic,
  title={Anisotropic elasticity: theory and applications},
  author={Ting, Thomas CT},
  volume={45},
  year={1996},
  publisher={Oxford university press}
}

@article {TS2004,
    AUTHOR = {Theotokoglou, E. E. and Stampouloglou, I. H.},
     TITLE = {The plane wedge problem loaded at its apex---the
              self-similarity property and the characteristic vector},
   JOURNAL = {J. Elasticity},
  FJOURNAL = {Journal of Elasticity. The Physical and Mathematical Science
              of Solids},
    VOLUME = {76},
      YEAR = {2004},
    NUMBER = {1},
     PAGES = {21--43 (2005)},
      ISSN = {0374-3535,1573-2681},
   MRCLASS = {74B05 (74E05 74E10)},
  MRNUMBER = {2131451},
       DOI = {10.1007/s10659-004-0044-0},
       URL = {https://doi.org/10.1007/s10659-004-0044-0},
}

@article{lazar2006note,
  title={A note on line forces in gradient elasticity},
  author={Lazar, Markus and Maugin, G{\'e}rard A},
  journal={Mechanics Research Communications},
  volume={33},
  number={5},
  pages={674--680},
  year={2006},
  publisher={Elsevier}
}

@article{unger2002similarity,
  title={Similarity solution of the flamant problem by means of a one-parameter group transformation},
  author={Unger, David J},
  journal={Journal of elasticity and the physical science of solids},
  volume={66},
  number={1},
  pages={93--97},
  year={2002},
  publisher={Springer}
}

@article{vasiliev2021flamant,
  title={On the Flamant problem for a half-plane loaded with a concentrated force},
  author={Vasiliev, VV and Lurie, SA and Salov, VA},
  journal={Acta Mechanica},
  volume={232},
  number={5},
  pages={1761--1771},
  year={2021},
  publisher={Springer}
}

@article{unger2025nonlinear,
  title={Nonlinear Flamant problem},
  author={Unger, David J},
  journal={Journal of Mechanics of Materials and Structures},
  volume={20},
  number={3},
  pages={353--362},
  year={2025},
  publisher={Mathematical Sciences Publishers}
}

@article {MaP22,
    AUTHOR = {Mainini, Edoardo and Percivale, Danilo},
     TITLE = {Linearization of elasticity models for incompressible
              materials},
   JOURNAL = {Z. Angew. Math. Phys.},
  FJOURNAL = {Zeitschrift f\"ur Angewandte Mathematik und Physik. ZAMP.
              Journal of Applied Mathematics and Physics. Journal de
              Math\'ematiques et de Physique Appliqu\'ees},
    VOLUME = {73},
      YEAR = {2022},
    NUMBER = {4},
     PAGES = {Paper No. 132, 33},
      ISSN = {0044-2275,1420-9039},
   MRCLASS = {74B15 (49J10 49J45 74B20)},
  MRNUMBER = {4444510},
       DOI = {10.1007/s00033-022-01768-y},
       URL = {https://doi.org/10.1007/s00033-022-01768-y},
}

@misc{AlLaPaWo26,
  author = { Alicandro, Roberto and Lazzaroni, Giuliano and Palombaro, Mariapia and Wozniak, Piotr },
  title = { Derivation of linear elasticity from energy functionals with infinitely many wells },
  journal = {  },
  year = { 2026 },
  pages = {  },
  URL = { http://cvgmt.sns.it/paper/7577/ },
  note = { cvgmt preprint}
}

@book{Bar23,
  author       = {Barber, J. R.},
  title        = {Elasticity},
  edition      = {4th},
  series       = {Solid Mechanics and Its Applications},
  volume       = {172},
  publisher    = {Springer Cham},
  address      = {Cham},
  year         = {2023},
  isbn         = {978-3-031-15214-6},
  doi          = {10.1007/978-3-031-15214-6},
  pages        = {xx+637},
}

@book {DL76,
    AUTHOR = {Duvaut, G. and Lions, J.-L.},
     TITLE = {Inequalities in mechanics and physics},
    SERIES = {Grundlehren der Mathematischen Wissenschaften},
    VOLUME = {219},
      NOTE = {Translated from the French by C. W. John},
 PUBLISHER = {Springer-Verlag, Berlin-New York},
      YEAR = {1976},
     PAGES = {xvi+397},
      ISBN = {3-540-07327-2},
   MRCLASS = {69.00},
  MRNUMBER = {521262},
}

@article {JeS21,
    AUTHOR = {Jesenko, Martin and Schmidt, Bernd},
     TITLE = {Geometric linearization of theories for incompressible elastic
              materials and applications},
   JOURNAL = {Math. Models Methods Appl. Sci.},
  FJOURNAL = {Mathematical Models and Methods in Applied Sciences},
    VOLUME = {31},
      YEAR = {2021},
    NUMBER = {4},
     PAGES = {829--860},
      ISSN = {0218-2025,1793-6314},
   MRCLASS = {74B20 (49J45 70G75)},
  MRNUMBER = {4265063},
MRREVIEWER = {Isabelle\ Gruais},
       DOI = {10.1142/S0218202521500202},
       URL = {https://doi.org/10.1142/S0218202521500202},
}

@article{FSSt25,
  title={Linearization of quasistatic fracture evolution in brittle materials},
  author={Friedrich, Manuel and Steinke, Pascal and Stinson, Kerrek},
  journal={Annales de l'Institut Henri Poincar{\'e} C},
  year={2025}
}

@article {MP20,
    AUTHOR = {Mainini, Edoardo and Percivale, Danilo},
     TITLE = {Variational linearization of pure traction problems in
              incompressible elasticity},
   JOURNAL = {Z. Angew. Math. Phys.},
  FJOURNAL = {Zeitschrift f\"{u}r Angewandte Mathematik und Physik. ZAMP.
              Journal of Applied Mathematics and Physics. Journal de
              Math\'{e}matiques et de Physique Appliqu\'{e}es},
    VOLUME = {71},
      YEAR = {2020},
    NUMBER = {5},
     PAGES = {Paper No. 146, 26},
      ISSN = {0044-2275,1420-9039},
   MRCLASS = {49J45 (35J50 35Q74 74K30 74K35 74R10)},
  MRNUMBER = {4141605},
MRREVIEWER = {Elvira\ Zappale},
       DOI = {10.1007/s00033-020-01377-7},
       URL = {https://doi.org/10.1007/s00033-020-01377-7},
}

@article {MP21,
    AUTHOR = {Mainini, Edoardo and Percivale, Danilo},
     TITLE = {Sharp conditions for the linearization of finite elasticity},
   JOURNAL = {Calc. Var. Partial Differential Equations},
  FJOURNAL = {Calculus of Variations and Partial Differential Equations},
    VOLUME = {60},
      YEAR = {2021},
    NUMBER = {5},
     PAGES = {Paper No. 164, 31},
      ISSN = {0944-2669,1432-0835},
   MRCLASS = {49J45 (74B20 74G65)},
  MRNUMBER = {4290377},
       DOI = {10.1007/s00526-021-02037-y},
       URL = {https://doi.org/10.1007/s00526-021-02037-y},
}

@article {GN11,
    AUTHOR = {Gloria, Antoine and Neukamm, Stefan},
     TITLE = {Commutability of homogenization and linearization at identity
              in finite elasticity and applications},
   JOURNAL = {Ann. Inst. H. Poincar\'{e} C Anal. Non Lin\'{e}aire},
  FJOURNAL = {Annales de l'Institut Henri Poincar\'{e} C. Analyse Non
              Lin\'{e}aire},
    VOLUME = {28},
      YEAR = {2011},
    NUMBER = {6},
     PAGES = {941--964},
      ISSN = {0294-1449,1873-1430},
   MRCLASS = {35B27 (35Q74 49J45 74E30 74Q05 74Q20)},
  MRNUMBER = {2859933},
MRREVIEWER = {Isabelle\ Gruais},
       DOI = {10.1016/j.anihpc.2011.07.002},
       URL = {https://doi.org/10.1016/j.anihpc.2011.07.002},
}

@article {MN11,
    AUTHOR = {M\"{u}ller, Stefan and Neukamm, Stefan},
     TITLE = {On the commutability of homogenization and linearization in
              finite elasticity},
   JOURNAL = {Arch. Ration. Mech. Anal.},
  FJOURNAL = {Archive for Rational Mechanics and Analysis},
    VOLUME = {201},
      YEAR = {2011},
    NUMBER = {2},
     PAGES = {465--500},
      ISSN = {0003-9527,1432-0673},
   MRCLASS = {35B27 (49J45 74B20 74Q05)},
  MRNUMBER = {2820355},
MRREVIEWER = {Carlos\ V\'{a}zquez Cend\'{o}n},
       DOI = {10.1007/s00205-011-0438-7},
       URL = {https://doi.org/10.1007/s00205-011-0438-7},
}

@article {D14,
    AUTHOR = {Davoli, Elisa},
     TITLE = {Linearized plastic plate models as {$\Gamma$}-limits of 3{D}
              finite elastoplasticity},
   JOURNAL = {ESAIM Control Optim. Calc. Var.},
  FJOURNAL = {ESAIM. Control, Optimisation and Calculus of Variations},
    VOLUME = {20},
      YEAR = {2014},
    NUMBER = {3},
     PAGES = {725--747},
      ISSN = {1292-8119,1262-3377},
   MRCLASS = {74C15 (49J45 74C10 74G65 74K20)},
  MRNUMBER = {3264221},
MRREVIEWER = {Cesare\ Davini},
       DOI = {10.1051/cocv/2013081},
       URL = {https://doi.org/10.1051/cocv/2013081},
}

@article {ADF23,
    AUTHOR = {Almi, Stefano and Davoli, Elisa and Friedrich, Manuel},
     TITLE = {Non-interpenetration conditions in the passage from nonlinear
              to linearized {G}riffith fracture},
   JOURNAL = {J. Math. Pures Appl. (9)},
  FJOURNAL = {Journal de Math\'{e}matiques Pures et Appliqu\'{e}es.
              Neuvi\`eme S\'{e}rie},
    VOLUME = {175},
      YEAR = {2023},
     PAGES = {1--36},
      ISSN = {0021-7824,1776-3371},
   MRCLASS = {74A30 (49J10 49J45 70G75 74B20 74G65 74R10)},
  MRNUMBER = {4598927},
       DOI = {10.1016/j.matpur.2023.05.001},
       URL = {https://doi.org/10.1016/j.matpur.2023.05.001},
}

@article {DF25,
    AUTHOR = {Davoli, Elisa and Friedrich, Manuel},
     TITLE = {Two-well linearization for solid-solid phase transitions},
   JOURNAL = {J. Eur. Math. Soc. (JEMS)},
  FJOURNAL = {Journal of the European Mathematical Society (JEMS)},
    VOLUME = {27},
      YEAR = {2025},
    NUMBER = {2},
     PAGES = {615--707},
      ISSN = {1435-9855,1435-9863},
   MRCLASS = {35Q74 (49J45 49Q20 74E99)},
  MRNUMBER = {4859576},
MRREVIEWER = {Xiaoguang\ Allan\ Zhong},
       DOI = {10.4171/jems/1385},
       URL = {https://doi.org/10.4171/jems/1385},
}

@article {F17,
    AUTHOR = {Friedrich, Manuel},
     TITLE = {A derivation of linearized {G}riffith energies from nonlinear
              models},
   JOURNAL = {Arch. Ration. Mech. Anal.},
  FJOURNAL = {Archive for Rational Mechanics and Analysis},
    VOLUME = {225},
      YEAR = {2017},
    NUMBER = {1},
     PAGES = {425--467},
      ISSN = {0003-9527,1432-0673},
   MRCLASS = {74B15 (49J10 74R10)},
  MRNUMBER = {3634030},
MRREVIEWER = {Alexander\ Michailovich\ Khludnev},
       DOI = {10.1007/s00205-017-1108-1},
       URL = {https://doi.org/10.1007/s00205-017-1108-1},
}

@article {MS13,
    AUTHOR = {Mielke, Alexander and Stefanelli, Ulisse},
     TITLE = {Linearized plasticity is the evolutionary {$\Gamma$}-limit of
              finite plasticity},
   JOURNAL = {J. Eur. Math. Soc. (JEMS)},
  FJOURNAL = {Journal of the European Mathematical Society (JEMS)},
    VOLUME = {15},
      YEAR = {2013},
    NUMBER = {3},
     PAGES = {923--948},
      ISSN = {1435-9855,1435-9863},
   MRCLASS = {74C05 (49J45)},
  MRNUMBER = {3085096},
MRREVIEWER = {Jos\'{e}\ Carlos Pedro Cardoso Matias},
       DOI = {10.4171/JEMS/381},
       URL = {https://doi.org/10.4171/JEMS/381},
}

@article {PT11,
    AUTHOR = {Paroni, Roberto and Tomassetti, Giuseppe},
     TITLE = {From non-linear elasticity to linear elasticity with initial
              stress via {$\Gamma$}-convergence},
   JOURNAL = {Contin. Mech. Thermodyn.},
  FJOURNAL = {Continuum Mechanics and Thermodynamics},
    VOLUME = {23},
      YEAR = {2011},
    NUMBER = {4},
     PAGES = {347--361},
      ISSN = {0935-1175,1432-0959},
   MRCLASS = {74B20 (49J45 74B10 74G65)},
  MRNUMBER = {2823756},
MRREVIEWER = {Stan\ Chiri\c{t}\u{a}},
       DOI = {10.1007/s00161-011-0184-y},
       URL = {https://doi.org/10.1007/s00161-011-0184-y},
}

@article {PT09,
    AUTHOR = {Paroni, Roberto and Tomassetti, Giuseppe},
     TITLE = {A variational justification of linear elasticity with residual
              stress},
   JOURNAL = {J. Elasticity},
  FJOURNAL = {Journal of Elasticity. The Physical and Mathematical Science
              of Solids},
    VOLUME = {97},
      YEAR = {2009},
    NUMBER = {2},
     PAGES = {189--206},
      ISSN = {0374-3535,1573-2681},
   MRCLASS = {74B20 (49S05 74B10)},
  MRNUMBER = {2563158},
MRREVIEWER = {Aida\ M.\ Timofte},
       DOI = {10.1007/s10659-009-9217-1},
       URL = {https://doi.org/10.1007/s10659-009-9217-1},
}

@article {S08,
    AUTHOR = {Schmidt, Bernd},
     TITLE = {Linear {$\Gamma$}-limits of multiwell energies in nonlinear
              elasticity theory},
   JOURNAL = {Contin. Mech. Thermodyn.},
  FJOURNAL = {Continuum Mechanics and Thermodynamics},
    VOLUME = {20},
      YEAR = {2008},
    NUMBER = {6},
     PAGES = {375--396},
      ISSN = {0935-1175,1432-0959},
   MRCLASS = {74B20 (49J45 74G65)},
  MRNUMBER = {2461715},
       DOI = {10.1007/s00161-008-0087-8},
       URL = {https://doi.org/10.1007/s00161-008-0087-8},
}

@article {CFLoPaM2025,
    AUTHOR = {Casado-D\'{\i}az, Juan and Francfort, Gilles A. and
              Lopez-Pamies, Oscar and Mora, Maria Giovanna},
     TITLE = {Liquid filled elastomers: from linearization to elastic
              enhancement},
   JOURNAL = {Arch. Ration. Mech. Anal.},
  FJOURNAL = {Archive for Rational Mechanics and Analysis},
    VOLUME = {249},
      YEAR = {2025},
    NUMBER = {1},
     PAGES = {Paper No. 5, 40},
      ISSN = {0003-9527,1432-0673},
   MRCLASS = {74F10 (49J45)},
  MRNUMBER = {4841735},
MRREVIEWER = {Merab\ Svanadze},
       DOI = {10.1007/s00205-024-02064-x},
       URL = {https://doi.org/10.1007/s00205-024-02064-x},
}

@article {KrM25,
    AUTHOR = {Kru\v{z}\'ik, Martin and Mainini, Edoardo},
     TITLE = {Linearization of finite elasticity with surface tension},
   JOURNAL = {J. Nonlinear Sci.},
  FJOURNAL = {Journal of Nonlinear Science},
    VOLUME = {35},
      YEAR = {2025},
    NUMBER = {3},
     PAGES = {Paper No. 63, 30},
      ISSN = {0938-8974,1432-1467},
   MRCLASS = {74G22 (49J45)},
  MRNUMBER = {4890987},
       DOI = {10.1007/s00332-025-10156-5},
       URL = {https://doi.org/10.1007/s00332-025-10156-5},
}

@article {BoS88,
    AUTHOR = {Boas, Harold P. and Straube, Emil J.},
     TITLE = {Integral inequalities of {H}ardy and {P}oincar\'e{} type},
   JOURNAL = {Proc. Amer. Math. Soc.},
  FJOURNAL = {Proceedings of the American Mathematical Society},
    VOLUME = {103},
      YEAR = {1988},
    NUMBER = {1},
     PAGES = {172--176},
      ISSN = {0002-9939,1088-6826},
   MRCLASS = {46E35 (26D10 35H05)},
  MRNUMBER = {938664},
MRREVIEWER = {Josef\ Vold\v rich},
       DOI = {10.2307/2047547},
       URL = {https://doi.org/10.2307/2047547},
}

@book {GiT01,
    AUTHOR = {Gilbarg, David and Trudinger, Neil S.},
     TITLE = {Elliptic partial differential equations of second order},
    SERIES = {Classics in Mathematics},
      NOTE = {Reprint of the 1998 edition},
 PUBLISHER = {Springer-Verlag, Berlin},
      YEAR = {2001},
     PAGES = {xiv+517},
      ISBN = {3-540-41160-7},
   MRCLASS = {35-02 (35Jxx)},
  MRNUMBER = {1814364},
}

@article {MaM21,
    AUTHOR = {Maor, Cy and Mora, Maria Giovanna},
     TITLE = {Reference configurations versus optimal rotations: a
              derivation of linear elasticity from finite elasticity for all
              traction forces},
   JOURNAL = {J. Nonlinear Sci.},
  FJOURNAL = {Journal of Nonlinear Science},
    VOLUME = {31},
      YEAR = {2021},
    NUMBER = {3},
     PAGES = {Paper No. 62, 28},
      ISSN = {0938-8974,1432-1467},
   MRCLASS = {74B20 (49J45 74B05)},
  MRNUMBER = {4261731},
MRREVIEWER = {Giuseppe\ Saccomandi},
       DOI = {10.1007/s00332-021-09716-2},
       URL = {https://doi.org/10.1007/s00332-021-09716-2},
}

@article {MR2023,
    AUTHOR = {Mora, Maria Giovanna and Riva, Filippo},
     TITLE = {Pressure live loads and the variational derivation of linear
              elasticity},
   JOURNAL = {Proc. Roy. Soc. Edinburgh Sect. A},
  FJOURNAL = {Proceedings of the Royal Society of Edinburgh. Section A.
              Mathematics},
    VOLUME = {153},
      YEAR = {2023},
    NUMBER = {6},
     PAGES = {1929--1964},
      ISSN = {0308-2105,1473-7124},
   MRCLASS = {74B20 (49J45)},
  MRNUMBER = {4666065},
MRREVIEWER = {Elvira\ Zappale},
       DOI = {10.1017/prm.2022.79},
       URL = {https://doi.org/10.1017/prm.2022.79},
}

@book {Leo23,
    AUTHOR = {Leoni, Giovanni},
     TITLE = {A first course in fractional {S}obolev spaces},
    SERIES = {Graduate Studies in Mathematics},
    VOLUME = {229},
 PUBLISHER = {American Mathematical Society, Providence, RI},
      YEAR = {[2023] \copyright 2023},
     PAGES = {xv+586},
      ISBN = {[9781470468989]; [9781470472535]; [9781470472528]},
   MRCLASS = {46-01 (30H05 35R11 42Bxx 42C40 46E35)},
  MRNUMBER = {4567945},
MRREVIEWER = {E.\ S.\ Dubtsov},
       DOI = {10.1090/gsm/229},
       URL = {https://doi.org/10.1090/gsm/229},
}

@book {Lew23,
    AUTHOR = {Lewicka, Marta},
     TITLE = {Calculus of {V}ariations on {T}hin {P}restressed {F}ilms--{A}symptotic
              {M}ethods in {E}lasticity},
    SERIES = {Progress in Nonlinear Differential Equations and their
              Applications},
    VOLUME = {101},
 PUBLISHER = {Birkh\"auser/Springer, Cham},
      YEAR = {[2023] \copyright 2023},
     PAGES = {viii+448},
      ISBN = {978-3-031-17494-0; 978-3-031-17495-7},
   MRCLASS = {74-02 (35Q74 49J10 49J45 74Bxx 74K35)},
  MRNUMBER = {4592573},
       DOI = {10.1007/978-3-031-17495-7},
       URL = {https://doi.org/10.1007/978-3-031-17495-7},
}

@article {CoS06,
    AUTHOR = {Conti, Sergio and Schweizer, Ben},
     TITLE = {Rigidity and gamma convergence for solid-solid phase
              transitions with {SO}(2) invariance},
   JOURNAL = {Comm. Pure Appl. Math.},
  FJOURNAL = {Communications on Pure and Applied Mathematics},
    VOLUME = {59},
      YEAR = {2006},
    NUMBER = {6},
     PAGES = {830--868},
      ISSN = {0010-3640,1097-0312},
   MRCLASS = {74G65 (49J45 74N15)},
  MRNUMBER = {2217607},
MRREVIEWER = {Giovanni\ Alberti},
       DOI = {10.1002/cpa.20115},
       URL = {https://doi.org/10.1002/cpa.20115},
}

@article {FJM02,
    AUTHOR = {Friesecke, Gero and James, Richard D. and M\"uller, Stefan},
     TITLE = {A theorem on geometric rigidity and the derivation of
              nonlinear plate theory from three-dimensional elasticity},
   JOURNAL = {Comm. Pure Appl. Math.},
  FJOURNAL = {Communications on Pure and Applied Mathematics},
    VOLUME = {55},
      YEAR = {2002},
    NUMBER = {11},
     PAGES = {1461--1506},
      ISSN = {0010-3640,1097-0312},
   MRCLASS = {74K20 (49J45 74B20 74G65)},
  MRNUMBER = {1916989},
MRREVIEWER = {Georg\ K.\ Dolzmann},
       DOI = {10.1002/cpa.10048},
       URL = {https://doi.org/10.1002/cpa.10048},
}

@article {CDM14,
    AUTHOR = {Conti, Sergio and Dolzmann, Georg and M\"{u}ller, Stefan},
     TITLE = {Korn's second inequality and geometric rigidity with mixed
              growth conditions},
   JOURNAL = {Calc. Var. Partial Differential Equations},
  FJOURNAL = {Calculus of Variations and Partial Differential Equations},
    VOLUME = {50},
      YEAR = {2014},
    NUMBER = {1-2},
     PAGES = {437--454},
      ISSN = {0944-2669,1432-0835},
   MRCLASS = {74B20 (35B45 35Q74)},
  MRNUMBER = {3194689},
MRREVIEWER = {Marin\ Marin},
       DOI = {10.1007/s00526-013-0641-5},
       URL = {https://doi.org/10.1007/s00526-013-0641-5},
}

@article {DNP02,
    AUTHOR = {Dal Maso, G. and Negri, M. and Percivale, D.},
     TITLE = {Linearized elasticity as {$\Gamma$}-limit of finite
              elasticity},
      NOTE* = {Calculus of variations, nonsmooth analysis and related topics},
   JOURNAL = {Set-Valued Anal.},
  FJOURNAL = {Set-Valued Analysis. An International Journal Devoted to the
              Theory of Multifunctions and its Applications},
    VOLUME = {10},
      YEAR = {2002},
    NUMBER = {2-3},
     PAGES = {165--183},
      ISSN = {0927-6947,1572-932X},
   MRCLASS = {49J45 (74B20)},
  MRNUMBER = {1926379},
MRREVIEWER = {Baisheng\ Yan},
       DOI = {10.1023/A:1016577431636},
       URL = {https://doi.org/10.1023/A:1016577431636},
}

@article {ADDe2012,
    AUTHOR = {Agostiniani, Virginia and Dal Maso, Gianni and DeSimone,
              Antonio},
     TITLE = {Linear elasticity obtained from finite elasticity by
              {$\Gamma$}-convergence under weak coerciveness conditions},
   JOURNAL = {Ann. Inst. H. Poincar\'{e} C Anal. Non Lin\'{e}aire},
  FJOURNAL = {Annales de l'Institut Henri Poincar\'{e} C. Analyse Non
              Lin\'{e}aire},
    VOLUME = {29},
      YEAR = {2012},
    NUMBER = {5},
     PAGES = {715--735},
      ISSN = {0294-1449,1873-1430},
   MRCLASS = {35B27 (35Q74 49J45 74B05)},
  MRNUMBER = {2971028},
MRREVIEWER = {Isabelle\ Gruais},
       DOI = {10.1016/j.anihpc.2012.04.001},
       URL = {https://doi.org/10.1016/j.anihpc.2012.04.001},
}

@book {bethuel1994ginzburg,
    AUTHOR = {Bethuel, Fabrice and Brezis, Ha\"im and H\'elein,
              Fr\'ed\'eric},
     TITLE = {Ginzburg-{L}andau vortices},
    SERIES = {Progress in Nonlinear Differential Equations and their
              Applications},
    VOLUME = {13},
 PUBLISHER = {Birkh\"auser Boston, Inc., Boston, MA},
      YEAR = {1994},
     PAGES = {xxviii+159},
      ISBN = {0-8176-3723-0},
   MRCLASS = {58E20 (35Q55 49-02 58E50 82D50)},
  MRNUMBER = {1269538},
       DOI = {10.1007/978-1-4612-0287-5},
       URL = {https://doi.org/10.1007/978-1-4612-0287-5},
}

@article{xu2024derivation,
  title={Derivation of an effective plate theory for parallelogram origami from bar and hinge elasticity},
  author={Xu, Hu and Tobasco, Ian and Plucinsky, Paul},
  journal={Journal of the Mechanics and Physics of Solids},
  volume={192},
  pages={105832},
  year={2024},
  publisher={Elsevier}
}

@book{bruce1992curves,
  title={Curves and Singularities: a geometrical introduction to singularity theory},
  author={Bruce, James William and Giblin, Peter J},
  year={1992},
  publisher={Cambridge university press}
}

@book{dal2012introduction,
  title={An introduction to $\Gamma$-convergence},
  author={Dal Maso, Gianni},
  volume={8},
  year={2012},
  publisher={Springer Science \& Business Media}
}

@article{xu2025modeling,
  title={Modeling and computation of the effective elastic behavior of parallelogram origami metamaterials},
  author={Xu, Hu and Marazzato, Frederic and Plucinsky, Paul},
  journal={Journal of the Mechanics and Physics of Solids},
 volume={204},
  pages={106295},
  year={2025}
}

@article{zheng2023modelling,
  title={Modelling planar kirigami metamaterials as generalized elastic continua},
  author={Zheng, Yue and Niloy, Imtiar and Tobasco, Ian and Celli, Paolo and Plucinsky, Paul},
  journal={Proceedings of the Royal Society A},
  volume={479},
  number={2272},
  pages={20220665},
  year={2023},
  publisher={The Royal Society}
}

@article{zheng2022continuum,
  title={Continuum field theory for the deformations of planar kirigami},
  author={Zheng, Yue and Niloy, Imtiar and Celli, Paolo and Tobasco, Ian and Plucinsky, Paul},
  journal={Physical review letters},
  volume={128},
  number={20},
  pages={208003},
  year={2022},
  publisher={APS}
}

@article{li2026effective,
  title={The effective energy of a lattice metamaterial},
  author={Li, Xuenan and Kohn, Robert V},
  journal={Journal of Elasticity},
  volume={158},
  number={1},
  pages={3},
  year={2026},
  publisher={Springer}
}

@article{li2025nonlinear,
  title={A nonlinear homogenization-based perspective on the soft modes and effective energies of some conformal metamaterials},
  author={Li, Xuenan and Kohn, Robert V},
  journal={arXiv preprint arXiv:2509.16907},
  year={2025}
}

@article{czajkowski2022conformal,
  title={Conformal elasticity of mechanism-based metamaterials},
  author={Czajkowski, Michael and Coulais, Corentin and van Hecke, Martin and Rocklin, D Zeb},
  journal={Nature communications},
  volume={13},
  number={1},
  pages={211},
  year={2022},
  publisher={Nature Publishing Group UK London}
}

@article{dull2024variational,
  title={A variational perspective on auxetic metamaterials of checkerboard-type},
  author={D{\"u}ll, Wolf-Patrick and Engl, Dominik and Kreisbeck, Carolin},
  journal={Archive for Rational Mechanics and Analysis},
  volume={248},
  number={3},
  pages={46},
  year={2024},
  publisher={Springer}
}

@misc{blog,
  author = {},
  title = {Plane Elasticity Problems},
  howpublished = {iMechanica},
  year = {2006},
  month = {October 19},
  note = {https://imechanica.org/node/319}
}

@article{berlyand1995effective,
  title={Effective elastic moduli of a soft medium with hard polygonal inclusions and extremal behavior of effective Poisson's ratio},
  author={Berlyand, L and Promislow, K},
  journal={Journal of elasticity},
  volume={40},
  number={1},
  pages={45--73},
  year={1995},
  publisher={Springer}
}

@article{berlyand1992asymptotics,
  title={Asymptotics of the homogenized moduli for the elastic chess-board composite},
  author={Berlyand, LV and Kozlov, SM},
  journal={Archive for rational mechanics and analysis},
  volume={118},
  number={2},
  pages={95--112},
  year={1992},
  publisher={Springer}
}

@book{sandier2007vortices,
    AUTHOR = {Sandier, Etienne and Serfaty, Sylvia},
     TITLE = {Vortices in the magnetic {G}inzburg-{L}andau model},
    SERIES = {Progress in Nonlinear Differential Equations and their
              Applications},
    VOLUME = {70},
 PUBLISHER = {Birkh\"auser Boston, Inc., Boston, MA},
      YEAR = {2007},
     PAGES = {xii+322},
      ISBN = {978-0-8176-4316-4; 0-8176-4316-8},
   MRCLASS = {82D55 (35B25 35B27 35J60 58E50 82-02)},
  MRNUMBER = {2279839},
MRREVIEWER = {Stanley\ A.\ Alama},
}

@book {Gra14,
    AUTHOR = {Grafakos, Loukas},
     TITLE = {Classical {F}ourier analysis},
    SERIES = {Graduate Texts in Mathematics},
    VOLUME = {249},
   EDITION = {Third},
 PUBLISHER = {Springer, New York},
      YEAR = {2014},
     PAGES = {xviii+638},
      ISBN = {978-1-4939-1193-6; 978-1-4939-1194-3},
   MRCLASS = {42-01 (42Bxx)},
  MRNUMBER = {3243734},
MRREVIEWER = {Atanas\ G.\ Stefanov},
       DOI = {10.1007/978-1-4939-1194-3},
       URL = {https://doi.org/10.1007/978-1-4939-1194-3},
}

\end{document}